\documentclass[11pt, twoside]{article}

\usepackage{amssymb}
\usepackage{amsfonts}
\usepackage{amsmath}
\usepackage{anysize}
\usepackage{amsthm}
\usepackage{color}
\usepackage{mathrsfs}
\usepackage{txfonts}
\usepackage{bbm}
\usepackage{latexsym}

\usepackage{indentfirst}

\usepackage[colorlinks=true,
  linkcolor=red,
  citecolor=blue,
  urlcolor=magenta]{hyperref}

\allowdisplaybreaks

\numberwithin{equation}{section}

\newtheorem{theorem}{Theorem}[section]
\newtheorem{lemma}[theorem]{Lemma}

\newtheorem{proposition}[theorem]{Proposition}

\theoremstyle{definition}
\newtheorem{remark}[theorem]{Remark}
\newtheorem{definition}[theorem]{Definition}

\renewcommand{\appendix}{\par
   \setcounter{section}{0}%
   \setcounter{subsection}{0}%
   \setcounter{subsubsection}{0}%
   \gdef\thesection{\@Alph\c@section}%
   \gdef\thesubsection{\@Alph\c@section.\@arabic\c@subsection}%
   \gdef\theHsection{\@Alph\c@section.}%
   \gdef\theHsubsection{\@Alph\c@section.\@arabic\c@subsection}%
   \csname appendixmore\endcsname
 }

\begin{document}
\title{\bf\Large Product Hardy Spaces
on Spaces of Homogeneous Type: Discrete Product Calder\'on-Type Reproducing
Formula, Atomic Characterization, and Product
Calder\'on--Zygmund Operators
\footnotetext{\hspace{-0.35cm} 2010 {\it Mathematics
Subject Classification}. Primary 42B30;
Secondary 42B20, 42C40, 42B35, 46E36.\endgraf
{\it Key words and phrases.}
space of homogeneous type, Hardy space,
reproducing formula, wavelet,
atom,
Calder\'on--Zygmund operator.\endgraf
This project is supported by the National Natural
Science Foundation of China (Grant Nos. 12201061,
12431006, 12371093, and 12326308), the Beijing Natural Science Foundation
(Grant No. 1262011), and the Fundamental Research Funds for the
Central Universities (Grant Nos. 2253200028 and 2023ZCJH02).}}
\author{Ziyi He, Dachun Yang
and Taotao Zheng}
\date{}
\maketitle
	
\vspace{-0.8cm}

\begin{center}
\begin{minipage}{13cm}
{\small {\bf Abstract}\quad
Let $i\in\{1,2\}$ and $X_i$ be a space of homogeneous
type in the sense of Coifman and Weiss
with the upper dimension $\omega_i$. Also let $\eta_i$ be the smoothness index of the Auscher--Hyt\"onen
wavelet function $\psi^{k_i}_{\alpha_i}$ on $X_i$.
In this article, for any $p\in(\max\{\frac{\omega_1}{\omega_1+\eta_1},\frac{\omega_2}{\omega_2+\eta_2}\}, 1]$,
by regarding the product Carleson measure space
$\mathrm{CMO}^p_{L^2}(X_1\times X_2)$ as the test function space
and its dual space $(\mathrm{CMO}^p_{L^2}(X_1\times X_2))'$
as the corresponding distribution space,
we introduce the product Hardy space $H^p(X_1\times X_2)$ in terms of wavelet
coefficients. Moreover, we establish an atomic characterization of
this product Hardy space and,
as an application, obtain a criterion for the boundedness
of linear operators from product Hardy spaces to corresponding Lebesgue spaces.
To escape the wavelet reproducing formula,
which is not useful for this atomic characterization
because the wavelets have no bounded support,
we establish a new discrete product Calder\'on-type reproducing
formula, which holds in the product Hardy space and has bounded support.
This reproducing formula also leads to the boundedness of product
Calder\'on--Zygmund operators on the product Hardy space.
}
\end{minipage}
\end{center}

\tableofcontents

\arraycolsep=1pt

\section{Introduction}\label{s-intro}

This article is based on \cite{HYY24}, in which the product
Littlewood--Paley theory and the product
Calder\'on reproducing formula in the setting of spaces of
homogeneous type were developed. The main purpose of this
article is to establish an atomic characterization
of product Hardy spaces on spaces of homogeneous type and,
as applications, obtain the boundedness of product
Calder\'on--Zygmund operators on these product Hardy spaces.

Recall that the classical Hardy space $H^p(\mathbb{R}^n)$
was pioneered by Stein and Weiss \cite{sw60}
and later significantly advanced by Fefferman and Stein \cite{FS72}.
The real-variable theory of Hardy spaces plays an important
role in harmonic analysis.
It is well known that the Hardy space $H^p(\mathbb{R}^n)$
has proved a suitable substitute of the Lebesgue space
$L^p(\mathbb{R}^n)$ with $p\in(0,1]$ in the study of the boundedness
of Calder\'on--Zygmund operators.
The classical Hardy spaces $H^p(\mathbb{R}^n)$ can be equivalently
defined, respectively, in terms of various maximal functions,
square functions, and atomic decompositions.
By means of the characterization of the grand maximal function,
the atomic characterization was obtained by Coifman
\cite{c74} for $n=1$  and by Latter \cite{l78} for $n \ge 2$.

As a generalization of the Euclidean space $\mathbb{R}^n$,
the concept of a space of homogeneous type
$(X, d, \mu)$ was introduced by Coifman and Weiss \cite{cw71,cw77},
which provides a natural general setting
for the theory of singular integrals and function spaces.
We begin with recalling the definitions of spaces of homogeneous type.
A \emph{quasi-metric space} $(X,d)$ is a non-empty set $X$ equipped with a \emph{quasi-metric}
$d:\ X \times X \rightarrow [0, \infty)$ satisfying that, for any $x,y,z\in X$,
\begin{enumerate}
\item[(i)] $d(x,y)=0$ if and only if $x=y$,
\item[(ii)] $d(x,y)=d(y,x)$,
\item[(iii)] there exists a  constant $A_0\in[1,\infty)$, independent of $x$, $y$, and $z$,
such that
$$d(x,z)\le A_0[d(x,y)+d(y,z)].$$
\end{enumerate}
The \emph{quasi-metric ball $B$}, with center $x\in X$ and radius $r\in(0,\infty)$,
is defined by setting
$$
B:=B(x,r):=\{y \in X:\ d(x, y)<r \}.
$$
For any ball $B$ and any $\tau\in(0,\infty)$, denote by $\tau B$ the ball
with the same center as $B$
but of radius $\tau$ times that of $B$.
Given a quasi-metric space $(X,d)$ and a non-negative measure $\mu$, we call $(X,d,\mu)$ a \emph{space of homogeneous type in the sense of Coifman and Weiss} if $\mu$ satisfies the {\it doubling condition}:
there exists a positive constant $C_{(\mu)}\in[1,\infty)$ such that,
for any ball $B\subset X$,
$$\mu(2B)\le C_{(\mu)}\mu(B).$$
This doubling condition further implies that, for any ball $B$
and $\lambda\in[1,\infty)$,
\begin{equation}\label{eq-doub}
\mu(\lambda B)\le C_{(\mu)}\lambda^\omega\mu(B),
\end{equation}
where $\omega:=\log_2{C_{(\mu)}}$ is called the \emph{upper dimension} of $X$.

In \cite{cw77}, Coifman and Weiss introduced the atomic
Hardy space $H^p_{\rm cw}(X)$, with $p\in(0,1]$,
by means of the bounded linear functionals on the Campanato space.
From then on, the real-variable theory of
atomic Hardy spaces on spaces of homogeneous type attracted
a lot of attention.
For example, Mac\'{i}as and Segovia \cite{MS79b,MS79}
established the maximal function characterization for
the Hardy space $H^p_{\rm cw}(X)$, with $p\in(0,1]$ but near to $1$, on Ahlfors-$1$ regular metric spaces.
To develop a complete real-variable theory of  atomic
Hardy spaces, Han et al. \cite{HMY06} introduced the concept of RD-spaces,
which are spaces of homogeneous type with the measure $\mu$
satisfying the additional \emph{reverse doubling condition}.
Later, in \cite{HMY08}, Han et al.\ established the
Calder\'on reproducing formulae on RD-spaces.
Based on these Calder\'on reproducing formulae, Han et al. \cite{HMY06}
characterized the atomic Hardy spaces with various Littlewood--Paley
functions and demonstrated the boundedness of Calder\'on--Zygmund
operators on these Hardy spaces.
Since then, the real-variable theory of various function spaces on RD-spaces has
developed rapidly; see, for example,
\cite{HMY06, HMY08, ylk17, SHYY17, yz08, YZ10,YZ11,  ZXT24, ZSY16}.

Recently, Hyt\"onen and Kairema \cite{hk12} established the dyadic
system on a space of homogeneous type $X$.
Later, Auscher and Hyt\"onen \cite{ah13,ah15} constructed the wavelet system on $X$,
which has the exponential decay and the
$\eta$-H\"older regularity for some $\eta\in(0,1)$. This result is
remarkable because no additional
assumptions for the quasi-metric $d$ or the measure $\mu$ is needed.
Using this wavelet system, on one hand, Han et al. \cite{hlw18} established
the wavelet reproducing formula, which holds in the sense of both
test functions and distributions.
Using this, Han et al. \cite{hhl16} introduced the
Hardy space $H^p_{\rm w}(X)$ by using the
\emph{wavelet square function}.
In \cite{hhl16}, they also proved that $H^p_{\rm w}(X)=H^p_{\rm cw}(X)$
when $p\in(0,1]$ but near $1$.
Later, Han et al. \cite{hhl17} introduced the Hardy space $H_{\rm a}^{p}(X)$ with $p\in(0,1]$
by using Carleson measure spaces
as test function spaces and their dual spaces as distribution spaces,
and obtained the equivalence
$$H_{\mathrm{cw}}^{p}(X) = H_{\rm w}^{p}(X)=H_{\rm a}^{p}(X)$$
when $p\in(0,1]$ but near $1$ in the sense of   equivalent (quasi-)norms.

On the other hand, He et al. \cite{hlyy19} introduced
a new kind of approximations of the
identity with exponential decay
(for short, exp-ATIs) on a space of homogeneous type $X$.
By using these exp-ATIs, He et al. \cite{hlyy19} also established
Calder\'on reproducing formulae on $X$.
Using these Calder\'on reproducing formulae, He et al. in \cite{hhllyy19} and \cite{hyy21}
developed a complete real-variable theory of (local) Hardy spaces.
Soon the real-variable theory of Hardy spaces and other function spaces
on spaces of homogeneous type has developed rapidly; see
\cite{awyy23,fmy20,hwyy21,hyy21,whhy21,yan24,yhyy22,yhyy23,zhy20} for more details.

We should also mention that, without having recourse to the Auscher--Hyt\"onen
wavelet system from \cite{ah13,ah15}, some function spaces and
their applications on spaces of homogeneous type have also been well developed in
Nakai \cite{n06}, Karak et al. \cite{k19,k20,kd23,kd24},
Krantz et al. \cite{dk21,k24,kl01-1,kl01-2}, Chang et al. \cite{cl19,cyy24,lcfy17,lcfy18,zcy17},
and Cao et al. \cite{cgr22,cgl21}; see also
Alvarado et al. \cite{agh20,am15,ayy22,ayy24}
for a sharp theory of function spaces on spaces of homogeneous type
and Bui et al. \cite{bdk20,bdl18,bdl20} for the real-variable theory of Hardy spaces
associated with operators.

We now turn to the product case. The product Hardy space $H^p(\mathbb{R}^n \times \mathbb{R}^m)$
with $p\in(0,1]$ was first introduced by Gundy and Stein \cite{gs79}.
Chang and  Fefferman \cite{CF80} applied the Calder\'on reproducing formula together with Lusin area
functions to obtain the atomic characterization of the product Hardy space
$H^1(\mathbb{R}^2_+ \times \mathbb{R}^2_+)$. Furthermore, by using the Lusin area function again,
Chang and Fefferman \cite{CF82} also introduced the product Hardy space
$H^p(\mathbb{R}^2_+ \times \mathbb{R}^2_+)$ for any given $p\in(0,\infty)$.
The corresponding atomic characterization of $H^p(\mathbb{R}^2_+ \times \mathbb{R}^2_+)$ with
$p \in (0, 1)$ was also obtained. Their applications to the boundedness of linear operators
were given, for example, in \cite{cyz10,clow24,hll10,hy05}.
For more results on the real-variable theory for product function spaces on Euclidean
spaces and domains, we refer to
\cite{bdl17,cg22,F87,gkp21,gkp-2312,ho08,ho17,lmv20-2,lmv20}.

In particular,  Nagal and Stein \cite{NS04} developed the Littlewood--Paley theory of Lebesgue spaces
in the setting of the product Carnot--Carathe\'odory spaces
$(X_1 \times X_2, d_1\times d_2, \mu_1 \times \mu_2)$,
which can be seen as a product of RD-spaces. On this product Carnot--Carathe\'odory spaces,
Han et al. \cite{hll16-Pisa}  applied the Journ\'e-type covering lemma to establish the atomic characterization
of  product Hardy spaces (see \cite[Definition 2.17]{hll13-Trans}).
We should mention that the key tool to establish the atomic
characterization of product Hardy
spaces in \cite{hll13-Trans} is the Calder\'on reproducing
formulae whose kernels have bounded support, and these reproducing formulae converge in the sense of both $L^2(X_1\times X_2)$
and product Hardy spaces.

From now on, we \emph{always} let $\{X_1,X_2\}$ be two spaces of homogeneous type and we work in $X_1\times X_2$ (the product space of homogeneous type) and,
unless necessary, we will not explicitly specify this underlying space in what follows.

Using the wavelet coefficients and the product distribution space,
Han et al. \cite[Definition 5.1]{hlw18}
introduced the product Hardy spaces $H^p_{\mathrm{hlw}}$ with $p\in(0,1]$
on product spaces of homogeneous type. To verify the reasonability of
this definition of product Hardy spaces and to obtain their dual spaces,
Han et al. \cite{hlw18} established the product wavelet reproducing formula,
which enables Han et al. to get rid of the dependence on the reverse
doubling assumption of measures and the regularity assumption
of quasi-metrics under consideration. More recently, Han et al. \cite{hlpw21} further
introduced a new kind of product atoms, which have
additional enlargement parameters instead of bounded supports
in \cite{hll16-Pisa}. Moreover, using these new product atoms
they characterized the product Hardy space $H_{\mathrm{hlw}}^p$,
which is impressive because the wavelet functions have no bounded support.
However, the boundedness of product Calder\'on--Zygmund operators
is still missing.

Let $i\in\{1,2\}$ and $\eta_i$ be the smoothness index of the
wavelet function $\psi^{k_i}_{\alpha_i}$ on $X_i$ constructed
by Auscher and Hyt\"onen in \cite{ah13}.
Let $\omega_i$ be the upper dimension of $X_i$. In this article, for any $p\in(\max\{\frac{\omega_1}{\omega_1+\eta_1},\frac{\omega_2}{\omega_2+\eta_2}\}, 1]$,
by regarding the product Carleson measure space
$\mathrm{CMO}^p_{L^2}$ as the test function space
and its dual space $(\mathrm{CMO}^p_{L^2})'$
as the corresponding distribution space,
we introduce the product Hardy space $H^p$ in terms of wavelet
coefficients, which is further proved to coincide with
$H^p_{\mathrm{hlw}}$ in \cite{hlw18}. Moreover,
we establish an atomic characterization of $H^p$ and,
as an application, obtain a criterion for the boundedness
of linear operators from product Hardy spaces to corresponding Lebesgue spaces.
Since the Auscher--Hyt\"onen wavelets have no bounded supports, the wavelet reproducing formula
is not useful for this atomic characterization.
To remedy this, by means of some ideas from  Han et al. \cite[Proposition 2.5]{hhl16},
we establish a new discrete product Calder\'on-type reproducing
formula, which holds on the product Hardy space and has bounded support.
The advantage of this new Calder\'on-type reproducing formula
exists in that a given function $f$ can be reproduced
by another function $g$ or $h$, both are not necessary to be equal
to $f$, but they have equivalent (quasi-)norms
(see Theorem \ref{thm-ctype}). This flexibility allows
that this new Calder\'on-type reproducing formula can be
established by using approximations of the identity with bounded support,
which in turn helps to derive the atomic characterization of product Hardy spaces.
Thus, this new reproducing formula is of independent interest,
which may be useful in solving other analysis problems
and indeed also leads to the boundedness of product
Calder\'on--Zygmund operators on the product Hardy space.
Interestingly enough, all results in this article hold on spaces
of homogeneous type with no additional assumptions.

The organization of the remainder of this article is as follows.

In Section \ref{s-pre},
we first recall some basic  concepts  on the dyadic cube system 
introduced by Hyt\"onen and Kairema
\cite{hk12}, the orthonormal wavelet basis constructed 
by Auscher and  Hyt\"{o}nen \cite{ah13}, and
the  product wavelet reproducing formula.

In Section \ref{s-hardy}, we construct the new pre-Hardy space $H^p_{L^2}$ and the pre-Carleson
measure space  $\mathrm{CMO}^p_{L^2}$ with $p\in(0,1]$,
which is used as the test function space and its dual is
used to introduce the product Hardy space $H^p$ (see Definition \ref{def-hcmo}).
The key step is how to establish the dual upper bounded estimate
(see Proposition \ref{prop-dual}). Besides, we also show that $H^p$ is independent
of the choice of the wavelet system (see Proposition \ref{prop-hpweq}).
Moreover,  we obtain the equivalence between $H^p$ and
$H^p_{\rm hlw}$ from \cite[Definition 5.1]{hlw18} in Proposition \ref{prop-h=h}.

In Section \ref{s-at}, we introduce a class of product $(p,q)$-atoms
for any $p\in(0,1]$ and $q\in(1,\infty)$ and, using them,
we introduce the product
atomic Hardy space $H^{p,q}_{\mathrm{at}}$. Moreover,
we establish an atomic characterization for the new
product Hardy space $H^p$ with $p\in(0,1]$ but near $1$,
that is, $H^p= H^{p,q}_{\mathrm{at}}$ with equivalent (quasi-)norms (see Theorem
\ref{thm-h=a}). Precisely,  on the one hand, the Journ\'e-type
covering lemma (see Lemma \ref{lem-j})
can lead to $ H^{p,q}_{\mathrm{at}} \subset H^p$. On the other hand, to prove
$ H^p \subset H^{p,q}_{\mathrm{at}}$ and overcome the absence of the
bounded support condition for the product wavelet
reproducing formula, we need to establish a new Calder\'on reproducing formula,
which holds in $L^2 \cap H^p $ and has the desired bound support
(see Theorem \ref{thm-ctype}). Applying this atomic characterization,
we establish a criterion that a linear operator is bounded
from product Hardy spaces to corresponding Lebesgue spaces
(see Theorem \ref{thm-bddhl}).
As an application of this criterion, we obtain the boundedness of an operator
belonging to Journ\'e's class from product Hardy spaces to the corresponding
Lebesgue spaces (see Theorem \ref{thm-bdj}).
Interestingly enough, the range $(\max\{\frac{\omega_1}{\omega_1+\eta_1},
\frac{\omega_2}{\omega_2+\eta_2}\},1]$ of $p$
is natural and optimal (see Remark \ref{r-prange}).

In Section \ref{s-czo}, we study the boundedness
of product Calder\'on--Zygmund operators on $H^p$ (see Theorem \ref{thm-czhp})
by using the aforementioned new Calder\'on reproducing formula obtained
in Theorem \ref{thm-ctype}.

We end this introduction by making some conventions on the notation.
For any $p\in[1,\infty]$, we use $p'$ to denote its \emph{conjugate index},
that is, $1/p+1/p'=1$.
The symbol $C$ denotes a positive constant which is independent
of the main parameters, but it may vary from
line to line.  The symbol $A \lesssim B$
means that there exists a positive constant $C$ such
that $A \le CB$. The symbol $A \sim B$ is used as an abbreviation of $A \lesssim B \lesssim A$.
For any $a, b\in\mathbb R$, denote the \emph{minimum} of $a$
and $b$ by $a\wedge b$ and the \emph{maximum} by $a\vee b$.
Also, for any set
$E$ of $X$, we use $\mathbf{1}_E$ to denote its
\emph{characteristic function} and $E^\complement$ the set
$X\setminus E$. Besides, $\varliminf$ denotes the \emph{limit inferior}.
Finally, in all proofs we consistently retain the notation
introduced in the original theorem (or related statement).
Moreover, in all proofs, we always write
$\sum_{k\in {\mathcal I}}\sum_{\alpha\in {\mathcal G}}$,
with index sets ${\mathcal I}$ and ${\mathcal G}$,
simply into $\sum_{{k\in {\mathcal I}},{\alpha\in {\mathcal G}}}$
and, only the first time it appears, we write out the full ranges of the
indices $k\in {\mathcal I}$ and $\alpha\in {\mathcal G}$;
subsequently, as long as no confusion arises, we abbreviate
it to $\sum_{k,\alpha}$.

\section{Some Basic Concepts of Dyadic Systems and Wavelet Functions}\label{s-pre}

Before we present the main results, we introduce the dyadic system
on spaces of homogeneous type.
The next lemma shows that the dyadic cubes on a given space of homogeneous type $X$, 
which comes from
\cite{hk12}.

\begin{lemma}\label{lem-cube}
Fix constants $0<c_0\le C_0<\infty$ and $\delta\in(0,1)$ such that $12A_0^3C_0\delta\le c_0$. 
Assume that
a set of points, $\{z_\alpha^k:\ k\in\mathbb Z,\ \alpha\in\mathcal A_k\}\subset X$ 
with $\mathcal A_k$ for any $k\in\mathbb Z$ being a
countable set of indices, has the following properties: for any $k\in\mathbb Z$,
\begin{enumerate}
\item $d(z_\alpha^k,z_\beta^k)\ge c_0\delta^k$ if $\alpha\neq\beta$;
\item $\min_{\alpha\in\mathcal A_k} d(x,z_\alpha^k)\le C_0\delta^k$ for any $x\in X$.
\end{enumerate}
Then there exists a family of sets $\{Q_\alpha^k:\  k\in\mathbb Z,\alpha\in\mathcal A_k\}$ satisfying
\begin{enumerate}
\item[\rm (a)] for any $k\in\mathbb Z$, $\bigcup_{\alpha\in\mathcal A_k} Q_\alpha^k=X$ and $\{ Q_\alpha^k:\ \alpha\in\mathcal A_k\}$
is mutually disjoint;
\item[\rm (b)] if $k,l\in\mathbb Z$ and $l\ge k$, then either $Q_\beta^l\subset Q_\alpha^k$ or
$Q_\beta^l\cap Q_\alpha^k=\emptyset$;
\item[\rm (c)] for any $k\in\mathbb Z$ and $\alpha\in\mathcal A_k$, $B(z_\alpha^k,c_\natural\delta^k)\subset Q_\alpha^k\subset
B(z_\alpha^k,C^\natural\delta^k)=:B(Q_\alpha^k)$, where $c_\natural:=(3A_0^2)^{-1}c_0$, $C^\natural:=2A_0C_0$, and
$z_\alpha^k$ is called ``the center'' of $Q_\alpha^k$.
\item[\rm (d)] if $Q_\beta^l\subset Q_\alpha^k$ for some $k,l\in\mathbb Z$ with $k\le l$, $\alpha\in\mathcal A_k$, and
$\beta\in\mathcal A_l$, then $B(Q_\beta^l)\subset B(Q_\alpha^k)$.
\end{enumerate}
\end{lemma}
Throughout this article, we \emph{assume} that $\delta$ is a very small positive number, for instance, $\delta\le(2A_0)^{-10}$.
In the following, for any $k\in\mathbb Z$, we define
\begin{equation}\label{eq-xy}
\mathcal X^k:=\{z_\alpha^k\}_{\alpha\in\mathcal A_k},\ \mathcal G_k:=\mathcal A_{k+1}\setminus\mathcal A_k,\ \textup{and}
\ \mathcal Y^k:=\{z_\alpha^{k+1}\}_{\alpha\in\mathcal G_k}=:\{y_\alpha^k\}_{\alpha\in\mathcal G_k}.
\end{equation}
Motivated by this, Auscher and Hyt\"{o}nen \cite{ah13} constructed 
the following \emph{wavelet} functions on spaces of
homogeneous type.
\begin{theorem}[{\cite[Theorem 7.1]{ah13}}]\label{thm-wave}
There exist constants $a\in(0,1]$, $\eta\in(0,1)$, $C, \nu\in(0,\infty)$, and wavelet functions
$\{\psi_\alpha^k:\ k\in\mathbb Z, \alpha\in\mathcal G_k\}$ satisfying that, for any $k\in\mathbb Z$ and $\alpha\in\mathcal G_k$,
\begin{enumerate}
\item \emph{(the decay condition)} for any $x\in X$,
$$
\left|\psi_\alpha^k(x)\right|\le
\frac C{\sqrt{V_{\delta^k}(y_\alpha^k)}}\exp\left\{-\nu\left[\frac{d(x,y_\alpha^k)}
{\delta^k}\right]^a\right\};
$$
\item \emph{(the H\"{o}lder-regularity condition)} for any $x,x'\in X$ with $d(x,x')\le\delta^k$,
$$
\left|\psi_\alpha^k(x)-\psi_\alpha^k(x')\right|\le\frac C{\sqrt{V_{\delta^k}(y_\alpha^k)}}\left[\frac{d(x,x')}{\delta^k}\right]^\eta
\exp\left\{-\nu\left[\frac{d(x,y_\alpha^k)}{\delta^k}\right]^a\right\};
$$
\item \emph{(the cancellation condition)} $\int_X \psi_\alpha^k(x)\,d\mu(x)=0$.
\end{enumerate}
Moreover, the wavelet  functions $\{\psi_\alpha^k:\ k\in\mathbb Z,  
\alpha\in\mathcal G_k\}$ form an orthonormal
basis of $L^2(X)$, that is, for any $f\in L^2(X)$,
\begin{equation}\label{eq-wrep}
f=\sum_{k\in\mathbb Z}\sum_{\alpha\in\mathcal G_k}\left<f,\psi_\alpha^k\right>\psi_\alpha^k
\end{equation}
in $L^2(X)$. Besides, $\{\psi_\alpha^k:\ k\in\mathbb Z,\alpha\in\mathcal G_k\}$ also form an unconditional basis of
$L^p(X)$ for any given $p\in(1,\infty)$.
\end{theorem}

Now we turn to the product setting. Suppose that $(X_1,d_1,\mu_1)$ [resp. $(X_2,d_2,\mu_2)$] is a space of
homogeneous type with upper dimension $\omega_1$ (resp. $\omega_2$) and wavelet smooth index $\eta_1$
(resp. $\eta_2$). Let $X:=X_1\times X_2$ and $\mu:=\mu_1\times\mu_2$. Let $f_1$ (resp. $f_2$) be a function on $X_1$
(resp. $X_2$). The function $f_1\otimes f_2$ is defined by setting, for any $(x_1,x_2)\in X_1\times X_2$,
$f_1\otimes f_2(x_1,x_2):=f_1(x_1)f_2(x_2)$. To simplify our symbols, we always assume
$(X_1,d_1,\mu_1)=(X_2,d_2,\mu_2)$, which implies that $\omega_1=\omega_2=:\omega$ and $\eta_1=\eta_2=:\eta$.
We still use subscripts $1$ and $2$ to distinguish the first and the second variables of elements in $X_1\times X_2$.
If $(X_1,d_1,\mu_1)\neq(X_2,d_2,\mu_2)$, then all the results in this article also hold  with
$\frac{\omega}{\omega+\eta}$  replaced by
$\max\{\frac{\omega_1}{\omega_1+\eta_1},\frac{\omega_2}{\omega_2+\eta_2}\}$.
For any $k_1,  k_2\in\mathbb Z$, $\alpha_1\in\mathcal G_{1,k_1}$, and $\alpha_2\in\mathcal G_{2,k_2}$,
let $Q_{\alpha_1,\alpha_2}^{k_1,k_2}:=Q_{1,\alpha_1}^{k_1}\times Q_{2,\alpha_{2}}^{k_{2}}$ and,
for any $(x_1, x_2) \in X_1 \times X_2$,
$$
\psi_{\alpha_1,\alpha_2}^{k_1,k_2}(x_1,x_2):=\psi_{1,\alpha_{1}}^{k_{1}}(x_1)
\psi_{2,\alpha_{2}}^{k_{2}}(x_2).
$$
By \eqref{eq-wrep}, we find that
$\{\psi_{\alpha_1,\alpha_2}^{k_1,k_2}:\ k_1,k_2\in\mathbb Z,\alpha_1\in\mathcal G_{1,k_1},\alpha_2\in\mathcal G_{2,k_2}\}$ form an orthonormal
base of $L^2$, that is, for any $f\in L^2$,
\begin{align}\label{eq-re}
f=\sum_{\genfrac{}{}{0pt}{}{k_1\in{\mathbb Z}}{\alpha_1\in\mathcal G_{1,k_1}}}
\sum_{\genfrac{}{}{0pt}{}{k_2\in{\mathbb Z}}{\alpha_2\in\mathcal G_{2,k_2}}}
\left<f,\psi_{\alpha_1,\alpha_2}^{k_1,k_2}\right>\psi_{\alpha_1,\alpha_2}^{k_1,k_2}
\end{align}
and
$$
\|f\|_{L^2}=\left(\sum_{\genfrac{}{}{0pt}{}{k_1\in{\mathbb Z}}{\alpha_1\in\mathcal G_{1,k_1}}}
\sum_{\genfrac{}{}{0pt}{}{k_2\in{\mathbb Z}}{\alpha_2\in\mathcal G_{2,k_2}}}
\left|\left<f,\psi_{\alpha_1,\alpha_2}^{k_1,k_2}\right>\right|^2\right)^{\frac{1}{2}}.
$$
We mention that a result similar to \eqref{eq-re} in product Euclidean spaces was obtained by Ho \cite{ho07}.
For any $x, y\in X$, $\gamma\in(0,\infty)$,
and $r\in(0,\infty)$, let
$$
V_r(x):=\mu(B(x,r))\textup{ and } V(x,y):=
\begin{cases}
\mu(B(x,d(x,y))) & \textup{if }x\neq y,\\
0 & \textup{if } x=y,
\end{cases}
$$
where $B(x,r):=\{y\in X:\ d(x,y)<r\}$.
We also assume
\begin{equation}\label{eq-defpg}
P_\gamma(x,y;r):=\frac{1}{V_r(x)+V(x,y)}\left[\frac{r}{r+d(x,y)}\right]^\gamma.
\end{equation}
Moreover, for any $k\in\mathbb Z$ and $x, y\in X$, let
\begin{gather}\label{eq-defek1}
\mathcal E_k^*(x,y):=\exp\left\{-\nu\left[\frac{d(x,y)}{\delta^k}\right]^a\right\},\
\mathcal E_k(x,y):=\mathcal E_k^*(x,y)\exp\left\{-\nu\left[\frac{d(x,\mathcal Y^k)}{\delta^k}\right]^a\right\},\\
\widetilde{\mathcal E}_k^*(x,y):=\frac{1}{V_{\delta^k}(x)}\mathcal E_k^*(x,y),
\textup{ and } \widetilde{\mathcal E}_k(x,y):=\frac{1}{V_{\delta^k}(x)}\mathcal E_k(x,y),\nonumber
\end{gather}
 where, for any $k\in\mathbb Z$, $\mathcal Y^k$ is as in \eqref{eq-xy} and $\nu\in(0,\infty)$  and $a\in(0,1]$
are two constants independent of $k$, $x$,
and $y$, but may vary from line to line. Moreover, for any $i\in\{1,2\}$, $\gamma\in(0,\infty)$, and $k\in\mathbb Z$,
we use $P_{i,\gamma}$, $\mathcal E_{i,k}^*$, $\mathcal E_{i,k}$, $\widetilde{\mathcal E}_{i,k}^*$,
and $\widetilde{\mathcal E}_{i,k}$ to denote, respectively, $P_{\gamma}$, $\mathcal E_{k}^*$, $\mathcal E_{k}$, $\widetilde{\mathcal E}_{k}^*$,
and $\widetilde{\mathcal E}_{k}$ associated with the space of homogeneous type $(X_i,d_i,\mu_i)$. Let $\mathcal B$ be a
Banach space. Denote by $\mathcal L(\mathcal B)$ the space of all bounded sublinear operators on $\mathcal B$, equipped with the
norm that, for any $T\in\mathcal L(\mathcal B)$, $\|T\|_{\mathcal L(\mathcal B)}:=\sup_{\|x\|_\mathcal B\le 1}\|Tx\|_{\mathcal B}$.
Throughout this article, we always assume $\mu(X)=\infty$, which is equivalent to
$\mathop{\rm diam}X=\infty$ (see \cite[Lemma 5.1]{ny97} or \cite[Lemma 8.1]{ah13})

\section{Product Hardy and Carleson Measure Spaces}\label{s-hardy}

In this section, we introduce the product Hardy space and the product Carleson measure  space, and give their basic properties.
 For any $f\in L^2$, we define the \emph{wavelet Littlewood--Paley function $S(f)$} by
setting, for any $(x_1,x_2)\in X_1\times X_2$,
$$
S(f)(x_1,x_2):=\left[\sum_{\genfrac{}{}{0pt}{}{k_1\in{\mathbb Z}}{\alpha_1\in\mathcal G_{1,k_1}}}
\sum_{\genfrac{}{}{0pt}{}{k_2\in{\mathbb Z}}{\alpha_2\in\mathcal G_{2,k_2}}}
\left|\left<f,\psi_{\alpha_1,\alpha_2}^{k_1,k_2}\right>\widetilde{\mathbf{1}}_{Q_{1,\alpha_1}^{k_1+1}}(x_1)
\widetilde{\mathbf{1}}_{Q_{2,\alpha_2}^{k_2+1}}(x_2)\right|^2\right]^{ \frac{1}{2}}.
$$
Here and thereafter, for any $i\in\{1,2\}$, $k_i\in\mathbb Z$ and $\alpha_i\in\mathcal G_{i,k_i}$,
$\widetilde{\mathbf 1}_{Q_{i,\alpha_i}^{k_i+1}}=\mathbf{1}_{Q_{i,\alpha_i}^{k_i+1}}/\sqrt{\mu_i(Q_{i,\alpha_i}^{k_i+1})}$.
For any given $p\in(0,1]$,  we define the \emph{pre-Hardy space $H^p_{L^2}(X_1\times X_2)$} by setting
\begin{equation*}
H^p_{L^2}:=\left\{f\in L^2:\ \|S(f)\|_{L^p}<\infty\right\}.
\end{equation*}
Moreover, for any $f\in L^2$, define
$$
\mathcal C_p(f):=\sup_{\Omega}
\left[\frac{1}{[\mu(\Omega)]^{\frac 2p-1}}
\sum_{Q_{\alpha_1,\alpha_2}^{k_1+1,k_2+1}\subset\Omega}
\left|\left<f,\psi_{\alpha_1,\alpha_2}^{k_1,k_2}\right>\right|^2\right]^{\frac{1}{2}},
$$
where the supremum is taken over all open sets $\Omega\subset X_1\times X_2$ with $\mu(\Omega)<\infty$. The
\emph{pre-Carleson measure space $\mathrm{CMO}^p_{L^2}$} defined by setting
\begin{equation*}
\mathrm{CMO}^p_{L^2}:=\left\{f\in L^2:\ \mathcal C_p(f)<\infty\right\}.
\end{equation*}
Also denote by $(H^p_{L^2})'$ [resp. $(\mathrm{CMO}^p_{L^2})'$] the \emph{dual space} of $H^p_{L^2}$
[resp. $\mathrm{CMO}^p_{L^2}$] equipped with the weak-$*$ topology.

For any measurable function $f$, define the \emph{strong maximal function $\mathcal M_{\rm str}(f)$} by setting, for any
$(x_1,x_2)\in X_1\times X_2$,
$$
\mathcal M_{\rm str}(f)(x_1,x_2):=\sup_{\genfrac{}{}{0pt}{}{B_1\ni x_1}{B_2\ni x_2}}
\frac{1}{\mu_1(B_1)\mu_2(B_2)}
\int_{B_1}\int_{B_2}|f(y_1,y_2)|\,d\mu_2(y_2)\,d\mu_1(y_1),
$$
where the supremum is taken over all balls $B_1\subset X_1$ containing $x_1$ and $B_2\subset X_2$ containing
$x_2$. By the Fubini theorem and the Fefferman--Stein vector-valued maximal inequalities respectively on $X_1$ and $X_2$, we
can obtain the Fefferman--Stein vector-valued maximal inequality for $\mathcal M_{\rm str}$; we omit the details here.

\begin{proposition}[{\cite[Theorem 2.8]{HYY24}}]\label{prop-fs}
Let $p\in(1,\infty)$ and $u\in(1,\infty]$. Then there exists a constant $C\in(0,\infty)$ such that, for any sequence
of measurable functions $\{f_n\}_{n=1}^\infty$ on $X_1\times X_2$,
$$
\left\|\left\{\sum_{n=1}^\infty [\mathcal M_{\rm str}(f_n)]^u\right\}^{\frac 1u}\right\|_{L^p}\le C
\left\|\left(\sum_{n=1}^\infty |f_n|^u\right)^{\frac 1u}\right\|_{L^p}.
$$
Consequently, $\mathcal M_{\rm str}$ is bounded on $L^p$ for any given $p\in(1,\infty)$, that is, there exists a
constant $C\in(0,\infty)$ such that, for any $f\in L^p$, $\|\mathcal M_{\rm str}(f)\|_{L^p}\le C\|f\|_{L^p}$.
\end{proposition}

The following proposition takes a very important role in this article, 
whose proof is similar to that of
\cite[Lemma 5.2]{HMY08} and we omit the details here.

\begin{proposition}\label{prop-estmax}
Let $\gamma\in(0,\infty)$ and $r\in(\frac{\omega}{\omega+\gamma},1]$ with $\omega$ as in \eqref{eq-doub}.
Then there exists a positive constant $C$ such that, for any $k_1,k_1',k_2,k_2'\in\mathbb Z$,
$\{a_{\alpha_1,\alpha_2}^{k_1,k_2}:\ \alpha_1\in\mathcal A_{1,k_1},\alpha_2\in\mathcal A_{2,k_2}\}\subset\mathbb C$,
$\{x_{\alpha_1,\alpha_2}^{k_1,k_2}\in Q_{\alpha_1,\alpha_2}^{k_1,k_2}:\ \alpha_1\in\mathcal A_{1,k_1},
\alpha_2\in\mathcal A_{2,k_2}\}$, and $(x_1,x_2)\in X_1\times X_2$,
\begin{align*}
&\sum_{\genfrac{}{}{0pt}{}{\alpha_1\in\mathcal A_{1,\alpha_1}}{\alpha_2\in\mathcal A_{2,\alpha_2}}}\mu\left(Q_{\alpha_1,\alpha_2}^{k_1,k_2}\right)
\left|a_{\alpha_1,\alpha_2}^{k_1,k_2}\right|P_{1,\gamma}\left(x_1,x_{1,\alpha_1}^{k_1};\delta^{k_1\wedge k_1'}\right)
P_{2,\gamma}\left(x_2,x_{2,\alpha_2}^{k_2};\delta^{k_2\wedge k_2'}\right)\\
&\quad\le C\delta^{[(k_1\wedge k_1')-k_1]\omega(\frac 1r-1)}\delta^{[(k_2\wedge k_2')-k_2]\omega(\frac 1r-1)}
\left[\mathcal M_{\rm str}\left(\sum_{\genfrac{}{}{0pt}{}{\alpha_1\in\mathcal A_{1,\alpha_1}}{\alpha_2\in\mathcal A_{2,\alpha_2}}}
\left|a_{\alpha_1,\alpha_2}^{k_1,k_2}\right|^r\mathbf 1_{Q_{\alpha_1,\alpha_2}^{k_1,k_2}}\right)(x_1,x_2)\right]^{\frac 1r},
\end{align*}
where, for any $i\in\{1,2\}$, $P_{i,\gamma}$ is as in \eqref{eq-defpg} associated with $X_i$.
\end{proposition}

In the next, we prove a fundamental proposition about $H^p_{L^2}$
and $\mathrm{CMO}^p_{L^2}$. Here and thereafter,
denote by $\mathcal D$ the set of all dyadic rectangles on $X_1\times X_2$, that is,
$$
\mathcal D:=\left\{Q_{1,\alpha_1}^{k_1}\times Q_{2,\alpha_2}^{k_2}:\ k_1, k_2\in\mathbb Z, \alpha_1\in\mathcal A_{1,k_1},
\alpha_2\in\mathcal A_{2,k_2}\right\}.
$$
Let
$$
\widetilde{\mathcal D}:=\left\{Q_{1,\alpha_1}^{k_1+1}\times Q_{2,\alpha_2}^{k_2+1}:\ k_1, k_2\in\mathbb Z,\alpha_1\in\mathcal G_{1,k_1},
\alpha_2\in\mathcal G_{2,k_2}\right\}.
$$

\begin{proposition}\label{prop-dual}
Let $p\in(0,1]$. Then there exists a positive constant $C$ such that, for any $f, g\in L^2$,
\begin{equation}\label{eq-dual}
|\langle f,g\rangle|\le C\|S(f)\|_{L^p}\mathcal C_p(g).
\end{equation}
\end{proposition}

\begin{proof}
Without loss of generality, we may assume that $f\in H^p_{L^2}$,
$g\in\mathrm{CMO}^p_{L^2}$, and $\|f\|_{H^p_{L^2}}>0$.
For any $l\in\mathbb Z$, let
\begin{align*}
\Omega_l&:=\left\{(x_1,x_2)\in X_1\times X_2:\ S(f)(x_1,x_2)>2^l\right\},\\
\Omega_{-\infty}&:=\{(x_1,x_2)\in X_1\times X_2:\ S(f)(x_1,x_2)=0\},\\
\mathcal R_l&:=\left\{R\in\widetilde{\mathcal D}:\ \mu(R\cap\Omega_l)>\frac 12\mu(R) \textup{ and } \mu(R\cap\Omega_{l+1})
\le\frac 12\mu(R)\right\},
\end{align*}
and
$$
\mathcal R_{-\infty}:=\left\{R\in\widetilde{\mathcal D}:\ \mu(R\cap\Omega_{-\infty})\ge\frac 12 \mu(R) \right\}.
$$
Indeed, for any $R\in\widetilde{\mathcal D}$, there exists a unique $l\in\mathbb Z\cup\{-\infty\}$ such that $R\in\mathcal R_l$. Indeed,
if $R\in\widetilde{\mathcal D}$ but $R\notin\mathcal R_{-\infty}$, then
$$
\mu(\{(x_1,x_2)\in R:\ S(f)(x_1,x_2)>0\})>\frac 12\mu(R).
$$
Then we find that there exists an $l_0\in\mathbb Z$ such that $\mu(R\cap\Omega_{l_0})>\frac{\mu(R)}2$. Let
\begin{equation*}
l:=\max\left\{l_0\in\mathbb Z:\ \mu\left(R\cap\Omega_{l_0}\right)>\frac 12\mu(R)\right\}.
\end{equation*}
We next show that such $l$ exists and $R\in\mathcal R_l$. If $l$ does not exist, then, for any $l_0\in\mathbb Z$,
$\mu(R\cap\Omega_{l_0})>\frac{\mu(R)}{2}$, which further implies that
\begin{align*}
\left\|S(f)\right\|_{L^p}^p\ge\left\|S(f)\mathbf 1_{R\cap\Omega_{l_0}}\right\|_{L^p}^p
\ge 2^{l_0p-1}\mu(R)\to\infty
\end{align*}
as $l_0\to\infty$. This contradicts  $S(f)\in L^p$. Thus, such $l$ exists. Moreover, by the definition
of $l$, we have $R\in\mathcal R_l$.
Moreover, for any $(x_1,x_2)\in Q_{1,\alpha_1}^{k_1+1}\times Q_{2,\alpha_2}^{k_2+1}\in\mathcal R_l$ with $l\in\mathbb Z$,
by Lemma \ref{lem-cube}  we have
\begin{align*}
\mathcal M_{\rm str} \left(\mathbf{1}_{\Omega_l}\right)(x_1,x_2)&\ge\frac{1}{\mu(B_1(Q_{1,\alpha_1}^{k_1+1})\times B_2(Q_{2,\alpha_{2}}^{k_{2}+1}))}
\int_{B_1(Q_{1,\alpha_1}^{k_1+1})\times B_2(Q_{2,\alpha_{2}}^{k_{2}+1})}\mathbf{1}_{\Omega_l}(y_1,y_2)\,d\mu(y_1,y_2)\\
&\gtrsim\frac{1}{\mu(Q_{1,\alpha_1}^{k_1+1}\times Q_{2,\alpha_2}^{k_2+1})}
\int_{Q_{\alpha_1}^{1,k_1+1}\times Q_{\alpha_2}^{2,k_2+1}}\mathbf{1}_{\Omega_l}(y_1,y_2)\,d\mu(y_1,y_2) \gtrsim 1,
\end{align*}
which implies that there exists $c_0\in(0,1)$, independent of $l$, $k_1$, $k_2$, $\alpha_1$, and $\alpha_2$, such that
$$
Q_{1,\alpha_1}^{k_1+1}\times Q_{2,\alpha_2}^{k_2+1}\subset\left\{(x_1,x_2):  \mathcal M_{\rm str}(\mathbf{1}_{\Omega_l})(x_1,x_2)>c_0\right\}
=:\widetilde\Omega_l.
$$
From the definition of $\mathcal M_{\rm str}$ and Proposition \ref{prop-fs}, we deduce that,
for any $l\in\mathbb Z$, $\widetilde\Omega_l$ is open and
\begin{equation}\label{eq-omegasim}
\mu\left(\widetilde\Omega_l\right)\sim\mu(\Omega_l).
\end{equation}
Using \eqref{eq-re} and   H\"{o}lder's inequality, we conclude that
\begin{align}\label{eq-dual1}
|\langle f,g\rangle|&=\left|\sum_{\genfrac{}{}{0pt}{}{k_1\in{\mathbb Z}}{\alpha_1\in\mathcal G_{1,k_1}}}
\sum_{\genfrac{}{}{0pt}{}{k_2\in{\mathbb Z}}{\alpha_2\in\mathcal G_{2,k_2}}}
\left<f,\psi_{\alpha_1,\alpha_2}^{k_1,k_2}\right>\left<g,
\psi_{\alpha_1,\alpha_2}^{k_1,k_2}\right>\right|\nonumber\\
&\le\sum_{Q_{1,\alpha_1}^{k_1+1}\times Q_{2,\alpha_2}^{k_2+1}\in \mathcal R_{-\infty}}
\left|\left<f,\psi_{\alpha_1,\alpha_2}^{k_1,k_2}\right>\left<g,
\psi_{\alpha_1,\alpha_2}^{k_1,k_2}\right>\right|+\sum_{l\in\mathbb Z}
\sum_{Q_{1,\alpha_1}^{k_1+1}\times Q_{2,\alpha_2}^{k_2+1}\in \mathcal R_l}\cdots\nonumber\\
&=:\mathrm I+\mathrm{II}.
\end{align}

Now, we deal with $\rm I$. Indeed, let $R:=Q_{1,\alpha_1}^{k_1+1}\times Q_{2,\alpha_2}^{k_2+1}\in\mathcal R_{-\infty}$.
By this, we find that there exists an $(x_1,x_2)\in R$ such that $S(f)(x_1,x_2)=0$. Thus, by the
definition of $f$, we conclude that $\langle f,\psi_{\alpha_1,\alpha_2}^{k_1,k_2}\rangle=0$, which further shows
that ${\rm I}=0$.

We now estimate the term $\rm II$. Indeed, by   H\"older's inequality, we have
\begin{align*}
{\rm II}&\le\sum_{l\in\mathbb Z}\left(\sum_{Q_{1,\alpha_1}^{k_1+1}\times Q_{2,\alpha_2}^{k_2+1}\in\mathcal R_l}
\left|\left<f,\psi_{\alpha_1,\alpha_2}^{k_1,k_2}\right>\right|^2\right)^{\frac{1}{2}}
\left(\sum_{Q_{1,\alpha_1}^{k_1+1}\times Q_{2,\alpha_2}^{k_2+1}\in\mathcal R_l}
\left|\left<g,\psi_{\alpha_1,\alpha_2}^{k_1,k_2}\right>\right|^2\right)^{\frac{1}{2}}\nonumber\\
&\le\left[\sum_{l\in\mathbb Z}\left(\sum_{Q_{1,\alpha_1}^{k_1+1}\times Q_{2,\alpha_2}^{k_2+1}\in\mathcal R_l}
\left|\left<f,\psi_{\alpha_1,\alpha_2}^{k_1,k_2}\right>\right|^2\right)^{\frac{p}{2}}
\left(\sum_{Q_{1,\alpha_1}^{k_1+1}\times Q_{2,\alpha_2}^{k_2+1}\in\mathcal R_l}
\left|\left<g,\psi_{\alpha_1,\alpha_2}^{k_1,k_2}\right>\right|^2\right)^{\frac{p}{2}}
\right]^{\frac{1}{p}}\nonumber\\
&=:\left[\sum_{l\in\mathbb Z}\mathrm Y_{1,l}\mathrm Y_{2,l}\right]^{\frac{1}{p}}.\nonumber
\end{align*}

For the term $\mathrm Y_{2,l}$, using the definition of $\mathcal R_l$ and \eqref{eq-omegasim}, we obtain, for any $l\in\mathbb Z$,
\begin{equation}\label{eq-esty2l}
\mathrm Y_{2,l}\le\left(\sum_{Q_{1,\alpha_1}^{k_1+1}\times Q_{2,\alpha_2}^{k_2+1}\subset\widetilde\Omega_l}
\left|\left<g,\psi_{\alpha_1,\alpha_2}^{k_1,k_2}\right>\right|^2\right)^{\frac{p}{2}}
\le\mathcal C_p(g)^p\mu\left(\widetilde\Omega_l\right)^{1-\frac{p}{2}}\sim\mathcal C_p(g)^p\mu(\Omega_l)^{1-\frac{p}{2}},
\end{equation}
which is the desired estimate for $\mathrm Y_{2,l}$.

For $\mathrm Y_{1,l}$, we claim that, for any $l\in\mathbb Z$,
\begin{equation}\label{eq-cl}
\sum_{Q_{1,\alpha_1}^{k_1+1}\times Q_{2,\alpha_2}^{k_2+1}\in\mathcal R_l}
\left|\left<f,\psi_{\alpha_1,\alpha_2}^{k_1,k_2}\right>\right|^2\lesssim 2^{2l}\mu(\Omega_l).
\end{equation}
Indeed, on one hand, by the definition of $\Omega_{l+1}$ and \eqref{eq-omegasim}, we find that
$$
\int_{\widetilde\Omega_l\setminus\Omega_{l+1}}[S(f)(x_1,x_2)]^2\,d\mu(x_1,x_2)\lesssim 2^{2l}\mu\left(\widetilde\Omega_l\right)
\sim 2^{2l}\mu(\Omega_l).
$$
On the other hand, from the definitions of $S(f)$ and $\widetilde\Omega_{l}$, we also infer that
\begin{align*}
&\int_{\widetilde\Omega_l\setminus\Omega_{l+1}}[S(f)(x_1,x_2)]^2\,d\mu(x_1,x_2)\\
&\quad\ge\int_{\widetilde\Omega_l\setminus\Omega_{l+1}}\sum_{Q_{\alpha_1,\alpha_2}^{k_1+1,k_2+1}
\in\mathcal R_l}
\left|\left<f,\psi_{\alpha_1,\alpha_2}^{k_1,k_2}\right>\right|^2
\left[\mu\left(Q_{\alpha_1,\alpha_2}^{k_1+1,k_2+1}\right)\right]^{-1}
\mathbf{1}_{Q_{\alpha_1,\alpha_2}^{k_1+1,k_2+1}}(x_1,x_2)\,d\mu(x_1,x_2)\\
&\quad=\sum_{Q_{\alpha_1,\alpha_2}^{k_1+1,k_2+1}\in\mathcal R_l}\left|\left<f,\psi_{\alpha_1,\alpha_2}^{k_1,k_2}\right>\right|^2
\left[\mu\left(Q_{\alpha_1,\alpha_2}^{k_1+1,k_2+1}\right)\right]^{-1}
\mu\left(\widetilde\Omega_l\cap\Omega_{l+1}^\complement\cap Q_{\alpha_1,\alpha_2}^{k_1+1,k_2+1}\right)\\
&\quad\ge\frac 12\sum_{Q_{\alpha_1,\alpha_2}^{k_1+1,k_2+1}\in\mathcal R_l}\left|\left<f,\psi_{\alpha_1,\alpha_2}^{k_1,k_2}\right>\right|^2.
\end{align*}
The above two inequalities show \eqref{eq-cl}. Thus, using \eqref{eq-dual1}, \eqref{eq-esty2l}, and \eqref{eq-cl},
we conclude that
\begin{align*}
|\langle f,g\rangle|
&\lesssim\mathcal C_p(g)\left\{\sum_{k\in\mathbb Z} 2^{pl}[\mu(\Omega_l)]^{\frac{2}{p}}[\mu(\Omega_l)]^{1-\frac{2}{p}}\right\}
^{\frac{1}{p}}
\sim\mathcal C_p(g)\left[\sum_{k\in\mathbb Z} 2^{pl}\mu(\Omega_l)\right]^{\frac{1}{p}}\\
&\sim\|S(f)\|_{L^p}\mathcal C_p(f),
\end{align*}
which completes the proof of \eqref{eq-dual} and hence   Proposition \ref{prop-dual}.
\end{proof}

Obviously, by the definitions of $H^p_{L^2}$ and ${\rm CMO}^p_{L^2}$, we find that, for any
$k_1,k_2\in\mathbb Z$, $\alpha_1\in\mathcal G_{1,k_1}$, $\alpha_2\in\mathcal G_{2,k_2}$, and $p\in(0,1]$,
$\psi_{\alpha_1,\alpha_2}^{k_1,k_2}$ belongs to both $H^p_{L^2}$
and $\mathrm{CMO}^p_{L^2}$. Therefore, for any $f\in(\mathrm{CMO}^p_{L^2})'$ or $f\in(H^p_{L^2})'$,
the symbol $\langle f,\psi_{\alpha_1,\alpha_2}^{k_1,k_2}\rangle$ can be properly understood within the framework of duality.
Consequently, for any $f\in (\mathrm{CMO}^p_{L^2})'$
[resp. $f\in(H^p_{L^2})'$], $S(f)$ [resp. $\mathcal C_p(f)$]
is well-defined. Thus, we may define the \emph{Hardy
space $H^p$} and the \emph{Carleson measure space $\mathrm{CMO}^p$} as follows.

\begin{definition}\label{def-hcmo}
Let $p\in(0,1]$. The \emph{Hardy space} $H^p$ is defined as
\begin{align*}
H^p&:=\left\{f\in\left(\mathrm{CMO}^p_{L^2}\right)':\ \textup{\eqref{eq-re} holds in } \left(\mathrm{CMO}^p_{L^2}\right)'
\textup{and}
~~
\|f\|_{H^p}:=\|S(f)\|_{L^p}<\infty \right\}.
\end{align*}
The  \emph{Carleson measure space} $\mathrm{CMO}^p$ is defined as
\begin{align*}
\mathrm{CMO}^p&:=\left\{g\in\left(H^p_{L^2}\right)':\
\textup{\eqref{eq-re} holds in } \left(H^p_{L^2}\right)'
~~\textup{and}~~
\|g\|_{\mathrm{CMO}^p}:=\mathcal C_p(g)<\infty
\right\}.
\end{align*}
\end{definition}

\begin{remark}
In \cite{hlw18}, Han et al. introduced product Hardy spaces by using the 
product  test function space
$\mathring{G}(\beta_1,\beta_2,\gamma_1, \gamma_2)$, which exactly coincides 
with  those product Hardy spaces \cite{hll16-Pisa}.
However, if one considers a wider
generalization of Hardy spaces, it is \emph{difficult} to find a direct space of test functions such that a proper
Calder\'on reproducing formula holds in this space. Using the method presented here, we only need to show that the proper Calder\'on reproducing formula holds in $L^2$ which is independent of Hardy spaces, and then use an expansion way to define the desired Hardy spaces.
\end{remark}

The following proposition gives the dense subspaces of product Hardy
and Carleson measure spaces.

\begin{proposition}\label{prop-dense}
Let $p\in(0,1]$. Then the following statements hold:
\begin{enumerate}
\item $H^p_{L^2}$ is a dense subspace of $H^p$;
\item $\mathrm{CMO}^p_{L^2}$ is dense in $\mathrm{CMO}^p$ in the weak sense that, for any
$g\in\mathrm{CMO}^p$, there exists a sequence $\{g_N\}_{N=1}^\infty\subset\mathrm{CMO}^p_{L^2}$ such that
\begin{align*}
 \sup_{N\in\mathbb N}\|g_N\|_{\mathrm{CMO}^p}\le\|g\|_{\mathrm{CMO}^p}~~\text{and}~~
 \lim_{N\to\infty} g_N=g~~\text{in}~~(H^p_{L^2})'.
\end{align*}
\end{enumerate}
\end{proposition}

\begin{proof}
Fix $(x^{(0)}_1,x^{(0)}_2)\in X_1\times X_2$ and, for any $N\in\mathbb N$, let
\begin{align}\label{eq-defin}
 \mathcal I_N:=\left\{(k_1,k_2,\alpha_1,\alpha_2):\ |k_1|\le N, |k_2|\le N, d_1\left(y_{1,\alpha_1}^{k_1},x^{(0)}_1\right)\le N,
 d_2\left(y_{2,\alpha_2}^{k_2},x^{(0)}_2\right)\le N\right\}.
\end{align}
By the doubling condition \eqref{eq-doub}, we find that, for any $N\in\mathbb N$, $\mathcal I_N$ is finite.

We first prove (i). For any $f\in H^p$, let
$$
f_N:=\sum_{(k_1,k_2,\alpha_1,\alpha_2)\in\mathcal I_N}\left<f,\psi_{\alpha_1,\alpha_2}^{k_1,k_2}\right>\psi_{\alpha_1,\alpha_2}^{k_1,k_2}.
$$
Since the summation here is finite, it follows that $f_N\in L^2$. Moreover, by the
orthogonality of $\{\psi_{\alpha_1,\alpha_2}^{k_1,k_2}\}$ and \eqref{eq-re}, we conclude that, for any
$N\in\mathbb N$, $0\le S(f_N)\le S(f)$ and $\lim_{N\to\infty}
S(f_N)=S(f)$ pointwise. Thus, by the fact $S(f)\in L^p$ and the
dominated convergence theorem, we have $\lim_{N\to\infty}\|S(f_N-f)\|_{L^p}=0$,
which implies that $f_N\to f$ in $H^p$ as $N\to\infty$. This finishes the proof of (i).

Now we prove (ii). For any $g\in\mathrm{CMO}^p$ and $N\in\mathbb N$, let
$$
g_N:=\sum_{(k_1,k_2,\alpha_1,\alpha_2)\in\mathcal I_N}\left<g,\psi_{\alpha_1,\alpha_2}^{k_1,k_2}\right>\psi_{\alpha_1,\alpha_2}^{k_1,k_2}.
$$
Then, applying an argument similar to that used in the proof of (i), we find that $g_N\in\mathrm{CMO}^p_{L^2}$
and $\|g_N\|_{\mathrm{CMO}^p} \le\|g\|_{\mathrm{CMO}^p}$. Moreover, by \eqref{eq-re}, we have $\lim_{N\to\infty} g_N=g$ in
$(H^p_{L^2})'$. This finishes the proof of (ii) and hence  Proposition \ref{prop-dense}.
\end{proof}

In the remainder of this section, we give some fundamental properties of product Hardy spaces, which will be used in
the remainder of this article. The following proposition is a generalization of Proposition \ref{prop-dual}, whose proof
is similar to that of Proposition \ref{prop-dual} and we omit the details here.

\begin{proposition}\label{prop-fg}
Let $p\in(0,1]$. Then there exists a constant $C\in(0,\infty)$ such that, for any $f\in H^p$ and
$g\in{\rm CMO}^p_{L^2}$, [resp.\ $f\in H^p_{L^2}$ and $g\in\mathrm{CMO}^p $],
$$
|\langle f,g\rangle|\le C\|S(f)\|_{L^p}\mathcal C_p(g) \ [\textit{resp.\ }
|\langle g,f\rangle|\le C\|S(f)\|_{L^p}\mathcal C_p(g)].
$$
\end{proposition}

The next proposition shows the completeness of product Hardy spaces.

\begin{proposition}\label{prop-com1}
Let $p\in(0,1]$ and $\{f_n\}_{n=1}^\infty$ be a Cauchy sequence on $H^p$. Then there exists
$f\in H^p$ such that $\lim_{n\to\infty}\|f_n-f\|_{H^p}=0$, that is, $f$ satisfies that $f\in({\rm CMO}^p_{L^2})'$,
$S(f)\in L^p$,
$$
f=\sum_{\genfrac{}{}{0pt}{}{k_1\in{\mathbb Z}}{\alpha_1\in\mathcal G_{1,k_1}}}
\sum_{\genfrac{}{}{0pt}{}{k_2\in{\mathbb Z}}{\alpha_2\in\mathcal G_{2,k_2}}}
\left<f,\psi_{\alpha_1,\alpha_2}^{k_1,k_2}\right>\psi_{\alpha_1,\alpha_2}^{k_1,k_2}
$$
in $({\rm CMO}^p_{L^2})'$, and $\lim\limits_{n\to\infty}\|S(f_n-f)\|_{L^p}=0$.
\end{proposition}

\begin{proof}
Since $\{f_n\}_{n=1}^\infty\subset H^p$ is a Cauchy sequence of $H^p$, it follows from Proposition
\ref{prop-fg} that, for any $m,n\in\mathbb N$ and $\varphi\in {\rm CMO}^p_{L^2}$,
$$
|\langle f_m-f_n,\varphi\rangle|\lesssim\|S(f_m-f_n)\|_{L^p}\mathcal C_p(\varphi)\to 0
$$
as $m,n\to\infty$. This shows that $\{\langle f_n,\varphi\rangle\}_{n=1}^\infty$ is a Cauchy sequence of $\mathbb C$.
Thus, we may define $f$ by setting, for any $\varphi\in{\rm CMO}^p_{L^2}$,
$$
\langle f,\varphi\rangle:=\lim_{n\to\infty}\langle f_n,\varphi\rangle.
$$
By this and Proposition \ref{prop-fg}, we find that $f\in({\rm CMO}^p_{L^2})'$ and
$\lim\limits_{n\to\infty}f_n=f$ in $({\rm CMO}^p_{L^2})'$. Moreover, from Fatou's lemma, we deduce that,
for any $(x_1,x_2)\in X_1\times X_2$,
\begin{align*}
S(f)(x_1,x_2)
&=\left[\sum_{k_1,\alpha_1}\sum_{k_2,\alpha_2}\left|\left<f,\psi_{\alpha_1,\alpha_2}^{k_1,k_2}\right>
\widetilde{\mathbf 1}_{Q_{\alpha_1,\alpha_2}^{k_1+1,k_2+1}} (x_1, x_2)\right|^2\right]^{\frac 12}\\
&\le\varliminf\limits_{n\to\infty}\left[\sum_{k_1,\alpha_1}\sum_{k_2,\alpha_2}
\left|\left<f_n,\psi_{\alpha_1,\alpha_2}^{k_1,k_2}\right>
\widetilde{\mathbf 1}_{Q_{\alpha_1,\alpha_2}^{k_1+1,k_2+1}}(x_1, x_2)\right|^2\right]^{\frac 12}
=\varliminf\limits_{n\to\infty}S(f_n)(x_1,x_2),
\end{align*}
which, combined with  Fatou's lemma again, further implies that
$$
\|S(f)\|_{L^p}\le
\varliminf\limits_{n\to\infty}
\|S(f_n)\|_{L^p}<\infty.
$$
Similarly, since $\{f_n\}_{n=1}^\infty$ is a Cauchy sequence of $H^p$, it then follows that
\begin{equation*}
\lim_{n\to\infty}\|S(f_n-f)\|_{L^p}= 0.
\end{equation*}

Finally, we show that the wavelet reproducing formula holds for any $f \in ({\rm CMO}^p_{L^2})'$. Indeed, fix
$\varphi\in {\rm CMO}^p_{L^2}$. We have
\begin{align*}
&\left|\langle f,\varphi\rangle-\sum_{k_1,\alpha_1}\sum_{k_2,\alpha_2}
\left<f,\psi_{\alpha_1,\alpha_2}^{k_1,k_2}\right>\left<\varphi,
\psi_{\alpha_1,\alpha_2}^{k_1,k_2}\right>\right|\\
&\quad\le|\langle f-f_n,\varphi\rangle|
+\left|\langle f_n,\varphi\rangle-\sum_{k_1,\alpha_1}\sum_{k_2,\alpha_2}
\left<f,\psi_{\alpha_1,\alpha_2}^{k_1,k_2}\right>\left<\varphi,
\psi_{\alpha_1,\alpha_2}^{k_1,k_2}\right>\right|\\
&\qquad+\left|\sum_{k_1,\alpha_1}\sum_{k_2,\alpha_2}
\left<f_n-f,\psi_{\alpha_1,\alpha_2}^{k_1,k_2}\right>\left<\varphi,
\psi_{\alpha_1,\alpha_2}^{k_1,k_2}\right>\right|\\
&\quad\lesssim\|S(f-f_n)\|_{L^p}\mathcal C_p(\varphi)\to 0
\end{align*}
as $n\to\infty$. This finishes the proof of Proposition \ref{prop-com1}.
\end{proof}

Next, we show that Hardy spaces are independent of the choice of wavelet systems. Let
$$
\mathcal W:=\{\psi_{\alpha_1,\alpha_2}^{k_1,k_2}:\ k_1,k_2\in\mathbb Z,\alpha_1\in\mathcal G_{1,k_1},
\alpha_2\in\mathcal G_{2,k_2}\}
$$
and
$$
\widetilde{\mathcal W}:=\left\{\widetilde{\psi}_{\widetilde{\alpha}_1,\widetilde{\alpha}_2}^{\widetilde k_1,\widetilde k_2}
:\ \widetilde k_1,\widetilde k_2\in\mathbb Z,\widetilde\alpha_1\in\widetilde{\mathcal G}_{1,\widetilde k_1},
\widetilde \alpha_2\in\widetilde{\mathcal G}_{2,\widetilde k_2}\right\}
$$
be two wavelet systems on $X_1\times X_2$. By Definition \ref{def-hcmo}, for any $p\in(0,1]$,
we use $H^p(\mathcal W)$ to denote the Hardy space $H^p$ associated with the wavelet system $\mathcal W$. The next
proposition shows the equivalence of two (quasi-)norms.

\begin{proposition}\label{prop-hpweq}
Let $\mathcal W$ and $\widetilde{\mathcal W}$ be two wavelet systems, and
let $p\in(0,1]$. Then there exist constants
$C_1,C_2\in(0,\infty)$ such that, for any $f\in L^2$,
\begin{equation}\label{eq-hpwep}
C_1\|S(f)\|_{L^p}\le\left\|\widetilde{S}(f)\right\|_{L^p}\le C_2\|S(f)\|_{L^p}.
\end{equation}
\end{proposition}

\begin{proof}
By symmetry, we only show the second inequality of \eqref{eq-hpwep}. To this end, let $f\in L^2$.
By the wavelet reproducing formula, we find that
$$
f(\cdot)=\sum_{\genfrac{}{}{0pt}{}{\widetilde k_1\in{\mathbb Z}}{\widetilde\alpha_1\in\widetilde{\mathcal G}_{1,k_1}}}
\sum_{\genfrac{}{}{0pt}{}{\widetilde k_2\in{\mathbb Z}}{\widetilde\alpha_2\in\widetilde{\mathcal G}_{2,k_2}}}
\left<f,\widetilde{\psi}_{\widetilde{\alpha}_1,
\widetilde{\alpha}_2}^{\widetilde{k}_1,\widetilde{k}_2}\right>
\widetilde{\psi}_{\widetilde{\alpha}_1,\widetilde{\alpha}_2}^{\widetilde{k}_1,\widetilde{k}_2}(\cdot)
$$
holds in $L^2$ and, consequently, for any $k_1,k_2\in\mathbb Z$, $\alpha_1\in\mathcal G_{1,k_1}$, and $\alpha_2\in\mathcal G_{2,k_2}$,
\begin{align}\label{eq-wpt}
\left\langle f,\psi_{\alpha_1,\alpha_2}^{k_1,k_2}\right\rangle
&=\sum_{\widetilde k_1,\widetilde\alpha_1}\sum_{\widetilde k_2,\widetilde\alpha_2}
\left<f,\widetilde{\psi}_{\widetilde{\alpha}_1,\widetilde{\alpha}_2}^{\widetilde{k}_1,
\widetilde{k}_2}\right>
\left<\widetilde{\psi}_{\widetilde{\alpha}_1,\widetilde{\alpha}_2}^{\widetilde{k}_1,\widetilde{k}_2},
\psi_{\alpha_1,\alpha_2}^{k_1,k_2}\right>\nonumber\\
&={\mu\left(Q_{\alpha_1,\alpha_2}^{k_1+1,k_2+1}\right)}^{\frac 12}\sum_{\widetilde k_1,\widetilde\alpha_1}
\sum_{\widetilde k_2,\widetilde\alpha_2}
{\mu\left(\widetilde Q_{\widetilde{\alpha}_1,\widetilde{\alpha}_2}^{\widetilde{k}_1+1,\widetilde{k}_2+1}\right)
}^{\frac 12}
\left<f,\widetilde{\psi}_{\widetilde{\alpha}_1,\widetilde{\alpha}_2}^{\widetilde{k}_1,
\widetilde{k}_2}\right>
\left<\widetilde{\Psi}_{\widetilde{\alpha}_1,\widetilde{\alpha}_2}^{\widetilde{k}_1,\widetilde{k}_2},
\Psi_{\alpha_1,\alpha_2}^{k_1,k_2}\right>,
\end{align}
where $\widetilde{\Psi}_{\widetilde{\alpha}_1,\widetilde{\alpha}_2}^{\widetilde{k}_1,\widetilde{k}_2}
:=\widetilde{\psi}_{\widetilde{\alpha}_1,\widetilde{\alpha}_2}^{\widetilde{k}_1,\widetilde{k}_2}
{\mu(\widetilde Q_{\widetilde{\alpha}_1,\widetilde{\alpha}_2}^{\widetilde{k}_1+1,\widetilde{k}_2+1})}^{-\frac 12}$
and $\Psi_{\alpha_1,\alpha_2}^{k_1,k_2}:=\psi_{\alpha_1,\alpha_2}^{k_1,k_2}
{\mu(Q_{\alpha_1,\alpha_2}^{k_1+1,k_2+1})}^{-\frac 12}$. Note that, applying an argument similar to that used in the proof
of \cite[Theorem 3.1]{hlyy19}, we conclude that, for any given $\eta'\in(0,\eta)$ and $\Gamma\in(0,\infty)$,
$$
\left|\left<\widetilde{\Psi}_{\widetilde{\alpha}_1,\widetilde{\alpha}_2}^{\widetilde{k}_1,
\widetilde{k}_2},
\Psi_{\alpha_1,\alpha_2}^{k_1,k_2}\right>\right|\lesssim\delta^{|k_1-\widetilde{k}_1|\eta'}
\delta^{|k_2-\widetilde{k}_2|\eta'}
P_{1,\Gamma}\left(y_{1,\alpha_1}^{k_1},\widetilde{y}_{1,\widetilde{\alpha}_1}^{\widetilde{k}_1};
\delta^{k_1\wedge\widetilde{k}_1}\right)P_{2,\Gamma}\left(y_{2,\alpha_2}^{k_2},
\widetilde{y}_{2,\widetilde{\alpha}_2}^{\widetilde{k}_2};\delta^{k_2\wedge\widetilde{k}_2}\right).
$$
From this, \eqref{eq-wpt}, and Proposition \ref{prop-estmax}, we infer that, 
for any $k_1,k_2\in\mathbb Z$ and $(x_1,x_2)\in X_1\times X_2$,
\begin{align}\label{eq-fpp}
&\sum_{\alpha_1\in\mathcal G_{1,k_1}}\sum_{\alpha_2\in\mathcal G_{2,k_2}}\left|\left<f,\psi_{\alpha_1,\alpha_2}^{k_1,k_2}
\right>\widetilde{\mathbf 1}_{Q_{\alpha_1,\alpha_2}^{k_1+1,k_2+1}}(x_1,x_2)\right|^2\nonumber\\
&\quad\lesssim\sum_{\alpha_1\in\mathcal G_{1,k_1}}\sum_{\alpha_2\in\mathcal G_{2,k_2}}
\left\{\sum_{\widetilde k_1,\widetilde\alpha_1}
\sum_{\widetilde k_2,\widetilde\alpha_2} \delta^{|k_1-\widetilde{k}_1|\eta'}\delta^{|k_2-\widetilde{k}_2|\eta'}
\mu\left(\widetilde Q_{\widetilde{\alpha}_1,\widetilde{\alpha}_2}^{\widetilde{k}_1+1,\widetilde{k}_2+1}\right)
\left|\left<f,\frac{\widetilde{\psi}_{\widetilde{\alpha}_1,
\widetilde{\alpha}_2}^{\widetilde{k}_1,\widetilde{k}_2}}
{\sqrt{\mu(\widetilde Q_{\widetilde{\alpha}_1,\widetilde{\alpha}_2}^{\widetilde{k}_1+1,
\widetilde{k}_2+1})}}\right>\right|\right.
\nonumber\\
&\quad\quad\times P_{1,\Gamma}\left(y_{1,\alpha_1}^{k_1},\widetilde{y}_{1,\widetilde{\alpha}_1}^{\widetilde{k}_1};
\delta^{k_1\wedge\widetilde{k}_1}\right)P_{2,\Gamma}\left(y_{2,\alpha_2}^{k_2},
\widetilde{y}_{2,\widetilde{\alpha}_2}^{\widetilde{k}_2};\delta^{k_2\wedge\widetilde{k}_2}\right)
\mathbf 1_{Q_{\alpha_1,\alpha_2}^{k_1+1,k_2+1}}(x_1,x_2)\Bigg\}^2\nonumber\\
&\quad\lesssim\left\{\sum_{\widetilde k_1,\widetilde\alpha_1}\sum_{\widetilde k_2,\widetilde\alpha_2}
\delta^{|k_1-\widetilde{k}_1|\eta'}\delta^{|k_2-\widetilde{k}_2|\eta'}
\mu\left(\widetilde Q_{\widetilde{\alpha}_1,\widetilde{\alpha}_2}^{\widetilde{k}_1+1,\widetilde{k}_2+1}\right)
\left|\left<f,\frac{\widetilde{\psi}_{\widetilde{\alpha}_1,
\widetilde{\alpha}_2}^{\widetilde{k}_1,\widetilde{k}_2}}
{\sqrt{\mu(\widetilde Q_{\widetilde{\alpha}_1,\widetilde{\alpha}_2}^{\widetilde{k}_1+1,
\widetilde{k}_2+1})}}\right>\right|\right.
\nonumber\\
&\quad\quad\times P_{1,\Gamma}\left(x_1,\widetilde{y}_{1,\widetilde{\alpha}_1}^{\widetilde{k}_1};
\delta^{k_1\wedge\widetilde{k}_1}\right)P_{2,\Gamma}\left(x_2,\widetilde{y}_{2,
\widetilde{\alpha}_2}^{\widetilde{k}_2};
\delta^{k_2\wedge\widetilde{k}_2}\right)\Bigg\}^2\nonumber\\
&\quad\lesssim\left\{\sum_{\widetilde{k}_1,\widetilde{k}_2\in\mathbb Z}
\delta^{|k_1-\widetilde{k}_1|\eta'}\delta^{|k_2-\widetilde{k}_2|\eta'}
\delta^{[(\widetilde{k}_1\wedge k_1)-\widetilde{k}_1]\omega(\frac 1r-1)}\delta^{[(\widetilde{k}_2\wedge k_2)-\widetilde{k}_2]
\omega(\frac 1r-1)}\right.\nonumber\\&\quad\quad\left.{}\times\left[\mathcal M_{\rm str}
\left(\sum_{\widetilde\alpha_1\in\widetilde{\mathcal G}_{1,\widetilde{k}_1}}
\sum_{\widetilde\alpha_2\in\widetilde{\mathcal G}_{2,\widetilde{k}_2}}
\left|\left<f,\widetilde\psi_{\widetilde\alpha_1,\widetilde\alpha_2}^{\widetilde k_1,\widetilde k_2}\right>
\widetilde{\mathbf 1}_{\widetilde Q_{\widetilde\alpha_1,\widetilde\alpha_2}^{\widetilde k_1+1,\widetilde k_2+1}}\right|^r\right)
(x_1,x_2)\right]^{\frac 1r}\right\}^2,
\end{align}
where $\eta'\in(0,\eta)$, $\Gamma\in(0,\infty)$, and $r\in(\frac{\omega}{\omega+\Gamma},1]$. Choose $\eta'$, $\Gamma$, and $r$
such that $\eta'>\omega(\frac 1r-1)$ and $r<p$. By this, \eqref{eq-fpp}, and H\"older's inequality, we further conclude that, for any
$(x_1,x_2)\in X_1\times X_2$,
\begin{align*}
&\sum_{\genfrac{}{}{0pt}{}{k_1\in{\mathbb Z}}{\alpha_1\in\mathcal G_{1,k_1}}}
\sum_{\genfrac{}{}{0pt}{}{k_2\in{\mathbb Z}}{\alpha_2\in\mathcal G_{2,k_2}}}
\left|\left<f,\psi_{\alpha_1,\alpha_2}^{k_1,k_2}\right>\widetilde{\mathbf 1}_{Q_{\alpha_1,\alpha_2}^{k_1+1,k_2+1}}(x_1,x_2)\right|^2\\
&\quad\lesssim\sum_{k_1,k_2\in\mathbb Z}\left\{
\sum_{\widetilde{k}_1,\widetilde{k}_2\in\mathbb Z}\delta^{|k_2-\widetilde{k}_2|\eta'}
\delta^{[(\widetilde{k}_1\wedge k_1)-\widetilde{k}_1]\omega(\frac 1r-1)}\delta^{[(\widetilde{k}_2\wedge k_2)-\widetilde{k}_2]
\omega(\frac 1r-1)}\right.\\
&\qquad\left.{}\times\left[\mathcal M_{\rm str}\left(\sum_{\widetilde{\alpha}_1\in\widetilde{\mathcal G}_{1,\widetilde{k}_1}}
\sum_{\widetilde{\alpha}_2\in\widetilde{\mathcal G}_{2,\widetilde{k}_2}}
\left|\left<f,\widetilde\psi_{\widetilde\alpha_1,\widetilde\alpha_2}^{\widetilde k_1,\widetilde k_2}\right>
\widetilde{\mathbf 1}_{\widetilde Q_{\widetilde\alpha_1,\widetilde\alpha_2}^{\widetilde k_1+1,\widetilde k_2+1}}\right|^r\right)
(x_1,x_2)\right]^{\frac 1r}\right\}^2\\
&\quad\lesssim\sum_{\widetilde{k}_1,\widetilde{k}_2\in\mathbb Z}\left[\mathcal M_{\rm str}
\left(\sum_{\widetilde{\alpha}_1\in\widetilde{\mathcal G}_{1,\widetilde{k}_1}}
\sum_{\widetilde{\alpha}_2\in\widetilde{\mathcal G}_{2,\widetilde{k}_2}}
\left|\left<f,\widetilde\psi_{\widetilde\alpha_1,\widetilde\alpha_2}^{\widetilde k_1,\widetilde k_2}\right>
\widetilde{\mathbf 1}_{\widetilde Q_{\widetilde\alpha_1,\widetilde\alpha_2}^{\widetilde k_1+1,\widetilde k_2+1}}\right|^r\right)
(x_1,x_2)\right]^{\frac 2r},
\end{align*}
which, together with Proposition \ref{prop-fs}, further implies that
\begin{align*}
\|S(f)\|_{L^p}&=\left\|\left[\sum_{k_1,\alpha_1}\sum_{k_2,\alpha_2}
\left|\left<f,\psi_{\alpha_1,\alpha_2}^{k_1,k_2}\right>
\widetilde{\mathbf 1}_{Q_{\alpha_1,\alpha_2}^{k_1+1,k_2+1}}\right|^2\right]^{\frac 12}\right\|_{L^p}\\
&\lesssim\left\|\left\{\sum_{\widetilde{k}_1,\widetilde{k}_2\in\mathbb Z}
\left[\mathcal M_{\rm str}\left(\sum_{\widetilde{\alpha}_1\in\widetilde{\mathcal G}_{1,\widetilde{k}_1}}
\sum_{\widetilde{\alpha}_2\in\widetilde{\mathcal G}_{2,\widetilde{k}_2}}
\left|\left<f,\widetilde\psi_{\widetilde\alpha_1,\widetilde\alpha_2}^{\widetilde k_1,\widetilde k_2}\right>
\widetilde{\mathbf 1}_{\widetilde Q_{\widetilde\alpha_1,\widetilde\alpha_2}^{\widetilde k_1+1,\widetilde k_2+1}}\right|^r\right)
\right]^{\frac 2r}\right\}^{\frac r2}\right\|_{L^{\frac pr}}^{\frac 1r}\\
&\lesssim\left\|\left[\sum_{\widetilde{k}_1,\widetilde{k}_2\in\mathbb Z}
\left(\sum_{\widetilde{\alpha}_1\in\widetilde{\mathcal G}_{1,\widetilde{k}_1}}
\sum_{\widetilde{\alpha}_2\in\widetilde{\mathcal G}_{2,\widetilde{k}_2}}
\left|\left<f,\widetilde\psi_{\widetilde\alpha_1,\widetilde\alpha_2}^{\widetilde k_1,\widetilde k_2}\right>
\widetilde{\mathbf 1}_{\widetilde Q_{\widetilde\alpha_1,\widetilde\alpha_2}^{\widetilde k_1+1,\widetilde k_2+1}}\right|^r\right)
^{\frac 2r}\right]^{\frac r2}\right\|_{L^{\frac pr}}^{\frac 1r}\\
&\sim\left\|\left(\sum_{\widetilde k_1,\widetilde\alpha_1}\sum_{\widetilde k_2,\widetilde\alpha_2}
\left|\left<f,\widetilde\psi_{\widetilde\alpha_1,\widetilde\alpha_2}^{\widetilde k_1,\widetilde k_2}\right>
\widetilde{\mathbf 1}_{\widetilde Q_{\widetilde\alpha_1,\widetilde\alpha_2}^{\widetilde k_1+1,\widetilde k_2+1}}\right|^2
\right)^{\frac r2}\right\|_{L^{\frac pr}}^{\frac 1r}\sim\left\|\widetilde{S}(f)\right\|_{L^p}.
\end{align*}
This finishes the proof of Proposition \ref{prop-hpweq}.
\end{proof}

By Proposition \ref{prop-hpweq}, we find that, for any $p\in(0,1]$ and any two wavelet systems $\mathcal W$ and $\widetilde{\mathcal W}$ on $X_1\times X_2$, $H^p_{L^2}(\mathcal W)=H^p_{L^2}(\widetilde{\mathcal W})$ with equivalent (quasi-)norms. The next theorem shows that
product Hardy spaces associated with different wavelet systems also coincide.

\begin{theorem}\label{thm-hpw}
Let $p\in(0,1]$ and $\mathcal W$ and $\widetilde{\mathcal W}$ be two wavelet systems on $X_1\times X_2$. Then
$H^p(\mathcal W)=H^p(\widetilde{\mathcal W})$ with equivalent (quasi-)norms.
\end{theorem}

\begin{proof}
By symmetry, we only prove that $H^p(\mathcal W)\subset H^p(\widetilde{\mathcal W})$ and, for any $f\in H^p(\mathcal W)$,
\begin{equation}\label{eq-sws}
\left\|\widetilde{S}(f)\right\|_{L^p}\lesssim\|S(f)\|_{L^p}.
\end{equation}
Since $H^p_{L^2}(\mathcal W)$ is dense in $H^p(\mathcal W)$, it follows that there exists a sequence $\{f_n\}_{n=1}^\infty\subset
H^p_{L^2}(\mathcal W)$ such that $\lim_{n\to\infty}\|S(f-f_n)\|_{L^p}=0$. Now, for any $\varphi\in{\rm CMO}^p_{L^2}(\mathcal W)$,
let $\langle f,\varphi\rangle:=\lim_{n\to\infty}\langle f_n,\varphi\rangle$. We first show that $\langle f,\varphi\rangle$ is
well-defined. Indeed, by Propositions \ref{prop-dual} and \ref{prop-hpweq}, we conclude that
$$
\left|\langle f_m-f_n,\varphi\rangle\right|\lesssim\left\|\widetilde{S}(f_n-f_m)\right\|_{L^p}\widetilde{\mathcal C}_p(\varphi)
\sim\|S(f_m-f_n)\|_{L^p}\widetilde{\mathcal{C}}_p(\varphi)\to 0
$$
as $m,n\to\infty$. This implies that $\{\langle f_n,\varphi\rangle\}_{n=1}^\infty$ is a Cauchy sequence of $\mathbb C$. Using this
and the completeness of $\mathbb C$, we find that the limit $\lim_{n\to\infty}\langle f_n,\varphi\rangle$ exists. Similarly, we may
also conclude that $\langle f,\varphi\rangle$ is independent of the choice of $\{f_n\}_{n=1}^\infty$. Thus, we may assume that, for
any $n\in\mathbb N$, $\|S(f_n)\|_{L^p}\lesssim\|f\|_{L^p}$. From this and Proposition \ref{prop-dual}, we deduce that
$$
|\langle f,\varphi\rangle|=\lim_{n\to\infty}|\langle f_n,\varphi\rangle|\lesssim\|S(f)\|_{L^p}\widetilde{\mathcal C}_p(\varphi).
$$
This shows that $f\in({\rm CMO}^p_{L^2}(\widetilde{\mathcal W}))'$. Moreover, by Fatou's lemma, we conclude that, for any
$(x_1,x_2)\in X_1\times X_2$,
\begin{align*}
\left[\widetilde S(f)(x)\right]^2&=\sum_{\genfrac{}{}{0pt}{}{\widetilde k_1\in{\mathbb Z}}
{\widetilde\alpha_1\in\widetilde{\mathcal G}_{1,\widetilde k_1}}}
\sum_{\genfrac{}{}{0pt}{}{\widetilde k_2\in{\mathbb Z}}{\widetilde\alpha_2\in\widetilde{\mathcal G}_{2,\widetilde k_2}}}
\left|\left<f,\widetilde\psi_{\widetilde\alpha_1,\widetilde\alpha_2}^{\widetilde k_1,\widetilde k_2}\right>
\widetilde{\mathbf 1}_{\widetilde Q_{\widetilde\alpha_1,\widetilde\alpha_2}^{\widetilde k_1+1,\widetilde k_2+1}}(x_1,x_2)\right|^2\\
&=\sum_{\widetilde k_1,\widetilde\alpha_1}\sum_{\widetilde k_2,\widetilde\alpha_2}
\lim_{n\to\infty}\left|\left<f_n,\widetilde\psi_{\widetilde\alpha_1,\widetilde\alpha_2}^{\widetilde k_1,\widetilde k_2}\right>
\widetilde{\mathbf 1}_{\widetilde Q_{\widetilde\alpha_1,\widetilde\alpha_2}^{\widetilde k_1+1,\widetilde k_2+1}}(x_1,x_2)\right|^2\\
&\le
\varliminf\limits_{n\to\infty}
\sum_{\widetilde k_1,\widetilde\alpha_1}\sum_{\widetilde k_2,\widetilde\alpha_2}
\left|\left<f_n,\widetilde\psi_{\widetilde\alpha_1,\widetilde\alpha_2}^{\widetilde k_1,\widetilde k_2}\right>
\widetilde{\mathbf 1}_{\widetilde Q_{\widetilde\alpha_1,\widetilde\alpha_2}^{\widetilde k_1+1,\widetilde k_2+1}}(x_1,x_2)\right|^2
=\varliminf\limits_{n\to\infty}
\left[\widetilde S(f_n)(x_1,x_2)\right]^2.
\end{align*}
By this, Fatou's lemma, and Proposition \ref{prop-hpweq}, we obtain
$$
\left\|\widetilde S(f)\right\|_{L^p}
\le
\varliminf\limits_{n\to\infty}\left\|\widetilde S(f_n)\right\|_{L^p}
\sim
\varliminf\limits_{n\to\infty}\|S(f_n)\|_{L^p}\lesssim\|S(f)\|_{L^p},
$$
that is, \eqref{eq-sws} holds.
In addition, applying an argument similar to that used in the proof of Proposition \ref{prop-com1}, we further find that the wavelet
reproducing formula \eqref{eq-re} associated with the wavelet system $\widetilde{\mathcal W}$ holds in
$({\rm CMO}^p_{L^2}(\widetilde{\mathcal W}))'$. This,
together with \eqref{eq-sws}, implies that $f\in H^p(\mathcal W)$ and
$$
\|f\|_{H^p(\widetilde{\mathcal W})}=\left\|\widetilde S(f)\right\|_{L^p}\lesssim\|S(f)\|_{L^p}
\sim\|f\|_{H^p(\mathcal W)}.
$$
This finishes the proof of $H^p(\mathcal W)\subset H^p(\widetilde{\mathcal W})$ and $\|\cdot\|_{H^p(\widetilde{\mathcal W})}
\lesssim\|\cdot\|_{H^p(\mathcal W)}$, which completes the proof of Theorem \ref{thm-hpw}.
\end{proof}

At the end of this section, we show that product Hardy spaces in Definition \ref{def-hcmo} precisely coincide with those
in \cite{hlw18}.

\begin{proposition}\label{prop-h=h}
Let $p\in(\frac{\omega}{\omega+\eta},1]$ with $\omega$ and $\eta$, respectively, the same as in \eqref{eq-doub} and Theorem
\ref{thm-wave}. Let $H_{\rm hlw}^p$ be the product Hardy space in \cite[Definition 5.1]{hlw18}. Then $H^p=H_{\rm hlw}^p$ in the
sense of equal (quasi-)norms.
\end{proposition}

\begin{proof}
For any $N \in \mathbb N$, let $\mathcal I_N$ be the same as in \eqref{eq-defin} and
$$
\mathcal F_N:=\left\{\sum_{(k_1,k_2,\alpha_1,\alpha_2)\in\mathcal I_N} c_{k_1,k_2,\alpha_1,\alpha_2}
\psi_{\alpha_1,\alpha_2}^{k_1,k_2}:\ \{c_{k_1,k_2,\alpha_1,\alpha_2}\}_{{(k_1,k_2,\alpha_1,\alpha_2)\in\mathcal I_N}}
\subset\mathbb C\right\}
$$
Also suppose  $\mathcal F:=\bigcup_{N\in\mathbb N}\mathcal F_N$. It is obvious that both $\mathcal F\subset H^p$ and $\mathcal F\subset H^p_{\rm hlw}$ hold.
 Moreover, on one hand, applying an argument similar to that used in the proof of Proposition
\ref{prop-dense}(i), we conclude that $\mathcal F$ is dense in $H^p$.
On the other hand, using \cite[(3.26)]{hlw18} and an argument
similar to that in \cite[pp.\ 159--160]{hlw18}, we find that $\mathcal F$ is also
dense in $H^p_{\rm hlw}$. Thus, $\mathcal F$ is a common dense subspace of $H^p$
and $H^p_{\rm hlw}$ and, for any $f\in\mathcal F$,
$$
\|f\|_{H^p}=\left\|\left(\sum_{(k_1,k_2,\alpha_1,\alpha_2)\in\mathcal I_N}
\left|c_{k_1,k_2,\alpha_1,\alpha_2}\widetilde{\mathbf 1}_{Q_{\alpha_1,\alpha_2}^{k_1+1,k_2+1}}\right|^2
\right)^{\frac 12}\right\|_{L^p}=\|f\|_{H^p_{\rm hlw}}.
$$

Now, for any $f\in H^p$, we have $s:=\{\langle f,\psi_{\alpha_1,\alpha_2}^{k_1,k_2}\rangle:
\ k_1,k_2\in\mathbb Z,\alpha_1\in\mathcal G_{1,k_1},\alpha_2\in\mathcal G_{2,k_2}\}\in s^p$, where
$s^p$ is the same as in \cite[p.\ 160]{hlw18}. Then, from \cite[Proposition 5.7]{hlw18}, we infer that
$$
\widetilde f:=\sum_{\genfrac{}{}{0pt}{}{k_1\in{\mathbb Z}}{\alpha_1\in\mathcal G_{1,k_1}}}
\sum_{\genfrac{}{}{0pt}{}{k_2\in{\mathbb Z}}{\alpha_2\in\mathcal G_{2,k_2}}}
\left<f,\psi_{\alpha_1,\alpha_2}^{k_1,k_2}\right>\psi_{\alpha_1,\alpha_2}^{k_1,k_2}\in H^p_{\rm hlw}
$$
and
$$
\left\|\widetilde{f}\right\|_{H^p_{\rm hlw}}
\le \|s\|_{s^p}=\left\|\left(\sum_{k_1,\alpha_1}\sum_{k_2,\alpha_2}\left|
\left<f,\psi_{\alpha_1,\alpha_2}^{k_1,k_2}\right>
\widetilde{\mathbf 1}_{Q_{\alpha_1,\alpha_2}^{k_1+1,k_2+1}}\right|^2\right)^{\frac 12}\right\|_{L^p}=\|f\|_{H^p}.
$$

On the contrary, let $g\in H^p_{\rm hlw}$. Then, for any $N\in\mathbb N$, let
$$
g_N:=\sum_{(k_1,k_2,\alpha_1,\alpha_2)\in\mathcal I_N}
\left<g,\psi_{\alpha_1,\alpha_2}^{k_1,k_2}\right>\psi_{\alpha_1,\alpha_2}^{k_1,k_2}.
$$
Since $g\in H^p_{\rm hlw}$, it then follows that
$$
\left\|\left(\sum_{k_1,\alpha_1}\sum_{k_2,\alpha_2}\left|\left<g,\psi_{\alpha_1,\alpha_2}^{k_1,k_2}\right>
\widetilde{\mathbf 1}_{Q_{\alpha_1,\alpha_2}^{k_1+1,k_2+1}}\right|^2\right)^{\frac 12}\right\|_{L^p}<\infty.
$$
This, together with  Fubini's theorem, further implies that $\{g_N\}_{N=1}^\infty$ is a Cauchy sequence of $H^p$.
Then, by Proposition \ref{prop-com1}, there exists $\widetilde{g}\in H^p$ such that
$$
\left\|\widetilde{g}\right\|_{H^p}\le
\left\|\left(\sum_{k_1,\alpha_1}\sum_{k_2,\alpha_2}\left|
\left<g,\psi_{\alpha_1,\alpha_2}^{k_1,k_2}\right>
\widetilde{\mathbf 1}_{Q_{\alpha_1,\alpha_2}^{k_1+1,k_2+1}}\right|^2\right)^{\frac 12}\right\|_{L^p}
=\|g\|_{H^p_{\rm hlw}}
$$
and, for any $\varphi\in{\rm CMO}^p_{L^2}$,
\begin{align*}
\left<\widetilde{g},\varphi\right>&=\lim_{N\to\infty}\langle g_N,\varphi\rangle
=\lim_{N\to\infty}\sum_{(k_1,k_2,\alpha_1,\alpha_2)\in\mathcal I_N}
\left<g,\psi_{\alpha_1,\alpha_2}^{k_1,k_2}\right>\left<\varphi,\psi_{\alpha_1,\alpha_2}^{k_1,k_2}\right>\\
&=\sum_{k_1,\alpha_1}\sum_{k_2,\alpha_2}\left<g,\psi_{\alpha_1,\alpha_2}^{k_1,k_2}\right>
\left<\varphi,\psi_{\alpha_1,\alpha_2}^{k_1,k_2}\right>.
\end{align*}
This finishes the proof of Proposition \ref{prop-h=h}.
\end{proof}

\section{Atomic Characterization of Product Hardy Spaces}\label{s-at}

In this section, we concern about the atomic decomposition characterization of product Hardy spaces.
Let $\Omega\subset X_1\times X_2$ be an open set with $\mu(\Omega)<\infty$.
Denote by $\mathcal M(\Omega)$ the set of all maximal dyadic rectangles in $\Omega$. More precisely, a rectangle
$Q_1\times Q_2\in\mathcal M(\Omega)$ means that $Q_1\times Q_2\subset\Omega$ and, if
$\widetilde Q_1\times\widetilde Q_2\subset\Omega$ with $\widetilde Q_1\supset Q_1$ and $\widetilde Q_2\supset Q_2$, then $Q_1=\widetilde Q_1$
and $Q_2=\widetilde Q_2$. Let $\mathcal M_1(\Omega)$ [resp.\ $\mathcal M_2(\Omega)$] denote the set of all rectangles maximal in the
first (resp.\ second) direction, that is, $Q_1\times Q_2\in\mathcal M_1(\Omega)$ [resp.\ $Q_1\times Q_2\in\mathcal M_2(\Omega)$]
if $Q_1\times Q_2\subset\Omega$ and, if $\widetilde Q_1\times Q_2\subset\Omega$ and $\widetilde Q_1\supset Q_1$
(resp.\ $Q_1\times\widetilde Q_2\subset\Omega$ and $\widetilde Q_2\supset Q_2$), then $\widetilde Q_1=Q_1$ (resp.\ $\widetilde Q_2=Q_2$).
Moreover, suppose $Q_1\times Q_2\in\mathcal M_1(\Omega)$ [resp.\ $Q_1\times Q_2\in\mathcal M_2(\Omega)$]. Let
$\widehat{Q_2}$ (resp.\ $\widehat{Q_1}$) be the maximal dyadic cube on $X_2$ (resp.\ $X_1$) such that
$$
\mu\left(\left(Q_1\times\widehat{Q_2}\right)\cap\Omega\right)>\frac 12\mu\left(Q_1\times\widehat{Q_2}\right)\
\left[\textup{resp. }\mu\left(\left(\widehat{Q_1}\times Q_2\right)\cap\Omega\right)
>\frac 12\mu\left(\widehat{Q_1}\times Q_2\right)\right].
$$

Now we give the definition of $(p,q)$-atoms with any $p\in(0,1]$ and $q\in(1,\infty)$.

\begin{definition}\label{def-atom}
Let $p\in(0,1]$, $q\in(1,\infty)$, $\Omega\subset X_1\times X_2$ be an open set with $\mu(\Omega)<\infty$, and $C_0\in(0,\infty)$ be a given constant determined later.
\begin{enumerate}
\item[(I)] A function $a\in L^q$ is called a \emph{$(p,q,C_0)$-atom supported
in $\Omega$} if $a$ has the following properties:
\begin{enumerate}
\item[(i)] $\mathop{\rm supp} a:=\{(x_1,x_2)\in X_1\times X_2:\ a(x_1,x_2)\neq 0\}\subset\Omega$;
\item[(ii)] $\|a\|_{L^q}\le[\mu(\Omega)]^{\frac 1q-\frac 1p}$;
\item[(iii)] if $q\in[2,\infty)$, then $a$ has a decomposition
$a=\sum_{R\in\mathcal M(\Omega)}a_R$, where $\{a_R\}_{R\in\mathcal M(\Omega)}$ satisfies the following conditions:
\begin{enumerate}
\item[(a)] $\mathop{\rm supp} a_R\subset [C_0B_1(Q_1)]\times [C_0B_2(Q_2)]$ for any $R:=Q_1\times Q_2\in\mathcal M(\Omega)$;
in what follows, for any $i\in\{1,2\}$ and any cube $Q_i$ in $X_i$, $B_i(Q_i)$ is the same as in Lemma \ref{lem-cube}(ii);
\item[(b)] for any $R\in\mathcal M(\Omega)$, $x_1\in X_1$, and $x_2\in X_2$,
$$
\int_{X_2} a_R(x_1,y_2)\,d\mu_2(y_2)=0=\int_{X_1} a_R(y_1,x_2)\,d\mu_1(y_1);
$$
\item[(c)] $[\sum_{R\in\mathcal M(\Omega)}\|a_R\|_{L^q(\Omega)}^q]^{\frac 1q}\le[\mu(\Omega)]^{\frac 1q-\frac 1p}$;
\end{enumerate}
\item[(iii)$'$] if $q\in(1,2)$, then $a$ has a decomposition
$$
a=\sum_{R_1\in\mathcal M_1(\Omega)}a_{R_1}+\sum_{R_2\in\mathcal M_2(\Omega)}a_{R_2},
$$
where, for any $i\in\{1,2\}$, $\{a_{R_i}\}_{R_i\in\mathcal M_i(\Omega)}$ satisfy (a) and (b) in (iii) with $R$
replaced by $R_i$ and $\mathcal M(\Omega)$ by $\mathcal M_i(\Omega)$, and
\begin{enumerate}
\item[(c)$'$] for any given $\epsilon\in(0,\infty)$, there exists a
positive constant $C_{\epsilon}$, independent of $a$ and $\Omega$, such that
$$
\left\{\sum_{Q_{1}\times Q_{2}=R\in\mathcal M_i(\Omega)}\left\|a_{R}\right\|_{L^q}^q
\left[\frac{\ell(Q_{j})}{\ell(\widehat{Q_{j}})}\right]^\epsilon\right\}^{\frac1q}\le C_{\epsilon}[\mu(\Omega)]^{\frac1q-\frac1p},
$$
where $(i,j)\in\{(1,2),(2,1)\}$.
\end{enumerate}
\end{enumerate}
\item[(II)] The atomic Hardy space $H^{p,q}_\mathrm{at}(X)$ is defined to be the set of all
$f\in({\rm CMO}^p_{L^2})'$ such that there exist  $(p,q,C_0)$-atoms $\{a_{j}\}_{j=1}^\infty$ and
$\{\lambda_{j}\}_{j=1}^\infty\subset\mathbb C$ such that
\begin{equation}\label{eq-adec}
f=\sum_{j=1}^\infty \lambda_{j}a_{j} \textup{ in } \left(\mathrm{CMO}^p_{L^2}\right)'.
\end{equation}
The (quasi-)norm of $H^{p,q}_\mathrm{at}$ is defined by setting, for any $f\in H^{p,q}_\mathrm{at}$,
$$
\|f\|_{H^{p,q}_\mathrm{at}}:=\inf\left\{\left(\sum_{j=1}^\infty |\lambda_{j}|^p\right)^{\frac 1p}\right\},
$$
where the infimum is taken over all atomic decomposition as in \eqref{eq-adec}.
\end{enumerate}
\end{definition}

\begin{remark}
In \cite{hlpw21}, Han et al. introduced a new kind of product atoms,
whose size conditions require enlargement
parameters $(\ell_1,\ell_2)\in\mathbb Z_+^2$. Via these atoms, Han et al. \cite{hlpw21}
obtained an atomic
characterization of product Hardy spaces on spaces of homogeneous type.
Comparing their atoms with the atoms in Definition \ref{def-atom}, we find
that these atoms here coincide with theirs  in the case where $(\ell_1,\ell_2):=(0,0)$.
Moreover, if $X_1$ and $X_2$ are RD-spaces,
then the atoms in Definition \ref{def-atom}  exactly coincide with those in \cite{hll16-Pisa}, which are precisely the atoms in \cite{CF80} and \cite{CF82}
on Euclidean spaces.
\end{remark}

The main result of this section is the following theorem.

\begin{theorem}\label{thm-h=a}
Let $p\in(\frac{\omega}{\omega+\eta},1]$ with $\omega$ and $\eta$, respectively, as in \eqref{eq-doub} and Theorem \ref{thm-wave},
and let $q\in(1,\infty)$. Then $f\in H^p$ if and only if there exist a sequence  of $(p,q,C_0)$-atoms $\{a_{j}\}_{j=1}^\infty$ and $\{\lambda_{j}\}_{j=1}^\infty\subset\mathbb C$ such that
$f=\sum_{j=1}^\infty\lambda_{j}a_{j}$ in $({\rm CMO}^p_{L^2})'$.
Moreover, there exist  constants $C_1,C_2\in(0,\infty)$, independent of $f$, $\{a_{j}\}_{j=1}^\infty$, and
$\{\lambda_{j}\}_{j=1}^\infty$, such that
\begin{equation*}
C_1\|f\|_{H^P}\le\left[\sum_{k=1}^\infty\left|\lambda_{j}\right|^p\right]^{\frac 1p}\le C_2\|f\|_{H^p}.
\end{equation*}
In other words, $H^p=H^{p,q}_\mathrm{at}(X)$ with equivalent (quasi-)norm.
\end{theorem}
The organization of the remainder of this section is as follows.
From Subsection \ref{ss-h>a} to Subsection \ref{ss-h<a}, we concentrate on proving Theorem \ref{thm-h=a}. In Subsection \ref{ss-h>a},
we show that $H^{p,q}_\mathrm{at}\subset H^p$. In Subsection \ref{ss-ctrf}, we establish a new
Caldre\'on-type reproducing formula to deal with the proof of $H^p\subset H^{p,q}_\mathrm{at}$.
In Subsection \ref{ss-h<a}, we show that $H^p\subset H^{p,q}_\mathrm{at}$.
Finally, in Subsection \ref{ss-bhl}, as an application of the atomic
characterization of $H^p$, we obtain a criterion for the
boundedness of linear operators from $H^p$
to $L^p$.

\subsection{Proof of $H^{p,q}_\mathrm{at}\subset H^p$
(``If"  Part of Theorem \ref{thm-h=a})}\label{ss-h>a}

In this section, we prove the ``if"  part of Theorem \ref{thm-h=a}. More precisely, we mainly concern about the
following proposition.

\begin{proposition}\label{prop-hsubat}
Let $p\in(\frac{\omega}{\omega+\eta},1]$, $q\in(1,\infty)$ with $\omega$ and $\eta$, respectively,
as in \eqref{eq-doub} and Theorem \ref{thm-wave}, $q\in(1,\infty)$, and $C_0\in(1,\infty)$. Suppose that
$\{a_{j}\}_{j=1}^\infty$ is a sequence of $(p,q,C_0)$-atoms and
a sequence
$\{\lambda_{j}\}_{j=1}^\infty\subset\mathbb C$ satisfies that $\sum_{j=1}^\infty|\lambda_{j}|^p<\infty$. Then there exists
$f\in({\rm CMO}^p_{L^2})'$ such that $f=\sum_{j=1}^\infty\lambda_{j}a_{j}$ in $({\rm CMO}^p_{L^2})'$.
Moreover, there exists a constant $C\in(0,\infty)$, independent of $\{a_{j}\}_{j=1}^\infty$ and
$\{\lambda_{j}\}_{j=1}^\infty$, such that
\begin{equation*}
\|f\|_{H^p}\le C\left(\sum_{j=1}^\infty\left|\lambda_{j}\right|^p\right)^{\frac 1p}.
\end{equation*}
\end{proposition}

To show Proposition \ref{prop-hsubat}, we need the following Propositions
\ref{prop-at} and \ref{prop-replq}, whose proofs are presented after the
proof of Proposition \ref{prop-hsubat}.

\begin{proposition}\label{prop-at}
Let $p\in(\frac{\omega}{\omega+\eta},1]$ with $\omega$ and $\eta$, respectively,  as in \eqref{eq-doub}
and Theorem \ref{thm-wave}, $q\in(1,\infty)$, and $C_0\in(1,\infty)$. Then there exists a constant $C\in(0,\infty)$
such that, for any $(p,q,C_0)$-atom $a$, $\|S(a)\|_{L^p}\le C$.
\end{proposition}

\begin{proposition}\label{prop-replq}
Let $p\in(\frac{\omega}{\omega+\eta},1]$, $q\in(1,\infty)$, and $a\in L^q$ with $S(a)\in L^p$, where
$\omega$ and $\eta$ are, respectively, the same as in \eqref{eq-doub} and Theorem \ref{thm-wave}. Then $a\in H^p$
in the following setting: there exists $\widetilde a\in({\rm CMO}^p_{L^2})'$ such that, for any
$\varphi\in{\rm CMO}^p_{L^2}\cap L^{q'}$,
$$
\left<\widetilde a,\varphi\right>=\int_{X_1}\int_{X_2}a(x_1,x_2)\varphi(x_1,x_2)\,d\mu_2(x_2)\,d\mu_1(x_1)
=\langle a,\varphi\rangle.
$$
\end{proposition}

Now, we use   Propositions \ref{prop-at} and \ref{prop-replq}
to prove  Proposition \ref{prop-hsubat}.

\begin{proof}[Proof of Proposition \ref{prop-hsubat}]
Let $\{a_{j}\}_{j=1}^\infty$ be a sequence of $(p,q,C_0)$-atoms and $\{\lambda_{j}\}_{j=1}^\infty\subset\mathbb C$ be such that
$\sum_{j=1}^\infty|\lambda_{j}|^p<\infty$.
By Propositions \ref{prop-at} and  \ref{prop-replq}, we   conclude that, for any
$j\in\mathbb N$, $a_{j}\in H^p$.
For any $n\in\mathbb N$, let
$f_n:=\sum_{j=1}^n\lambda_{j}a_{j}$.
Then $\{f_n\}_{n=1}^\infty\subset H^p$.
For any $m,n\in\mathbb N$ with $m>n$,
by   Minkowski's inequality and Proposition \ref{prop-at}, we have
$$
\|S(f_m-f_n)\|_{L^p}^p\le\sum_{j=n+1}^m|\lambda_{j}|^p\left\|S(a_{j})\right\|_{L^p}
\lesssim\sum_{j=n+1}^m|\lambda_{j}|^p\to 0
$$
as $m,n\to\infty$. This shows that $\{f_n\}_{n=1}^\infty$ is a Cauchy sequence
of $H^p$. Combining this with
Proposition \ref{prop-com1}, we find that there exists $f\in H^p$ such that
$$
f=\lim_{n\to\infty}f_n
=\lim_{n\to\infty}\sum_{j=1}^n\lambda_{j}a_{j}
=\sum_{j=1}^\infty\lambda_{j}a_{j}
$$
in $({\rm CMO}^p_{L^2})'$. Moreover,
$$
\|f\|_{H^p}^p=\|S(f)\|_{L^p}^p\lesssim\lim_{n\to\infty}\|S(f_n)\|_{L^p}^p
\lesssim\sum_{j=1}^\infty|\lambda_{j}|^p.
$$
This finishes the proof of Proposition \ref{prop-hsubat}.
\end{proof}

To show  Proposition \ref{prop-at}, we need several lemmas.

\begin{lemma}\label{lem-basic1}
Let $\Omega\subset X_1\times X_2$ be an open set. Then the following assertions hold:
\begin{enumerate}
\item For any $(x_1,x_2)\in\Omega$, there exists $R\in\mathcal D$ such that $(x_1,x_2)\in R\subset\Omega$.
\item If $ Q_1\times Q_2\in\mathcal D$ such that $ \mu((Q_1\times Q_2)\cap\Omega)>\frac{1}{2}\mu(Q_1\times Q_2)$, then, for any fixed $\tau\in[1,\infty)$,
there exists a constant $c_{\tau}\in(0,1)$, independent of $R$ and $\Omega$, such that, for any
$(x_1,x_2)\in[\tau B_1(Q_1)]\times[\tau B_2(Q_2)]$,
$\mathcal M_{\rm str}(\mathbf 1_{\Omega})(x_1,x_2)>c_{\tau}$.
\end{enumerate}
\end{lemma}

\begin{proof}
We first prove (i). By the topology of $X_1\times X_2$, we only need to consider the single parameter case.
Indeed, let $\Omega_1\subset X_1$ be an open set and fix $x_1\in\Omega_1$. Then there exists  $r\in(0,\infty)$
such that $B_1(x_1,r)\subset\Omega$. By Lemma \ref{lem-cube}(a), we find that, for any $k_1\in\mathbb Z$, there exists
$\alpha_1\in\mathcal A_{1,k_1}$ such that $x_1\in Q_{1,\alpha_1}^{k_1}$. Moreover, by Lemma \ref{lem-cube}(c), we also have
$Q_{1,\alpha_1}^{k_1}\subset B_1(z_{1,\alpha_1}^{k_1},C^\natural\delta^{k_1})\subset B(x,2A_0C^{\natural}\delta^{k_1})$.
Choosing $k\in\mathbb Z$ sufficiently large such that $2A_0C^{\natural}\delta^{k_1}\le r$, we obtain
$x_1\in Q_{1,\alpha_1}^{k_1}\subset B_1(x_1,r)\subset\Omega_1$. This finishes the proof of (i).

Next, we show (ii). By Lemma \ref{lem-cube}(c) and \eqref{eq-doub}, for any
$(x_1,x_2)\in[\tau B_1(Q_1)]\times[\tau B_2(Q_2)]$, we find that
\begin{align*}
\mathcal M_{\rm str}(\mathbf 1_\Omega)(x_1,x_2)&\ge\frac 1{\mu_1(B_1(Q_{1,\alpha_{1}}^{k_{1}}))\mu_2(B_2(Q_{2,\alpha_{2}}^{k_{2}}))}
\int_{B_1(Q_{1,\alpha_{1}}^{k_{1}})}\int_{B_2(Q_{2,\alpha_{2}}^{k_{2}})}\mathbf 1_\Omega(y_1,y_2)\,d\mu_2(y_2)\,d\mu_1(y_1)\\
&=\frac{\mu_1(Q_{1,\alpha_{1}}^{k_{1}})}{\mu_1(B_1(Q_{1,\alpha_{1}}^{k_{1}}))}
\frac{\mu_2(Q_{2,\alpha_{2}}^{k_{2}})}{\mu_2(B_2(Q_{2,\alpha_{2}}^{k_{2}}))}
\frac{\mu(R\cap\Omega)}{\mu(R)}\ge c_0,
\end{align*}
where $c_0\in(0,1)$ is a constant independent of $R$ and $\Omega$.
This finishes the proof of (ii).

Finally, we prove (iii), Indeed, applying \eqref{eq-doub},
Theorem \ref{thm-wave}, and an argument similar to
that used in the proof of (ii), we conclude that
\begin{align*}
\mathcal M_{\rm str}(\mathbf 1_\Omega)(x_1,x_2)&\ge\frac 1{\mu_1(\tau B_1(Q_{1,\alpha_{1}}^{k_{1}}))\mu_2(\tau B_2(Q_{2,\alpha_{2}}^{k_{2}}))}
\int_{\tau B_1(Q_{1,\alpha_{1}}^{k_{1}})}\int_{\tau B_2(Q_{2,\alpha_{2}}^{k_{2}})}\mathbf 1_\Omega(y_1,y_2)\,d\mu_2(y_2)\,d\mu_1(y_1)\\
&=\frac{\mu_1(Q_{1,\alpha_{1}}^{k_{1}})}{\mu_1(\tau B_1(Q_{1,\alpha_{1}}^{k_{1}}))}\frac{\mu_2(Q_{2,\alpha_{2}}^{k_{2}})}{\mu_2(\tau B_2(Q_{2,\alpha_{2}}^{k_{2}}))}
\frac{\mu(R\cap\Omega)}{\mu(R)}\ge c_{\tau},
\end{align*}
where $c_{\tau}\in(0,1)$ is a constant only depending on $\tau$. This finishes the proof of (ii) and hence Lemma \ref{lem-basic1}.
\end{proof}

The next lemma is the well-known Journ\'{e}-type lemma. When $X_1$ and $X_2$ are RD-spaces, the proof of this
lemma was shown in \cite[Lemma 2.2]{hll16-Pisa}. That proof is also valid for any arbitrary spaces of homogeneous type $\{X_1, X_2\}$.
We omit the details here.

\begin{lemma}[the Journ\'{e}-type lemma]\label{lem-j}
Let $\Omega$ be an open subset of $X_1\times X_2$ with finite measure, $c_1\in(0,\infty)$, and $w:\ (0,\infty)\to(0,\infty)$ a
non-decreasing function satisfying $\sum_{j=0}^\infty w(\delta^{j})\le c_1$. Then there exists a constant $C\in(0,\infty)$,
only depending on $c_1$, such that
$$
\sum_{Q_1\times Q_2\in\mathcal M_1(\Omega)}\mu(Q_1\times Q_2)w\left(\frac{\ell(Q_2)}{\ell(\widehat{Q_2})}\right)
\le C\mu(\Omega)\textit{ and }
\sum_{Q_1\times Q_2\in\mathcal M_2(\Omega)}\mu(Q_1\times Q_2)w\left(\frac{\ell(Q_1)}{\ell(\widehat{Q_1})}\right)
\le C\mu(\Omega).
$$
Consequently, for any given $\epsilon\in(0,\infty)$, there exists a positive constant $C_{\epsilon}$, only depending on $\epsilon$,
such that
$$
\sum_{Q_1\times Q_2\in\mathcal M_1(\Omega)}\mu(Q_1\times Q_2)\left[\frac{\ell(Q_2)}{\ell(\widehat{Q_2})}\right]^\epsilon
\le C_{\epsilon}\mu(\Omega) \textit{ and }
\sum_{Q_1\times Q_2\in\mathcal M_2(\Omega)}\mu(Q_1\times Q_2)\left[\frac{\ell(Q_1)}{\ell(\widehat{Q_1})}\right]^\epsilon
\le C_{\epsilon}\mu(\Omega).
$$
\end{lemma}

The next lemma gives the wavelet characterization of $L^q$ when $q\in(1,\infty)$.

\begin{lemma}[{\cite[Theorem 4.13]{HYY24}}]\label{lem-bds}
Let $q\in(1,\infty)$. Then there exist constants $C_1,C_2\in(0,\infty)$ such that, for any $f\in L^q$,
$$
C_1\|f\|_{L^q}\le \|S(f)\|_{L^q}\le C_2\|f\|_{L^q}.
$$
\end{lemma}

We split the proof of Proposition \ref{prop-at} into two cases for
$q$: $q\in(1,2)$ and $q\in[2, \infty)$.
We first prove this proposition in the case $q\in[2,\infty)$.

\begin{proof}[Proof (I) of Proposition \ref{prop-at} With $q\in[2,\infty)$]
Let $p$ be   the same as in the present proposition, $q\in[2,\infty)$, $\Omega$ be an open subset of $X_1\times X_2$ with $\mu(\Omega)<\infty$,
and $a$ a $(p, q, C_0)$-atom supported in $\Omega$.
Let
$$
\Omega^*:=\{(x_1,x_2)\in X_1\times X_2:\ \mathcal M_{\rm str}(\mathbf 1_{\Omega})>c\}
$$
and
$$
\Omega^{**}:=\{(x_1,x_2)\in X_1\times X_2:\ \mathcal M_{\rm str}(\mathbf 1_{\Omega^*})>c\},
$$
where $c:=\min\{c_0,c_{\tau}\}$ with $c_0$ and $c_{\tau}$  as in Lemma \ref{lem-basic1} and
$\tau\in(0,\infty)$ determined later. By the boundedness
of $\mathcal M_{\rm str}$ on $L^2$ (see Proposition \ref{prop-fs}), we find that
$
\mu(\Omega^{**})\sim\mu(\Omega^*)\sim\mu(\Omega).
$
Using this, H\"{o}lder's inequality, and Lemma \ref{lem-bds}, we obtain
\begin{equation}\label{eq-sa1}
\left\|S(a)\mathbf 1_{\Omega^{**}}\right\|_{L^p}^p\le\|S(a)\|_{L^q}^{p}
\left[\mu\left(\Omega^{**}\right)\right]^{1-\frac pq}\lesssim[\mu(\Omega)]^{\frac pq-1}[\mu(\Omega)]^{1-\frac pq}\sim 1.
\end{equation}

Now, we estimate $\|S(a)\mathbf 1_{(\Omega^{**})^\complement}\|_{L^p}$.
To achieve this, by the definition of $(p, q, C_0)$-atoms and Minkowski's inequality, we have
$S(a)\le\sum_{R\in\mathcal M(\Omega)} S(a_R)$, which, together with $p \in (0, 1]$, further implies that
$$
\left\|S(a)\mathbf 1_{(\Omega^{**})^\complement}\right\|_{L^p}^p
\le\sum_{R\in\mathcal M(\Omega)}\left\|S(a_R)\mathbf 1_{(\Omega^{**})^\complement}\right\|_{L^p}^p.
$$
We fix $R=Q_1\times Q_2\in\mathcal M(\Omega)$. Using the fact that $\mathcal M(\Omega)\subset\mathcal M_1(\Omega)$, we find that
$$
\mu\left(\left(Q_1\times\widehat{Q_2}\right)\cap\Omega\right)>\frac 12
\mu\left(Q_1\times\widehat{Q_2}\right).
$$
By this and Lemma \ref{lem-basic1},
we conclude that $Q_1\times\widehat{Q_2}\subset\Omega^*\subset\Omega^{**}$. Similarly, we have
$\widehat{Q_1}\times Q_2\subset\Omega^{**}$. From this, we further infer that $(\Omega^{**})^\complement
\subset(Q_1\times\widehat{Q_2})^\complement\cap(\widehat{Q_1}\times Q_2)^\complement$. Now, let
$(x_1,x_2)\in(\Omega^{**})^\complement$. Then we have $(x_1,x_2)\in (Q_1\times\widehat{Q_2})^\complement$.
We assume $Q_1=Q_{1,\beta_1}^{j_1}$, $Q_2=Q_{2,\beta_2}^{j_2}$, and $\widehat{Q_2}=Q_{2,\beta_2'}^{j_2'}$ for some
$j_1, j_2, j_2'\in\mathbb Z$, $\beta_1\in\mathcal A_{1,j_1}$, $\beta_2\in\mathcal A_{2,j_2}$, and $\beta_2'\in\mathcal A_{2,j_2'}$.
To estimate $S(a_R)(x_1,x_2)$, we consider the following three cases for $x_1$ and $x_2$.

{\it Case (I.1) $x_1\in Q_1^\complement$ and $x_2\in(\widehat{Q_2})^\complement$.} In this case, we write
\begin{align*}
[S(a_R)(x_1,x_2)]^2=& \sum_{\genfrac{}{}{0pt}{}{k_1\in{\mathbb Z}}{\alpha_1\in\mathcal G_{1,k_1}}}
\sum_{\genfrac{}{}{0pt}{}{k_2\in{\mathbb Z}}{\alpha_2\in\mathcal G_{2,k_2}}}
\left|\left<a_R,\frac{\psi_{\alpha_1,\alpha_2}^{k_1,k_2}}
{\sqrt{\mu_1(Q_{1,\alpha_1}^{k_1+1})}\sqrt{\mu_2(Q_{2,\alpha_2}^{k_2+1})}}\right>\right|^2
\mathbf 1_{Q_{\alpha_1,\alpha_2}^{k_1+1,k_2+1}}(x_1,x_2).
\end{align*}
Fix $k_1, k_2\in\mathbb Z$, $x_1\in Q_{1,\alpha_1}^{k_1+1}$, and $x_2\in Q_{2,\alpha_2}^{k_2+1}$ for some
$\alpha_1\in\mathcal G_{1,k_1}$ and $\alpha_2\in\mathcal G_{2,k_2}$. Now, we consider the following four cases for the
relations between $k_1$ and $j_1$  and  between $k_2$ and $j_2$.

{\it Case (I.1.1) $k_1\le j_1$ and $k_2\le j_2$.} In this case, we have $\delta^{k_1}\ge\delta^{j_1}$ and
$\delta^{k_2}\ge\delta^{j_2}$. Since $d_1(x_1,z_{1,\alpha_{1}}^{k_{1}})\lesssim\delta^{k_1}$ and $d_2(x_2,z_{2,\alpha_{2}}^{k_{2}})\lesssim\delta^{k_2}$,
it then follows that, for any $y_1\in Q_1$ and $y_2\in Q_2$,
$$
\delta^{k_1}+d_1\left(z_{1,\alpha_{1}}^{k_{1}},y_{1,\beta_{1}}^{j_{1}}\right)
\sim\delta^{k_1}+d_1\left(x_1,y_{1,\beta_{1}}^{j_{1}}\right)
$$
and
$$
\delta^{k_2}+d_2\left(z_{2,\alpha_{2}}^{k_{2}},y_{2,\beta_{2}}^{j_{2}}\right)
\sim\delta^{k_2}+d_2\left(x_2,y_{2,\beta_{2}}^{j_{2}}\right).
$$
Using this and the cancellation of $a_R$, we conclude that, for any $\alpha_1\in\mathcal G_{1,k_1}$ and $\alpha_2\in\mathcal G_{2,k_2}$,
\begin{align*}
&\left|\left<a_R,\frac{\psi_{\alpha_1,\alpha_2}^{k_1,k_2}}
{\sqrt{\mu_1(Q_{1,\alpha_1}^{k_1+1})}\sqrt{\mu_2(Q_{2,\alpha_2}^{k_2+1})}}\right>\right|\nonumber\\
&\quad\le\int_{Q_1}\int_{Q_2}|a_R(y_1,y_2)|
\frac{|\psi_{1,\alpha_1}^{k_1}(y_1)-\psi_{1,\alpha_1}^{k_1}(z_{1,\beta_1}^{j_1})|}
{\sqrt{\mu_1(Q_{1,\alpha_1}^{k_1+1})}}
\frac{|\psi_{2,\alpha_2}^{k_2}(y_2)-\psi_{2,\alpha_2}^{k_2}(z_{2,\beta_2}^{j_2})|}
{\sqrt{\mu_2(Q_{2,\alpha_2}^{k_2+1})}}\,d\mu_2(y_2)\,d\mu_1(y_1)\nonumber\\
&\quad\lesssim\int_{Q_1}\int_{Q_2}\left[\frac{d_1(y_1,z_{1,\beta_1}^{j_1})}{\delta^{k_1}}\right]^{\eta}
\left[\frac{d_2(y_2,z_{2,\beta_2}^{j_2})}{\delta^{k_2}}\right]^{\eta}\widetilde{\mathcal E}_{1,k_1}^*
\left(y_{1,\alpha_{1}}^{k_{1}},z_{1,\beta_{1}}^{j_{1}}\right)
\widetilde{\mathcal E}_{2,k_2}^*\left(y_{2,\alpha_{2}}^{k_{2}},z_{2,\beta_{2}}^{j_{2}}\right)\nonumber\\
&\qquad\times |a_R(y_1,y_2)|\,d\mu_2(y_2)\,d\mu_1(y_1)\nonumber\\
&\quad\lesssim\delta^{(j_1-k_1)\eta}\delta^{(j_2-k_2)\eta}\widetilde{\mathcal E}_{1,k_1}^*\left(x_1,z_{1,\beta_{1}}^{j_{1}}\right)
\widetilde{\mathcal E}_{2,k_2}^*\left(x_2,z_{2,\beta_{2}}^{j_{2}}\right)\|a_R\|_{L^1}\nonumber\\
&\quad\lesssim\delta^{(j_1-k_1)\eta}\delta^{(j_2-k_2)\eta}\widetilde{\mathcal E}_{1,k_1}^*\left(x_1,z_{1,\beta_{1}}^{j_{1}}\right)
\widetilde{\mathcal E}_{2,k_2}^*\left(x_2,z_{2,\beta_{2}}^{j_{2}}\right)
\left[\mu_1\left(Q_{1,\beta_1}^{j_1+1}\right)\mu_2\left(Q_{2,\beta_2}^{j_2+1}\right)\right]^{1- \frac{1}{q}} \|a_R\|_{L^q};
\end{align*}
here and thereafter, for any $i\in\{1,2\}$ and $k_i\in\mathbb Z$, $\widetilde{\mathcal E}_{i,k_i}^*$
is the same as in \eqref{eq-defek1} associated with $X_i$.
By this, we find that, for any fixed $\eta'\in(0,\eta)$,
\begin{align}\label{eq-1.1-1r}
&\sum_{k_1=-\infty}^{j_1}\sum_{\alpha_1\in\mathcal G_{1,k_1}}\sum_{k_2=-\infty}^{j_2}\sum_{\alpha_2\in\mathcal G_{2,k_2}}
\left|\left<a_R,\frac{\psi_{\alpha_1,\alpha_2}^{k_1,k_2}}
{\sqrt{\mu_1(Q_{1,\alpha_1}^{k_1+1})}\sqrt{\mu_2(Q_{2,\alpha_2}^{k_2+1})}}\right>\right|^2
\mathbf 1_{Q_{1,\alpha_1}^{k_1+1}}(x_1)\mathbf 1_{Q_{2,\alpha_2}^{k_2+1}}(x_2)\nonumber\\
&\quad\lesssim\left[\mu_1\left(Q_{1,\beta_1}^{j_1}\right)\mu_2
\left(Q_{2,\beta_2}^{j_2}\right)\right]^{2-\frac 2q}\|a_R\|_{L^q}^2\nonumber\\
&\qquad\times\sum_{k_1=-\infty}^{j_1}\sum_{k_2=-\infty}^{j_2}\delta^{2(j_1-k_1)\eta}\delta^{2(j_2-k_2)\eta}
\left[\widetilde{\mathcal E}_{1,k_1}^*\left(x_1,z_{1,\beta_{1}}^{j_{1}}\right)\widetilde{\mathcal E}_{2,k_2}^*\left(x_2,z_{2,\beta_{2}}^{j_{2}}\right)\right]^2\nonumber\\
&\quad\lesssim\left[\mu_1\left(Q_{1,\beta_1}^{j_1}\right)\mu_2\left(Q_{2,\beta_2}^{j_2}
\right)\right]^{2-\frac 2q}\|a_R\|_{L^q}^2
\left[\frac 1{V_1(x_1,z_{1,\beta_1}^{j_1})}\right]^2\left[\frac 1{V_2(x_2,z_{2,\beta_2}^{j_2})}\right]^2\nonumber\\
&\qquad\times\sum_{k_1=-\infty}^{j_1}\sum_{k_2=-\infty}^{j_2}\delta^{2(j_1-k_1)\eta}\delta^{2(j_2-k_2)\eta}
\left[\frac{\delta^{k_1}}{d_1(x_1,z_{1,\beta_1}^{j_1})}\right]^{2\eta'}
\left[\frac{\delta^{k_2}}{d_2(x_2,z_{2,\beta_2}^{j_2})}\right]^{2\eta'}\nonumber\\
&\quad\lesssim\left[\mu_1\left(Q_{1,\beta_1}^{j_1}\right)\mu_2\left(Q_{2,\beta_2}^{j_2}
\right)\right]^{2-\frac 2q}\|a_R\|_{L^q}^2
\left[\frac 1{V_1(x_1,z_{1,\beta_1}^{j_1})}\right]^2\left[\frac 1{V_2(x_2,z_{2,\beta_2}^{j_2})}\right]^2\nonumber\\
&\qquad\times\left[\frac{\delta^{j_1}}{d_1(x_1,z_{1,\beta_1}^{j_1})}\right]^{2\eta'}
\left[\frac{\delta^{j_2}}{d_2(x_2,z_{2,\beta_2}^{j_2})}\right]^{2\eta'}.
\end{align}
This is the desired estimate in this case.

{\it Case (I.1.2) $k_1\le j_1$ and $k_2>j_2$.} In this case, we have $\delta^{k_1}\ge\delta^{j_1}$ and
$\delta^{k_2}<\delta^{j_2}$. Applying an argument similar to that used in the proof of Case
(I.1.1), we obtain
$$
\delta^{k_1}+d_1\left(z_{1,\alpha_{1}}^{k_{1}},y_{1,\beta_{1}}^{j_{1}}\right)
\sim\delta^{k_1}+d_1\left(x_1,y_{1,\beta_{1}}^{j_{1}}\right).
$$
By this, the cancellation of $a_R$, and Theorem \ref{thm-wave}, we find that, for any given
$\Gamma\in(0,\infty)$,
\begin{align}\label{eq-1.1-2}
&\left|\left<a_R,\frac{\psi_{\alpha_1,\alpha_2}^{k_1,k_2}}
{\sqrt{\mu_1(Q_{1,\alpha_1}^{k_1+1})}\sqrt{\mu_2(Q_{2,\alpha_2}^{k_2+1})}}\right>\right|\nonumber\\
&\quad\le\int_{Q_1}\int_{Q_2}|a_R(y_1,y_2)|
\frac{|\psi_{1,\alpha_1}^{k_1}(y_1)-\psi_{1,\alpha_1}^{k_1}(z_{1,\beta_1}^{j_1})|}
{\sqrt{\mu_1(Q_{1,\alpha_1}^{k_1+1})}}
\frac{|\psi_{2,\alpha_2}^{k_2}(y_2)|}{\sqrt{\mu_2(Q_{2,\alpha_2}^{k_2+1})}}
\,d\mu_2(y_2)\,d\mu_1(y_1)\nonumber\\
&\quad\lesssim\int_{Q_1}\int_{Q_2}\left[\frac{d_1(y_1,z_{1,\beta_1}^{j_1})}{\delta^{k_1}}\right]^{\eta}
\widetilde{\mathcal E}^*_{1,k_1}\left(y_{1,\alpha_{1}}^{k_{1}},y_1\right)\widetilde{\mathcal E}^*_{2,k_2}\left(y_{2,\alpha_{2}}^{k_{2}},y_2\right)
|a_R(y_1,y_2)|\,d\mu(y_2)\,d\mu(y_1)\nonumber\\
&\quad\lesssim\delta^{(j_1-k_1)\eta}\delta^{(k_2-j_2)\Gamma}P_{1,\Gamma}
\left(x_1,z_{1,\beta_1}^{j_1};\delta^{k_1}\right)
P_{2,\Gamma}\left(x_2,z_{2,\beta_2}^{j_2};\delta^{j_2}\right)\|a\|_{L^1}\nonumber\\
&\quad\lesssim\delta^{(j_1-k_1)\eta}\delta^{(k_2-j_2)\Gamma}P_{1,\Gamma}
\left(x_1,z_{1,\beta_1}^{j_1};\delta^{k_1}\right)
P_{2,\Gamma}\left(x_2,z_{2,\beta_2}^{j_2};\delta^{j_2}\right) \left[\mu_1\left(Q_{1,\beta_1}^{j_1}\right)\mu_2\left(Q_{2,\beta_2}^{j_2}\right)\right]^{1-\frac 1q}
\|a_R\|_{L^q},
\end{align}
where we used the fact that, for any $y_2\in Q_{2,\beta_2}^{j_2}$,
$d_2(x_2,z_{2,\beta_2}^{j_2})>2A_0d_2(y_2,z_{2,\beta_2}^{j_2})$. Using this, we conclude that,
for any fixed $\eta'\in(0,\eta)$,
\begin{align}\label{eq-1.1-2r}
&\sum_{k_1=-\infty}^{j_1}\sum_{\alpha_1\in\mathcal G_{1,k_1}}\sum_{k_2=j_2+1}^\infty\sum_{\alpha_2\in\mathcal G_{2,k_2}}
\left|\left<a_R,\frac{\psi_{\alpha_1,\alpha_2}^{k_1,k_2}}
{\sqrt{\mu_1(Q_{1,\alpha_1}^{k_1+1})}\sqrt{\mu_2(Q_{2,\alpha_2}^{k_2+1})}}\right>\right|^2
\mathbf 1_{Q_{1,\alpha_1}^{k_1+1}}(x_1)\mathbf 1_{Q_{2,\alpha_2}^{k_2+1}}(x_2)\nonumber\\
&\quad\lesssim\left[\mu_1\left(Q_{1,\beta_1}^{j_1}\right)\mu_2
\left(Q_{2,\beta_2}^{j_2}\right)\right]^{2-\frac 2q}\|a_R\|_{L^q}^2\nonumber\\
&\qquad\times\sum_{k_1=-\infty}^{j_1}\sum_{k_2=j_2+1}^\infty\delta^{2(j_1-k_1)\eta}\delta^{2(k_2-j_2)\Gamma}
\left[P_{1,\Gamma}\left(x_1,z_{1,\beta_1}^{j_1};\delta^{k_1}\right)\right]^2
\left[P_{2,\Gamma}\left(x_2,z_{2,\beta_2}^{j_2};\delta^{j_2}\right)\right]^2\nonumber\\
&\quad\lesssim\left[\mu_1\left(Q_{1,\beta_1}^{j_1}\right)
\mu_2\left(Q_{2,\beta_2}^{j_2}\right)\right]^{2-\frac 2q}\|a_R\|_{L^q}^2
\left[\frac 1{V_1(x_1,z_{1,\beta_1}^{j_1})}\right]^2\left[\frac 1{V_2(x_2,z_{2,\beta_2}^{j_2})}\right]^2\nonumber\\
&\qquad\times\sum_{k_1=-\infty}^{j_1}\sum_{k_2=j_2+1}^\infty\delta^{2(j_1-k_1)\eta}\delta^{2(k_2-j_2)\Gamma}
\left[\frac{\delta^{k_1}}{d_1(x_1,z_{1,\beta_1}^{j_1})}\right]^{2\eta'}
\left[\frac{\delta^{j_2}}{d_2(x_2,z_{2,\beta_2}^{j_2})}\right]^{2\eta'}\nonumber\\
&\quad\lesssim\left[\mu_1\left(Q_{1,\beta_1}^{j_1}\right)
\mu_2\left(Q_{2,\beta_2}^{j_2}\right)\right]^{2-\frac 2q}\|a_R\|_{L^q}^2
\left[\frac 1{V_1(x_1,z_{1,\beta_1}^{j_1})}\right]^2\left[\frac 1{V_2(x_2,z_{2,\beta_2}^{j_2})}\right]^2\nonumber\\
&\qquad\times\left[\frac{\delta^{j_1}}{d_1(x_1,z_{1,\beta_1}^{j_1})}\right]^{2\eta'}
\left[\frac{\delta^{j_2}}{d_2(x_2,z_{2,\beta_2}^{j_2})}\right]^{2\eta'}
\sum_{k_1=-\infty}^{j_1}\delta^{2(j_1-k_1)(\eta-\eta')}
\sum_{k_2=j_2+1}^\infty\delta^{2(k_2-j_2)\Gamma}\nonumber\\
&\quad\sim\left[\mu_1\left(Q_{1,\beta_1}^{j_1}\right)
\mu_2\left(Q_{2,\beta_2}^{j_2}\right)\right]^{2-\frac 2q}\|a_R\|_{L^q}^2
\left[\frac 1{V_1(x_1,z_{1,\beta_1}^{j_1})}\right]^2\left[\frac 1{V_2(x_2,z_{2,\beta_2}^{j_2})}\right]^2\nonumber\\
&\qquad\times\left[\frac{\delta^{j_1}}{d_1(x_1,z_{1,\beta_1}^{j_1})}\right]^{2\eta'}
\left[\frac{\delta^{j_2}}{d_2(x_2,z_{2,\beta_2}^{j_2})}\right]^{2\eta'}.
\end{align}
This is also the desired estimate.

{\it Case (I.1.3) $k_1>j_1$ and $k_1\le j_2$.} In this case, applying an argument similar to that used in the estimation of
\eqref{eq-1.1-2}, we obtain, when $k_1>j_1$ and $k_2\le j_2$,
\begin{align*}
\left|\left<a_R,\frac{\psi_{\alpha_1,\alpha_2}^{k_1,k_2}}
{\sqrt{\mu_1(Q_{\alpha_1}^{k_1+1})}\sqrt{\mu_2(Q_{\alpha_2}^{k_2+1})}}\right>\right|
&\lesssim\delta^{(j_2-k_2)\eta}\left[\mu_1\left(Q_{1,\beta_1}^{j_1}
\right)\mu_2\left(Q_{2,\beta_2}^{j_2}\right)\right]^{1-\frac 1q}\|a_R\|_{L^q}\nonumber\\
&\quad\times\widetilde{\mathcal E}^*_{1,k_1}\left(x_1,z_{1,\beta_{1}}^{j_{1}}\right)
\widetilde{\mathcal E}^*_{2,k_2}\left(x_2,z_{2,\beta_{2}}^{j_{2}}\right),
\end{align*}
which further implies that, for any given $\eta'\in(0,\eta)$,
\begin{align}\label{eq-1.1-3r}
&\sum_{k_1=j_1+1}^\infty\sum_{\alpha_1\in\mathcal G_{1,k_1}}\sum_{k_2=-\infty}^{j_2}\sum_{\alpha_2\in\mathcal G_{2,k_2}}
\left|\left<a_R,\frac{\psi_{\alpha_1,\alpha_2}^{k_1,k_2}}
{\sqrt{\mu_1(Q_{1,\alpha_1}^{k_1+1})}\sqrt{\mu_2(Q_{2,\alpha_2}^{k_2+1})}}\right>\right|^2
\mathbf 1_{Q_{1,\alpha_1}^{k_1+1}}(x_1)\mathbf 1_{Q_{2,\alpha_2}^{k_2+1}}(x_2)\nonumber\\
&\quad\lesssim\left[\mu_1\left(Q_{1,\beta_1}^{j_1}\right)\mu_2\left(Q_{2,\beta_2}^{j_2}
\right)\right]^{2-\frac 2q}\|a_R\|_{L^q}^2
\left[\frac 1{V_1(x_1,z_{1,\beta_1}^{j_1})}\right]^2\left[\frac 1{V_2(x_2,z_{2,\beta_2}^{j_2})}\right]^2\nonumber\\
&\qquad\times\left[\frac{\delta^{j_1}}{d_1(x_1,z_{1,\beta_1}^{j_1})}\right]^{2\eta'}
\left[\frac{\delta^{j_2}}{d_2(x_2,z_{2,\beta_2}^{j_2})}\right]^{2\eta'}.
\end{align}
This is also the desired estimate.

{\it Case (I.1.4) $k_1>j_1$ and $k_2>j_2$.} In this case, by the size conditions of $a$ and
$\psi_{\alpha_1,\alpha_2}^{k_1,k_2}$, we have, for any given $\Gamma\in(0,\infty)$,
\begin{align*}
&\left|\left<a_R,\frac{\psi_{\alpha_1,\alpha_2}^{k_1,k_2}}
{\sqrt{\mu_1(Q_{1,\alpha_1}^{k_1+1})}\sqrt{\mu_2(Q_{2,\alpha_2}^{k_2+1})}}\right>\right|\nonumber\\
&\quad\le\int_{Q_1}\int_{Q_2}|a_R(y_1,y_2)|
\frac{\big|\psi_{1,\alpha_1}^{k_1}(y_1)\big|}{\sqrt{\mu_1(Q_{1,\alpha_1}^{k_1+1})}}
\frac{\big|\psi_{2,\alpha_2}^{k_2}(y_2)\big|}{\sqrt{\mu_2(Q_{2,\alpha_2}^{k_2+1})}}
\,d\mu_2(y_2)\,d\mu_1(y_1)\nonumber\\
&\quad\lesssim\int_{Q_1}\int_{Q_2}P_{1,\Gamma}\left(x_1,y_1;\delta^{k_1}\right)
P_{2,\Gamma}\left(x_2,y_2;\delta^{k_2}\right)\,d\mu_2(y_2)\,d\mu_1(y_1)\nonumber\\
&\quad\lesssim\delta^{(k_1-j_1)\Gamma}\delta^{(k_2-j_2)\Gamma}P_{1,\Gamma}
\left(x_1,z_{1,\beta_1}^{j_1};\delta^{j_1}\right)
P_{2,\Gamma}\left(x_2,z_{2,\beta_2}^{j_2};\delta^{j_2}\right)\|a\|_{L^1}\nonumber\\
&\quad\lesssim\delta^{(j_1-k_1)\Gamma}\delta^{(j_2-k_2)\Gamma}P_{1,\Gamma}
\left(x_1,z_{1,\beta_1}^{j_1};\delta^{j_1}\right)
P_{2,\Gamma}\left(x_2,z_{2,\beta_2}^{j_2};\delta^{j_2}\right)\nonumber\\
&\qquad\times\left[\mu_1\left(Q_{1,\beta_1}^{j_1}\right)\mu_2
\left(Q_{2,\beta_2}^{j_2}\right)\right]^{1-\frac 1q}\|a_R\|_{L^q}.
\end{align*}
From this, we further deduce that, for any fixed $\eta'\in(0,\eta)$,
\begin{align}\label{eq-1.1-4r}
&\sum_{k_1=j_1+1}^\infty\sum_{\alpha_1\in\mathcal G_{1,k_1}}\sum_{k_2=j_2+1}^\infty\sum_{\alpha_2\in\mathcal G_{2,k_2}}
\left|\left<a_R,\frac{\psi_{\alpha_1,\alpha_2}^{k_1,k_2}}
{\sqrt{\mu_1(Q_{1,\alpha_1}^{k_1+1})}\sqrt{\mu_2(Q_{2,\alpha_2}^{k_2+1})}}\right>\right|^2
\mathbf 1_{Q_{1,\alpha_1}^{k_1+1}}(x_1)\mathbf 1_{Q_{2,\alpha_2}^{k_2+1}}(x_2)\nonumber\\
&\quad\lesssim\left[\mu_1\left(Q_{1,\beta_1}^{j_1}\right)\mu_2\left(Q_{2,\beta_2}^{j_2}
\right)\right]^{2-\frac 2q}\|a_R\|_{L^q}^2
\left[\frac 1{V_1(x_1,z_{1,\beta_1}^{j_1})}\right]^2\left[\frac 1{V_2(x_2,z_{2,\beta_2}^{j_2})}\right]^2\nonumber\\
&\qquad\times\left[\frac{\delta^{j_1}}{d_1(x_1,z_{1,\beta_1}^{j_1})}\right]^{2\eta'}
\left[\frac{\delta^{j_2}}{d_2(x_2,z_{2,\beta_2}^{j_2})}\right]^{2\eta'}.
\end{align}

Combining \eqref{eq-1.1-1r}, \eqref{eq-1.1-2r}, \eqref{eq-1.1-3r}, and \eqref{eq-1.1-4r}, we conclude that,
for any given $\eta'\in(0,\eta)$ and $(x_1,x_2)\in Q_1^\complement\times(\widehat{Q_2})^\complement$,
\begin{align}\label{eq-sarp1}
S(a_R)(x_1,x_2)&\lesssim\left[\mu_1\left(Q_{1,\beta_1}^{j_1}\right)\mu_2\left(Q_{2,\beta_2}^{j_2}
\right)\right]^{1-\frac 1q}\|a_R\|_{L^q}
\frac 1{V_1(x_1,z_{1,\beta_1}^{j_1})}\frac 1{V_2(x_2,z_{2,\beta_2}^{j_2})}\nonumber\\
&\quad\times\left[\frac{\delta^{j_1}}{d_1(x_1,z_{1,\beta_1}^{j_1})}\right]^{\eta'}
\left[\frac{\delta^{j_2}}{d_2(x_2,z_{2,\beta_2}^{j_2})}\right]^{\eta'}.
\end{align}
Now, let $R_{c,1}:=[Q_1^\complement\times(\widehat{Q_2})^\complement]\cap\Omega^{**}$. Due to the assumption that
$p>\frac{\omega}{\omega+\eta}$, we choose $\eta'$ such that $p>\frac{\omega}{\omega+\eta'}$. Using this and
\eqref{eq-sarp1}, we obtain
\begin{align}\label{eq-sar1}
 \left\|S(a_R)\mathbf{1}_{R_{c,1}}\right\|_{L^p}^p
&\lesssim\left[\mu_1\left(Q_{1,\beta_1}^{j_1}\right)\mu_2\left(Q_{2,\beta_2}^{j_2}\right)\right]^{p-\frac pq}\|a_R\|_{L^q}^p\nonumber\\
&\quad\times\int_{d(x_1,z_{1,\beta_1}^{j_1})>C\delta^{j_1}}\left[\frac 1{V_1(x_1,z_{1,\beta_1}^{j_1})}\right]^p
\left[\frac{\delta^{j_1}}{d_1(x_1,z_{1,\beta_1}^{j_1})}\right]^{p\eta'}\,d\mu_1(x_1)\nonumber\\
&\quad\times\int_{d(x_2,z_{2,\beta_2}^{j_2})>C\delta^{j_2'}}\left[\frac 1{V_2(x_2,z_{2,\beta_{2}}^{j_{2}})}\right]^p
\left[\frac{\delta^{j_2}}{d_1(x_2,z_{2,\beta_2}^{j_2})}\right]^{p\eta'}\,d\mu_2(x_2)\nonumber\\
&\lesssim\left[\mu_1\left(Q_{1,\beta_1}^{j_1}\right)\right]^{1-\frac pq}\delta^{(j_2-j_2')p\eta'}
\left[\mu_2\left(Q_{2,\beta_2}^{j_2}\right)\right]^{p-\frac pq}\left[V_2\left(z_{2,\beta_2}^{j_2},\delta^{j_2'}\right)\right]^{1-p}
\|a_R\|_{L^q}^p\nonumber\\
&\lesssim\left[\mu_1\left(Q_{1,\beta_1}^{j_1}\right)\mu_2\left(Q_{2,\beta_2}^{j_2}\right)\right]^{1-\frac pq}
\delta^{(j_2-j_2')[p\eta'-\omega(1-p)]}\|a_R\|_{L^q}^p\nonumber\\
&\sim[\mu(R)]^{1-\frac qp}\left[\frac{\ell(Q_2)}{\ell(\widehat{Q_2})}\right]^{p\eta'-\omega(1-p)}
\|a_R\|_{L^q}^p.
\end{align}
This is the desired case for Case (I.1).

{\it Case (I.2) $x_1\in Q_1$ and $x_2\in(\widehat{Q_2})^\complement$.} In this case, for any $k_1\in\mathbb Z$,
$\alpha_1\in\mathcal G_{1,k_1}$, and $y_2\in X_2$, let
$$
F_{\alpha_1}^{k_1}(y_2):=\int_{X_1}a_R(y_1,y_2)\frac{\psi_{1,\alpha_1}^{k_1}(y_1)}
{\sqrt{\mu_1(Q_{1,\alpha_1}^{k_1+1})}}
\,d\mu_1(y_1).
$$
By this and Fubini's theorem, we have, for any $k_1,k_2\in\mathbb Z$, 
$\alpha_1\in\mathcal G_{1,k_1}$, and $\alpha_2\in\mathcal G_{2,k_2}$,
\begin{align*}
&\left<a_R,\frac{\psi_{\alpha_1,\alpha_2}^{k_1,k_2}}
{\sqrt{\mu_1(Q_{1,\alpha_1}^{k_1+1})}\sqrt{\mu_2(Q_{2,\alpha_2}^{k_2+1})}}\right>\\
&\quad=\int_{X_2}\int_{X_1} a_R(y_1,y_2)
\frac{\psi_{1,\alpha_1}^{k_1}(y_1)}{\sqrt{\mu_1(Q_{1,\alpha_1}^{k_1+1})}}\,d\mu_1(y_1)
\frac{\psi_{2,\alpha_2}^{k_2}(y_2)}{\sqrt{\mu_2(Q_{2,\alpha_2}^{k_2+1})}}\,d\mu_2(y_2)\\
&\quad=\int_{Q_{2,\beta_{2}}^{j_{2}}} F_{\alpha_1}^{k_1}(y_2)\frac{\psi_{2,\alpha_2}^{k_2}(y_2)}{\sqrt{\mu_2(Q_{2,\alpha_2}^{k_2+1})}
}\,d\mu_2(y_2)=:\mathrm J_{\alpha_1,\alpha_2}^{k_1,k_2}.
\end{align*}
We separate $k_2\in\mathbb Z$ into two cases.

{\it Case (I.2.1) $k_2\le j_2$.} In this case, we have $\delta^{k_2}\ge\delta^{j_2}$. Note that, from Fubini's theorem and
the cancellation of $a_R$, we infer that, for any $k_1\in\mathbb Z$ and $\alpha_1\in\mathcal G_{1,k_1}$,
$\int_{X_2}F_{\alpha_1}^{k_1}(y_2)\,d\mu(y_2)=0$.
Applying this and the regularity condition of $\psi_{\alpha_1,\alpha_2}^{k_1,k_2}$, we conclude that, for any
$\alpha_1\in\mathcal G_{1,k_1}$ and $\alpha_2\in\mathcal G_{2,k_2}$,
\begin{align*}
\left|\mathrm J_{\alpha_1,\alpha_2}^{k_1,k_2}\right|&\le\int_{X_2}\left|F_{\alpha_1}^{k_1}(y_2)\right|
\frac{|\psi_{2,\alpha_2}^{k_2}(y_2)-\psi_{2,\alpha_2}^{k_2}(z_{2,\beta_2}^{j_2})|}
{\sqrt{\mu_2(Q_{\alpha_2}^{k_2+1})}}\,d\mu_2(y_2)\\
&\lesssim\int_{X_2}\left|F_{\alpha_1}^{k_1}(y_2)\right|\left[\frac{d_2(y_2,z_{2,
\beta_2}^{j_2})}{\delta^{k_2}}\right]^{\eta}
\widetilde{\mathcal E}^*_{2,k_2}\left(y_{2,\alpha_{2}}^{k_{2}},z_{2,\beta_{2}}^{j_{2}}\right)\,d\mu_2(y_2)\\
&\lesssim\delta^{(j_2-k_2)\eta}\widetilde{\mathcal E}^*_{2,k_2}\left(x_2,z_{2,\beta_{2}}^{j_{2}}\right)\int_{X_2}
\left|F_{\alpha_1}^{k_1}(y_2)\right|\,d\mu_2(y_2).
\end{align*}
Using this, we further obtain, for any fixed $\eta'\in(0,\eta)$,
\begin{align}\label{eq-1.2.1r}
&\sum_{\genfrac{}{}{0pt}{}{k_1\in{\mathbb Z}}{\alpha_1\in\mathcal G_{1,k_1}}}
\sum_{k_2=-\infty}^{j_2}\sum_{\alpha_2\in\mathcal G_{2,k_2}}
\left|\mathrm J_{\alpha_1,\alpha_2}^{k_1,k_2}\right|^2\mathbf 1_{Q_{1,\alpha_1}^{k_1+1}}(x_1)
\mathbf 1_{Q_{2,\alpha_2}^{k_2+1}}(x_2)\nonumber\\
&\quad\lesssim\sum_{k_2\le j_2}\delta^{2(j_2-k_2)\eta}
\left[\widetilde{\mathcal E}^*_{2,k_2}\left(x_2,z_{2,\beta_{2}}^{j_{2}}\right)\right]^2
\sum_{k_1\in\mathbb Z}\sum_{\alpha_1\in\mathcal G_{1,k_1}}\left[\int_{X_2}\left|F_{\alpha_1}^{k_1}(y_2)\right|
\,d\mu_2(y_2)\right]^2\mathbf 1_{Q_{1,\alpha_1}^{k_1+1}}(x_1)\nonumber\\
&\quad\lesssim\left[\frac 1{V_2(x_2,z_{2,\beta_2}^{j_2})}\right]^2
\left[\frac{\delta^{j_2}}{d_2(x_2,z_{2,\beta_2}^{j_2})}\right]^{2\eta'}
\mu_2\left(Q_{2,\beta_2}^{j_2}\right)
\int_{Q_{2,\beta_{2}}^{j_{2}}}\sum_{k_1,\alpha_1}\left|F_{\alpha_1}^{k_1}(y_2)\right|^2
\mathbf 1_{Q_{1,\alpha_1}^{k_1+1}}(x_1)\,d\mu_2(y_2).
\end{align}
This is the desired estimate in this case.

{\it Case (I.2.2) $k_2>j_2$.} In this case, we have $\delta^{k_2}<\delta^{j_2}$. Since, for any $y_2\in Q_{2,\beta_2}^{k_2}$
$d_2(x_2,z_{2,\beta_2}^{j_2})>2A_0d_2(y_2,z_{2,\beta_2}^{j_2})$, it then follows that
$d_2(x_2,z_{2,\beta_2}^{j_2})\sim d_2(x_2,y_2)$. By this and the size condition of $\psi_{2,\alpha_2}^{k_2}$, we find that
\begin{align*}
\left|\mathrm J_{\alpha_1,\alpha_2}^{k_1,k_2}\right|&\le\int_{Q_{2,\beta_{2}}^{j_{2}}}\left|F_{\alpha_1}^{k_1}(y_2)\right|
\frac{|\psi_{2,\alpha_2}^{k_2}(y_2)|}{\sqrt{\mu_2(Q_{2,\alpha_2}^{k_2+1})}}\,d\mu_2(y_2)
\lesssim\int_{X_2}\left|F_{\alpha_1}^{k_1}(y_2)\right|
\widetilde{\mathcal E}_{2,k_2}^*\left(y_2,z_{2,\alpha_{2}}^{k_{2}}\right)\,d\mu_2(y_2)\\
&\lesssim\widetilde{\mathcal E}^*_{2,k_2}\left(x_2,z_{2,\beta_{2}}^{j_{2}}\right)\int_{Q_{2,\beta_{2}}^{j_{2}}}
\left|F_{\alpha_1}^{k_1}(y_2)\right|\,d\mu_2(y_2),
\end{align*}
which further implies that
\begin{align}\label{eq-1.2.2r}
&\sum_{k_1,\alpha_1}\sum_{k_2=j_2+1}^\infty\sum_{\alpha_2\in\mathcal G_{2,k_2}}
\left|\mathrm J_{\alpha_1,\alpha_2}^{k_1,k_2}\right|^2
\mathbf 1_{Q_{1,\alpha_1}^{k_1+1}}(x_1)\mathbf 1_{Q_{2,\alpha_2}^{k_2+1}}(x_2)\nonumber\\
&\quad\lesssim\sum_{k_2=j_2+1}^\infty\left[\widetilde{\mathcal E}^*_{2,k_2}\left(x_2,z_{2,\beta_{2}}^{j_{2}}\right)\right]^2
\sum_{k_1\in\mathbb Z}\sum_{\alpha_1\in\mathcal G_{1,k_1}}\left[\int_{Q_{2,\beta_{2}}^{j_{2}}}\left|F_{\alpha_1}^{k_1}(y_2)\right|
\,d\mu_2(y_2)\right]^2\mathbf 1_{Q_{1,\alpha_1}^{k_1+1}}(x_1)\nonumber\\
&\quad\lesssim\left[\frac 1{V_2(x_2,z_{2,\beta_2}^{j_2})}\right]^2\left[\frac{\delta^{j_2}}{d_2(x_2,
z_{2,\beta_2}^{j_2})}\right]^{2\eta'}
\mu_2\left(Q_{2,\beta_2}^{j_2}\right)
\int_{Q_{2,\beta_{2}}^{j_{2}}}\sum_{k_1,\alpha_1}\left|F_{\alpha_1}^{k_1}(y_2)\right|^2
\mathbf 1_{Q_{1,\alpha_1}^{k_1+1}}(x_1)\,d\mu_2(y_2).
\end{align}
This is also the desired estimate in this case.

Combining \eqref{eq-1.2.1r} and \eqref{eq-1.2.2r}, we conclude that, for any fixed $\eta'\in(0,\eta)$,
any $x_1\in Q_1$ and $x_2\in(\widehat{Q_2})^\complement$,
\begin{align}\label{eq-1.2}
S(a_R)(x_1,x_2)&\lesssim\frac 1{V_2(x_2,z_{2,\beta_2}^{j_2})}\left[\frac{\delta^{j_2}}{d_2(x_2,z_{2,\beta_2}^{j_2})}\right]^{\eta'}
\left[\mu\left(Q_{2,\beta_2}^{j_2}\right)\right]^{\frac 12}\nonumber\\
&\quad\times\left[\int_{Q_2}\sum_{k_1,\alpha_1}\left|F_{\alpha_1}^{k_1}(y_2)\right|^2
\mathbf 1_{Q_{1,\alpha_1}^{k_1+1}}(x_1)\,d\mu_2(y_2)\right]^{\frac 12}.
\end{align}
Now, we let $R_{c,2}:=[Q_1\times(\widehat{Q_2})^\complement]\cap\Omega^{**}$ and estimate
$\|S(a_R)\mathbf 1_{R_{c,2}}\|_{L^p}$. Indeed, by \eqref{eq-1.2}, the assumption that $q\ge 2$, and H\"{o}lder's
inequality, we have, for any $x_2\in(\widehat{Q_2})^\complement$,
\begin{align*}
&\int_{Q_1}[S(a_R)(x_1,x_2)]^q\,d\mu_1(x_1)\\
&\quad\lesssim\left[\frac 1{V_2(x_2,z_{2,\beta_2}^{j_2})}\right]^q\left[\frac{\delta^{j_2}}{d_2(x_2,z_{2,\beta_2}^{j_2})}
\right]^{q\eta'}
\left[\mu_2\left(Q_{2,\beta_2}^{j_2}\right)\right]^{\frac q2}\\
&\qquad\times\int_{Q_1}\left[\int_{Q_2}\sum_{k_1,\alpha_1}\left|F_{\alpha_1}^{k_1}(y_2)\right|^2
\mathbf 1_{Q_{1,\alpha_1}^{k_1+1}}(x_1)\,d\mu_2(y_2)\right]^{\frac q2}\,d\mu_1(x_1)\\
&\quad\lesssim\left[\frac 1{V_2(x_2,z_{2,\beta_2}^{j_2})}\right]^q\left[\frac{\delta^{j_2}}
{d_2(x_2,z_{2,\beta_2}^{j_2})}\right]^{q\eta'}
\left[\mu_2\left(Q_{2,\beta_2}^{j_2}\right)\right]^{q-1}\\
&\qquad\times\int_{Q_2}\int_{Q_1}\left[\sum_{k_1,\alpha_1}
\left|\left<a_R(\cdot,y_2),\frac{\psi_{1,\alpha_1}^{k_1}}{\sqrt{\mu(Q_{1,
\alpha_1}^{k_1+1})}}\right>\right|^2
\mathbf 1_{Q_{1,\alpha_1}^{k_1+1}}(x_1)\right]^{\frac q2}\,d\mu_1(x_1)\,d\mu_2(y_2)\\
&\quad\lesssim\left[\frac 1{V_2(x_2,z_{2,\beta_2}^{j_2})}\right]^q\left[\frac{\delta^{j_2}}{d_2(x_2,
z_{2,\beta_2}^{j_2})}\right]^{q\eta'}
\left[\mu_2\left(Q_{2,\beta_2}^{j_2}\right)\right]^{q-1}\int_{Q_2}\|a_R(\cdot,
y_2)\|_{L^q(X_1)}^q\,d\mu(y_2)\\
&\quad\lesssim\left[\frac 1{V_2(x_2,z_{2,\beta_2}^{j_2})}\right]^q\left[\frac{\delta^{j_2}}{d_2(x_2,z_{2,
\beta_2}^{j_2})}\right]^{q\eta'}
\left[\mu_2\left(Q_{2,\beta_2}^{j_2}\right)\right]^{q-1}\|a_R\|_{L^q}^q.
\end{align*}
which, together with H\"{o}lder's inequality again, further implies that
\begin{align*}
&\int_{Q_1}[S(a_R)(x_1,x_2)]^p\,d\mu_1(x_1)\\
&\quad\le\left[\int_{Q_1}[S(a_R)(x_1,x_2)]^q\,d\mu_1(x_1)\right]^{\frac pq}[\mu_1(Q_1)]^{1-\frac pq}\\
&\quad\lesssim\left[\frac 1{V_2(x_2,z_{2,\beta_2}^{j_2})}\right]^p
\left[\frac{\delta^{j_2}}{d_2(x_2,z_{2,\beta_2}^{j_2})}\right]^{p\eta'}
\left[\mu\left(Q_{2,\beta_2}^{j_2}\right)\right]^{p-\frac pq}\|a_R\|_{L^q}^p[\mu_1(Q_1)]^{1-\frac pq}.
\end{align*}
Using this and choosing $\eta'\in(0,\eta)$ such that $p>\frac{\omega}{\omega+\eta'}$, we conclude that
\begin{align}\label{eq-sar2}
\left\|S(a_R)\mathbf 1_{R_{c,2}}\right\|_{L^p}^p
&\le\int_{(\widehat{Q_2})^\complement}\int_{Q_1}|a_R(x_1,x_2)|^p\,d\mu_1(x_1)\,d\mu(x_2)\nonumber\\
&\nonumber\lesssim\|a_R\|_{L^q}^p[\mu_1(Q_1)]^{1-\frac pq}\left[\mu_2\left(Q_{2,\beta_2}^{j_2}\right)\right]^{p-\frac pq}\\
&\nonumber\quad\times\int_{d(x_2,z_{2,\beta_2}^{j_2})>C\delta^{j_2'}}\left[\frac 1{V_2(x_2,z_{2,\beta_2}^{j_2})}\right]^p
\left[\frac{\delta^{j_2}}{d_2(x_2,z_{2,\beta_2}^{j_2})}\right]^{p\eta'}\,d\mu(x_2)\\
&\nonumber\lesssim[\mu_1(Q_1)]^{1-\frac pq}\left[\mu_2\left(Q_{2,\beta_2}^{j_2}\right)\right]^{p-\frac pq}\|a_R\|_{L^q}^p\\
&\nonumber\quad\times\delta^{(j_2-j_2')p\eta'}\left[\mu_2\left(B\left(z_{2,\beta_2}^{j_2},
\delta^{j_2'}\right)\right)\right]^{1-p}\\
&\nonumber\lesssim[\mu_1(Q_1)]^{1-\frac pq}\left[\mu_2\left(Q_{2,\beta_2}^{j_2}
\right)\right]^{1-\frac pq}
\delta^{(j_2-j_2')[p\eta'-\omega(1-p)]}\|a_R\|_{L^q}^p\\
&\sim[\mu(R)]^{1-\frac pq}\left[\frac{\ell(Q_2)}{\ell(\widehat{Q_2})}\right]^{p\eta'-\omega(1-p)}
\|a_R\|_{L^q}^p.
\end{align}
This is the desired estimate for Case (I.2).

{\it Case (I.3) $x_1\in Q_1^\complement$ and $x_2\in\widehat{Q_2}$.} In this case, note that
$Q_1\times\widehat{Q_2}\subset\Omega^*$. Then there exists a unique dyadic cube $P_2\supset Q_2$ such that
$Q_1\times P_2\in\mathcal M_2(\Omega^*)$. Let $\widehat{Q_1}$ be the maximal dyadic cube containing $Q_1$
such that
$$
\mu\left(\left(\widehat{Q_1}\times P_2\right)\cap\Omega^*\right)>\frac 12\mu(Q_1\times P_2).
$$
Then we have $\widehat{Q_1}\times P_2\subset\Omega^{**}$. However, since $(x_1,x_2)\notin\Omega^{**}$ and
$x_2\in\widehat{Q_2}\subset P_2$, it then follows that $x_1\notin\widehat{Q_1}$. Suppose that
$P_2:=Q_{2,\gamma_2}^{l_2}$ and $\widehat{Q_1}:=Q_{1,\beta_1'}^{j_1'}$ for some $j_1', l_2\in\mathbb Z$,
$\beta_1'\in\mathcal A_{1,j_1'}$,
and $\gamma_2\in\mathcal A_{2,l_2}$. We also consider two case for whether
$x_2\in Q_2$.

{\it Case (I.3.1) $x_2\in Q_2$.} In this case, applying an argument
similar to that used in the proof of Case (I.2),
but with
$X_1$ and $X_2$ interchanged, we find that, for any given $\eta'\in(0,\eta)$,
\begin{equation*}
\int_{Q_2}[S(a_R)(x_1,x_2)]^q\,d\mu(x_2)
\lesssim\left[\frac 1{V_1(x_1,z_{1,\beta_1}^{j_1})}\right]^q\left[\frac{\delta^{j_1}}{d_1(x_1,
z_{1,\beta_1}^{j_1})}\right]^{q\eta'}
[\mu_1(Q_1)]^{q-1}\|a_R\|_{L^q}^q.
\end{equation*}
From this and H\"{o}lder's inequality, we deduce that
\begin{align}\label{eq-estr3-1}
\left\|S(a_R)\mathbf 1_{(\widehat{Q_1})^\complement\times Q_2}\right\|_{L^p}^p
&=\int_{(\widehat{Q_1})^\complement}\int_{Q_2}|S(a_R)(x_1,x_2)|^p\,d\mu_1(x_1)\,d\mu(x_2)\nonumber\\
&\lesssim\|a_R\|_{L^q}^p[\mu_1(Q_1)]^{p-\frac pq}[\mu_2(Q_2)]^{1-\frac pq}\nonumber\\
&\quad\times\int_{d(x_1,z_{1,\beta_1}^{j_1})>C\delta^{j_1'}}\left[\frac 1{V_1(x_1,z_{1,\beta_1}^{j_1})}\right]^p
\left[\frac{\delta^{j_1}}{d_1(x_1,z_{1,\beta_1}^{j_1})}\right]^{p\eta'}\,d\mu(x_2)\nonumber\\
&\sim[\mu(R)]^{1-\frac pq}\left[\frac{\ell(Q_1)}{\ell(\widehat{Q_1})}\right]^{p\eta'-\omega(1-p)}\|a_R\|_{L^q}^p.
\end{align}
This is the desired estimate for this case.

{\it Case (I.3.2) $x_2\in\widehat{Q_2}\setminus Q_2$.} In this case, applying an argument similar to that used
in the proof of Case (I.1), but with $X_1$ and $X_2$ interchanged, we conclude that, for any given $\eta'\in(0,\eta)$,
\begin{align*}
S(a_R)(x_1,x_2)&\lesssim[\mu(R)]^{1-\frac 1q}\|a_R\|_{L^q}
\frac 1{V_1(x_1,z_{1,\beta_1}^{j_1})}\frac 1{V_2(x_2,z_{2,\beta_2}^{j_2})}\\
&\quad\times\left[\frac{\delta^{j_1}}{d_1(x_1,z_{1,\beta_1}^{j_1})}\right]^{\eta'}
\left[\frac{\delta^{j_2}}{d_2(x_2,z_{2,\beta_2}^{j_2})}\right]^{\eta'}
\end{align*}
and, consequently,
\begin{equation}\label{eq-estr3-2}
\left\|S(a_R)\mathbf 1_{(\widehat{Q_1})^\complement\times Q_2^\complement}\right\|_{L^p}^p
\lesssim[\mu(R)]^{1-\frac pq}\left[\frac{\ell(Q_1)}{\ell(\widehat{Q_1})}\right]^{p\eta'-\omega(1-p)}
\|a_R\|_{L^q}^p.
\end{equation}
This is also the desired estimate.

Now, let $R_{c,3}:=(Q_1^\complement\times\widehat{Q_2})\cap(\Omega^{**})^\complement$.
Combining \eqref{eq-estr3-1} and \eqref{eq-estr3-2}, we obtain
\begin{equation}\label{eq-sar3}
\left\|S(a_R)\mathbf 1_{R_{c,3}}\right\|_{L^p}^p\lesssim[\mu(R))]^{1-\frac pq}
\left[\frac{\ell(Q_1)}{\ell(\widehat{Q_1})}\right]^{p\eta'-\omega(1-p)}\|a_R\|_{L^q}^p,
\end{equation}
where $\eta'\in(0,\eta)$ satisfies $p>\frac{\omega}{\omega+\eta'}$. This is the desired estimate for
Case (I.3).

By the definitions of $\{R_{c,i}\}_{i=1}^3$, we have $(\Omega^{**})^\complement=R_{c,1}\cup R_{c,2}\cup R_{c,3}$.
Thus, using \eqref{eq-sar1}, \eqref{eq-sar2}, and \eqref{eq-sar3}, we conclude that
\begin{align*}
\left\|S(a_R)\mathbf 1_{(\Omega^{**})^\complement}\right\|_{L^p}^p
&=\sum_{i=1}^3\left\|S(a_R)\mathbf 1_{R_{c,i}}\right\|_{L^p}^p\\
&\lesssim[\mu(R)]^{1-\frac pq}\left[\frac{\ell(Q_2)}{\ell(\widehat{Q_2})}\right]^{p\eta'-\omega(1-p)}
\|a_R\|_{L^q}^p\\
&\quad+\left[\mu\left(Q_1\times\widehat{Q_2}\right)\right]^{1-\frac pq}
\left[\frac{\ell(Q_1)}{\ell(\widehat{Q_1})}\right]^{p\eta'-\omega(1-p)}\|a_R\|_{L^q}^p.
\end{align*}
By this and H\"{o}lder's inequality, we have
\begin{align*}
 \left\|S(a)\mathbf 1_{(\Omega^{**})^\complement}\right\|_{L^p}^p
&\le\sum_{R\in\mathcal M(\Omega)}\left\|S(a_R)\mathbf 1_{(\Omega^{**})^\complement}\right\|_{L^p}^p\nonumber\\
&\lesssim\sum_{Q_1\times Q_2=R\in\mathcal M(\Omega)}[\mu(R)]^{1-\frac pq}\left[\frac{\ell(Q_2)}{\ell(\widehat{Q_2})}\right]^{p\eta'-\omega(1-p)}
\|a_R\|_{L^q}^p\nonumber\\
&\quad+\sum_{Q_1\times Q_2=R\in\mathcal M(\Omega)}\left[\mu\left(Q_1\times\widehat{Q_2}\right)\right]^{1-\frac pq}
\left[\frac{\ell(Q_1)}{\ell(\widehat{Q_1})}\right]^{p\eta'-\omega(1-p)}\|a_R\|_{L^q}^p\nonumber\\
&=:\mathrm Y_1+\mathrm Y_2.
\end{align*}
To estimate $\mathrm Y_1$, we first claim that $\mathcal M(\Omega)\subset\mathcal M_2(\Omega)$. Indeed, suppose $Q_1\times Q_2=:R
\in\mathcal M(\Omega)$. To show $R\in\mathcal M_2(\Omega)$, we assume $\widetilde Q_2\supset Q_2$ such that
$Q_1\times \widetilde Q_2\subset\Omega$. Since $R\in\mathcal M(\Omega)$, it then follows that $\widetilde Q_2=Q_2$.
By this, H\"{o}lder's inequality, and Lemma \ref{lem-j}, we find that
\begin{align*}
\mathrm Y_1&\lesssim\left[\sum_{R\in\mathcal M(\Omega)}\|a_R\|_{L^q}^q\right]^{\frac pq}
\left\{\sum_{Q_1\times Q_2=R\in\mathcal M_2(\Omega)}\mu(R)
\left[\frac{\ell(Q_2)}{\ell(\widehat{Q_2})}\right]^{[p\eta'-\omega(1-p)](\frac qp)'}\right\}^{1-\frac pq}\\
&\lesssim[\mu(\Omega)]^{\frac pq-1}[\mu(\Omega)]^{1-\frac pq}\sim 1.
\end{align*}
This is the desired estimate for $\mathrm Y_1$.

To estimate the term $\mathrm Y_2$, for any $Q_1\times Q_2\in\mathcal M(\Omega)$, let $m_2^{Q_1}(Q_2):=P_2$ be
as in the proof of Case (I.3). Notice that, for any fixed $Q_1$, $\{Q_2:\ Q_1\times Q_2\in\mathcal M(\Omega)\}$ are disjoint.
By this and H\"{o}lder's inequality, we have
\begin{align*}
\mathrm Y_2&\lesssim\sum_{Q_1\times Q_2=R\in\mathcal M(\Omega)}\left[\mu(Q_1\times Q_2)\right]^{1-\frac pq}
\left[\frac{\ell(Q_1)}{\ell(\widehat{Q_1})}\right]^{p\eta'-\omega(1-p)}\|a_R\|_{L^q}^p\\
&\lesssim\left[\sum_{R\in\mathcal M(\Omega)}\|a_R\|_{L^q}^q\right]^{\frac pq}
\left\{\sum_{Q_1\times Q_2=R\in\mathcal M(\Omega)}\mu(Q_1\times Q_2)
\left[\frac{\ell(Q_1)}{\ell(\widehat{Q_1})}\right]^{[p\eta'-\omega(1-p)](\frac qp)'}\right\}^{1-\frac pq}\\
&\lesssim[\mu(\Omega)]^{\frac pq-1}
\left\{\sum_{Q_1\in\mathcal D_1}\mu_1(Q_1)\sum_{\{Q_2:\ Q_1\times Q_2\in\mathcal M(\Omega)\}}\mu_2(Q_2)
\left[\frac{\ell(Q_1)}{\ell(\widehat{Q_1})}\right]^{[p\eta'-\omega(1-p)](\frac qp)'}\right\}^{1-\frac pq}\\
&\lesssim[\mu(\Omega)]^{\frac pq-1}\left\{\sum_{Q_1\in\mathcal D_1}\mu_1(Q_1)\sum_{P_2}
\sum_{\substack{\{Q_2:\ Q_1\times Q_2\in\mathcal M(\Omega),\\ m_2^{Q_1}(Q_2)=P_2\}}}\mu_2\left(Q_2\right)
\left[\frac{\ell(Q_1)}{\ell(\widehat{Q_1})}\right]^{[p\eta'-\omega(1-p)](\frac qp)'}\right\}^{1-\frac pq}\\
&\lesssim[\mu(\Omega)]^{\frac pq-1}\left\{\sum_{Q_1\in\mathcal D_1}\mu_1(Q_1)\sum_{\{P_2:\ Q_1\times P_2\in\mathcal M_2(\Omega)\}}\mu_2(P_2)
\left[\frac{\ell(Q_1)}{\ell(\widehat{Q_1})}\right]^{[p\eta'-\omega(1-p)](\frac qp)'}\right\}^{1-\frac pq}\\
&\sim[\mu(\Omega)]^{\frac pq-1}\left\{\sum_{Q_1\times P_2\in\mathcal M_2(\Omega^*)}\mu(Q_1\times P_2)
\left[\frac{\ell(Q_1)}{\ell(\widehat{Q_1})}\right]^{[p\eta'-\omega(1-p)](\frac qp)'}\right\}^{1-\frac pq}\\
&\lesssim[\mu(\Omega)]^{\frac pq-1}[\mu(\Omega^*)]^{1-\frac pq}\lesssim 1.
\end{align*}
Combining the estimates of $\mathrm Y_1$ and $\mathrm Y_2$, we obtain
$\|S(a)\mathbf 1_{(\Omega^{**})^\complement}\|_{L^p}\lesssim 1$.
This finishes the proof of Proposition
\ref{prop-at} in the case $q\in[2,\infty)$.
\end{proof}

Now, we prove Proposition \ref{prop-at} in the case $q\in(1,2)$. To this end, we need the following
vector-valued estimate. Let $\nu$ be a measure on $\mathbb N$. For any $n\in\mathbb N$, we write $\nu(\{n\})$ simply as
$\nu(n)$. We assume that, for any $n\in\mathbb N$, $\nu(n)\in(0,\infty)$. For any $r\in(0,\infty)$. The space $\ell^r(\nu)$
is defined to be the set of all sequences $a:=\{a_n\}_{n\in\mathbb N}\subset\mathbb C$ such that
$$
\|a\|_{\ell^r(\nu)}:=\left[\sum_{n\in\mathbb N}|a_n|^r\nu(n)\right]^{1/r}<\infty.
$$
Let $X$ be a space of homogeneous type, $q\in(0,\infty)$, and $\mathcal B$ a Banach space. Denote by
$L^q(X,\mathcal B)$ the space of all measurable $\mathcal B$-valued functions $F$ on $X$ such that
$$
\|F\|_{L^q(X,\mathcal B)}:=\left[\int_X \|F(x)\|_{\mathcal B}^q\,d\mu(x)\right]^{1/q}<\infty.
$$
Let $\vec F:=\{F_n\}_{n\in\mathbb N}$ be a sequence of measurable functions on $X$ and $g$ a measurable function.
The \emph{symbol $\langle \vec F,g\rangle$} is defined by setting
$$
\left\langle\vec F,g\right\rangle:=\{\langle F_n,g\rangle\}_{n\in\mathbb N}
:=\left\{\int_X F_n(x)g(x)\,d\mu(x)\right\}_{n\in\mathbb N}
$$
if each of the terms is well-defined.

The next proposition shows our desired vector-valued estimate.

\begin{proposition}\label{prop-vector}
Let $X$ be a space of homogeneous type and $\{\psi_{\alpha}^k:\ k\in\mathbb Z,\alpha\in\mathcal G_k\}$ be defined as in
Theorem \ref{thm-wave}. Suppose that $\nu$ is a measure on $\mathbb N$. Then, for any given $p\in(1,2]$, there exists a
constant $C_{p}\in(0,\infty)$, independent of $\nu$, such that, for any $\vec F:=\{F_n\}_{n\in\mathbb N}\in L^q(X,\ell^2(\nu))$,
\begin{equation*}
\left\|\left[\sum_{k\in\mathbb Z}\sum_{\alpha\in\mathcal G_k}\left\|\left<\vec F,\psi_\alpha^k\right>\right\|_{\ell^2(\nu)}^2
\frac{\mathbf 1_{Q_\alpha^{k+1}}}{\mu(Q_\alpha^{k+1})}\right]^{\frac 12}\right\|_{L^q(X)}
\le C_p\|\vec F\|_{L^q(X,\ell^2(\nu))}.
\end{equation*}
\end{proposition}

\begin{proof}
For any $\vec F:=\{F_n\}_{n\in\mathbb N}\in L^q(X,\ell^2(\nu))$ and $x\in X$, by Lemma \ref{lem-cube}(a), we have
\begin{align*}
\sum_{\alpha\in\mathcal G_k}\left\|\left<F,\psi_\alpha^k\right>\right\|_{\ell^2(\nu)}^2
\frac{\mathbf 1_{Q_\alpha^{k+1}}(x)}{\mu(Q_\alpha^{k+1})}
&=\sum_{\alpha\in\mathcal G_k}\sum_{n=1}^\infty\left|\int_X F_n(y)\psi_{\alpha}^k(y)\,d\mu(x)\right|^2\nu(n)
\frac{\mathbf 1_{Q_\alpha^{k+1}}(x)}{\mu(Q_\alpha^{k+1})}\\
&=\sum_{n=1}^\infty\left|\sum_{\alpha\in\mathcal G_k}\int_X F_n(y)\psi_{\alpha}^k(y)
\frac{\mathbf 1_{Q_\alpha^{k+1}}(x)}{\sqrt{\mu(Q_\alpha^{k+1})}}\,d\mu(y)\right|^2\nu(n)\\
&=\sum_{n=1}^\infty\left|\int_X F_n(y)L_k(x,y)\,d\mu(y)\right|^2\nu(n)\\
&=\left\|\int_X\vec F(y)L_k(x,y)\,d\mu(y)\right\|_{\ell^2(\nu)}^2,
\end{align*}
where, for any $k\in\mathbb Z$ and $x, y\in X$,
$$
L_k(x,y):=\sum_{\alpha\in\mathcal G_k}\psi_{\alpha}^k(y)\frac{\mathbf 1_{Q_\alpha^{k+1}}(x)}{\sqrt {\mu(Q_\alpha^{k+1})}}
$$
and
$$
\int_X\vec F(y)L_k(x,y)\,d\mu(y):=\left\{\int_X F_n(y)L_k(x,y)\,d\mu(y)\right\}_{n\in\mathbb N}.
$$
Using this, we obtain, for any $\vec F\in L^q(X,\ell^2(\nu))$,
\begin{equation*}
\left\|\left[\sum_{\genfrac{}{}{0pt}{}{k\in{\mathbb Z}}{\alpha_1\in\mathcal G_{k}}}
\left\|\left<\vec F,\psi_\alpha^k\right>\right\|_{\ell^2(\nu)}^2
\frac{\mathbf 1_{Q_\alpha^{k+1}}(\cdot)}{\mu(Q_\alpha^{k+1})}\right]^{\frac 12}\right\|_{L^q(X)}
=\left\|\left[\sum_{k\in\mathbb Z}\left\|\int_X \vec F(y)
L_k(\cdot,y)\,d\mu(y)\right\|_{\ell^2(\nu)}^2\right]^{\frac 12}\right\|_{L^q(X)}.
\end{equation*}
We define the vector-valued operator $\vec T$ by setting, for any $\vec F\in L^q(X,\ell^2(\nu))$ and $x\in X$,
$$
\vec{T}\vec F(x):=\left\{T_k\vec F(x)\right\}_{k\in\mathbb Z}:=\left\{\int_X\vec F(y)L_k(x,y)\,d\mu(y)\right\}_{k\in\mathbb Z}.
$$
It remains to show that $\vec T$ is bounded from $L^q(X,\ell^2(\nu))$ to
$L^q(X,\ell^2(\mathbb Z,\ell^2(\nu)))$ for any given $q\in(1,2]$.

We first show that $\vec T$ is bounded from $L^2(X,\ell^2(\nu))$ to $L^2(X,\ell^2(\mathbb Z,\ell^2(\nu)))$. Indeed, using
the wavelet characterization of $L^2(X)$, we find that
\begin{align*}
\left\|\vec T \vec F\right\|_{L^2(X,\ell^2(\mathbb Z,\ell^2(\nu)))}^2
&=\int_X\sum_{k,\alpha}\left\|\left<\vec F,\psi_\alpha^k\right>\right\|_{\ell^2(\nu)}^2
\frac{\mathbf 1_{Q_\alpha^{k+1}}(x)}{\mu(Q_\alpha^{k+1})}\,d\mu(x)\\
&=\sum_{k,\alpha}\sum_{n=1}^\infty\int_X\left|\left<F_n,\psi_\alpha^k\right>\right|^2
\frac{\mathbf 1_{Q_{\alpha}^{k+1}}(x)}{\mu(Q_\alpha^{k+1})}\,d\mu(x)\nu(n)\\
&=\int_X\sum_{n=1}^\infty|F_n(x)|^2\nu(n)\,d\mu(x)=\int_X\|F(x)\|_{\ell^2(\nu)}^2
\,d\mu(x)=\|F\|_{L^2(X,\ell^2(\nu))},
\end{align*}
which implies that $\vec T$ is bounded from $L^2(X,\ell^2(\nu))$ to $L^2(X,\ell^2(\mathbb Z,\ell^2(\nu)))$.

To prove that $\vec T$ is bounded from $L^q(X,\ell^2(\nu))$ to $L^q(X,\ell^2(\mathbb Z,\ell^2(\nu)))$ for any given
$q\in(1,2)$, we show that $\vec T$ can be seen as a vector-valued Calder\'{o}n--Zygmund operator.
Moreover, by Theorem \ref{thm-wave}, Lemma \ref{lem-cube}(b), and
\cite[Lemma 6.4]{ah13}, we conclude that, for any $k\in\mathbb Z$ and $x,y\in X$,
\begin{align*}
|L_k(x,y)|
&=\sum_{\alpha\in\mathcal G_k}\left|\psi_\alpha^k(y)\right|
\frac{\mathbf 1_{Q_\alpha^{k+1}}(x)}{\sqrt{\mu(Q_\alpha^{k+1})}}
\lesssim\sum_{\alpha\in\mathcal G_k}\frac 1{V_{\delta^k}(y_\alpha^k)}\mathcal E_k^*(y,y_{\alpha}^k)\mathbf 1_{Q_\alpha^{k+1}}(x)\\
&
\lesssim\sum_{\alpha\in\mathcal G_k}\frac 1{V_{\delta^k}(y_\alpha^k)}\mathcal E_k^*(y,y_{\alpha}^k)\mathcal E_k^*(x,y_{\alpha}^k)
\lesssim\widetilde{\mathcal E}_k^*(x,y)
\lesssim \widetilde{\mathcal E}_k (x,y).
\end{align*}
This shows that $\vec T$ is well-defined for any $\vec F:=\{F_n\}_{n\in\mathbb N}\in L^q(X,\ell^2(\nu))$ with any given
$q\in(1,2)$. Moreover, applying an argument similar to that used in the proof of \cite[Theorem 1.1]{gly09}, we find that
it remains to prove that, for any $x, y, y'\in X$ with $x\neq y$ and $d(y,y')<(2A_0)^{-1}d(x,y)$,
\begin{equation}\label{eq-hormander}
\left[\sum_{k\in\mathbb Z}|L_k(x,y)-L_k(x,y')|^2\right]^{ \frac{1}{2}}\lesssim\left[\frac{d(y,y')}{d(x,y)}\right]^\eta\frac 1{V(x,y)}.
\end{equation}
To this end, by \cite[Theorem 3.3]{hlw18}, we obtain, for any given $\Gamma\in(0,\infty)$ and for any $k\in\mathbb Z$,
$\alpha\in\mathcal G_k$, and $y,y'\in X$,
$$
\frac{|\psi_\alpha^k(y)|}{\sqrt{\mu(Q_\alpha^{k+1})}}\lesssim P_{\Gamma}(y_\alpha^k,y;\delta^k)
$$
and, if $d(y,y')\le(2A_0)^{-1}[\delta^k+d(y_\alpha^k,y)]$,
$$
\frac{|\psi_\alpha^k(y)-\psi_\alpha^k(y')|}{\sqrt{\mu(Q_\alpha^{k+1})}}
\lesssim\left[\frac{d(y,y')}{\delta^k+d(y_\alpha^k,y)}\right]^\eta
P_{\Gamma}(y_\alpha^k,y;\delta^k).
$$
Combining the above two estimates, we find that,
for any $x\in Q_\alpha^{k+1}$ and $y,y'\in X$ such that
$d(y,y')\le(2A_0)^{-1}[\delta^k+d(x,y)]$,
$$
\frac{|\psi_\alpha^k(y)-\psi_\alpha^k(y')|}{\sqrt{\mu(Q_\alpha^{k+1})}}
\lesssim\left[\frac{d(y,y')}{\delta^k+d(x,y)}\right]^\eta
P_{\Gamma}(x,y;\delta^k).
$$
Thus, from this, we deduce that, for any fixed $\Gamma\in(0,\infty)$ and for any $k\in\mathbb Z$
and $x, y, y'\in X$ with $x\neq y$ and $d(y,y')\le(2A_0)^{-1}d(x,y)$,
\begin{equation*}
|L_k(x,y)-L_k(x,y')|=\sum_{\alpha\in\mathcal G_k}\left|\psi_\alpha^k(y)-\psi_\alpha^k(y')\right|
\frac{\mathbf 1_{Q_\alpha^{k+1}}(x)}{\sqrt{\mu(Q_\alpha^{k+1})}}
\lesssim\left[\frac{d(y,y')}{\delta^k+d(x,y)}\right]^\eta P_\Gamma(x,y;\delta^k),
\end{equation*}
which further implies that
\begin{align*}
&\left[\sum_{k\in\mathbb Z}|L_k(x,y)-L_k(x,y')|^2\right]^{\frac12}\\
&\quad\le\sum_{k\in\mathbb Z}|L_k(x,y)-L_k(x,y')|\lesssim\sum_{k\in\mathbb Z}\left[\frac{d(y,y')}{\delta^k+d(x,y)}\right]^\eta
P_\Gamma(x,y;\delta^k)\\
&\quad\lesssim\sum_{\{k:\ \delta^k>d(x,y)\}}\left[\frac{d(y,y')}{\delta^k}\right]^\eta\frac 1{V(x,y)}+
\sum_{\{k:\ \delta^k\le d(x,y)\}}\left[\frac{d(y,y')}{\delta^k}\right]^\eta\frac 1{V(x,y)}\left[\frac{\delta^k}{d(x,y)}\right]^\Gamma\\
&\quad\sim\left[\frac{d(y,y')}{d(x,y)}\right]^\eta\frac{1}{V(x,y)}.
\end{align*}
Therefore, \eqref{eq-hormander} holds. Consequently, $\vec T$ is bounded
from $L^q(X,\ell^2(\nu))$ to
$L^q(X,\ell^2(\mathbb Z,\ell^2(\nu)))$ for any given $q\in(1,2)$
with the operator norm independent of $\nu$.
This finishes the proof of Proposition \ref{prop-vector}.
\end{proof}

Now, we show Proposition \ref{prop-at} in the case $q\in(1,2)$.

\begin{proof}[Proof (II) of Proposition \ref{prop-at} With $q\in(1,2)$]
Let all the symbols be the same as in the proof of the case $q\in[2,\infty)$. Applying an argument similar to that used in the
proof of \eqref{eq-sa1}, we find that
$$
\left\|S(a)\mathbf 1_{\Omega^{**}}\right\|_{L^p}\lesssim 1.
$$
This is the desired estimate.

Now, we estimate $\|S(a)\mathbf 1_{(\Omega^{**})^\complement}\|_{L^p}$. By Minkowski's inequality, we conclude that,
for any $(x_1,x_2)\in X_1\times X_2$,
$$
S(a)(x_1,x_2)=S\left(\sum_{R\in\mathcal M_1(\Omega)}a_R+\sum_{R\in\mathcal M_2(\Omega)}a_R\right)
\le \sum_{R\in\mathcal M_1(\Omega)} S(a_R)(x_1,x_2)+\sum_{R\in\mathcal M_2(\Omega)} S(a_R)(x_1,x_2).
$$
By symmetry, we assume $R:=Q_1\times Q_2\in\mathcal M_1(\Omega)$. Further assume that $Q_1=Q_{1,\beta_1}^{j_1}$ and
$Q_2=Q_{2,\beta_2}^{j_2}$ for some $j_1, j_2\in\mathbb Z$, $\beta_1\in\mathcal A_{1,j_1}$, and $\beta_2\in\mathcal A_{2,j_2}$.
Let $\widehat{Q_2}$ be the maximal cube containing $Q_2$ such that
$$
\mu\left(\left(Q_1\times\widehat Q_2\right)\cap\Omega\right)>\frac 12\mu\left(Q_1\times\widehat{Q_2}\right).
$$
We consider three cases for $x_1$ and $x_2$.

{\it Case (II.1) $x_1\in Q_1^\complement$ and $x_2\in(\widehat{Q_2})^\complement$.} In this case,
applying an argument similar to that used in the proof of Case (I.1), we find that, for any fixed $\eta'\in(0,\eta)$,
\begin{align*}
S(a_R)(x_1,x_2)&\lesssim\left[\mu_1\left(Q_{1,\beta_1}^{j_1}\right)\mu_2\left(Q_{2,
\beta_2}^{j_2}\right)\right]^{1-\frac 1q}\|a_R\|_{L^q}
\frac 1{V_1(x_1,z_{1,\beta_1}^{j_1})}\frac 1{V_2(x_2,z_{2,\beta_2}^{j_2})}\\
&\quad\times\left[\frac{\delta^{j_1}}{d_1(x_1,z_{1,\beta_1}^{j_1})}\right]^{\eta'}
\left[\frac{\delta^{j_2}}{d_2(x_2,z_{2,\beta_2}^{j_2})}\right]^{\eta'}.
\end{align*}
Let $R_{c,1}:=[Q_1^\complement\times(\widehat{Q_2})^\complement]\cap\Omega^{**}$. Choosing $\eta'\in(0,\eta)$ such that
$p>\frac{\omega}{\omega+\eta'}$ and using an argument similar to the proof of \eqref{eq-sar1}, we conclude that
\begin{equation}\label{eq-sar1a}
\left\|S(a_R)\mathbf{1}_{R_{c,1}}\right\|_{L^p}^p
\lesssim[\mu(R)]^{1-\frac pq}\left[\frac{\ell(Q_2)}{\ell(\widehat{Q_2})}\right]^{p\eta'-\omega(1-p)}\|a_R\|_{L^q}^p.
\end{equation}
This is the desired estimate in this case.

{\it Case (II.2) $x_1\in Q_1$ and $x_2\in(\widehat{Q_2})^\complement$.} In this case, for any $z_1\in X_1$,
$k_2\in\mathbb Z$, and $\alpha_2\in\mathcal G_{2,k_2}$, we write
$$
F_{\alpha_2}^{k_2}(z_1):=\left<a_R(z_1,\cdot),\psi_{2,\alpha_2}^{k_2}\right>_2
=\int_{X_2} a_R(z_1,y_2)\psi_{2,\alpha_2}^{k_2}(y_2)\,d\mu(y_2)
$$
and $\nu_{x_2}(k_2,\alpha_2):=[\mu_2(Q_{2,\alpha_2}^{k_2+1})]^{-\frac 12}\mathbf 1_{Q_{2,\alpha_2}^{k_2+1}}(x_2)$. Therefore,
from Proposition \ref{prop-vector}, we deduce that
\begin{align*}
&\int_{X_1} [S(a_R)(z_1,x_2)]^q\,d\mu_1(x_1)\\
&\quad=\int_{X_1}\left[\sum_{\genfrac{}{}{0pt}{}{k_1\in{\mathbb Z}}{\alpha_1\in\mathcal G_{1,k_1}}}
\sum_{\genfrac{}{}{0pt}{}{k_2\in{\mathbb Z}}{\alpha_2\in\mathcal G_{2,k_2}}}
\left|\left<\left<a_R,\psi_{2,\alpha_2}^{k_2}\right>_2,\psi_{1,\alpha_1}^{k_1}\right>_1\right|^2
\frac{\mathbf 1_{Q_{2,\alpha_2}^{k_2+1}}(x_2)}{\mu_2(Q_{2,\alpha_2}^{k_2+1})}
\frac{\mathbf 1_{Q_{1,\alpha_1}^{k_1+1}}(z_1)}{\mu_1(Q_{1,\alpha_1}^{k_1+1})}\right]^{\frac q2}\,d\mu_1(z_1)\\
&\quad=\int_{X_1}\left[\sum_{k_1,\alpha_1}\sum_{k_2,\alpha_2}
\left|\left<F_{\alpha_2}^{k_2},\psi_{1,\alpha_1}^{k_1}\right>_1\right|^2
\frac{\mathbf 1_{Q_{2,\alpha_2}^{k_2+1}}(x_2)}{\mu_2(Q_{2,\alpha_2}^{k_2+1})}
\frac{\mathbf 1_{Q_{1,\alpha_1}^{k_1+1}}(z_1)}{\mu_1(Q_{1,\alpha_1}^{k_1+1})}\right]^{\frac q2}\,d\mu_1(z_1)\\
&\quad=\int_{X_1}\left[\sum_{k_1,\alpha_1}
\left\|\left<\vec F,\psi_{1,\alpha_1}^{k_1}\right>_1\right\|_{\ell^2(\nu_{x_2})}^2
\frac{\mathbf 1_{Q_{1,\alpha_1}^{k_1+1}}(z_1)}{\mu(Q_{1,\alpha_1}^{k_1+1})}\right]^{\frac q2}\,d\mu_1(z_1)\\
&\quad\lesssim\int_{X_1}\left\|\vec F(z_1)\right\|_{\ell^2(\nu_{x_2})}^q\,d\mu_1(z_1)\\
&\quad\sim\int_{X_1}\left[\sum_{k_2,\alpha_2}\left|\left<a_R(z_1,\cdot),
\frac{\psi_{2,\alpha_2}^{k_2}}{\sqrt{\mu_2(Q_{2,\alpha_2}^{k_2+1})}}\right>\right|^2
\mathbf 1_{Q_{2,\alpha_2}^{k_2+1}}(x_2)\right]^{\frac q2}\,d\mu_1(z_1);
\end{align*}
here $\vec F$ is defined by setting, for any $z_1\in X_1$,
$$
\vec F(z_1):=\left\{F_{\alpha_2}^{k_2}(z_1):\ k_2\in\mathbb Z,\alpha_2\in\mathcal G_{2,k_2}\right\}.
$$
Applying an argument similar to that used in the proof in Case (I.2), we find that, for any given
$\eta'\in(0,\eta)$ and for any $z_1\in X_1$ and $x_2\in(\widehat{Q_2})^\complement$,
\begin{align*}
\sum_{k_2,\alpha_2}\left|\left<a(z_1,\cdot),
\frac{\psi_{2,\alpha_2}^{k_2}}{\sqrt{\mu_2(Q_{2,\alpha_2}^{k_2+1})}}\right>\right|^2\mathbf 1_{Q_{2,\alpha_2}^{k_2+1}}(x_2)
\lesssim\left[\frac{1}{V_2(x_2,z_{2,\beta_2}^{j_2})}\right]^2\left[\frac{\delta^{j_2}}
{d_2(x_2,z_{2,\beta_2}^{j_2})}\right]^{2\eta'}
\|a_R(z_1,\cdot)\|_{L^1(X_2)}^2.
\end{align*}
Combining the two above inequalities, we conclude that
\begin{align*}
&\int_{X_1}[S(a_R)(z_1,x_2)]^q\,d\mu_1(z_1)\\
&\quad\sim\int_{X_1}\left[\sum_{k_2,\alpha_2}\left|\left<a_R(z_1,\cdot),
\frac{\psi_{2,\alpha_2}^{k_2}}{\sqrt{\mu_2(Q_{2,\alpha_2}^{k_2+1})}}\right>\right|^2
\mathbf 1_{Q_{2,\alpha_2}^{k_2+1}}(x_2)\right]^{\frac q2}\,d\mu_1(z_1)\\
&\quad\lesssim\left[\frac{1}{V_2(x_2,z_{2,\beta_2}^{j_2})}\right]^q
\left[\frac{\delta^{j_2}}{d_2(x_2,z_{2,\beta_2}^{j_2})}\right]^{q\eta'}
\int_{X_1}\|a_R(z_1,\cdot)\|_{L^1(X_2)}^q\,d\mu_1(z_1)\\
&\quad\lesssim\left[\frac{1}{V_2(x_2,z_{2,\beta_2}^{j_2})}\right]^q
\left[\frac{\delta^{j_2}}{d_2(x_2,z_{2,\beta_2}^{j_2})}\right]^{q\eta'}
[\mu(Q_2)]^{q-1}\int_{X_1}\|a_R(z_1,\cdot)\|_{L^q(X_2)}^q\,d\mu_1(z_1)\\
&\quad\sim\left[\frac{1}{V_2(x_2,z_{2,\beta_2}^{j_2})}\right]^q
\left[\frac{\delta^{j_2}}{d_2(x_2,z_{2,\beta_2}^{j_2})}\right]^{q\eta'}
[\mu(Q_2)]^{q-1}\|a_R\|_{L^q}^q.
\end{align*}
By this and H\"{o}lder's inequality, we have, for any $x_2\in(\widehat{Q_2})^\complement$,
\begin{align*}
&\int_{Q_1}[S(a_R)(z_1,x_2)]^p\,d\mu(z_1)\\
&\quad\le
[\mu(Q_1)]^{1-\frac pq}
\left\{\int_{X_1}[S(a_R)(z_1,x_2)]^q\,d\mu(z_1)\right\}^{\frac pq}\\
&\quad\lesssim
[\mu(Q_1)]^{1-\frac pq}
\left[\frac{1}{V_2(x_2,z_{2,\beta_2}^{j_2})}\right]^p\left[\frac{\delta^{j_2}}
{d_2(x_2,z_{2,\beta_2}^{j_2})}\right]^{p\eta'}
[\mu(Q_2)]^{p-\frac pq}\|a_R\|_{L^q}^p.
\end{align*}
Choose $\eta'\in(0,\eta)$ such that $p>\frac{\omega}{\omega+\eta'}$ and let
$R_{c,2}:=[Q_1\times(\widehat{Q_2})^\complement]\cap\Omega^{**}$. Then, by Fubini's theorem, we conclude
that
\begin{align}\label{eq-sar2a}
&\left\|S(a_R)\mathbf 1_{R_{c,2}}\right\|_{L^p}^p\nonumber\\
&\quad\lesssim\|a_R\|_{L^q}[\mu(Q_2)]^{p-\frac pq}\int_{d(x_2,z_{2,\beta_2}^{j_2})>C\delta^{j_2'}}
\left[\frac{1}{V_2(x_2,z_{2,\beta_2}^{j_2})}\right]^p\left[\frac{\delta^{j_2}}
{d_2(x_2,z_{2,\beta_2}^{j_2})}\right]^{p\eta'}
\,d\mu_2(x_2)\nonumber\\
&\quad\lesssim[\mu(R)]^{1-\frac pq}\|a_R\|_{L^q}^p\left[\frac{\ell(Q_2)}{\ell(\widehat{Q_2})}\right]^{p\eta'-\omega(1-p)}.
\end{align}
This is the desired estimate for this case.

{\it Case (II.3) $x_1\in Q_1^\complement$ and $x_2\in\widehat{Q_2}$.} In this case, since
$Q_1\times Q_2\in\mathcal M_1(\Omega)$, it follows that there exists a cube $P_2\supset Q_2$ such that
$Q_1\times P_2\in\mathcal M_2(\Omega)$. Then choose $\widehat{Q_1}\supset Q_1$ such that
$$
\mu\left(\left(\widehat{Q_1}\times P_2\right)\cap\Omega^*\right)
>\frac 12\mu\left(\widehat{Q_1}\times P_2\right)
$$
By the definition of $\widehat{Q_2}$, we find that, for any $(x_1,x_2)\in Q_1\times\widehat{Q_2}$,
\begin{align*}
\mathcal M_{\rm str}(\mathbf 1_\Omega)(x_1,x_2)&\ge\frac{1}{\mu_1(B(Q_1))\mu_2(B_2(\widehat{Q_2}))}
\int_{B_1(Q_1)}\int_{B_2(\widehat{Q_2})}\mathbf 1_{\Omega}(y_1,y_2)\,d\mu_2(y_2)\,d\mu_1(y_1)\\
&\gtrsim\frac{\mu((Q_1\times\widehat{Q_2})\cap\Omega)}{\mu(Q_1\times\widehat{Q_2})}\gtrsim 1.
\end{align*}
Moreover, we further conclude that $[CB_1(\widehat{Q_1})]\times[CB_2(P_2)]\subset\Omega^{**}$. Thus, since
$(x_1,x_2)\in(\Omega^{**})^\complement$ and $x_2\in C_2B_2(P_2)$, it follows that
$x_1\in[CB_1(\widehat{Q_1})]^\complement$. Let $\widehat{Q_1}:=Q_{1,\beta_1'}^{j_1'}$ with $j_1'\in\mathbb Z$ and
$\beta_1'\in\mathcal G_{1,j_1'}$. We consider the following two cases for $x_2$.

{\it Case (II.3.1) $x_2\in CB_2(Q_2)$.} In this case, applying an  argument similar
to that used in the proof of Case (I.2), we find that, for any given $\eta'\in(0,\eta)$
and for any $x_1\in[B_1(\widehat{Q_1})]^\complement$,
$$
\int_{CB_2(Q_2)}[S(a_R)(x_1,x_2)]^q\,d\mu_2(x_2)
\lesssim\left[\frac{1}{V_1(x_1,z_{1,\beta_1}^{j_1})}\right]^q
\left[\frac{\delta^{j_1}}{d_1(x_1,z_{1,\beta_1}^{j_1})}\right]^{q\eta'}\mu_1(Q_1)\|a_R\|_{L^q}^q,
$$
which, together with H\"older's inequality, further implies that
\begin{align}\label{eq-sar231}
&\left\|S(a_R)\mathbf 1_{[CB_1(\widehat{Q_1})]^\complement\times[CB_2(Q_2)]}\right\|_{L^p}^p\nonumber\\
&\quad=\int_{[CB_1(\widehat{Q_1})]^\complement}\int_{CB_2(Q_2)}[S(a_R)(x_1,x_2)]^p
\,d\mu_2(x_2)\,d\mu_1(x_1)\nonumber\\
&\quad\lesssim\|a_R\|_{L^q}^p[\mu_1(Q_1)]^{p-\frac pq}[\mu_2(Q_2)]^{1-\frac pq}\nonumber\\
&\quad\quad\times\int_{d_1(x_1,z_{1,\beta_1}^{j_1})>C\delta^{j_1'}}
\left[\frac{1}{V_1(x_1,z_{1,\beta_1}^{j_1})}\right]^p
\left[\frac{\delta^{j_1}}{d_1(x_1,z_{1,\beta_1}^{j_1})}\right]^{p,\eta'}\,d\mu_1(x_1)\nonumber\\
&\quad\lesssim[\mu(R)]^{1-\frac pq}\left[\frac{\ell(Q_1)}{\ell(\widehat{Q_1})}\right]^{p\eta'-\omega(1-p)}\|a_R\|_{L^q}^p.
\end{align}
This is the desired estimate for this case.

{\it Case (II.3.2) $x_2\in [CB_2(\widehat{Q_2})]\setminus[CB_2(Q_2)]$.} In this case, applying an
argument similar to that used in the proof of Case (I.1), we conclude that, for any given $\eta'\in(0,\eta)$
and for any $x_1\in[CB_1(\widehat{Q_1})]^\complement$,
\begin{equation*}
S(a_R)(x_1,x_2)\lesssim[\mu(R)]^{1-\frac 1q}\|a_R\|_{L^q}\frac{1}{V_1(x_1,z_{1,\beta_1}^{j_1})}\frac{1}{V_2(x_2,z_{2,\beta_2}^{j_2})}
\left[\frac{\delta^{j_1}}{d_1(x_1,z_{1,\beta_1})}\right]^{\eta'}
\left[\frac{\delta^{j_2}}{d_2(x_2,z_{2,\beta_2}^{j_2})}\right]^{\eta'},
\end{equation*}
and consequently
\begin{equation}\label{eq-sar232}
\left\|S(a_R)\mathbf{1}_{[CB_1(\widehat{Q_1})]^\complement\times\{[CB_2(\widehat{Q_2})]
\setminus[CB_2(Q_2)]\}}\right\|_{L^p}^p
\lesssim[\mu(R)]^{1-\frac qp}\left[\frac{\ell(Q_2)}{\ell(\widehat{Q_2})}\right]^{p\eta'-\omega(1-p)}
\|a_R\|_{L^q}^p.\nonumber
\end{equation}
This is also the desired estimate in this case.

To summarize, let $R_{c,3}:=(Q_1^\complement\times\widehat{Q_2})\cap(\Omega^{**})^\complement$.
Combining \eqref{eq-sar231} and \eqref{eq-sar232} and applying an argument similar to that used in the proof of Case (I.3),
we obtain, for any given $\eta'\in(0,\eta)$ satisfying $p>\frac{\omega}{\omega+\eta'}$,
\begin{equation}\label{eq-sar3a}
\left\|S(a_R)\mathbf 1_{R_{c,3}}\right\|_{L^p}^p\lesssim[\mu(R))]^{1-\frac pq}
\left[\frac{\ell(Q_1)}{\ell(\widehat{Q_1})}\right]^{p\eta'-\omega(1-p)}\|a_R\|_{L^q}^p.
\end{equation}

Finally, applying \eqref{eq-sar1a}, \eqref{eq-sar2a}, \eqref{eq-sar3a},
and an argument similar to that used in the proof
of \eqref{eq-sarp1}, we conclude that
$$
\left\|S(a)\mathbf 1_{(\Omega^{**})^\complement}\right\|_{L^p}\lesssim 1.
$$
This finishes the proof of Proposition \ref{prop-at} in the case $p\in(1,2)$ and hence  Proposition \ref{prop-at}.
\end{proof}

Finally, we present the proof of Proposition \ref{prop-replq}.

\begin{proof}[Proof of Proposition \ref{prop-replq}]
By the definition of $H^p$, we only need to show that $\widetilde a\in({\rm CMO}^p_{L^2})'$ and
$$
 \widetilde a=\sum_{\genfrac{}{}{0pt}{}{k_1\in{\mathbb Z}}{\alpha_1\in\mathcal G_{1,k_1}}}
\sum_{\genfrac{}{}{0pt}{}{k_2\in{\mathbb Z}}{\alpha_2\in\mathcal G_{2,k_2}}}
\left<a,\psi_{\alpha_1,\alpha_2}^{k_1,k_2}\right>\psi_{\alpha_1,\alpha_2}^{k_1,k_2}
$$
in $({\rm CMO}^p_{L^2})'$. To this end, for any $N\in\mathbb N$, let
$$
a_N:=\sum_{(j_1,j_2,\beta_1,\beta_2)\in\mathcal I_N}\left<a,\psi_{\beta_1,\beta_2}^{j_1,j_2}\right>\psi_{\beta_1,\beta_2}^{j_1,j_2}
$$
with $\mathcal I_N$ as in \eqref{eq-defin}. Since $\mathcal I_N$ is finite, it then follows that $a_N\in H^p_{L^2}$.
Moreover, for any $M,N\in\mathbb N$ with $M>N$, $k_1,k_2\in\mathbb Z$, $\alpha_1\in\mathcal G_{1,k_1}$,
and $\alpha_2\in\mathcal G_{2,k_2}$,
$$
\left< a_M-a_N,\psi_{\alpha_1,\alpha_2}^{k_1,k_2}\right>
=\begin{cases}
\left<a,\psi_{\alpha_1,\alpha_2}^{k_1,k_2}\right>
& \textup{if }(k_1,k_2,\alpha_1,\alpha_2)\in\mathcal I_M\setminus\mathcal I_N,\\
0 & \textup{otherwise},
\end{cases}
$$
which further implies that, for any $(x_1,x_2)\in X_1\times X_2$,
$$
 S(a_M-a_N)(x_1,x_2)=\left[\sum_{(k_1,k_2,\alpha_1,\alpha_2)\in\mathcal I_M\setminus\mathcal I_N}
\left|
\left< a,\psi_{\alpha_1,\alpha_2}^{k_1,k_2}\right>\right|^2\widetilde{\mathbf 1}_{Q_{\alpha_1,\alpha_2}^{k_1+1,k_2+1}}
(x_1,x_2)\right]^{\frac 12}.
$$
Using $ S(a)\in L^p$ and the dominated convergence theorem, we find that
$$
\lim_{m,n\to\infty}\|S(a_M-a_N)\|_{L^p}=0.
$$
This shows that $\{a_N\}_{N=1}^\infty$ is a Cauchy sequence of $H^p$. By Proposition \ref{prop-com1}, we find that
there exists $\widetilde a\in({\rm CMO}^p_{L^2})'$ such that $\lim_{N\to\infty}a_N=\widetilde{a}$ in
$({\rm CMO}^p_{L^2})'$ and
\begin{equation}\label{eq-rea}
\widetilde{a}=\sum_{k_1,\alpha_1}\sum_{k_2,\alpha_2}
\left<\widetilde{a},\psi_{\alpha_1,\alpha_2}^{k_1,k_2}\right>\psi_{\alpha_1,\alpha_2}^{k_1,k_2}
\end{equation}
in $({\rm CMO}^p_{L^2})'$. Now, we show that $\widetilde{a}=a$ in the following setting, for any
$\varphi\in{\rm CMO}^p_{L^2}\cap L^{q'}$,
\begin{equation}\label{eq-a=wa}
\left<\widetilde{a},\varphi\right>=\int_{X_1}\int_{X_2}a(x_1,x_2)\varphi(x_1,x_2)\,d\mu_2(x_2)\,d\mu_1(x_1).
\end{equation}
To achieve this, by \eqref{eq-rea} and the wavelet reproducing formula \eqref{eq-re}, it suffices to show that, for any
$k_1,k_2\in\mathbb Z$, $\alpha_1\in\mathcal G_{1,k_1}$, and $\alpha_2\in\mathcal G_{2,k_2}$,
$$
\left<\widetilde{a},\psi_{\alpha_1,\alpha_2}^{k_1,k_2}\right>=\left\langle a,\psi_{\alpha_1,\alpha_2}^{k_1,k_2}\right\rangle.
$$
Indeed, by the definition of $\widetilde a$, we conclude that
\begin{align*}
\left<\widetilde{a},\psi_{\alpha_1,\alpha_2}^{k_1,k_2}\right>&=\lim_{N\to\infty}\left\langle a_N,\psi_{\alpha_1,\alpha_2}^{k_1,k_2}\right\rangle
=\lim_{N\to\infty}\int_{X_1}\int_{X_2}a_N(x_1,x_2)\psi_{\alpha_1,\alpha_2}^{k_1,k_2}(x_1,x_2)
\,d\mu_2(x_2)\,d\mu_1(x_1)\\
&=\int_{X_1}\int_{X_2}a(x_1,x_2)\psi_{\alpha_1,\alpha_2}^{k_1,k_2}(x_1,x_2)\,d\mu_2(x_2)\,d\mu_1(x_1)
 =\left<a,\psi_{\alpha_1,\alpha_2}^{k_1,k_2}\right>.
\end{align*}
This proves \eqref{eq-a=wa}. Thus, we conclude that $ \widetilde a\in H^p$. This finishes the proof of Proposition \ref{prop-replq}.
\end{proof}

\subsection{A Discrete Product Calder\'{o}n-Type Reproducing Formula} \label{ss-ctrf}

In the next two subsections, we focus on the proof of the
``only if'' part of Theorem \ref{thm-h=a}. Since
any product exp-ATI does not have bounded support, the
Calder\'on reproducing formula derived from it is no
longer useful to this end. To remedy this, motivated
by Han et al. \cite[Proposition 2.5]{hhl16}, in this section
we establish a new Calder\'on-type reproducing formula,
which is of independent interest and may be useful in solving other analysis
problems.

To this end, we first recall the kernel with bounded support in \cite[Lemma 5.8]{HYY24}.

\begin{proposition}[{\cite[Proposition 5.6]{HYY24}}]
There exists a sequence of bounded linear operators $\{S_k\}_{k \in \mathbb{Z}}$ on $L^2(X)$ such that,
for any $k \in \mathbb Z$, the kernel of $S_k$, still denoted by $S_k$,
has the following properties:
\begin{enumerate}
\item for any $x, y\in X$, $S_k(x,y)=S_k(y,x)$ and $S_k(x,y)\neq 0$ implies that
$d(x,y)\le(2A_0)^3C^\natural\delta^k$ with $C^\natural$ and $\delta$ the same as in Lemma \ref{lem-cube},
\item for any $x,y\in X$,
$$
0\le S_k(x,y)\le\frac C{V_{\delta^k}(x)+V_{\delta^k}(y)},
$$
\item for any $x, x', y\in X$ with $d(x,x')\le\delta^k$,
$$
\left|S_k(x,y)-S_k(x',y)\right|\le\left[\frac{d(x,x')}{\delta^k}\right]^\eta\frac C{V_{\delta^k}(x)+V_{\delta^k}(y)}
$$
with $\eta$ the same as in Theorem \ref{thm-wave},
\item for any $x, x', y, y'\in X$ with $d(x,x')\le\delta^k$ and $d(y,y')\le\delta^k$,
$$
\left|\left[S_k(x,y)-S_k(x',y)\right]-\left[S_k(x,y')-S_k(x',y')\right]\right|
\le\left[\frac{d(x,x')}{\delta^k}\right]^\eta
\left[\frac{d(y,y')}{\delta^k}\right]^\eta\frac C{V_{\delta^k}(x)+V_{\delta^k}(y)},
$$
\item for any $x\in X$, $\int_X S_k(x,y)\,d\mu(y)=1$.
\end{enumerate}
\end{proposition}
For any $k\in\mathbb Z$, let
\begin{equation} \label{eq-pd}
D_k:=S_k-S_{k-1},
\end{equation}
then, in the sense of $L^2(X)$,
the identity operator $I$ can be written as $I = \sum_{k\in \mathbb{Z}} D_k$.

Now, we return to the product setting. For any $k_1, k_2\in\mathbb Z$, let $D_{k_1,k_2}:=D_{1,k_1}\otimes D_{2,k_2}$,
where, for any $i\in\{1,2\}$, $\{D_{i,k_i}\}_{k_i\in\mathbb Z}$ is the same as in \eqref{eq-pd}
associated with $X_i$.
More precisely, the kernel of $D_{k_1,k_2}$, still denoted by $D_{k_1,k_2}$, is defined by
setting, for any $(x_1, y_1),(x_2,y_2)\in X_1\times X_2$,
\begin{equation}\label{eq-pdd}
D_{k_1,k_2}(x_1,x_2,y_1,y_2):=D_{1,k_1}(x_1,y_1)D_{2,k_2}(x_2,y_2).
\end{equation}

The following Calder\'on-type reproducing formula
is the main result of this  subsection.
\begin{theorem}\label{thm-ctype}
Let $\{D_{k_1,k_2}\}_{k_1, k_2\in\mathbb Z}$ be the same as in \eqref{eq-pdd}. Then there exist $N,j_0\in\mathbb N$ such that, for
any $k_1,k_2\in\mathbb Z$, $\alpha_1\in \mathcal A_{1,k_1+j_0}$,
$\alpha_2\in\mathcal A_{2,k_2+j_0}$,
$x_{1,\alpha_1}^{k_1+j_0}\in Q_{1,\alpha_1}^{k_1+j_0}$, $x_{2,\alpha_2}^{k_2+j_0}\in Q_{2,\alpha_2}^{k_2+j_0}$,
and $f\in L^2\cap H^p$,
\begin{align}\label{eq-ctype}
f(x_1,x_2)
&=\sum_{\genfrac{}{}{0pt}{}{k_1\in{\mathbb Z}}{\alpha_1\in\mathcal  A_{1,k_1+j_0}}}
\sum_{\genfrac{}{}{0pt}{}{k_2\in{\mathbb Z}}{\alpha_2\in\mathcal A_{2,k_2+j_0}}}\mu_1\left(Q_{1, \alpha_1}^{k_1+j_0}\right)
\mu_2\left(Q_{2, \alpha_2}^{k_2+j_0}\right)\nonumber\\
&\quad\times D_{k_1,k_2}^N\left(x_1, x_2, x_{1, \alpha_1}^{k_1+j_0},x_{2, \alpha_2}^{k_2+j_0}\right)
D_{k_1,k_2}(g)\left(x_{1, \alpha_1}^{k_1+j_0},x_{2, \alpha_2}^{k_2+j_0}\right)
\end{align}
and
\begin{align}\label{eq-ctype2}
f(x_1,x_2)&=\sum_{\genfrac{}{}{0pt}{}{k_1\in{\mathbb Z}}{\alpha_1\in\mathcal  A_{1,k_1+j_0}}}
\sum_{\genfrac{}{}{0pt}{}{k_2\in{\mathbb Z}}{\alpha_2\in\mathcal A_{2,k_2+j_0}}}
\mu_1\left(Q_{1, \alpha_1}^{k_1+j_0}\right)\mu_2\left(Q_{2, \alpha_2}^{k_2+j_0}\right)\nonumber\\
&\quad\times D_{k_1,k_2}\left(x_1, x_2,x_{1, \alpha_1}^{k_1+j_0},x_{2, \alpha_2}^{k_2+j_0}\right)
D_{k_1,k_2}^N(h)\left(x_{1, \alpha_1}^{k_1+j_0},x_{2, \alpha_2}^{k_2+j_0}\right),
\end{align}
converging in both $L^2$ and $H^p$,
where $\|g\|_{L^2}\sim\|f\|_{L^2}$,  $ \|g\|_{H^p}\sim\|f\|_{H^p}$, $\|h\|_{L^2}\sim\|f\|_{L^2}$, and  $\|h\|_{H^p}\sim\|f\|_{H^p}$,
with the positive equivalent constants independent of $g$, $x_{1,\alpha_1}^{k_1+j_0}$, and $x_{2,\alpha_2}^{k_2+j_0}$ and where, for any $k_1, k_2\in\mathbb Z$,
$$D_{k_1,k_2}^N:=\sum_{|l_1-k_1|\le N}\sum_{|l_2-k_2|\le N}D_{k_1,k_2}.$$
\end{theorem}

\begin{proof}
First, we let $\mathbb Z_N^2:=\{(l_1, l_2):\ |l_1| \le N , |l_2| \le N \}$ and
$$
(\mathbb Z_N^2)^\complement=\{(l_1, l_2):\ |l_1| \le N , |l_2| > N+1\}\cup
 \{(l_1, l_2): |l_1| > N+1 , l_2 \in \mathbb Z \}.
$$
Let $f\in L^2$. Then, by the definition of $\{D_{k_1,k_2}\}_{k_1,k_2\in\mathbb Z}$,
we have
\begin{align*}
f(\cdot, *)&=\sum_{k_1=-\infty}^\infty\sum_{k_2=-\infty}^\infty D_{k_1,k_2}f(\cdot, *)=
\sum_{k_1,k_2\in\mathbb Z}\sum_{l_1=-\infty}^\infty\sum_{l_2=-\infty}^\infty D_{k_1+l_1,k_2+l_2}D_{k_1,k_2}f(\cdot, *)\\
&=\sum_{k_1,k_2\in\mathbb Z}\sum_{(l_1, l_2) \in \mathbb Z_N^2}
D_{k_1+l_1,k_2+l_2}D_{k_1,k_2}f(\cdot, *)+\sum_{k_1,k_2\in\mathbb Z}\sum_{(l_1,l_2)\in(\mathbb Z_N^2)^\complement}
D_{k_1+l_1,k_2+l_2}D_{k_1,k_2}f(\cdot, *)\\
&=\sum_{k_1,k_2\in\mathbb Z}\int_XD_{k_1,k_2}^N(\cdot,*,y_1,y_2)D_{k_1,k_2}f(y_1,y_2)\,d\mu(y_1,y_2)+R_{N}^{(1)}f(\cdot,*),
\end{align*}
where
$$
R_{N}^{(1)}f:=\sum_{k_1,k_2\in\mathbb Z}\sum_{(l_1,l_2)\in(\mathbb Z_N^2)^\complement}D_{k_1+l_1,k_2+l_2}D_{k_1,k_2}f.
$$
Splitting the integral region into dyadic rectangles, we further obtain
\begin{align*}
f(\cdot, *)&=\sum_{\genfrac{}{}{0pt}{}{k_1\in{\mathbb Z}}{\alpha_1\in\mathcal  A_{1,k_1+j_0}}}
\sum_{\genfrac{}{}{0pt}{}{k_2\in{\mathbb Z}}{\alpha_2\in\mathcal A_{2,k_2+j_0}}}
\int_{Q_{\alpha_1,\alpha_2}^{k_1+j_0,k_2+j_0}}D_{k_1,k_2}^N(\cdot, *, y_1,y_2)D_{k_1,k_2}f(y_1,y_2)\,d\mu(y_1,y_2)
+R_{N}^{(1)}f(\cdot,*)\\
&=\sum_{k_1,\alpha_1}\sum_{k_2,\alpha_2}\mu\left(Q_{\alpha_1,\alpha_2}^{k_1+j_0,k_2+j_0}\right)
D_{k_1,k_2}^N\left(\cdot, *,x_{1,\alpha_1}^{k_1+j_0},x_{2,\alpha_2}^{k_2+j_0}\right)
D_{k_1,k_2}f\left(x_{1,\alpha_1}^{k_1+j_0},x_{2,\alpha_2}^{k_2+j_0}\right)+R_N^{(1)}f(\cdot,*)\\
&\quad+\sum_{k_1,\alpha_1}\sum_{k_2,\alpha_2}\int_{Q_{\alpha_1,\alpha_2}^{k_1+j_0,k_2+j_0}}
\left[D^N_{k_1,k_2}(\cdot, *, y_1,y_2)-D^N_{k_1,k_2}\left(\cdot, *,x_{1,\alpha_1}^{k_1+j_0},
x_{2,\alpha_2}^{k_2+j_0}\right)\right]\\
&\quad\times D_{k_1,k_2}f(y_1,y_2)\,d\mu(y_1,y_2)\\
&\quad+\sum_{k_1,\alpha_1}\sum_{k_2,\alpha_2}D_{k_1,k_2}^N
 \left(\cdot, *, x_{1,\alpha_1}^{k_1+j_0},x_{2,\alpha_2}^{k_2+j_0}\right)\\
&\quad\times\int_{Q_{\alpha_1,\alpha_2}^{k_1+j_0,k_2+j_0}}\left[D_{k_1,k_2}f(y_1,y_2)-D_{k_1,k_2}f
\left(x_{1,\alpha_1}^{k_1+j_0},x_{2,\alpha_2}^{k_2+j_0}\right)\right]\,d\mu(y_1,y_2)\\
&=:T_Nf(\cdot, *)+R_{N}^{(1)}f(\cdot, *)+R_{N}^{(2)}f(\cdot, *)+R_{N}^{(3)}f(\cdot, *),
\end{align*}
where all the series converge in $L^2$. Let $R_N:=R_{N}^{(1)}+R_{N}^{(2)}+R_{N}^{(3)}$. Now, we show
that, when $N$ is sufficiently large, we have
\begin{equation}\label{eq-ctrfaim}
\|R_N\|_{\mathcal L(L^2)}<\frac 12\textup{ and }\|R_N\|_{\mathcal L(H^p)}<\frac 12.
\end{equation}

We begin with proving the first inequality in \eqref{eq-ctrfaim}. To this end, we first estimate $\|R_N^{(1)}\|_{\mathcal{L}(L^2) }$.
Indeed, noting that $\{D_{k_1,k_2}\}_{k_1,\ k_2\in\mathbb Z}$ is of variable separable, applying an argument similar
to that used in the proof of \cite[Lemma 4.1(i)]{hlyy19}, we obtain, for any $k_1, k_2, j_1, j_2\in\mathbb Z$
and $(x_1,x_2), (y_1,y_2)\in X_1\times X_2$,
\begin{equation*}
\left|D_{j_1,j_2}D_{k_1,k_2}(x_1,x_2,y_1,y_2)\right|\lesssim\delta^{|j_1-k_1|\eta}\delta^{|j_2-k_2|\eta}
P_{1,\Gamma}(x_1,y_1;\delta^{j_1\wedge k_1})P_{2,\Gamma}(x_2,y_2;\delta^{j_2\wedge k_2}),
\end{equation*}
where $\Gamma\in(0,\infty)$ and the implicit positive constant may depend on $\Gamma$ and, for any $i\in \{1,2\}$,
$P_{i,\Gamma}$ is the same as in \eqref{eq-defpg}  associated with $X_i$. By this and Proposition  \ref{prop-fs},
we conclude that, for any $j_1, j_2, k_1, k_2\in\mathbb Z$,
$$
\left\|D_{j_1,j_2}D_{k_1,k_2}\right\|_{\mathcal L(L^2)}\lesssim\delta^{|j_1-k_1|\eta}\delta^{|j_2-k_2|\eta}.
$$
According to the symmetry of $D_{k_1,k_2}$, on one hand we find that, for any $k_1, k_2, k_1', k_2'\in\mathbb Z$
and $(l_1,l_2),(l_1',l_2')\in(\mathbb Z_N^2)^\complement$,
\begin{align*}
&\left\|\left(D_{k_1+l_1,k_2+l_2}D_{k_1,k_2}\right)^*D_{k_1'+l_1',k_2'+l_2'}D_{k_1',k_2'}\right\|_{\mathcal L(L^2)}\\
&\quad=\left\|D_{k_1,k_2}D_{k_1+l_1,k_2+l_2}D_{k_1'+l_1',k_2'+l_2'}D_{k_1',k_2'}\right\|_{\mathcal L(L^2)}
\lesssim\delta^{|l_1|\eta}\delta^{|l_1'|\eta}\delta^{|l_2|\eta}\delta^{|l_2'|\eta},
\end{align*}
and, on the other hand, we have
\begin{align*}
\left\|\left(D_{k_1+l_1,k_2+l_2}D_{k_1,k_2}\right)^*D_{k_1'+l_1',k_2'+l_2'}D_{k_1',k_2'}\right\|_{\mathcal L(L^2)}
&\lesssim\left\|D_{k_1+l_1,k_2+l_2}D_{k_1'+l_1',k_2'+l_2'}\right\|_{\mathcal L(L^2)}\\
&\lesssim\delta^{|(k_1+l_1)-(k_1'-l_1')|\eta}\delta^{|(k_2+l_2)-(k_2'-l_2')|\eta}\\
&\lesssim \delta^{|k_1-k_1'|\eta}\delta^{|k_2-k_2'|\eta}\delta^{-|l_1|\eta}\delta^{-|l_1'|\eta}\delta^{-|l_2|\eta}
\delta^{-|l_2'|\eta}.
\end{align*}
Taking the geometric mean of the two above inequality, we obtain, for any given $\theta\in(0, \frac{1}{2})$,
\begin{align*}
&\left\|\left(D_{k_1+l_1,k_2+l_2}D_{k_1,k_2}\right)^*D_{k_1'+l_1',k_2'+l_2'}D_{k_1',k_2'}\right\|_{\mathcal L(L^2)}\\
&\quad\lesssim\delta^{|k_1-k_1'|\eta\theta}\delta^{|k_2-k_2'|\eta\theta}
\delta^{|l_1|\eta(1-2\theta)}\delta^{|l_1'|\eta(1-2\theta)}\delta^{|l_2|\eta(1-2\theta)}
\delta^{|l_2'|\eta(1-2\theta)}.
\end{align*}
Therefore, we conclude that, for any fixed $\theta\in(0, \frac{1}{2})$ and any $k_1, k_2, k_1', k_2'\in\mathbb Z$,
\begin{align*}
&\left\|\left(\sum_{(l_1,l_2)\in(\mathbb Z_N^2)^\complement}D_{k_1+l_1,k_2+l_2}D_{k_1,k_2}\right)^*
\left(\sum_{(l_1',l_2')\in(\mathbb Z_N^2)^\complement}D_{k_1'+l_1',k_2'+l_2'}D_{k_1',k_2'}\right)\right\|_{\mathcal L(L^2)}\\
&\quad\lesssim\sum_{(l_1,l_2)\in(\mathbb Z_N^2)^\complement}\sum_{(l_1',l_2')\in(\mathbb Z_N^2)^\complement}
\left\|\left(D_{k_1+l_1,k_2+l_2}D_{k_1,k_2}\right)^*D_{k_1'+l_1',k_2'+l_2'}D_{k_1',k_2'}\right\|_{\mathcal L(L^2)}\\
&\quad\lesssim\sum_{(l_1,l_2)\in(\mathbb Z_N^2)^\complement}\sum_{(l_1',l_2')\in(\mathbb Z_N^2)^\complement}
\delta^{|k_1-k_1'|\eta\theta}\delta^{|k_2-k_2'|\eta\theta}
\delta^{|l_1|\eta(1-2\theta)}\delta^{|l_1'|\eta(1-2\theta)}\delta^{|l_2|\eta(1-2\theta)}
\delta^{|l_2'|\eta(1-2\theta)}\\
&\quad\lesssim\delta^{N\eta(1-2\theta)}\delta^{|k_1-k_1'|\eta\theta}\delta^{|k_2-k_2'|\eta\theta}.
\end{align*}
Similarly, for any given $\theta\in(0,1)$ and any $k_1, k_2,k_1', k_2'\in\mathbb Z$,
\begin{align*}
&\left\|\left(\sum_{(l_1,l_2)\in(\mathbb Z_N^2)^\complement}D_{k_1+l_1,k_2+l_2}D_{k_1,k_2}\right)
\left(\sum_{(l_1',l_2')\in(\mathbb Z_N^2)^\complement}D_{k_1'+l_1',k_2'+l_2'}D_{k_1',k_2'}\right)^*\right\|_{\mathcal L(L^2)}\\
&\quad\lesssim\delta^{N\eta(1-\theta)}\delta^{|k_1-k_1'|\eta\theta}\delta^{|k_2-k_2'|\eta\theta}.
\end{align*}
By the above two inequalities and the Cotlar--Stein lemma (see \cite[Lemma 5.12]{HYY24}), we find that, for any
$\varphi\in L^2$,
$$
R_N^{(1)}\varphi=\sum_{k_1,k_2\in\mathbb Z}\sum_{(l_1,l_2)\in(\mathbb Z_N^2)^\complement}D_{k_1+l_1,k_2+l_2}
D_{k_1,k_2}\varphi
$$
in $L^2$ and, for any fixed $\epsilon'\in(0,\eta)$,
$$
\left\|R_N^{(1)}\right\|_{\mathcal L(L^2)}\lesssim\delta^{N\epsilon'}.
$$
This is the desired estimate.

Now, we show the second inequality in \eqref{eq-ctrfaim} for $R_N^{(1)}$. To this end, for any $f\in L^2$,
we write
\begin{align*}
R_N^{(1)}f&=\sum_{k_1,k_2\in\mathbb Z}\sum_{(l_1,l_2)\in(\mathbb Z_N^2)^\complement}D_{k_1+l_1,k_2+l_2}
D_{k_1,k_2}f\\
&=\sum_{k_1,k_2\in\mathbb Z}\sum_{|l_1|\ge N+1}\sum_{l_2=-\infty}^\infty D_{k_1+l_1,k_2+l_2}
D_{k_1,k_2}f+\sum_{k_1,k_2\in\mathbb Z}\sum_{|l_1|\le N}\sum_{|l_2|\ge N+1}\cdots\\
&=:R_N^{(1,1)}f+R_N^{(1,2)}f.
 \end{align*}
Due to similarity, we only estimate $\|R_N^{(1,1)}\|_{\mathcal L(H^p)}$. Indeed, for any
$j_1',j_2'\in\mathbb Z$, $\beta_1'\in\mathcal G_{1,k_1'}$, and $\beta_2'\in\mathcal G_{2,k_2'}$,
\begin{align*}
\left<R_N^{(1,1)}f,\psi_{\beta_1',\beta_2'}^{j_1',j_2'}\right>
&=\sum_{\genfrac{}{}{0pt}{}{j_1\in{\mathbb Z}}{\beta_1\in\mathcal G_{1,j_1}}}
\sum_{\genfrac{}{}{0pt}{}{j_2\in{\mathbb Z}}{\beta_2\in\mathcal G_{2,j_2}}}
\left<f,\psi_{\beta_1,\beta_2}^{j_1,j_2}\right>\left<R_N^{(1,1)}\psi_{\beta_1,\beta_2}^{j_1,j_2},
\psi_{\beta_1',\beta_2'}^{j_1',j_2'}\right>\\
&=\sum_{j_1,\beta_1}\sum_{j_2,\beta_2}\sum_{k_1,k_2\in\mathbb Z}\sum_{|l_1|\ge N+1}\sum_{l_2=-\infty}^\infty
\left<f,\psi_{\beta_1,\beta_2}^{j_1,j_2}\right>\left<D_{k_1+l_1,k_2+l_2}D_{k_1,k_2}
\psi_{\beta_1,\beta_2}^{j_1,j_2},
\psi_{\beta_1',\beta_2'}^{j_1',j_2'}\right>.
\end{align*}
Moreover, by the definition of $\{D_{k_1,k_2}\}_{k_1,k_2\in\mathbb Z}$, we find that, for any $(x_1,x_2),(y_1,y_2)\in X_1\times X_2$,
$$
D_{k_1+l_1,k_2+l_2}D_{k_1,k_2}(x_1,x_2,y_1,y_2)=D_{1,k_1+l_1}D_{1,k_1}(x_1,y_1)
D_{2,k_2+l_2}D_{2,k_2}(x_2,y_2).
$$
Using this, we conclude that
\begin{align*}
&\left<D_{k_1+l_1,k_2+l_2}D_{k_1,k_2}\psi_{\beta_1,\beta_2}^{j_1,j_2},
\psi_{\beta_1',\beta_2'}^{j_1',j_2'}\right>\\
&\quad=\left[\int_{X_1}D_{1,k_1+l_1}D_{1,k_1}\psi_{1,\beta_1}^{j_1}(x_1)
\psi_{1,\beta_1'}^{j_1'}(x_1)\,d\mu_1(x_1)\right]
\left[\int_{X_2}D_{2,k_2+l_2}D_{2,k_2}\psi_{2,\beta_2}^{j_2}(x_2)
\psi_{2,\beta_2'}^{j_2'}(x_2)\,d\mu_2(x_2)\right]\\
&\quad=\left[\mu_1\left(Q_{1,\beta_1}^{j_1+1}\right)\mu_1\left(Q_{1,\beta_1'}^{j_1'+1}\right)
\mu_2\left(Q_{2,\beta_2}^{j_2+1}\right)\mu_2\left(Q_{2,\beta_2'}^{j_2'+1}\right)\right]^{\frac{1}{2}}
\int_{X_1}D_{1,k_1+l_1}D_{1,k_1}\Psi_{1,\beta_1}^{j_1}(x_1)\Psi_{1,\beta_1'}^{j_1'}(x_1)\,d\mu_1(x_1)\\
&\quad\quad\times\int_{X_2}D_{2,k_2+l_2}D_{2,k_2}\Psi_{2,\beta_2}^{j_2}(x_2)
\Psi_{2,\beta_2'}^{j_2'}(x_2)\,d\mu_2(x_2),
\end{align*}
where, for any $i\in\{1,2\}$, $j_i\in\mathbb Z$, and $\beta_i\in\mathcal G_{i,j_i}$,
$\Psi_{i,\beta_i}^{j_i}:=\psi_{i,\beta_i}^{j_i}/\mu_i(Q_{i,\beta_i}^{j_i+1})^{\frac{1}{2}}$. We only estimate the
term related to $X_1$ because the other terms are similar by the symmetry of $X_1$ and $X_2$.
Now, for any $x_1\in X_1$, let
$$
F_{1,\beta_1}^{j_1}(x_1):=\sum_{k_1=-\infty}^\infty\sum_{|l_1|\ge N+1}D_{1,k_1+l_1}D_{1,k_1}
\Psi_{1,\beta_1}^{j_1}(x_1)
$$
and
$$
F_{2,\beta_2}^{j_2}(x_2):=\sum_{k_2=-\infty}^\infty\sum_{l_2=-\infty}^\infty
D_{2,k_2+l_2}D_{2,k_2}\Psi_{2,\beta_2}^{j_2}(x_2).
$$
For any $i\in\{1,2\}$, $j_i'\in\mathbb Z$, and $\beta_i'\in\mathcal G_{i,j_i'}$, we have
\begin{align}\label{eq-rn1hp1}
&\left[\mu\left(Q_{\beta_1',\beta_2'}^{j_1'+1,j_2'+1}\right)\right]^{-\frac{1}{2}}
\left<R_N^{(1,1)}f,\psi_{\beta_1',\beta_2'}^{j_1',j_2'}\right>\nonumber\\
&\quad=\sum_{j_1,\beta_1}\sum_{j_2,\beta_2}\sum_{k_1,k_2\in\mathbb Z}\sum_{|l_1|\ge N+1}
\sum_{l_2=-\infty}^\infty\left<f,\psi_{\beta_1,\beta_2}^{j_1,j_2}\right>
\left<D_{k_1+l_1,k_2+l_2}D_{k_1,k_2}\psi_{\beta_1',\beta_2'}^{j_1',j_2'},
\Psi_{\beta_1,\beta_2}^{j_1,j_2}\right>\nonumber\\
&\quad=\sum_{j_1,\beta_1}\sum_{j_2,\beta_2}\left[\mu\left(Q_{\beta_1',
\beta_2'}^{j_1'+1,j_2'+1}\right)\right]^{\frac{1}{2}}
\left<f,\psi_{\beta_1,\beta_2}^{j_1,j_2}\right>\left<F_{1,\beta_1}^{j_1},\Psi_{1,\beta_1'}^{j_1'}\right>
\left<F_{2,\beta_2}^{j_2},\Psi_{2,\beta_2'}^{j_2'}\right>.
\end{align}
Next, we estimate $\langle F_{1,\beta_1}^{j_1},\Psi_{1,\beta_1'}^{j_1'}\rangle$
and
$\langle F_{2,\beta_2}^{j_2},\Psi_{2,\beta_2'}^{j_2'}\rangle$.
For the first item, we only consider
the case $j_1\le j_1'$ because the proof when $j_1>j_1'$ is similar and we omit the details here.
By the cancellation of $\Psi_{1,\beta_1'}^{j_1'}$, we find that
\begin{align}\label{eq-fpsi}
\left|\left<F_{1,\beta_1}^{j_1},\Psi_{1,\beta_1'}^{j_1'}\right>\right|
&=\int_{X_1}
  \left|F_{1,\beta_1}^{j_1}(x_1)
     -F_{1,\beta_1}^{j_1}\left(y_{1,\beta_1'}^{j_1'}\right)\right|
  \left|\Psi_{1,\beta_1'}^{j_1'}(x_1)\right|\,d\mu_1(x_1)\nonumber\\
&\le
  \int_{d_1(x_1,y_{1,\beta_1'}^{j_1'}) \le\delta^{j_1}}
  \left|F_{1,\beta_1}^{j_1}(x_1)
     -F_{1,\beta_1}^{j_1}\left(y_{1,\beta_1'}^{j_1'}\right)\right|
  \left|\Psi_{1,\beta_1'}^{j_1'}(x_1)\right|\,d\mu_1(x_1)\nonumber\\
&\quad
  +\int_{d_1(x_1,y_{1,\beta_1'}^{j_1'})>\delta^{j_1}}
  \left|F_{1,\beta_1}^{j_1}(x_1)\right|
  \left|\Psi_{1,\beta_1'}^{j_1'}(x_1)\right|\,d\mu_1(x_1)\nonumber\\
&\quad
  +\left|F_{1,\beta_1}^{j_1}\left(y_{1,\beta_1'}^{j_1'}\right)\right|
  \int_{d_1(x_1,y_{1,\beta_1'}^{j_1'})>\delta^{j_1}}
  \left|\Psi_{1,\beta_1'}^{j_1'}(x_1)\right|\,d\mu_1(x_1)\nonumber\\
&=:\mathrm{IV}+\mathrm{V}+\mathrm{VI}.
\end{align}
In order to show the desired estimates for $\rm IV$, $\rm V$, and $\rm VI$, we need some estimates
on the size and the regularity of $F_{1,\beta_1}^{j_1}$, which need the following two estimates (see \cite[Lemma 5.9]{HYY24}):
for any given $\eta'\in(0,\eta)$, $\eta''\in(0,\eta')$, and $\Gamma\in(\eta',\infty)$ and
for any $x_1\in X_1$,
\begin{equation}\label{eq-ddesize}
\left|D_{1,k_1+l_1}D_{1,k_1}
   \Psi_{1,\beta_1}^{j_1}(x_1)
\right|\lesssim\delta^{|l_1|(\eta-\eta')}
\delta^{|j_1-[k_1\wedge(k_1+l_1)]|\eta'}
P_{1,\Gamma}\left(x_1,y_{1,\alpha_1}^{j_1};\delta^{k_1\wedge j_1\wedge(k_1+l_1)}\right),
\end{equation}
and, for any $x_1,  x_1'\in X_1\times X_2$ with $d_1(x_1,x_1')\le\delta^{k_1\wedge j_1\wedge(k_1+l_1)}$,
\begin{align}\label{eq-ddereg}
&\left|D_{1,k_1+l_1}D_{1,k_1}\Psi_{1,\beta_1}^{j_1}(x_1)-D_{1,k_1+l_1}D_{1,k_1}
\Psi_{1,\beta_1}^{j_1}(x_1)\right|\\
&\quad\lesssim
\delta^{|l_1|(\eta-\eta')}
\delta^{|j_1-[k_1\wedge(k_1+l_1)]|(\eta'-\eta'')}
\left[\frac{d_1(x_1,x_1')}{\delta^{k_1\wedge j_1\wedge(k_1+l_1)}}\right]^{\eta''}
P_{1,\Gamma}\left(x_1,y_{1,\alpha_1}^{j_1};\delta^{k_1\wedge j_1\wedge(k_1+l_1)}\right).\nonumber
\end{align}
In the next, we will explain how to use
\eqref{eq-ddesize} and \eqref{eq-ddereg} to estimate
$F_{1,\beta_1}^{j_1}$.
Now, for any $x_1\in X_1$, we have
\begin{align*}
F_{1,\beta_1}^{j_1}(x_1)&=\sum_{k_1=-\infty}^\infty\left[\sum_{l_1=N+1}^\infty D_{1,k_1+l_1}D_{1,k_1}\Psi_{1,\beta_1}^{j_1}(x_1)
+\sum_{l_1=-\infty}^{-N-1}D_{1,k_1+l_1}D_{1,k_1}\Psi_{1,\beta_1}^{j_1}(x_1)\right]\\
&=:F_{1,\beta_1}^{j_1,1}(x_1)+F_{2,\beta_1}^{j_1,2}(x_1).
\end{align*}

We next consider the size and the regularity of $F_{1,\beta_1}^{j_1}$. By symmetry, we only consider
the function $F_{1,\beta_1}^{j_1,1}$ because the estimates of $F_{1,\beta_1}^{j_1,2}$ is similar and we omit the details.
Indeed, by \eqref{eq-ddesize}, we conclude that, for any $x_1\in X_1$,
\begin{align}\label{eq-sizef1}
\left|F_{1,\beta_1}^{j_1,1}(x_1)\right|
&=\left|\sum_{k_1=-\infty}^\infty\sum_{l_1\ge N+1}D_{1,k_1+l_1}D_{1,k_1}\Psi_{1,\beta_1}^{j_1}(x_1)\right|\\
&\lesssim\sum_{k_1=-\infty}^{j_1}\delta^{(j_1-k_1)\eta'}\sum_{l_1\ge N+1}\delta^{l_1(\eta-\eta')}P_{1,\Gamma}
\left(x_1,y_{1,\beta_1}^{j_1};\delta^{k_1}\right)\notag\\
&\quad+\sum_{k_1=j_1+1}^{\infty}\delta^{(k_1-j_1)\eta'}\sum_{l_1\ge N+1}\delta^{l_1(\eta-\eta')}
P_{1,\Gamma}\left(x_1,y_{1,\beta_1}^{j_1};\delta^{j_1}\right)\notag\\
&\lesssim\delta^{N(\eta-\eta')}P_{1,\Gamma}\left(x_1,y_{1,\beta_1}^{j_1};\delta^{j_1}\right)
\left[\sum_{k_1=-\infty}^{j_1}\delta^{(j_1-k_1)(\eta'-\Gamma)}
+\sum_{k_1=j_1+1}^{\infty}\delta^{(k_1-j_1)\eta'}\right]\notag\\
&\lesssim\delta^{N(\eta-\eta')}P_{1,\eta''}\left(x_1,y_{1,\beta_1}^{j_1};\delta^{j_1}\right).\nonumber
\end{align}
This is the desired estimate for the size of $F_{1,\beta_1}^{j_1,1}$.

Now, we consider the regularity of $F_{1,\beta_1}^{j_1,1}$. Suppose $x_1, x'_1\in X_1$ satisfying $d_1(x_1,x_1')\le\delta^{j_1}$.
Notice that, when $k_1\in\mathbb Z\cap(-\infty,j_1]$, we have $d_1(x_1,x_1')\le\delta^{k_1}$. By this and \eqref{eq-ddereg},
we have
\begin{align*}
&\left|F_{1,\beta_1}^{j_1,1}(x_1)-F_{1,\beta_1}^{j_1,1}(x_1')\right|\\
&\quad\le\sum_{k_1=-\infty}^{j_1}\sum_{l_1=N+1}^\infty\left|D_{1,k_1+l_1}D_{1,k_1}
\Psi_{1,\beta_1}^{j_1}(x_1)-D_{1,k_1+l_1}D_{1,k_1}\Psi_{1,\beta_1}^{j_1}(x_1')\right|\\
&\quad\quad+\sum_{k_1=j_1+1}^{\infty}\sum_{l_1=N+1}^\infty\left|D_{1,k_1+l_1}D_{1,k_1}
\Psi_{1,\beta_1}^{j_1}(x_1)-D_{1,k_1+l_1}D_{1,k_1}\Psi_{1,\beta_1}^{j_1}(x_1')\right|\\
&\quad\lesssim\sum_{k_1=-\infty}^{j_1}\delta^{(j_1-k_1)(\eta'-\eta'')}
\left[\frac{d_1(x_1,x_1')}{\delta^{k_1}}\right]^{\eta''}\sum_{l_1=N+1}^\infty
\delta^{l_1(\eta-\eta')}P_{1,\Gamma}\left(x_1,y_{1,\beta_1}^{j_1};\delta^{k_1}\right)\\
&\quad\quad+\sum_{k_1=j_1+1}^{\infty}\delta^{(k_1-j_1)(\eta'-\eta'')}
\left[\frac{d_1(x_1,x_1')}{\delta^{j_1}}\right]^{\eta''}\sum_{l_1=N+1}^\infty
\delta^{l_1(\eta-\eta')}P_{1,\Gamma}\left(x_1,y_{1,\beta_1}^{j_1};\delta^{j_1}\right)\\
&\quad\lesssim\delta^{N(\eta-\eta')}\sum_{k_1=-\infty}^{j_1}\delta^{(j_1-k_1)(\eta'-\eta'')}
\left[\frac{d_1(x_1,x_1')}{\delta^{j_1}}\right]^{\eta''}
P_{1,\Gamma}\left(x_1,y_{1,\beta_1}^{j_1};\delta^{j_1}\right)\\
&\quad\quad+\delta^{N(\eta-\eta')}\sum_{k_1=j_1+1}^{\infty}\delta^{(k_1-j_1)(\eta'-\eta'')}
\left[\frac{d_1(x_1,x_1')}{\delta^{j_1}}\right]^{\eta''}
P_{1,\Gamma}\left(x_1,y_{1,\beta_1}^{j_1};\delta^{j_1}\right)\\
&\quad\lesssim\delta^{N(\eta-\eta')}\left[\frac{d_1(x_1,x_1')}{\delta^{j_1}}\right]^{\eta''}
P_{1,\eta''}\left(x_1,y_{1,\beta_1}^{j_1};\delta^{j_1}\right).
\end{align*}
Thus, we find that, if $d_1(x_1,x_1')\le\delta^{j_1}$, then
\begin{align*}
\left|F_{1,\beta_1}^{j_1,1}(x_1)-F_{1,\beta_1}^{j_1,1}(x'_1)\right|
&\lesssim\delta^{N(\eta-\eta')}\left[\frac{d_1(x_1,x_1')}{\delta^{j_1}}\right]^{\eta''}
P_{1,\eta''}\left(x_1,y_{1,\beta_1}^{j_1};\delta^{j_1}\right)\\
&\lesssim\delta^{N(\eta-\eta')}\left[\frac{d_1(x_1,x_1')}{\delta^{j_1}}\right]^{\eta''}
\left[P_{1,\eta''}\left(x_1,y_{1,\beta_1}^{j_1};\delta^{j_1}\right)
+P_{1,\eta''}\left(x_1',y_{1,\beta_1}^{j_1};\delta^{j_1}\right)\right];
\end{align*}
on the other hand, when $d_1(x_1,x_1')>\delta^{j_1}$, by \eqref{eq-sizef1}, we obtain
\begin{equation*}
\left|F_{1,\beta_1}^{j_1,1}(x_1)-F_{1,\beta_1}^{j_1,1}(x'_1)\right|
\lesssim\delta^{N(\eta-\eta')}\left[P_{1,\eta''}\left(x_1, y_{1,\beta_1}^{j_1};\delta^{j_1}\right)
+P_{1,\eta''}\left(x_1',y_{1,\beta_1}^{j_1};\delta^{j_1}\right)\right].
\end{equation*}
Combining the above two inequalities, we conclude that, for any $x_1,\ x_1'\in X_1$,
\begin{align}\label{eq-regf1}
&\left|F_{1,\beta_1}^{j_1,1}(x_1)-F_{1,\beta_1}^{j_1,1}(x_1')\right|\nonumber\\
&\quad  \lesssim\delta^{N(\eta-\eta')}
\min\left\{1,\left[\frac{d_1(x_1,x_1')}{\delta^{j_1}}\right]^{\eta''}\right\}
\left[P_{1,\eta''}\left(x_1,y_{1,\beta_1}^{j_1};\delta^{j_1}\right)
+P_{1,\eta''}\left(x_1',y_{1,\beta_1}^{j_1};\delta^{j_1}\right)\right],
\end{align}
where $\eta'\in(0,\eta)$ and $\eta''\in(0,\eta')$. This concludes the desired estimate.

The size and the regularity of $F_{1,\beta_1}^{j_1}(x_1)$ provide the desired estimates for
$\mathrm{IV}$, $\mathrm{V}$, and $\mathrm{VI}$.
Indeed, by \eqref{eq-regf1} and the size of $\Psi_{1,\beta_1'}^{j_1'}$, we have
\begin{align*}
\mathrm{IV}&\lesssim\delta^{N(\eta-\eta')}\int_{d_1(x_1,y_{1,\beta_1'}^{j_1'})\le\delta^{j_1}}
\left[\frac{d_1(x_1,y_{1,\beta_1'}^{j_1'})}{\delta^{j_1}}\right]^{\eta''}\\
&\quad\times  \left[P_{1,\eta''}\left(y_{1,\beta_1}^{j_1},y_{1,\beta_1'}^{j_1'};\delta^{j_1}\right)
+P_{1,\eta''}\left(x_1,y_{1,\beta_1'}^{j_1'};\delta^{j_1}\right)\right]
\widetilde{\mathcal E}^*_1\left(x_1,y_{1,\beta_1'}^{j_1'}\right) d\mu_1(x_1)\nonumber\\
&\lesssim\delta^{N(\eta-\eta')}\delta^{(j_1'-j_1)\eta''}P_{1,\eta''}
\left(y_{1,\beta_1}^{j_1},y_{1,\beta_1'}^{j_1'};\delta^{j_1}\right).
\nonumber
\end{align*}
This is the desired estimate for $\rm IV$.

To deal with $\mathrm V$, by \eqref{eq-sizef1} and the size of $\Psi_{1,\beta_1'}^{j_1'}$, we conclude
that, for any given $\Gamma\in(0,\infty)$,
\begin{align*}
\mathrm{V}&\lesssim\delta^{N(\eta-\eta')}\int_{d_1(x_1,y_{1,\beta_1'}^{j_1'})>\delta^{j_1}}
P_{1,\eta''}\left(x_1,y_{1,\beta_1}^{j_1};\delta^{j_1}\right)
\widetilde{\mathcal E}^*_{1,j'_1}\left(x_1,y_{1,\beta_1'}^{j_1'}\right)\,d\mu_1(x_1)\\
&\lesssim\delta^{N(\eta-\eta')}\delta^{(j_1'-j_1)\Gamma}\int_{d_1(x_1,y_{1,\beta_1'}^{j_1'})>\delta^{j_1}}
P_{1,\eta''}\left(x_1,y_{1,\beta_1}^{j_1};\delta^{j_1}\right)
P_{1,\Gamma}\left(x_1,y_{1,\beta_1'}^{j_1'};\delta^{j_1}\right)\,d\mu_1(x_1)\nonumber\\
&\lesssim\delta^{N(\eta-\eta')}\delta^{(j_1'-j_1)\Gamma}P_{1,\eta''}
\left(y_{1,\beta}^{j_1},y_{1,\beta_1'}^{j_1'};\delta^{j_1}\right)
\int_{X_1}\frac{V_{1,\delta^{j_1}}(y_{1,\beta}^{j_1})+V_1(y_{1,\beta}^{j_1},y_{1,\beta_1'}^{j_1'})}
{V_{1,\delta^{j_1}}(y_{1,\beta}^{j_1})+V_1(x_1,y_{1,\beta_1}^{j_1})}\nonumber\\
&\quad\times\left[\frac{\delta^{k_1}+d_1(y_{1,\beta}^{j_1},y_{1,\beta_1'}^{j_1'})}
{\delta^{k_1}+d_1(x_1,y_{1,\beta}^{j_1})}\right]^{\eta''}
P_{1,\Gamma}\left(x_1,y_{1,\beta_1'}^{j_1'};\delta^{j_1}\right)\,d\mu_1(x_1)\nonumber\\
&\lesssim\delta^{N(\eta-\eta')}\delta^{(j_1'-j_1)\Gamma}P_{1,\eta''}
\left(y_{1,\beta}^{j_1},y_{1,\beta_1'}^{j_1'};\delta^{j_1}\right)
\int_{X_1}P_{1,\Gamma-\omega-\eta''}\left(x_1,y_{1,\beta_1'}^{j_1'};
\delta^{j_1}\right)\,d\mu_1(x_1)\nonumber\\
&\lesssim\delta^{N(\eta-\eta')}\delta^{(j_1'-j_1)\Gamma}P_{1,\eta''}
\left(y_{1,\beta}^{j_1},y_{1,\beta_1'}^{j_1'};\delta^{j_1}\right).\nonumber
\end{align*}
This is also the desired estimate.

Now, we estimate $\mathrm{VI}$. To this end, by \eqref{eq-sizef1} and the size of $\Psi_{1,\beta_1'}^{j_1'}$
again, we have, for any given $\Gamma\in(\eta',\infty)$,
\begin{align}\label{eq-est6}
\mathrm{VI}
&\lesssim
  \delta^{N(\eta-\eta')}P_{1,\eta''}\left(y_{1,\beta_1}^{j_1},y_{1,\beta_1'}^{j_1'};\delta^{j_1}\right)
  \int_{d_1(x_1,y_{1,\beta_1'}^{j_1'})>\delta^{j_1}}
  P_{1,\Gamma}\left(x_1,y_{1,\beta_1'}^{j_1'};\delta^{j_1'}\right)\,d\mu_1(x_1)\nonumber\\
&\lesssim\delta^{N(\eta-\eta')}\delta^{(j_1'-j_1)\Gamma}P_{1,\eta''}\left(y_{1,
\beta_1}^{j_1},y_{1,\beta_1'}^{j_1'};\delta^{j_1}\right)
\int_{d_1(x_1,y_{1,\beta_1'}^{j_1'})>\delta^{j_1}}
\frac{1}{V_1(y_{1,\beta_1'}^{j_1'},x_1)}\nonumber\\
&\quad\times\left[\frac{\delta^{j_1}}{d_1(y_{1,\beta_1'}^{j_1'},x_1)}\right]^\Gamma\,d\mu_1(x_1)\nonumber\\
&\lesssim\delta^{N(\eta-\eta')}\delta^{(j_1'-j_1)\eta''}P_{1,\Gamma}\left(y_{1,
\beta}^{j_1},y_{1,\beta_1'}^{j_1'};\delta^{j_1}\right).
\end{align}
This is the desired estimate.

Combining \eqref{eq-fpsi} through \eqref{eq-est6}, we conclude that, for any given $\eta'\in(0,\eta)$
and $\eta''\in(0,\eta')$,
\begin{equation}\label{eq-fpsi1}
\left|\left<F_{1,\beta_1}^{j_1},\Psi_{1,\beta_1'}^{j_1'}\right>\right|\lesssim
\delta^{N(\eta-\eta')}\delta^{|j_1'-j_1|\eta''}P_{1,\eta''}\left(y_{1,\beta}^{j_1},
y_{1,\beta_1'}^{j_1'};\delta^{j_1\wedge j_1'}\right).
\end{equation}

Applying an argument similar to that used in the estimations of \eqref{eq-sizef1} and \eqref{eq-regf1},
we conclude that $F_{2,\beta_2}^{j_2}$ has the following properties: for any given $\eta'\in(0,\eta)$
and $\eta''\in(0,\eta)$,
\begin{enumerate}
\item[(i)] for any $x_2\in X_2$,
\begin{equation*}
\left|F_{2,\beta_2}^{j_2}(x_2)\right|
\lesssim
P_{2,\eta'}\left(x_2,y_{2,\beta_2}^{j_2};\delta^{j_2}\right),
\end{equation*}
\item[(ii)] for any $x_2, x_2'\in X_2$,
\begin{align*}
 \left|F_{2,\beta_2}^{j_2}(x_2)-F_{2,\beta_2}^{j_2}(x_2')\right|
 \lesssim\min\left\{1,\left[\frac{d_2(x_2,x_2')}{\delta^{j_2}}\right]^{\eta'}\right\}
\left[P_{2,\eta''}\left(x_2,y_{2,\beta_2}^{j_2};\delta^{j_2}\right)
   +P_{2,\eta''}\left(x_2',y_{2,\beta_2}^{j_2};\delta^{j_2}\right)\right].
\end{align*}
\end{enumerate}
By these and an argument similar to that used
in the proof of \eqref{eq-fpsi1}, we can also obtain,
for any given $\eta'\in(0,\eta)$,
\begin{equation}\label{eq-fpsi2}
\left|\left<F_{2,\beta_2}^{j_2},\Psi_{2,\beta_2'}^{j_2'}\right>\right|\lesssim
\delta^{|j_2'-j_2|\eta'}P_{2,\eta'}\left(y_{2,\beta}^{j_2},y_{2,\beta_2'}^{j_2'};\delta^{j_2\wedge j_2'}\right).
\end{equation}
By this, \eqref{eq-fpsi1}, \eqref{eq-fpsi2},
\eqref{eq-rn1hp1}, Proposition \ref{prop-estmax}, and H\"older's inequality, we find that,
for any $j_1', j_2'\in\mathbb Z$, $\beta_1'\in\mathcal G_{1,j_1'}$, $\beta_2'\in\mathcal G_{1,j_2'}$, and $(x_1,x_2)\in X_1\times X_2$,
\begin{align*}
&\sum_{\genfrac{}{}{0pt}{}{j_1'\in{\mathbb Z}}{\beta_1'\in\mathcal G_{1,j_1'}}}
\sum_{\genfrac{}{}{0pt}{}{j_2'\in{\mathbb Z}}{\beta_2'\in\mathcal G_{2,j_2'}}}
\left[\left|\left<R_N^{(1,1)}f,\psi_{\beta_1',\beta_2'}^{j_1'+1,j_2'+1}\right>\right|
\widetilde{\mathbf 1}_{Q_{\beta_1',\beta_2'}^{j_1'+1,j_2'+1}}(x_1,x_2)\right]^2\\
&\quad\lesssim\sum_{j_1',\beta_1'}\sum_{j_2',\beta_2'}\left[\sum_{j_1,\beta_1}\sum_{j_2,\beta_2}
\mu\left(Q_{\beta_1,\beta_2}^{j_1+1,j_2+1}\right)\left[\mu\left(Q_{\beta_1,
\beta_2}^{j_1+1,j_2+1}\right)\right]^{-\frac12}\right.\\
&\quad\quad\times\left.\left|\left<f,\psi_{\beta_1,\beta_2}^{j_1,j_2}\right>
\left<F_{1,\beta_1}^{j_1},\Psi_{1,\beta_1'}^{j_1'}\right>\left<F_{2,\beta_2}^{j_2},
\Psi_{2,\beta_2'}^{j_2'}\right>\right|
{\mathbf 1}_{Q_{\beta_1',\beta_2'}^{j_1'+1,j_2'+1}}(x_1,x_2)\right]^2\\
&\quad\lesssim\delta^{N(\eta-\eta')}\sum_{j_1', j_2'\in\mathbb Z}\left[\sum_{j_1,\beta_1}\sum_{j_2,\beta_2}
\delta^{|j_1-j_1'|\eta''}\delta^{|j_2-j_2'|\eta'}\mu\left(Q_{\beta_1,\beta_2}^{j_1,j_2}\right)
\left[\mu\left(Q_{\beta_1,\beta_2}^{j_1+1,j_2+1}\right)\right]^{-\frac12}\left|\left<f,
\psi_{\beta_1,\beta_2}^{j_1,j_2}\right>\right|\right.\\
&\quad\quad\times P_{1,\eta''}\left(x_1,y_{1,\beta_1}^{j_1};\delta^{j_1\wedge j_1'}\right)
P_{2,\eta''}\left(x_2,y_{2,\beta_2}^{j_2};\delta^{j_2\wedge j_2'}\right)\Bigg]^2\\
&\quad\lesssim\delta^{N(\eta-\eta')}\sum_{j_1', j_2'\in\mathbb Z}\left\{\sum_{j_1,j_2\in\mathbb Z}
\delta^{|j_1-j_1'|\eta''-[j_1-(j_1\wedge j_1')]\omega(\frac 1r-1)}
\delta^{|j_2-j_2'|\eta''-[j_2-(j_2\wedge j_2')]\omega(\frac 1r-1)}\right.\\
&\quad\quad\times\left.\left[\mathcal M_{\rm str}\left(\sum_{\beta_1\in\mathcal G_{1,j_1}}
\sum_{\beta_2\in\mathcal G_{2,j_2}}\left|\left<f,\psi_{\beta_1,\beta_2}^{j_1,j_2}
\right>\widetilde{\mathbf 1}_{Q_{\beta_1,\beta_2}^{j_1+1,j_2+1}}\right|^r\right)(x_1,x_2)\right]^{\frac 1r}\right\}^2\\
&\quad\lesssim\delta^{N(\eta-\eta')}\sum_{j_1', j_2'\in\mathbb Z}\sum_{j_1, j_2\in\mathbb Z}
\delta^{|j_1-j_1'|\eta''-[j_1-(j_1\wedge j_1')]\omega(\frac 1r-1)}
\delta^{|j_2-j_2'|\eta''-[j_2-(j_2\wedge j_2')]\omega(\frac 1r-1)}\\
&\quad\quad\times\left[\mathcal M_{\rm str}\left(\sum_{\beta_1\in\mathcal G_{1,j_1}}
\sum_{\beta_2\in\mathcal G_{2,j_2}}\left|\left<f,\psi_{\beta_1,\beta_2}^{j_1,j_2}
\right>\widetilde{\mathbf 1}_{Q_{\beta_1,\beta_2}^{j_1+1,j_2+1}}
\right|^r\right)(x_1,x_2)\right]^{\frac 2r},
\end{align*}
where $r\in(\frac{\omega}{\omega+\eta''},1]$.
Choose $\eta'\in(0,\eta)$ and $\eta''\in(0,\eta')$ such that
$\frac{\omega}{\omega+\eta''}<p$. Thus, we may choose $r\in(p,1]$ and use Proposition \ref{prop-fs} to conclude that
\begin{align*}
 \left\|R_N^{(1,1)}f\right\|_{H^p}
&\lesssim\delta^{N(\eta-\eta')}\left\|\left\{\sum_{j_1',j_2'\in\mathbb Z}\sum_{j_1,j_2\in\mathbb Z}
\delta^{|j_1-j_1'|\eta''-[j_1-(j_1\wedge j_1')]\omega(\frac 1r-1)}
\delta^{|j_2-j_2'|\eta''-[j_2-(j_2\wedge j_2')]\omega(\frac 1r-1)}\right.\right.\\
&\quad\left.\left.{}\times
\left[\mathcal M_{\rm str}\left(\sum_{\beta_1\in\mathcal G_{1,j_1}}
\sum_{\beta_2\in\mathcal G_{2,j_2}}\left|\left<f,\psi_{\beta_1,\beta_2}^{j_1,j_2}\right>
\widetilde{\mathbf 1}_{Q_{\beta_1,\beta_2}^{j_1+1,j_2+1}}\right|^r\right)\right]^{\frac 2r}\right\}
^{\frac r2}\right\|_{L^{\frac pr}}^{\frac 1r}\\
&\lesssim\delta^{N(\eta-\eta')}\left\|\left\{\sum_{j_1,\beta_1}\sum_{j_2,\beta_2}
\left|\left<f,\psi_{\beta_1,\beta_2}^{j_1,j_2}\right>
\widetilde{\mathbf 1}_{Q_{\beta_1,\beta_2}^{j_1+1,j_2+1}}\right|^2\right\}^{\frac r2}\right\|_{L^{\frac pr}}
^{\frac 1r}\\
&\sim\delta^{N(\eta-\eta')}\left\|\left\{\sum_{j_1,\beta_1}\sum_{j_2,\beta_2}
\left|\left<f,\psi_{\beta_1,\beta_2}^{j_1,j_2}\right>\widetilde{\mathbf 1}_{Q_{\beta_1,\beta_2}^{j_1+1,j_2+1}}\right|^2\right\}^{\frac12}\right\|_{L^{p}}
\sim\delta^{N(\eta-\eta')}\|f\|_{H^p}.
\end{align*}

To summarize, we find that there exists $\eta'\in(0,\eta)$ such that
\begin{equation}\label{eq-rn1bdd}
\left\|R_N^{(1)}\right\|_{\mathcal L(L^2)}\lesssim\delta^{N\eta'} ~~\textup{and}~~
\left\|R_N^{(1)}\right\|_{\mathcal L(H^p)}\lesssim\delta^{N\eta'}.
\end{equation}

We next estimate $\|R_N^{(2)}\|_{\mathcal L(L^2)}$ and $\|R_N^{(3)}\|_{\mathcal L(L^2)}$.
We only estimate $\|R_N^{(2)}\|_{\mathcal L(L^2)}$ because the estimation
of $\|R_N^{(3)}\|_{\mathcal L(L^2)}$ is similar and we omit the details here.
Indeed, recall that, for any $f\in L^2$,
\begin{align*}
&R_N^{(2)}f(x_1, x_2)\\
&\quad=\sum_{k_1,\alpha_1}\sum_{k_2,\alpha_2}\int_{Q_{\alpha_1,\alpha_2}^{k_1+j_0,k_2+j_0}}
\left[D_{k_1,k_2}^N(x_1, x_2,y_1,y_2)
-D_{k_1,k_2}^N \left(x_1, x_2,x_{1,\alpha_1}^{k_1+j_0},x_{2,\alpha_2}^{k_2+j_0}\right)\right]\\
&\qquad\times D_{k_1,k_2}f(y_1,y_2)\,d\mu(y_1,y_2)\\
&\quad=\sum_{k_1,\alpha_1}\sum_{k_2,\alpha_2}\int_{Q_{\alpha_1,\alpha_2}^{k_1+j_0,k_2+j_0}}
\left[D_{k_1,k_2}^N(x_1, x_2, y_1, y_2)-D_{k_1, k_2}^N\left(x_1, x_2, y_1,x_{2,\alpha_2}^{k_2+j_0}\right)\right]
D_{k_1,k_2}f(y_1,y_2)\,d\mu(y_1,y_2)\\
&\quad\quad+\sum_{k_1,\alpha_1}\sum_{k_2,\alpha_2}\int_{Q_{\alpha_1,\alpha_2}^{k_1+j_0,k_2+j_0}}
\left[D_{k_1,k_2}^N\left(x_1, x_2,y_1,x_{2,\alpha_2}^{k_2+j_0}\right)-D_{k_1,k_2}^N
\left(x_1, x_2,x_{1,\alpha_1}^{k_1+j_0},x_{2,\alpha_2}^{k_2+j_0}\right)\right]\\
&\quad\quad\times D_{k_1,k_2}f(y_1,y_2)\,d\mu(y_1,y_2)\\
&\quad=:R_N^{(2,1)}f(x_1, x_2)+R_N^{(2,2)}f(x_1, x_2).
\end{align*}
We only give the details for $R_N^{(2,1)}f$.
For any $k_1, k_2\in\mathbb Z$ and $(x_1,x_2), (z_1,z_2)\in X_1\times X_2$, let
\begin{align*}
&G_{k_1,k_2}(x_1,x_2,z_1,z_2)\\
&\quad:=\sum_{\alpha_1\in\mathcal A_{1,k_1+j_0}}\sum_{\alpha_2\in\mathcal A_{2,k_2+j_0}}
\int_{Q_{\alpha_1,\alpha_2}^{k_1+j_0,k_2+j_0}}\left[D_{k_1,k_2}^N
\left(x_1,x_2,y_1,y_2\right)-D_{k_1,k_2}^N\left(x_1,x_2,y_1,x_{2,\alpha_2}^{k_2+j_0}\right)\right]\\
&\quad\quad\times D_{k_1,k_2}(y_1,y_2,z_1,z_2)\,d\mu(y_1,y_2);
\end{align*}
Moreover, by Fubini's theorem, we find that
\begin{align*}
&G_{k_1,k_2}(x_1,x_2,z_1,z_2)\\
&\quad=
\left[\sum_{\alpha_1\in\mathcal A_{1,k_1+j_0}}
    \int_{Q_{1,\alpha_1}^{k_1+j_0}}
    D_{1,k_1}^N(x_1,y_1)
    D_{1,k_1}(y_1,z_1)\,d\mu_1(y_1)
   \right]\\
&\quad\quad\times
   \left\{
      \sum_{\alpha_2\in\mathcal A_{2,k_2+j_0}}
      \int_{Q_{2,\alpha_2}^{k_2+j_0}}
      \left[D_{2,k_2}^N(x_2,y_2)-D_{2,k_2}^N
      \left(x_2,x_{2,\alpha_2}^{k_2+j_0}\right)\right]
      D_{2,k_2}(y_2,z_2)\,d\mu_1(y_2)\right
     \}\\
&\quad=:G_{1,k_1}(x_1,z_1)G_{2,k_2}(x_2,z_2).
\end{align*}
To obtain the desired estimates of $G_{k_1,k_2}$, it suffices to show the those of $G_{1,k_1}$ and $G_{2,k_2}$, respectively.
Indeed, by the definition of $G_{1,k_1}$, we find that $G_{1,k_1}$ has the  properties similar to those of $G_{2,k_2}$.
Thus, we consider the
properties of $G_{2, k_2}$. First, since $N$ is a fixed number, it follows that $D_{2,k_2}^N$ satisfies the
same condition as $D_{2,k_2}$ with the implicit positive constant only independent of $N$;
more precisely, for any $x_2, y_2\in X_2$, $D_{2,k_2}^N(x_2,y_2)\neq 0$ implies $d_2(x_2,y_2)\le c_{N}\delta^{k_2}$,
where $c_{N}$ is a positive constant depending on $N$.
By this, suppose $x_2, z_2\in X_2$ and $d_2(x_2,z_2)>2A_0 c_{N} \delta^{k_2}$.\
If $G_{2,k_2}(x_2,z_2)\neq 0$, then there exists $\alpha_2\in\mathcal A_{2,k_2+j_0}$ such that
$$
\int_{Q_{2,\alpha_2}^{k_2+j_0}}\left[D_{2,k_2}^N(x_2,y_2)
-D_{2,k_2}^N\left(x_2,x_{2,\alpha_2}^{k_2+j_0}\right)\right]D_{2,k_2}(y_2,z_2)\,d\mu_2(y_2)\neq 0,
$$
which further implies that there exists $y_2\in Q_{2,\alpha_2}^{k_2+j_0}$ such that
$$
D_{2,k_2}^N(x_2,y_2)-D_{2,k_2}^N\left(x_2,x_{2,\alpha_2}^{k_2+j_0}\right)\neq 0\textup{ and }
D_{2,k_2}(y_2,z_2)\neq 0.
$$
Using the second formulae, we find that $d_2(y_2,z_2)\le c_{N}\delta^{k_2}$. Meanwhile, since
$d_2(y_2,z_2)>(2A_0)^3c_{N}\delta^{k_2}$, it then follows that
$$
d_2(x_2,y_2)\ge\frac{2}{A_0}d_2(x_2,z_2)-d_2(y_2,z_2)>c_{(N)}\delta^{k_2},
$$
which further implies that $D_{2,k_2}(x_2,y_2)=0$. Moreover, by $x_{2,\alpha_2}^{k_2+j_0}
\in Q_{2,\alpha_2}^{k_2+j_0}$ and the choice of $j_0$, we find that
$$
d_2\left(x_{2,\alpha_2}^{k_2+j_0},z_2\right)\le A_0\left[d_2\left(x_{2,\alpha_2}^{k_2+j_0},z_{2,\alpha_2}^{k_2+j_0}\right)
+d_2\left(z_{2,\alpha_2}^{k_2+j_0},z_2\right)\right]\le 2A_0C^{\natural}\delta^{k_2+j_0}<\delta^{k_2}.
$$
From this, we deduce that
$$
d_2\left(x_2,x_{2,\alpha_2}^{k_2+j_0}\right)\ge\frac{2}{A_0}d_2(x_2,z_2)-d_2\left(x_{2,
\alpha_2}^{k_2+j_0},z_2\right)
>c_{N}\delta^{k_2},
$$
which further implies that $D_{2,k_2}^N(x_2,x_{2,\alpha_2}^{k_2+j_0})=0$. Therefore,
$D_{2,k_2}^N(x_2,y_2)-D_{2,k_2}^N(x_2,x_{2,\alpha_2}^{k_2+j_0})=0$, which leads a contradiction. Thus, when
$d_2(x_2,z_2)>2A_0\delta^{k_2}$, $D_{2,k_2}(x_2,z_2)=0$.
For simplicity of presentation, let $C\in(0,\infty)$ be
such that, for any $x_1, z_1\in X_1$ and $x_2, z_2\in X_2$, $G_{2,k_2}(x_2,z_2)\neq 0$
[resp.\ $G_{1,k_1}(x_1,z_1)=0$] implies $d_2(x_2,z_2)\le C\delta^{k_2}$ [resp.\ $d_1(x_1,z_1)\le C\delta^{k_1}$].
Moreover, since $G_{1,k_1}$ and $G_{2,k_2}$ have similar properties to $D_{1,k_1}$ and $D_{2,k_2}$
(for $G_{2,k_2}$, we have an additional term $\delta^{j_0\eta}$), it then follows that, for any
$k_1,  k_2,  j_1, j_2\in\mathbb Z$,
\begin{equation*}
\left\|(G_{k_1,k_2})^*G_{j_1,j_2}\right\|_{\mathcal L(L^2)}\lesssim\delta^{j_0\eta}\delta^{|k_1-j_1|\eta}
\delta^{|k_2-j_2|\eta}
\end{equation*}
and
\begin{equation*}
\left\|G_{k_1,k_2}(G_{j_1,j_2})^*\right\|_{\mathcal L(L^2)}\lesssim\delta^{j_0\eta}\delta^{|k_1-j_1|\eta}
\delta^{|k_2-j_2|\eta}.
\end{equation*}
By these and the Cotlar--Stein lemma (see \cite[Lemma 5.12]{HYY24}), we find that
$$
\left\|R_N^{(2,1)}\right\|_{\mathcal L(L^2)}\lesssim\delta^{j_0\eta}.
$$
This is the desired estimate.

Now, we estimate $\|R_N^{(2,1)}\|_{\mathcal L(H^p)}$. Indeed, for any $j_1',j_2'\in\mathbb Z$,
$\beta_1'\in\mathcal G_{1,j_1'}$, and $\beta_2'\in\mathcal G_{2,j_2'}$, we have
\begin{align}\label{eq-r2psi}
&\mu\left(Q_{\beta_1',\beta_2'}^{j_1'+1,j_2'+1}\right)^{-\frac{1}{2}}\left<R_N^{(2,1)}f,
\psi_{\beta_1',\beta_2'}^{j_1',j_2'}\right>
\nonumber\\
&\quad=\sum_{j_1,\beta_1}\sum_{j_2,\beta_2}\sum_{k_1,k_2\in\mathbb Z}\left<G_{k_1,k_2}\psi_{\beta_1,\beta_2}^{j_1,j_2},
\Psi_{\beta_1',\beta_2'}^{j_1',j_2'}\right>\left<f,\psi_{\beta_1,\beta_2}^{j_1,j_2}\right>\nonumber\\
&\quad=\sum_{j_1,\beta_1}\sum_{j_2,\beta_2}\mu\left(Q_{\beta_1,\beta_2}^{j_1+1,j_2+1}\right)
\mu\left(Q_{\beta_1,\beta_2}^{j_1+1,j_2+1}\right)^{-\frac{1}{2}}\left<f,
\psi_{\beta_1,\beta_2}^{j_1,j_2}\right>
\sum_{k_1, k_2 \in \mathbb{Z}}\left<G_{k_1,k_2}\Psi_{\beta_1,\beta_2}^{j_1,j_2},
\Psi_{\beta_1',\beta_2'}^{j_1',j_2'}\right>\nonumber\\
&\quad=\sum_{j_1,\beta_1}\sum_{j_2,\beta_2}\mu\left(Q_{\beta_1,\beta_2}^{j_1+1,j_2+1}\right)
\mu\left(Q_{\beta_1,\beta_2}^{j_1+1,j_2+1}\right)^{-\frac{1}{2}}\left<f,
\psi_{\beta_1,\beta_2}^{j_1,j_2}\right>\nonumber\\
&\quad\quad\times\sum_{k_1=-\infty}^\infty\left<G_{1,k_1}\Psi_{1,\beta_1}^{j_1},
\Psi_{1,\beta_1'}^{j_1'}\right>
\sum_{k_2=-\infty}^\infty\left<G_{2,k_2}\Psi_{2,\beta_2}^{j_2},\Psi_{2,\beta_2'}^{j_2'}\right>\nonumber\\
&\quad=\sum_{j_1,\beta_1}\sum_{j_2,\beta_2}\mu\left(Q_{\beta_1,\beta_2}^{j_1+1,j_2+1}\right)
\mu\left(Q_{\beta_1,\beta_2}^{j_1+1,j_2+1}\right)^{-\frac{1}{2}}\left<f,
\psi_{\beta_1,\beta_2}^{j_1,j_2}\right>
\left<H_{1,\beta_1}^{j_1},\Psi_{1,\beta_1'}^{j_1'}\right>\left<H_{2,
\beta_2}^{j_2},\Psi_{2,\beta_2'}^{j_2'}\right>,
\end{align}
here and thereafter, for any $i\in\{1,2\}$, $j_i\in\mathbb Z$, $\beta_i\in\mathcal G_{i,j_i}$, and $x_i\in X_i$,
$$
H_{i,\beta_i}^{j_i}(x_i):=\sum_{k_i=-\infty}^\infty G_{i,k_i}\Psi_{i,\beta_i}^{j_i}(x_i)
$$
Using an argument similar to the estimations of \eqref{eq-fpsi1} and \eqref{eq-fpsi2}, we find that, for any given $\eta'\in(0,\eta)$,
$$
\left|\left<H_{1,\beta_1}^{j_1},\Psi_{1,\beta_1'}^{j_1'}\right>\right|\lesssim\delta^{j_0\eta}
\delta^{|j_1-j_1'|\eta'}
P_{1,\eta'}\left(y_{1,\beta_1}^{j_1},y_{1,\beta_1'}^{j_1'};\delta^{j_1\wedge j_1'}\right)
$$
and
$$
\left|\left<H_{2,\beta_2}^{j_2},\Psi_{2,\beta_2'}^{j_2'}\right>\right|\lesssim\delta^{|j_2-j_2'|\eta'}
P_{2,\eta'}\left(y_{2,\beta_2}^{j_2},y_{2,\beta_2'}^{j_2'};\delta^{j_2\wedge j_2'}\right).
$$
By the above two inequalities, \eqref{eq-r2psi}, Proposition \ref{prop-estmax}, and H\"older's inequality, we conclude
that, for any $j_1',  j_2'\in\mathbb Z$ and $(x_1,x_2)\in X_1\times X_2$,
\begin{align*}
&\sum_{\beta_1'\in\mathcal G_{1,j_1'}}\sum_{\beta_2'\in\mathcal G_{2,j_2'}}
\left|\left<R_N^{(2,1)}f,\psi_{\beta_1',\beta_2'}^{j_1',j_2'}\right>
\widetilde{\mathbf 1}_{Q_{\beta_1',\beta_2'}^{j_1'+1,j_2'+1}}(x_1,x_2)\right|^2\\
&\quad\lesssim\delta^{j_0\eta}\sum_{\beta_1'\in\mathcal G_{1,j_1'}}\sum_{\beta_2'\in\mathcal G_{2,j_2'}}
\left|\sum_{j_1,\beta_1}\sum_{j_2,\beta_2}\delta^{|j_1-j_1'|\eta'}\delta^{|j_2-j_2'|\eta'}
\sum_{\beta_2\in\mathcal G_{2,j_2}}\mu\left(Q_{\beta_1,\beta_2}^{j_1+1,j_2+1}\right)
\mu\left(Q_{\beta_1,\beta_2}^{j_1+1,j_2+1}\right)^{-\frac 12}\right.\\
&\quad\quad\times\left.\left<f,\psi_{\beta_1,\beta_2}^{j_1,j_2}\right>
P_{1,\eta'}\left(y_{1,\beta_1}^{j_1},y_{1,\beta_1'}^{j_1'};\delta^{j_1\wedge j_1'}\right)
P_{2,\eta'}\left(y_{2,\beta_2}^{j_2},y_{2,\beta_2'}^{j_2'};\delta^{j_2\wedge j_2'}\right)
{\mathbf 1}_{Q_{\beta_1',\beta_2'}^{j_1'+1,j_2'+1}}(x_1,x_2)\right|^2\\
&\quad\sim\delta^{j_0\eta}\left|\sum_{j_1,\beta_1}\sum_{j_2,\beta_2}\delta^{|j_1-j_1'|\eta'}
\delta^{|j_2-j_2'|\eta'}
\mu\left(Q_{\beta_1,\beta_2}^{j_1+1,j_2+1}\right)\mu\left(Q_{\beta_1,\beta_2}^{j_1+1,j_2+1}
\right)^{-\frac12}
\left<f,\psi_{\beta_1,\beta_2}^{j_1,j_2}\right>\right.\\
&\quad\quad\times\Biggl.P_{1,\eta'}\left(x_1,y_{1,\beta_1}^{j_1};\delta^{j_1\wedge j_1'}\right)
P_{2,\eta'}\left(x_2,y_{2,\beta_2}^{j_2};\delta^{j_2\wedge j_2'}\right)\Biggr|^2\\
&\quad\lesssim\delta^{j_0\eta}\left\{\sum_{j_1,j_2\in\mathbb Z}\delta^{|j_1-j_1'|\eta-[j_1-(j_1\wedge j_1')]\omega(\frac{1}{r}-1)}
\delta^{|j_2-j_2'|\eta-[j_1-(j_1\wedge j_1')]\omega(\frac{1}{r}-1)}\right.\\
&\quad\quad\left.\times\left[\mathcal M_{\rm str}\left(\sum_{\beta_1\in\mathcal G_{1,j_1}}
\sum_{\beta_2\in\mathcal G_{2,j_2}}\left|\left<f,\psi_{\beta_1,\beta_2}^{j_1,j_2}\right>
\widetilde{\mathbf 1}_{Q_{\beta_1,\beta_2}^{j_1+1,j_2+1}}\right|^r\right)(x_1,x_2)\right]^{\frac 1r}\right\}^2\\
&\quad\lesssim\delta^{j_0 \eta}\sum_{j_1,j_2\in\mathbb Z}\delta^{|j_1-j_1'|\eta-[j_1-(j_1\wedge j_1')]\omega(\frac{1}{r}-1)}
\delta^{|j_2-j_2'|\eta-[j_1-(j_1\wedge j_1')]\omega(\frac{1}{r}-1)}\\
&\quad\quad\times\left[\mathcal M_{\rm str}\left(\sum_{\beta_1\in\mathcal G_{1,j_1}}
\sum_{\beta_2\in\mathcal G_{2,j_2}}\left|\left<f,\psi_{\beta_1,\beta_2}^{j_1,j_2}\right>
\widetilde{\mathbf 1}_{Q_{\beta_1,\beta_2}^{j_1+1,j_2+1}}\right|^r\right)(x_1,x_2)\right]^{\frac 2r},
\end{align*}
where $r\in (\frac{\omega}{\omega+\eta'},1]$. Choosing $\eta'\in(0,\eta)$ such that $\frac{\omega}{\omega+\eta'}<p$,
together with $r\in(p,1]$ and Proposition \ref{prop-fs}, we find that
\begin{align*}
\left\|R^{(2,1)}_Nf\right\|_{H^p}
&=\left\|\left(\sum_{j_1',\beta_1'}\sum_{j_2',\beta_2'}\left|\left<R_N^{(2,1)}f,
\psi_{\beta_1',\beta_2'}^{j_1',j_2'}\right>
\widetilde{\mathbf 1}_{Q_{\beta_1',\beta_2'}^{j_1'+1,j_2'+1}}\right|^2\right)^{\frac{1}{2}}\right\|_{L^p}\\
&\lesssim\delta^{j_0\eta}\left\|\left\{\sum_{j_1', j_2'\in\mathbb Z}\sum_{j_1,j_2\in\mathbb Z}
\delta^{|j_1-j_1'|\eta-[j_1-(j_1\wedge j_1')]\omega(\frac{1}{r}-1)}
\delta^{|j_2-j_2'|\eta-[j_1-(j_1\wedge j_1')]\omega(\frac{1}{r}-1)}\right.\right.\\
&\quad\times\left.\left.\left[\mathcal M_{\rm str}\left(\sum_{\beta_1\in\mathcal G_{1,j_1}}\sum_{\beta_2\in\mathcal G_{2,j_2}}
\left|\left<f,\psi_{\beta_1,\beta_2}^{j_1,j_2}\right>
\widetilde{\mathbf 1}_{Q_{\beta_1,\beta_2}^{j_1+1,j_2+1}}\right|^r\right)
\right]^{\frac{2}{r}}\right\}^{\frac{r}{2}}\right\|_{L^{\frac{p}{r}}}^{\frac{1}{r}}\\
&\lesssim\delta^{j_0\eta}\left\|\left[\sum_{j_1, j_2\in\mathbb Z}\left(\sum_{\beta_1\in\mathcal G_{1,j_1}}
\sum_{\beta_2\in\mathcal G_{2,j_2}}\left|\left<f,\psi_{\beta_1,\beta_2}^{j_1,j_2}\right>
\widetilde{\mathbf 1}_{Q_{\beta_1,\beta_2}^{j_1+1,j_2+1}}\right|^r\right)^{\frac{2}{r}}
\right]^{\frac{r}{2}}\right\|_{L^{\frac{p}{r}}}^{\frac{1}{r}}\\
&\sim\delta^{j_0\eta}\left\|\left[\sum_{j_1,\beta_1}\sum_{j_2,\beta_2}\left|\left<f,
\psi_{\beta_1,\beta_2}^{j_1,j_2}\right>
\widetilde{\mathbf 1}_{Q_{\beta_1,\beta_2}^{j_1+1,j_2+1}}\right|^{2}\right]^{\frac{1}{2}}\right\|_{L^{p}}
\sim\delta^{j_0\eta}\|f\|_{H^p}.
\end{align*}
Applying an argument similar to that used in the above estimation, we find that
\begin{equation}\label{eq-rn23l2}
\max\left\{\left\|R_N^{(2)}\right\|_{\mathcal L(L^2)},\left\|R_N^{(3)}\right\|_{\mathcal L(L^2)}\right\}
\lesssim\delta^{j_0\eta}
\end{equation}
and
\begin{equation}\label{eq-rn23hp}
\max\left\{\left\|R_N^{(2)}\right\|_{\mathcal L(H^p)},\left\|R_N^{(3)}\right\|_{\mathcal L(H^p)}\right\}
\lesssim\delta^{j_0\eta},
\end{equation}
where the implicit positive constants dependent on $N$ but independent of $j_0$.

Now, by \eqref{eq-rn1bdd}, we choose $N\in\mathbb N$ sufficiently large such that
$$
\left\|R_{N}^{(1)}\right\|_{\mathcal L(L^2)}<\frac 14
\textup{ and }\left\|R_{N}^{(1)}\right\|_{\mathcal L(H^p)}^p<\frac 1{2^{p+1}};
$$
fix such $N$ and, by \eqref{eq-rn23l2} and \eqref{eq-rn23hp}, we   choose $j_0\in\mathbb N$ such that
$$
\left\|R_{N}^{(2)}+R_N^{(3)}\right\|_{\mathcal L(L^2)}<\frac 14\textup{ and }
\left\|R_{N}^{(2)}+R_N^{(3)}\right\|_{\mathcal L(H^p)}^p\le\frac 1{2^{p+1}}.
$$
We then obtain \eqref{eq-ctrfaim}.

By \eqref{eq-ctrfaim}, we find that $T_N$ is invertible in $L^2$,
\begin{equation}\label{eq-tnl2}
\left\|T_N^{-1}\right\|_{\mathcal L(L^2)}\sim 1,
\end{equation}
and, for any $f\in L^2\cap H^p$,
$$
T_N^{-1}f=\sum_{j=0}^\infty(R_N)^jf
$$
in $L^2$. By this and Minkowski's inequality, we further obtain, for any $(x_1,x_2)\in X_1\times X_2$,
\begin{equation}\label{eq-tnhp}
S\left(T_N^{-1}f\right)(x_1,x_2)=S\left(\sum_{j=0}^\infty (R_N)^jf\right)(x_1,x_2)\le\sum_{j=0}^\infty
S\left((R_N)^jf\right)(x_1,x_2),
\end{equation}
which further implies that
\begin{align*}
 \left\|S\left(T_N^{-1}f\right)\right\|_{L^p}^p
\le\sum_{j=0}^\infty\left\|S\left((R_N)^jf\right)\right\|_{L^p}^p
&=\sum_{j=0}^\infty\left\|(R_N)^jf\right\|_{H^p}^p\lesssim 1.
\end{align*}
Letting $g:=T_N^{-1}f$, we then obtain \eqref{eq-ctype}.
By \eqref{eq-tnl2} and \eqref{eq-ctrfaim} we obtain
\begin{equation*}
\|g\|_{L^2}=\left\|T_N^{-1}f\right\|_{L^2}\le\left\|T_N^{-1}\right\|_{\mathcal L(L^2)}\|f\|_{L^2}
\lesssim\|f\|_{L^2}\sim\|T_Ng\|_{L^2}\lesssim\|I-R_N\|_{\mathcal L(L^2)}\|f\|_{L^2}
\lesssim\|f\|_{L^2},
\end{equation*}
and using \eqref{eq-tnhp} and \eqref{eq-ctrfaim} again we find that
\begin{align*}
\|g\|_{H^p}=\left\|S\left(T_N^{-1}\right)f\right\|_{L^p}\lesssim\|f\|_{H^p}
=\left\|(I-R_N)f\right\|_{H^p}\lesssim\|f\|_{H^p}.
\end{align*}
This finishes the proof of Theorem \ref{thm-ctype}.
\end{proof}

\subsection{Proof of
$H^p\subset H^{p,q}_{\mathrm at}$
(``Only If'' Part of Theorem \ref{thm-h=a})} \label{ss-h<a}

In this subsection, using the new Calder\'on-type reproducing
formula in Section \ref{ss-ctrf}, we give the proof of
$H^p\subset H^{p,q}_{\mathrm at}$ with $p$ and $q$ the same as in
Theorem \ref{thm-h=a} (namely the ``only if''
part of Theorem \ref{thm-h=a}). To achieve this, we first establish several
technical lemmas.

\begin{lemma}\label{lem-slp}
Let $p\in(\frac{\omega}{\omega+\eta},1]$ with $\omega$ and $\eta$ the same, respectively, as in \eqref{eq-doub}
and Theorem \ref{thm-wave}, and let $\{D_{k_1,k_2}\}_{k_1,k_2\in\mathbb Z}$ be the same as in Theorem \ref{thm-ctype}.
Suppose that $m\in\mathbb Z_+$ and, for any $k_1,k_2\in\mathbb Z$, $\alpha_1\in\mathcal A_{1,k_1}$, and
$\alpha_2\in\mathcal A_{2,k_2}$, $(x_{1,\alpha_1}^{k_1+m},x_{2,\alpha_2}^{k_2+m})\in Q_{\alpha_1,\alpha_2}^{k_1+m,k_2+m}$
is an arbitrary point. Then there exists a positive constant $C$, independent of
$(x_{1,\alpha_1}^{k_1+m_1},x_{2,\alpha_2}^{k_2+m_2})$, such that, for any $f\in H^p$,
\begin{align*}
&\left\|\left[\sum_{\genfrac{}{}{0pt}{}{k_1\in{\mathbb Z}}{\alpha_1\in\mathcal  A_{1,k_1+m}}}
\sum_{\genfrac{}{}{0pt}{}{k_2\in{\mathbb Z}}{\alpha_2\in\mathcal A_{2,k_2+m}}}\left|D_{k_1,k_2}f
\left(x_{1,\alpha_1}^{k_1+m_1},x_{2,\alpha_2}^{k_2+m_2}\right)\right|^2
\mathbf 1_{Q_{\alpha_1,\alpha_2}^{k_1+m,k_2+m}}\right]^{1/2}\right\|_{L^p} \le C\|f\|_{H^p}.\nonumber
\end{align*}
\end{lemma}

\begin{proof}
Since the positive constant $C$ may be allowed to depend on $m$, without loss of generality, we may assume $m=0$. Let $f\in H^p$.
Then we have  $f\in(\mathrm{CMO}^p_{L^2})'$ with $\vec\beta, \vec\gamma\in(\omega(1/p-1),\eta)^2$
and $\omega$ and $\eta$ as, respectively, in \eqref{eq-doub} and Theorem \ref{thm-wave},
and hence
$$
f(\cdot)=\sum_{\genfrac{}{}{0pt}{}{j_1\in{\mathbb Z}}{\beta_1\in\mathcal G_{1,j_1}}}
\sum_{\genfrac{}{}{0pt}{}{j_2\in{\mathbb Z}}{\beta_2\in\mathcal G_{2,j_2}}}
\left<f,\psi_{\beta_1,\beta_2}^{j_1,j_2}\right>\psi_{\beta_1,\beta_2}^{j_1,j_2}(\cdot)
$$
in $(\mathrm{CMO}^p_{L^2})'$.
Thus, we have
\begin{align*}
&D_{k_1,k_2}f\left(x_{1,\alpha_1}^{k_1},x_{2,\alpha_2}^{k_2}\right)\\
&\quad=\sum_{j_1,\beta_1}\sum_{j_2,\beta_2}
\left<D_{k_1,k_2}\left(\cdot,*,x_{1,\alpha_1}^{k_1},x_{2,\alpha_2}^{k_2}\right),
\psi_{\beta_1,\beta_2}^{j_1,j_2}\right>\left<f,\psi_{\beta_1,\beta_2}^{j_1,j_2}\right>\\
&\quad=\sum_{j_1,\beta_1}\sum_{j_2,\beta_2}\mu\left(Q_{\beta_1,\beta_2}^{j_1,j_2}\right)
\left<D_{k_1,k_2}\left(\cdot,*,x_{1,\alpha_1}^{k_1},x_{2,\alpha_2}^{k_2}\right),
\Psi_{\beta_1,\beta_2}^{j_1,j_2}\right>
\mu\left(Q_{\beta_1,\beta_2}^{j_1,j_2}\right)^{-\frac{1}{2}}\left<f,
\psi_{\beta_1,\beta_2}^{j_1,j_2}\right>,
\end{align*}
where, for any $j_1,j_2\in\mathbb Z$, $\beta_1\in\mathcal G_{1,j_1}$, and $\beta_2\in\mathcal G_{2,j_2}$,
$\Psi_{\beta_1,\beta_2}^{j_1,j_2}=\Psi_{1,\beta_1}^{j_1}\otimes\Psi_{2,\beta_2}^{j_2}$.
Using an argument similar to the one-parameter estimation in \cite[Proposition 2.5]{hhl16},
 we conclude that,
for any $j_1, j_2, k_1, k_2\in\mathbb Z$, $\alpha_1\in\mathcal A_{1,k_1}$, $\alpha_2\in\mathcal A_{2,k_2}$, $\beta_1\in\mathcal G_{1,j_1}$,
$\beta_2\in\mathcal G_{2,j_2}$, and $(x_{1,\alpha_1}^{k_1},x_{2,\alpha_2}^{k_2})\in Q_{\alpha_1,\alpha_2}^{k_1,k_2}$,
\begin{align*}
&\left|\left<D_{k_1,k_2}\left(\cdot,*,x_{1,\alpha_1}^{k_1},x_{2,\alpha_2}^{k_2}\right),
\Psi_{\beta_1,\beta_2}^{j_1,j_2}\right>\right|\\
&\quad\lesssim\delta^{|k_1-j_1|\eta}\delta^{|k_2-j_2|\eta}
P_{1,\Gamma}\left(x_{1,\alpha_1}^{k_1},y_{1,\beta_1}^{j_1};
\delta^{j_1\wedge k_1}\right)P_{2,\Gamma}\left(x_{2,\alpha_2}^{k_2},y_{2,\beta_2}^{j_2};
\delta^{j_2\wedge k_2}\right),
\end{align*}
where $\Gamma\in(0,\infty)$ is fixed, which, from Proposition \ref{prop-estmax}, 
further implies that,
for any $k_1,k_2\in\mathbb Z$ and $(x_1,x_2)\in X_1\times X_2$,
\begin{align*}
&\sum_{\alpha_1\in\mathcal A_{1,k_1}}\sum_{\alpha_2\in\mathcal A_{2,k_2}}
\left|D_{k_1,k_2}f\left(x_{1,\alpha_1}^{k_1},x_{2,\alpha_2}^{k_2}\right)\right|^2
\mathbf 1_{Q_{\alpha_1,\alpha_2}^{k_1,k_2}}(x_1,x_2)\\
&\quad\lesssim\sum_{\alpha_1\in\mathcal A_{1,k_1}}\sum_{\alpha_2\in\mathcal A_{2,k_2}}\left[\sum_{j_1,\beta_1}\sum_{j_2,\beta_2}
\delta^{|k_1-j_1|\eta}\delta^{|k_2-j_2|\eta}
P_{1,\Gamma}\left(x_{1,\alpha_1}^{k_1},y_{1,\beta_1}^{j_1};\delta^{j_1\wedge k_1}\right)\right.\\
&\quad\quad\Biggl.{}\times
P_{2,\Gamma}\left(x_{2,\alpha_2}^{k_2},y_{2,\beta_2}^{j_2};\delta^{j_2\wedge k_2}\right)
\mu\left(Q_{\beta_1,\beta_2}^{j_1+1,j_2+1}\right)^{-\frac{1}{2}}
\left|\left<f,\psi_{\beta_1,\beta_2}^{j_1,j_2}\right>\right|\mathbf 1_{Q_{\alpha_1,\alpha_2}^{k_1,k_2}}(x_1,x_2)\Biggr]^2\\
&\quad\lesssim\sum_{j_1, j_2\in\mathbb Z}\delta^{|k_1-j_1|\eta}\delta^{|k_2-j_2|\eta}
\left[\sum_{\beta_1\in\mathcal G_{1,j_1}}\sum_{\beta_2\in\mathcal G_{2,j_2}}
P_{1,\Gamma}\left(x_1,y_{1,\beta_1}^{j_1};\delta^{j_1\wedge k_1}\right)
P_{2,\Gamma}\left(x_2,y_{2,\beta_2}^{j_2};\delta^{j_2\wedge k_2}\right)\right.\\
&\quad\quad\Biggl.{}\times\mu\left(Q_{\beta_1,\beta_2}^{j_1+1,j_2+1}\right)^{- \frac{1}{2}}
\left<f,\psi_{\beta_1,\beta_2}^{j_1,j_2}\right>\Biggr]^2\\
&\quad\lesssim\sum_{j_1, j_2\in\mathbb Z}
\delta^{|k_1-j_1|\eta-\omega(\frac{1}{r}-1)[k_1-(k_1\wedge j_1)]}
\delta^{|k_2-j_2|\eta-\omega(\frac{1}{r}-1)[k_2-(k_2\wedge j_2)]}\\
&\quad\quad\times\left[\mathcal M_{\rm str}\left(\sum_{\beta_1\in\mathcal G_{1,j_1}}\sum_{\beta_2\in\mathcal G_{2,j_2}}
\left|\left<f,\psi_{\beta_1,\beta_2}^{j_1,j_2}\right>
\widetilde{\mathbf 1}_{Q_{\beta_1,\beta_2}^{j_1+1,j_2+1}}\right|^r\right)(x_1,x_2)\right]^{\frac 2r},
\end{align*}
where $r\in(\frac{\omega}{\omega+\Gamma},1]$.
Since $p\in(\frac{\omega}{\omega+\eta},1]$,
we can choose $\Gamma\in(0,\infty)$ sufficiently large such that $\frac\omega{\omega+\Gamma}<p$.
Then we  choose $r\in(p,1]$ and use Proposition \ref{prop-fs} to conclude that
\begin{align*}
&\left\|\left[\sum_{k_1=-\infty}^\infty\sum_{k_2=-\infty}^\infty
\sum_{\alpha_1\in\mathcal A_{1,k_1}}\sum_{\alpha_2\in\mathcal A_{2,k_2}}
\left|D_{k_1,k_2}f\left(x_{1,\alpha_1}^{k_1},x_{2,\alpha_2}^{k_2}\right)\right|^2
\mathbf 1_{Q_{\alpha_1,\alpha_2}^{k_1,k_2}}\right]^{\frac{1}{2}}\right\|_{L^p}\\
&\quad\lesssim\left\|\left\{\sum_{k_1,k_2\in\mathbb Z}\sum_{j_1,j_2\in\mathbb Z}\delta^{|k_1-j_1|\eta-\omega(\frac{1}{r}-1)[k_1-(k_1\wedge j_1)]}
\delta^{|k_2-j_2|\eta-\omega(\frac{1}{r}-1)[k_2-(k_2\wedge j_2)]}\right.\right.\\
&\quad\quad\left.\left.\times\left[\mathcal M_{\rm str}\left(\sum_{\beta_1\in\mathcal G_{1,j_1}}\sum_{\beta_2\in\mathcal G_{2,j_2}}
\left|\left<f,\psi_{\beta_1,\beta_2}^{j_1,j_2}\right>
\widetilde{\mathbf 1}_{Q_{\beta_1,\beta_2}^{j_1+1,j_2+1}}\right|^r\right)\right]^{\frac{2}{r}}
\right\}^{\frac{1}{2}}\right\|_{L^p}\\
&\quad\sim\left\|\left\{\sum_{j_1,j_2\in\mathbb Z}\left[\mathcal M_{\rm str}\left(\sum_{\beta_1\in\mathcal G_{1,j_1}}
\sum_{\beta_2\in\mathcal G_{2,j_2}}\left|\left<f,\psi_{\beta_1,\beta_2}^{j_1,j_2}\right>
\widetilde{\mathbf 1}_{Q_{\beta_1,\beta_2}^{j_1+1,j_2+1}}\right|^r\right)\right]^{\frac{2}{r}}
\right\}^{\frac{r}{2}}\right\|_{L^{\frac pr}}^{\frac{1}{r}}\\
&\quad\lesssim\left\|\left\{\sum_{j_1,\beta_1}\sum_{j_2,\beta_2}
\left|\left<f,\psi_{\beta_1,\beta_2}^{j_1,j_2}\right>
\widetilde{\mathbf 1}_{Q_{\beta_1,\beta_2}^{j_1+1,j_2+1}}\right|^2\right\}^{\frac{r}{2}}\right\|_{L^{\frac pr}}^{\frac{1}{r}}
\sim\|f\|_{H^p}.
\end{align*}
This finishes the proof of Lemma \ref{lem-slp}.
\end{proof}

Now, we give the atomic decomposition theorem of $H^p$.

\begin{proposition}\label{prop-atd}
Let $p\in(\frac{\omega}{\omega+\eta}, 1] $ with $\omega$ and $\eta$ as,
respectively, in \eqref{eq-doub} and Theorem \ref{thm-wave}, $q\in(1,\infty)$, and $C_0$ be sufficiently large.
Then, for any $f\in H^p$, there exist a sequence of $(p,q,C_0)$-atoms $\{a_j\}_{j=1}^\infty$ and
a sequence $\{\lambda_j\}_{j=1}^\infty$ in $\mathbb C$ such that $f=\sum_{j=1}^\infty\lambda_ja_j$
and the series converges in $(\mathrm{CMO}^p_{L^2})'$. Moreover, there exists a constant $C\in(0,\infty)$,
independent of $f$, such that
$$
\left\{\sum_{j=1}^\infty\left|\lambda_j\right|^p\right\}^{\frac{1}{p}}\le C\|f\|_{H^p}.
$$
\end{proposition}

\begin{proof}
We first assume $f\in L^2\cap H^p$.
By Theorem \ref{thm-ctype}, we find that there exists $h\in L^2\cap H^p$ such that
$\|h\|_{L^2}\sim\|f\|_{L^2}$,  $\|h\|_{H^p}\sim\|f\|_{H^p}$, and
\begin{align*}
f(\cdot,*)&=\sum_{\genfrac{}{}{0pt}{}{k_1\in{\mathbb Z}}{\alpha_1\in\mathcal  A_{1,k_1+j_0}}}
\sum_{\genfrac{}{}{0pt}{}{k_2\in{\mathbb Z}}{\alpha_2\in\mathcal A_{2,k_2+j_0}}}
\mu_1\left(Q_{\alpha_1}^{k_1+j_0}\right)\mu_2\left(Q_{\alpha_2}^{k_2+j_0}\right)D_{k_1,k_2}
\left(\cdot,*,x_{\alpha_1}^{k_1+j_0},x_{\alpha_2}^{k_2+j_0}\right)\\
&\quad\times D_{k_1,k_2}^N(h)\left(x_{\alpha_1}^{k_1+j_0},x_{\alpha_2}^{k_2+j_0}\right),\nonumber
\end{align*}
where all the symbols are the same as in \eqref{eq-ctype2},
$x_{\alpha_1,\alpha_2}^{k_1+j_0,k_2+j_0}:=(x_{1,\alpha_1}^{k_1+j_0},x_{2,\alpha_2}^{k_2+j_0})$ is an
arbitrary point in $Q_{\alpha_1,\alpha_2}^{k_1+j_0,k_2+j_0}$, and the summation converges in both
$L^2$ and $H^p$.

Now, for any $(x_1,x_2)\in X_1\times X_2$, let
\begin{align*}
F(x_1,x_2)&:=\left\{\sum_{k_1,\alpha_1}\sum_{k_2,\alpha_2}
\left|D_{k_1,k_2}^Nh\left(x_{\alpha_1,\alpha_2}^{k_1+j_0,k_2+j_0}\right)\right|^2
\mathbf 1_{Q_{\alpha_1,\alpha_2}^{k_1+j_0,k_2+j_0}}(x_1,x_2)\right\}^{\frac 12}.
\end{align*}
By Lemma \ref{lem-slp} and Theorem \ref{thm-ctype}, we conclude that
\begin{equation*}
\left\|F\right\|_{L^p}\lesssim\|h\|_{H^p}\sim\|f\|_{H^p}<\infty,
\end{equation*}
where the implicit positive constants depend only on $N$ and $j_0$.
For any $l\in\mathbb Z$, let
$$
\Omega_l:=\left\{(x_1,x_2)\in X_1\times X_2:\ F(x_1,x_2)>2^l\right\}
$$
and
$\widetilde\Omega_l:=\{(x_1,x_2)\in X_1\times X_2:\ \mathcal M_{\rm str}(\mathbf 1_{\Omega_l})(x_1,x_2)>c\}$, where $c\in(0,1)$ is
a fixed positive constant which is determined later. Further, for any $l\in\mathbb Z$, let
$$
\mathcal R_l:=\left\{Q\in\mathcal D:\ \mu(Q\cap\Omega_l)\ge \frac 12\mu(Q)
\textup{ and }\mu(Q\cap\Omega_{l+1})<\frac 12\mu(Q)\right\}.
$$
Notice that, for any $(x_1,x_2)\in\bigcup_{Q\in\mathcal R_l}Q$, we find that there exists $Q_1\times Q_2\in \mathcal R_l$
such that $(x_1,x_2)\in Q_1\times Q_2$ and
$$
\frac 1{\mu(Q_1\times Q_2)}\int_{Q_1\times Q_2}\mathbf 1_{\Omega_l}(y_1,y_2)\,d\mu(y_1,y_2) \ge\frac{1}{2}.
$$
Consequently, we have, for any $l\in\mathbb Z$,
\begin{align*}
\mathcal M_{\rm str}\left(\mathbf 1_{\Omega_l}\right)(x_1,x_2)&\ge\frac{1}{\mu(B_1(Q_1)\times B_2(Q_2))}
\int_{B_1(Q_1)\times B_2(Q_2)}\mathbf 1_{\Omega_l}(y_1,y_2)\,d\mu(x_1,x_2)\\
&\gtrsim\frac 1{\mu(Q_1\times Q_2)}\int_{Q_1\times Q_2}\mathbf 1_{\Omega_l}(y_1,y_2)
\,d\mu(y_1,y_2)\sim 1.
\end{align*}
Therefore, we may choose $c\in(0,1)$   such that, for any $l\in\mathbb Z$, $\mathcal M_{\rm str}(\mathbf 1_{\Omega_l})(x_1,x_2)>c$.
By this, we find that
\begin{equation}\label{eq-suppatm}
\bigcup_{Q\in\mathcal R_l}Q\subset\widetilde\Omega_l.
\end{equation}

On the other hand, for any $Q:=Q_{\alpha_1,\alpha_2}^{k_1+j_0,k_2+j_0}$ with
$k_1, k_2\in\mathbb Z$, $\alpha_1\in\mathcal A_{1,k_1+j_0}$, and $\alpha_2\in\mathcal A_{2,k_2+j_0}$,
let $D_Q:=D_{k_1,k_2}$, $D_Q^N:=D^N_{k_1,k_2}$, and $x_Q:=(x_{Q,1},x_{Q,2}):=
(x_{1,\alpha_1}^{k_1+j_0},x_{2,\alpha_2}^{k_2+j_0})\in Q$. Then, applying
an argument similar to that used
in the proof of \eqref{eq-dual1}, we obtain
\begin{align*}
f(\cdot,*)&=\sum_{k_1,\alpha_1}\sum_{k_2,\alpha_2}\mu_1\left(Q_{\alpha_1}^{k_1+j_0}\right)
\mu_2\left(Q_{\alpha_2}^{k_2+j_0}\right)
D_{k_1,k_2}\left(\cdot,*,x_{\alpha_1}^{k_1+j_0},x_{\alpha_2}^{k_2+j_0}\right)
D_{k_1,k_2}^N(h)\left(x_{\alpha_1}^{k_1+j_0},x_{\alpha_2}^{k_2+j_0}\right) \\
&=:\sum_{l=-\infty}^\infty\sum_{Q\in\mathcal R_l}\mu(Q)D_Q(\cdot,*, x_Q)D_Q^N h(x_Q)
\end{align*}
in $L^2$. Observe that, for any $l\in\mathbb Z$,
\begin{align*}
\sum_{Q\in\mathcal R_l}\mu(Q)\left|D_Q^N h(x_Q)\right|^2
&\le 2\sum_{Q\in\mathcal R_l}\mu(Q\setminus\Omega_{l+1})\left|D_k^N h(x_Q)\right|^2\\
&\lesssim\int_{\widetilde\Omega_l\setminus\Omega_{l+1}}\sum_{Q\in \mathcal R_l}\left|D_Q^N h\left(x_Q\right)\right|^2
\mathbf 1_{Q}(x_1,x_2)\,d\mu(x_1,x_2)\\
&\lesssim\int_{\widetilde\Omega_l\setminus\Omega_{l+1}}\left[F(x_1,x_2)\right]^2\,d\mu(x_1,x_2)
\lesssim 2^{2l}\mu\left(\widetilde\Omega_l\right).
\end{align*}
Thus, we claim that, for any $l\in\mathbb Z$, if we let $\lambda_l:=2^{l}\mu(\widetilde\Omega_l)^{\frac{1}{p}}$ and
$$
a_l(\cdot):=\frac 1{\lambda_l}\sum_{Q\in \mathcal R_l}\mu(Q)D_Q(\cdot, x_Q)D^N_Q h(x_Q),
$$
then $a_l$ is a harmless constant multiple of a $(p,q,C_0)$-atom  with some sufficiently large $C_0$.
Moreover, the above summation converges in
$L^q$. Indeed, by \eqref{eq-suppatm},
we have $\mathop{\rm supp} a_l\subset\widetilde\Omega_l$.
For any $g\in L^{q'}$ with $\|g\|_{L^{q'}}\le 1$,
we obtain
\begin{align}\label{eq-atomlq}
|\langle a_l,g\rangle|
&=2^{-l}\mu\left(\widetilde\Omega_l\right)^{-\frac{1}{p}}\left|\left<\sum_{Q\in\mathcal R_l}
\mu(Q)D_Q(\cdot,\cdot,x_Q)D_Q^Nh(x_Q), g\right>\right|\nonumber\\
&=2^{-l}\mu\left(\widetilde\Omega_l\right)^{-\frac{1}{p}}\left|\sum_{Q\in\mathcal R_l}
\mu(Q)D_Q g(x_Q) D^N_Q h(x_Q)\right|\nonumber\\
&=2^{-l}\mu\left(\widetilde\Omega_l\right)^{-\frac{1}{p}}
\left|\int_{X_1\times X_2}\sum_{Q\in\mathcal R_l}D_Q g(x_Q)D_Q^N h(x_Q)
\mathbf 1_{Q}(x_1,x_2)\,d\mu(x_1,x_2)\right|\nonumber\\
&\le 2^{-l}\mu\left(\widetilde\Omega_l\right)^{-\frac{1}{p}}
\left\{\int_{X_1\times X_2}\left[\sum_{Q\in \mathcal R_l}\left|D_Q g(x_Q)\right|^2
\mathbf 1_{Q}(x_1,x_2)\right]^{\frac{q'}{2}}\,d\mu(x_1,x_2)\right\}^{\frac{1}{q'}}\nonumber\\
&\quad\times\left\{\int_{X_1\times X_2}\left[\sum_{Q\in \mathcal R_l}\left|D_Q^N h(x_Q)\right|^2
\mathbf 1_{Q}(x_1,x_2)\right]^{\frac{q}{2}}\,d\mu(x_1,x_2)\right\}^{\frac{1}{q}}\nonumber\\
&=:\mathrm I\times \mathrm J.
\end{align}

Obviously, applying an argument similar to that used in the proof of Lemma \ref{lem-slp}, we find that
$\mathrm I\lesssim\|g\|_{L^{q'}}$.

Now, we estimate $\mathrm J$. Indeed, by \eqref{eq-suppatm}, we obtain,
for any $(x_1,x_2)\in Q\in\mathcal R_l$, $\mu(Q\cap\Omega_{l+1})\le \frac{\mu(Q)}{2}$, which further
implies that $\mu(Q\cap\Omega_{l+1}^\complement)\ge \frac{\mu(Q)}{2}$.
By this and \eqref{eq-suppatm}, we have, for any $(x_1, x_2) \in Q$,
$$
\mathcal M_{\rm str}\left(\mathbf 1_{Q\cap\widetilde\Omega_l\cap\Omega_{l+1}^\complement}\right)
(x_1,x_2)>c.
$$
From this and Proposition \ref{prop-fs}, we deduce that
\begin{align}\label{eq-atomestJ}
\mathrm J
&\lesssim\left\{\int_{X_1\times X_2}\left[\sum_{Q\in\mathcal R_l}\left|D_Q^N h(x_Q)\mathcal M_{\rm str}
\left(\mathbf 1_{Q\cap\widetilde\Omega_l\cap\Omega_{l+1}^\complement}\right)(x_1,x_2)\right|^2\right]^{\frac{q}{2}}
\,d\mu(x_1,x_2)\right\}^{\frac{1}{q}}\nonumber\\
&\lesssim\left\{\int_{X_1\times X_2}\left[\sum_{Q\in\mathcal R_l}\left|D_Q^N h(x_Q)
\mathbf 1_{Q\cap\widetilde\Omega_l\cap\Omega_{l+1}^\complement}(x_1,x_2)\right|^2\right]^{\frac{q}{2}}
\,d\mu(x_1,x_2)\right\}^{\frac{1}{q}}\nonumber\\
&\sim\left\{\int_{\widetilde\Omega_l\cap\Omega_{l+1}^\complement}
\left[\sum_{Q\in\mathcal R_l}\left|D_Q^N h(x_Q)
\mathbf 1_{Q}(x_1,x_2)\right|^2\right]^{\frac{q}{2}}\,d\mu(x_1,x_2)\right\}^{\frac{1}{q}}
\lesssim 2^l\mu\left(\widetilde\Omega_l\right)^{\frac{1}{q}}.
\end{align}
By \eqref{eq-atomlq} and \eqref{eq-atomestJ}, we find that
\begin{equation}\label{eq-atomlqre}
|\langle a_l, g\rangle|\lesssim \mu\left(\widetilde\Omega_l\right)^{\frac{1}{q}-\frac{1}{p}}\|g\|_{L^{q'}},
\end{equation}
which, together with the arbitrariness of $g\in L^{q'}$,  further conclude that
$$
\|a_l\|_{L^q}\lesssim\mu\left(\widetilde\Omega_l\right)^{\frac{1}{q}-\frac{1}{p}}.
$$
This is the desired estimate of the size condition.

Now, we prove that $a_l$ satisfies Definition \ref{def-atom}(ii). To this end, we consider two cases
for $q$.

{\it Case (1) $q\in[2,\infty)$.} In this case, for any $Q:=Q_1\times Q_2\in\mathcal R_l$, we first construct
$m(Q)$ as follows: Choose $P_1\in\mathcal D_1$ being the maximal cube such that
$P_1\times Q_2\subset\widetilde\Omega_l$; then choose $P_2\in\mathcal D_2$ being the maximal cube such that
$P_1\times P_2\subset\widetilde\Omega_l$, and denote $P_1\times P_2$ by $m(Q)$. Obviously, for any $Q\in\mathcal R_l$,
we have $Q\subset m(Q)\subset\widetilde\Omega_l$.

In the next, we show that $m(Q)\in\mathcal M(\widetilde\Omega_l)$. Otherwise, there exists
$\widetilde R:=\widetilde P_1\times \widetilde P_2\in\mathcal M(\widetilde\Omega_l)$ and
$\widetilde R\supsetneqq m(Q)$. Then $\widetilde P_1\supset P_1$,
$\widetilde P_2\supset P_2$, and at least either of $\widetilde P_1\supsetneqq P_1$ or
$\widetilde P_2\supsetneqq P_2$ holds. However, if $\widetilde P_1\supsetneqq P_1$, then
$\widetilde P_1\times Q_2\supsetneqq P_1\times Q_2\supset Q_1\times Q_2$,
and this contradicts the definition of $P_1$, which implies $\widetilde P_1=P_1$; similarly,
$\widetilde P_2=P_2$. Then we have $\widetilde R=R$, which makes a contradiction again. Thus, we have,
for any $Q\in\mathcal R_l$, $m(Q)\in\mathcal M(\widetilde\Omega_l)$.

Now, for any $R\in\mathcal M(\widetilde\Omega_l)$, we let
$$
a_{l,R}(\cdot,*):=2^{-l}\mu\left(\widetilde\Omega_l\right)^{-\frac{1}{p}}
\sum_{Q\in\mathcal R_l, m(Q)=R}\mu(Q)D_Q (\cdot,*,x_Q)D_Q^N h(x_Q).
$$
Applying an argument similar to that used in the proof of \eqref{eq-atomlqre}, we find that
\begin{align*}
&\sum_{R\in\mathcal M(\widetilde\Omega_l)}\|a_{l,R}\|_{L^q}^q\\
&\quad\lesssim2^{-lq}\mu\left(\widetilde\Omega_l\right)^{-\frac{q}{p}}
\int_{\widetilde\Omega_{l}\cap\Omega_{l+1}^\complement}\sum_{R\in\mathcal M(\widetilde\Omega_l)}
\left[\sum_{\{Q\in\mathcal R_l:\ m(Q)=R\}}\mu(Q)
\left|D_Q^N h(x_Q)\right|^2\mathbf 1_{Q}(x_1,x_2)\right]^{\frac{q}{2}}\,d\mu(x_1,x_2)\\
&\quad\lesssim 2^{-lq}\mu\left(\widetilde\Omega_l\right)^{-\frac{q}{p}}\int_{\widetilde\Omega_{l}
\cap\Omega_{l+1}^\complement}
\left[\sum_{R\in\mathcal M(\widetilde\Omega_l)}\sum_{\{Q\in\mathcal R_l:\ m(Q)=R\}}\left|D_Q^N h(x_Q)\right|^2
\mathbf 1_{Q}(x_1,x_2)\right]^{\frac{q}{2}}\,d\mu(x_1,x_2)\\
&\quad\lesssim\mu\left(\widetilde\Omega_l\right)^{1-\frac{q}{p}}.
\end{align*}
Moreover, by the cancellation of $D_Q$, we conclude that, for any $x_1\in X_1$ and $x_2\in X_2$,
$$
\int_{X_2}a_{l,R}(x_1,y_2)\,d\mu_2(y_2)=0=\int_{X_1}a_{l,R}(x_1,y_2)\,d\mu_2(y_2).
$$
This proves that  $a_{l,R}$ satisfies Definition \ref{def-atom}(ii) in this case.

{\it Case (2) $q\in(1,2)$.}
In this case, we find that, for any $Q:=Q_1\times Q_2\in\mathcal R_l$, there exist unique rectangles
$m_1(Q):=P_1\times Q_2\in\mathcal M_1(\widetilde\Omega_l)$
and $m_2(Q):=Q_1\times P_2\in\mathcal M_2(\widetilde\Omega_l) $ such that
$Q_1\subset P_1$ and $Q_2\subset P_2$. For any $l\in\mathbb Z$, $i\in\{1,2\}$, and $R_i\in\mathcal M_i(\widetilde\Omega_l)$,
let
$$
a_{i,l,R_i}(\cdot,*):=2^{-l-1}\mu\left(\widetilde\Omega_l\right)^{-\frac1p}
\sum_{\{Q\in\mathcal R_l:\ m_i(Q)=R_i\}}\mu(Q)D_Q(\cdot,*,x_Q)D_Q^N h(x_Q).
$$
Then we have
$$
a_l=\sum_{R_1\in\mathcal M_1(\widetilde\Omega_1)}a_{1,l,R_1}+\sum_{R_2\in\mathcal M_2(\widetilde\Omega_1)}a_{2,l,R_2}.
$$
Besides, applying an argument similar to that used in the estimation of \eqref{eq-atomlq}, we obtain
\begin{align*}
\left\|a_{i,l,R_i}\right\|_{L^q}^q
&\lesssim 2^{-lq}\mu\left(\widetilde\Omega_l\right)^{-\frac qp}\int_{X_1\times X_2}\left[\sum_{\{Q\in\mathcal R_l:\ m_i(Q)=R_i\}}
\left|D_Qh(x_Q)\right|^2\mathbf 1_{Q}(x_1,x_2)\right]^{\frac q2}\,d\mu(x_1,x_2)\\
&\lesssim 2^{-lq}\mu\left(\widetilde\Omega_l\right)^{-\frac qp}\int_{X_1\times X_2}\left[\sum_{\{Q\in\mathcal R_l:\ m_i(Q)=R_i\}}\left|D_Qh(x_Q)\mathcal M_{\rm str}
\left(\mathbf 1_{Q\cap R_i\cap\Omega_{l+1}^\complement}\right)(x_1,x_2)\right|^2\right]^{\frac q2} \,d\mu(x_1,x_2)\\
&\lesssim 2^{-lq}\mu\left(\widetilde\Omega_l\right)^{-\frac qp}
\int_{R_i\cap\Omega_{l+1}^\complement}\left[\sum_{\{Q\in\mathcal R_l:\ m_i(Q)=R_i\}}
\left|D_Qh(x_Q)\right|^2\mathbf 1_{Q}(x_1,x_2)\right]^{\frac q2}\,d\mu(x_1,x_2)\\
&\lesssim\mu\left(\widetilde\Omega_l\right)^{-\frac qp}\mu(R_i).
\end{align*}
By this and Lemma \ref{lem-j}, we find that
\begin{align*}
&\left\{\sum_{P_{i,1}\times P_{i,2}=R_i\in\mathcal M_i(\widetilde\Omega_l)}\left\|a_{R_i}\right\|_{L^q}^q
\left[\frac{\ell(P_{i,j})}{\ell(\widehat{P_{i,j}})}\right]^\epsilon\right\}^{\frac1q}\\
&\quad\lesssim\mu\left(\widetilde\Omega_l\right)^{-\frac 1p}\left\{\sum_{P_{i,1}\times P_{i,2}=R_i\in\mathcal M_i(\widetilde\Omega_l)}
\mu(R_i)\left[\frac{\ell(P_{i,j})}{\ell(\widehat{P_{i,j}})}\right]^\epsilon\right\}^{\frac1q}
\lesssim\mu\left(\widetilde\Omega_l\right)^{\frac 1q-\frac 1p},
\end{align*}
where we chose, for any $r\in(0,\infty)$, $w(r):=r^\epsilon$. Moreover, by
\eqref{eq-atomlqre} and the cancellation of $D_Q$, we have, for any $x_1\in X_1$ and $x_2\in X_2$,
$$
\int_{X_2} a_{i,l,R_i}(x_1,y_2)\,d\mu_2(y_2)=0=\int_{X_1}a_{i,l,R_i}(y_1,x_2)\,d\mu_1(y_1).
$$
Then we show that $a_l$ is a harmless multiple of a $(p,q,C_0)$-atom.

To summarize, for any $f\in H^p\cap L^2$, there exist a sequence of $(p,q,C_0)$-atoms $\{a_l\}_{l=-\infty}^\infty$
and a sequence $\{\lambda_l\}_{l=-\infty}^\infty\subset\mathbb C$ such that $f=\sum_{l=-\infty}^\infty \lambda_la_l$ and
\begin{align*}
\sum_{l=-\infty}^\infty|\lambda_l|^p&=\sum_{l=-\infty}^\infty  2^{lp}\mu\left(\widetilde\Omega_l\right)\\
&\lesssim\sum_{l=-\infty}^\infty  2^{lp}\int_{\{(x_1,x_2)\in X_1\times X_2:\ \mathcal M_{\rm str}(\mathbf 1_{\Omega_l})(x_1,x_2)>c\}}
\left[\mathcal M_{\rm str}(\mathbf 1_{\Omega_l})(x_1,x_2)\right]^{\frac{1}{p}}\,d\mu(x_1,x_2)\\
&\lesssim\sum_{l=-\infty}^\infty  2^{lp}\int_{X_1\times X_2}\left[\mathbf 1_{\Omega_l} (x_1,x_2)\right]^{\frac{1}{p}} \, d\mu(x_1,x_2)\\
& \lesssim\sum_{l=-\infty}^\infty  2^{lp}\mu(\Omega_l)\lesssim\|h\|^p_{H^p}\lesssim\|f\|^p_{H^p},
\end{align*}
where the penultimate inequality comes from
\begin{align*}
\mu(\Omega_l)&=\mu\left(\left\{(x_1,x_2)\in X_1\times X_2:\ F(x_1,x_2)>2^l\right\}\right)
\le 2^{-lp} \int_{X_1\times X_2}[F(x_1,x_2)]^p \,d\mu(x_1,x_2)\\
&=2^{-lp}\|F\|^p_{L^p}\lesssim 2^{-lp}\|h\|^p_{H^p}.
\end{align*}
We have proved that, if $f\in H^p\cap L^2$, then $f$ has an atomic decomposition.
So, next, we need to show that the
atomic decomposition for $f$ obtained above must converge in the dual of
$\mathrm{CMO}^p_{L^2}$. Indeed, for any $g \in \mathrm{CMO}^p_{L^2}$, by Proposition \ref{prop-dual}, we have
\begin{align*}
\left|\left< f-\sum_{|l| \le N} \lambda_l a_l, g \right> \right|
\lesssim \mathcal C_p(g)\left\|f-\sum_{|l| \le N} \lambda_l a_l\right\|_{H^p},
\end{align*}
where the last term tends to  $0$ as  $N$ tends to infinity because
$$
\lim_{N\to\infty}\left\|f-\sum_{|l| \le N} \lambda_l a_l\right\|_{H^p}=0.
$$
It means that $\sum_{l} \lambda_l a_l$ converges to $f$ in the dual of $\mathrm{CMO}^p_{L^2}$.

Now, if $f\in H^p$, since $H^p\cap L^2$ is dense in $H^p$ (see Proposition \ref{prop-dense}),
applying a standard density argument similar to that used in the proof of \cite[Theorem 4.13]{MS79b},
we obtain all the desired conclusions of in Proposition \ref{prop-atd}, which completes the proof of
Proposition \ref{prop-atd}.
\end{proof}

\begin{remark}\label{r-prange}
We assert that the range $p\in(\frac{\omega}{\omega+\eta},1]$ in Theorem \ref{thm-h=a}
is sharp in some sense. Indeed, let $X$ be a space of homogeneous type. The real-variable theory
of the single-parameter Hardy space $H^p(X)$ is only meaningful in $p\in(\frac{\omega}{\omega+\eta},1]$
because its atoms only have the cancellation of moment $0$. If $X=\mathbb R^n$, then atoms with cancellation of moment $0$
can only be used characterize $H^p(\mathbb R^n)$ when $p\in(\frac{n}{n+1},1]$.
In this sense, the range of $p$ in Theorem \ref{thm-h=a} is sharp and the best possible.
\end{remark}

\subsection{A Criterion of Boundedness of Operators
from $H^p$ to $L^p$ and Its Applications}\label{ss-bhl}

In this section, we give an application of atomic characterizations to the boundedness of linear operators. To this end,
we first introduce the concept of rectangle atoms.
\begin{definition}\label{def-rat}
Let $p\in(0,1]$ and $C_0\in(0,\infty)$. A function $a$ is called a \emph{$(p,2,C_0)$-rectangle atom} if it satisfies
\begin{enumerate}
\item [(i)] $\mathop{\rm supp} a\subset C_0B_1(Q_1)\times C_0B_2(Q_2)$,
\item [(ii)] $\|a\|_{L^2}\leq [\mu_1(Q_1)\mu_2(Q_2)]^{\frac{1}{2}-\frac{1}{p}}$,
\item [(iii)] for any $x_1\in X_1$ and $x_2\in X_2$,
$$
\int_{C_0B_1(Q_1)}a(y_1,x_2)\,d\mu_1(y_1)=0=\int_{C_0B_2(Q_2)}a(x_1,y_2)\,d\mu_2(y_2).
$$
\end{enumerate}
\end{definition}

In the past, it was conjectured that product Hardy spaces can be characterized in terms of rectangle atoms.
However, this conjecture was disproved by a counter-example constructed by Carleson \cite{car74}.
Although, the space $H^p$ cannot be decomposed into rectangle atoms, we can combine the atom decomposition in
Theorem \ref{thm-h=a} with  Fefferman's method (see \cite{F87}) by considering the action of linear operator $T$ on rectangle atoms to
establish the boundedness of linear operators from product Hardy spaces to corresponding Lebesgue spaces.

\begin{theorem}\label{thm-bddhl}
Let $r\in(0,\eta]$ and $p\in(\frac{\omega}{\omega+\eta},1]$ with $\omega$ and $\eta$ as, respectively,
in \eqref{eq-doub} and Theorem \ref{thm-wave}. Assume that $C_0\in(0,\infty)$ as in Theorem \ref{thm-h=a}
and that $T$ is a bounded linear operator on $L^2$. If, for any \emph{$(p,2,C_0)$}-rectangle atom $a$
supported in $C_0B_1(Q_1)\times C_0B_2(Q_2)$ and for any $\gamma\in[2A_0,\infty)$,
\begin{equation}\label{eq-bddhl}
\int_{\{[\gamma C_0B_1(Q_1)]\times [\gamma C_0B_2(Q_2)]\}^\complement}|Ta(x_1,x_2)|^p\, d\mu(x_1,x_2)\leq \gamma^{-r},
\end{equation}
then $T$ can be extended to a bounded linear operator from $H^p$ to $L^p$, still denoted by $T$, and there
exists a constant $C\in(0,\infty)$ such that $\|T\|_{H^p\to L^p}\le C$.
\end{theorem}
\begin{proof}
We first show that, for any $(p,2,C_0)$-atom $a$ supported in some open set $\Omega$ with $\mu(\Omega)<\infty$,
\begin{equation}\label{eq-talp}
\|Ta\|_{L^p}\lesssim 1.
\end{equation}
To achieve this, let
$$
\widetilde{\Omega} := \{(x_1,x_2)\in X_1\times X_2:\ \mathcal M_{\rm str}(\mathbf{1}_\Omega)(x_1,x_2)>c\}
$$
and
$$
\widetilde{\widetilde{\Omega}} := \left\{(x_1,x_2)\in X_1\times X_2:\ \mathcal M_{\rm str}
\left(\mathbf{1}_{\widetilde{\Omega}}\right)(x_1,x_2) > c\right\},
$$
where $c\in(0,\infty)$ is a constant which is determined later. Then, 
by Proposition \ref{prop-fs}, we find that
\begin{equation}\label{eq-omegasim2}
\mu\left(\widetilde{\widetilde{\Omega}}\right) \sim \mu\left(\widetilde{\Omega}\right) \sim \mu(\Omega).
\end{equation}
By this and the boundedness of $T$ on $L^2(X)$, we obtain
$$
\left\|(Ta)\mathbf{1}_{\widetilde{\widetilde{\Omega}}}\right\|_{L^2} \leq \|Ta\|_{L^2} \lesssim \|a\|_{L^2}
\lesssim\mu(\Omega)^{\frac{1}{2}-\frac{1}{p}}.
$$
This, together with H\"older's inequality and \eqref{eq-omegasim2}, further implies that
\begin{align}\label{eq-ta1}
\left\|(Ta)\mathbf{1}_{\widetilde{\widetilde{\Omega}}}\right\|_{L^p}
&=\left[\int_{\widetilde{\widetilde{\Omega}}}|T_a(x_1,x_2)|^p\,d\mu(x_1,x_2)\right]^{\frac{1}{p}}
\le \left[\int_{\widetilde{\widetilde{\Omega}}}| T_a(x_1,x_2)|^2\, d \mu(x_1) d \mu(x_2)\right]^{\frac{1}{2}}
\mu\left(\widetilde{\widetilde{\Omega}}\right)^{(1-\frac{2}{p})\cdot\frac{1}{2}}\nonumber\\
&\lesssim \mu(\Omega)^{\frac{1}{2}-\frac{1}{p}}\mu\left(\widetilde{\widetilde{\Omega}}
\right)^{\frac{1}{p}-\frac{1}{2}} \sim 1.
\end{align}

It therefore suffices to handle $\|(Ta ) \mathbf{1}_{(\widetilde{\widetilde{\Omega}})^\complement}\|_{L^p}$.
Notice that
\begin{equation}\label{eq-ta2}
\int_{(\widetilde{\widetilde{\Omega}})^\complement} | Ta(x_1,x_2)|^p\,d \mu(x_1,x_2)
\leq \sum_{R \in\mathcal M(\Omega)} \int_{(\widetilde{\widetilde{\Omega}})^\complement}| Ta_R(x_1,x_2)|^p\,d \mu(x_1,x_2),
\end{equation}
where $\{a_R\}_{R\in\mathcal M(\Omega)}$ is the same as in Definition \ref{def-atom}(iii).
For each $R:=Q_1\times Q_2\in \mathcal{M}(\Omega)\subset \mathcal{M}_1(\Omega)$, we can choose a maximal dyadic cube
$\widehat{Q_2}$ such that
$$
\mu\left(\left(Q_1\times\widehat{Q}_2\right) \cap\Omega\right) > \frac{1}{2}\mu\left(Q_1\times\widehat{Q}_2\right).
$$
Since $Q_1\times\widehat{Q}_2 \subset \widetilde{\Omega}$, it follows
that there exists a maximal dyadic cube $P_2$ for the
first variable such that $Q_1\times P_2 \in \mathcal{M}_2(\widetilde{\Omega})$.
Therefore, there also exists
a maximal dyadic cube $\widehat{Q}_1$ for the first variable containing  $Q_1$ such that
$$
\mu\left(\left(\widehat{Q}_1\times P_2\right) \cap \widetilde{\Omega}\right) > \frac{1}{2}\mu\left(\widehat{Q}_1\times P_2\right).
$$
We may choose $c\in(0,1)$ such that $[3A_0^2C_0B_1(\widehat{Q}_1)]\times[3A_0^2C_0B_2(P_2)]\subset\widetilde{\widetilde{\Omega}}$.
It follows from Lemma \ref{lem-cube}(d) that
\begin{align*}
&\int_{(\widetilde{\widetilde{\Omega}})^\complement} | Ta_R(x_1,x_2)|^p\,d \mu(x_1,x_2)\\
&\quad\leq \int_{\{[3A_0^2C_0B_1(\widehat{Q}_1)]\times[3A_0^2C_0B_2(P_2)]\}^\complement}|Ta_R(x_1,x_2)|^p d \mu(x_1,x_2)\\
&\quad\leq \int_{[3A_0^2C_0B_1(\widehat{Q}_1)]^\complement\times X_2} | Ta_R(x_1,x_2)|\,d \mu(x_1,x_2)
+\int_{X_1\times[3A_0^2C_0B_2(P_2)]^\complement}\cdots\\
&\quad\leq \int_{[3A_0^2C_0B_1(\widehat{Q}_1)]^\complement\times X_2} | Ta_R(x_1,x_2)|\,d \mu(x_1,x_2)
+\int_{X_1\times[3A_0^2C_0B_2(\widehat Q_2)]^\complement}\cdots\\
&\quad=:\mathrm I+\mathrm{II}.
\end{align*}

Without loss of generality, we may assume that, for any $R\in\mathcal M(\Omega)$, $a_R \neq 0$.
Thus, for any $R\in\mathcal M(\Omega)$, $\frac{a_R}{\|a_R\|_{L^2}}\mu(R)^{\frac{1}{2}-\frac{1}{p}}$ is a
$(p, 2,C_0)$-atom supported in $R$, which, together with \eqref{eq-bddhl}, shows that, for any $\gamma\in[2A_0,\infty)$,
\begin{equation}\label{eq-bddhl2}
\int_{\{[\gamma C_0B_1(Q_1)]\times [\gamma C_0B_2(Q_2)]\}^\complement} | Ta_R(x_1,x_2) |^p d\mu(x_1,x_2)
\le \gamma^{-r} \left\| a_R \right\|_{L^2}^p\mu(R)^{1 - \frac{p}{2}}.
\end{equation}

For the item $\rm II$, we assume that $Q_2=Q_{2,\beta_2}^{j_2}$ and $\widehat{Q_2} = Q_{2,\beta'_2}^{j_2'}$ for some
$j_2,j_2'\in\mathbb Z$, $\beta_2\in\mathcal A_{2,j_2}$, and $\beta_2'\in\mathcal A_{2,j_2'}$. Since $Q_2\subset \widehat{Q}_2$,
it follows that $j_2\ge j_2'$ and, from Lemma \ref{lem-cube}(d), that $B_2(Q_2) \subset B_2(\widehat{Q_2})$. Thus,
$d(z_{2,\beta_1}^{j_2}, z_{2,\beta_2'}^{j_2'}) \leq C^\natural \delta^{j_2'}$ and, consequently,
\begin{align*}
2A_0\delta^{j_2'-j_2}C_0B_2(Q_2)&=B_2\left(z_{2,\beta_2}^{j_2},2A_0C_0C^\natural\delta^{j_2'}\right)
\subset B_2\left(z_{2,\beta_2'}^{j_2'}, 3A_0^2C_0 C^\natural \delta^{j_2'}\right)=3A_0^2C_0B_2\left(\widehat{Q}_2\right),
\end{align*}
which further implies that $[3A_0^2C_0 B_2(\widehat{Q}_2)]^\complement \subset [2A_0\delta^{j'_2-j_2} C_0B_2(Q_2)]^\complement$.
By this and \eqref{eq-bddhl2}, we find that
\begin{align}\label{eq-es2}
{\rm II}&\le\int_{X_1\times[2A_0\delta^{j'_2-j_2} C_0B_2(Q_2)]^\complement}|Ta_R(x_1,x_2)|^p\,d\mu(x_1,x_2)\nonumber\\
&\le \int_{\{[2A_0\delta^{j'_2-j_2} C_0B_1(Q_1)]\times[2A_0\delta^{j'_2-j_2} C_0B_2(Q_2)]\}^\complement}
|Ta_R(x_1,x_2)|^p\,d\mu(x_1,x_2)\nonumber\\
&\lesssim\left[\frac{\ell(Q_2)}{\ell(\widehat{Q}_2)}\right]^r\|a_R\|_{L^2}\mu(R)^{1-\frac p2}.
\end{align}
Summing this over $R\in\mathcal M(\Omega)$ and using H\"older's inequality and Lemma \ref{lem-j}, we conclude that
\begin{align}\label{eq-es2s}
&\sum_{R \in \mathcal{M}(\Omega)} \int_{X_1\times[3A_0^2C_0B_2(P_2)]^\complement} | Ta_R(x_1,x_2) |^p
\, d \mu(x_1) d \mu(x_2)\nonumber\\
&\quad\lesssim\sum_{R \in \mathcal{M}_1(\Omega)} \left[\frac{\ell(Q_2)}{\ell(\widehat{Q_2})} \right]^r \left\|a_R\right\|_{L^2}^p
\mu(R)^{1-\frac{p}{2}}\nonumber\\
&\quad\lesssim \left[\sum_{R \in \mathcal{M}(\Omega)} \left\|a_R\right\|_{L^2}^2 \right]^{\frac{p}{2}}
\left\{\sum_{R \in \mathcal{M}(\Omega)} \left[\frac{\ell(Q_2)}{\ell(\widehat{Q_2})} \right]^{r(\frac{p}{2})'}
\mu(R)^{(1-\frac{p}{2})(\frac{2}{p})'} \right\}^{1-\frac{p}{2}}\nonumber\\
&\quad\lesssim\mu(\Omega)^{\frac{p}{2}-1}\mu(\Omega)^{1-\frac{p}{2}} \sim 1,
\end{align}
This is the desired estimate.

Next, we estimate $\rm I$. From an argument similar to that used in the estimation of \eqref{eq-es2}, we deduce that
\begin{equation}\label{eq-es1}
\mathrm{I}\lesssim\left[\frac{\ell(Q_1)}{\ell(\widehat{Q_1})}\right]^r\|a_R\|_{L_2}^p\mu(R)^{1-\frac{2}{p}}.
\end{equation}
Although this estimate is similar to \eqref{eq-es2}, when summing over $R\in\mathcal M(\Omega)$, the estimation differs from
\eqref{eq-es2s}. Suppose $Q_1\times Q_2\in\mathcal M(\Omega)$. Fix $Q_1$ and let $m_2^{Q_1}(Q_2)=P_2$ be the maximal dyadic cube
in $X_2$ such that $Q_1\times P_2\subset\widetilde{\Omega}$. It then follows that, for any fixed $Q_1$,
$\{P_2:\ P_2=m_2^{Q_1}(Q_2),Q_1\times Q_2\in\mathcal M(\Omega)\}$ is mutually disjoint. By this, \eqref{eq-es1},
H\"older's inequality, and Lemma \ref{lem-j}, we conclude that
\begin{align}\label{eq-es1s}
&\sum_{R \in \mathcal{M}(\Omega)} \int_{[3A_0^2C_0B_1(\widehat{Q}_1)]^\complement\times X_2}
| Ta_R(x_1,x_2) |^p\, d \mu(x_1,x_2)\nonumber\\
&\quad\lesssim \left[\sum_{R \in \mathcal{M}(\Omega)} \left\|a_R\right\|_{L^2}^2 \right]^{\frac{p}{2}}
\left\{ \sum_{R \in \mathcal{M}(\Omega)}\left[\frac{\ell(Q_1)}{\ell(\widehat{Q_1})} \right]^{r(\frac{p}{2})'}
\mu_1(Q_1) \mu_2(Q_2) \right\}^{1 - \frac{p}{2}}\nonumber\\
&\quad\lesssim\mu(\Omega)^{\frac{p}{2}-1}\left\{ \sum_{Q_1} \mu_1(Q_1)
\sum_{ \{Q_1:\ Q_1\times Q_2 \in \mathcal{M}(\Omega)\} }
\mu_2(Q_2) \left[ \frac{\ell(Q_1)}{\ell(\widehat{Q_1})} \right]^{r(\frac{p}{2})'} \right\}^{1 - \frac{p}{2}}\nonumber\\
&\quad\sim\mu(\Omega)^{\frac{p}{2} - 1} \left\{ \sum_{Q_1} \mu_1(Q_1) \sum_{P_2}
\left( \frac{\ell(Q_1)}{\ell(\widehat{Q_1})} \right)^{r(\frac{p}{2})'}
\sum_{\{Q_2:\ Q_1\times Q_2 \in \mathcal M(\Omega), m_2^{Q_1}(Q_2)=P_2\}} \mu_2(Q_2) 
\right\}^{1 - \frac{p}{2}}\nonumber\\
&\quad\lesssim\mu(\Omega)^{\frac{p}{2} - 1} \left\{ \sum_{Q_1} \mu_1(Q_1) \sum_{P_2}
\left[\frac{\ell(Q_1)}{\ell(\widehat{Q_1})} \right]^{r(\frac{p}{2})'} \mu_2(P_2) \right\}^{1-\frac{p}{2}}\nonumber\\
&\quad\lesssim\mu(\Omega)^{\frac{p}{2} - 1} \left\{ \sum_{Q_1\times P_2 \in \mathcal{M}_2(\widetilde{\Omega})}
\mu(Q_1\times P_2) \left[\frac{\ell(Q_1)}{\ell(\widehat{Q_1})} \right]^{r(\frac{p}{2})'}\right\}^{1-\frac p2}\nonumber\\
&\quad\lesssim\mu(\Omega)^{\frac{p}{2} - 1}\mu(\widetilde{\Omega})^{1 - \frac{p}{2}}
\leq\mu(\Omega)^{\frac{p}{2} -1}\mu(\Omega)^{1 - \frac{p}{2}} \sim 1.
\end{align}
This is also the desired estimate. By \eqref{eq-ta1}, \eqref{eq-ta2}, \eqref{eq-es1s}, and \eqref{eq-es2s}, we finally obtain
\eqref{eq-talp}.

Now, we suppose $f\in H^p\cap L^2$. Then, from an argument similar to that used in the proof of
Proposition \ref{prop-atd}, it follows that there exist a sequence of $(p,2,C_0)$-atoms $\{a_j\}_{j=1}^\infty$,
where $C_0$ is sufficiently large, and a sequence $\{\lambda_j\}_{j=1}^\infty\subset(0,\infty)$ such that
\begin{equation}\label{eq-fsim2}
\sum_{j=1}^\infty|\lambda_j|^p\lesssim\|f\|_{H^p}^p
\end{equation}
and
$f=\sum_{j=1}^\infty\lambda_ja_j$ in $L^2$. By the boundedness of $T$ on $L^2$, we have
$Tf=\sum_{j=1}^\infty\lambda_jTa_j$. Using this and the Riesz theorem, we find that there exists
$\{N_k\}_{k=1}^\infty\subset\mathbb N$ with $\lim_{k\to\infty}N_k=\infty$ such that
$$
Tf(x)=\lim_{k\to\infty}\sum_{j=1}^{N_k}\lambda_jTa_j(x)
$$
holds for almost every $x\in X_1\times X_2$, which, combined with $p\in(0,1]$, further implies that
$$
|Tf(x)|^p\le\sum_{j=1}^{\infty}|\lambda_j|^p|Ta_j(x)|^p.
$$
From this, \eqref{eq-talp}, and \eqref{eq-fsim2}, we deduce that
$$
\|Tf\|_{L^p}^p\le\sum_{j=1}^\infty|\lambda_j|^p\|Ta_j\|_{L^p}^p
\lesssim\sum_{j=1}^\infty|\lambda_j|^p\lesssim\|f\|_{H^p}^p.
$$
Using this and the density of $H^p\cap L^2$ in $H^p$, we obtain the desired extension of $T$ and its desired boundedness, which completes the proof of Theorem \ref{thm-bddhl}.
\end{proof}

In the remainder of this section, we give an application of Theorem \ref{thm-bddhl}. To this end,
we first recall the concepts of Calder\'on--Zygmund kernels on a given space of homogeneous type.

\begin{definition}\label{def-wczok}
Let $X$ be a space of homogeneous type and $\epsilon\in(0,\infty)$. A function
$K:\ (X\times X)\setminus\{(x,x):\ x\in X\}$ is called a \emph{weak $\epsilon$-Calder\'on--Zygmund kernel}
if there exists $C\in(0,\infty)$ such that $K$ satisfies the following H\"older's regularity condition:
for any $x,y,y'\in X$ with $x\neq y$ and $d(y,y')\le(2A_0)^{-1}d(x,y)$,
\begin{equation}\label{eq-holderreg}
|K(x,y)-K(x,y')|\le\left[\frac{d(y,y')}{d(x,y)}\right]^\epsilon\frac{C}{V(x,y)}.
\end{equation}
Define $C_K:=\inf\{C\in(0,\infty):\ \textup{\eqref{eq-holderreg} holds}\}$.
\end{definition}

\begin{definition}\label{def-wczo}
Let $X$ be a space of homogeneous type and $\epsilon\in(0,\eta]$. A linear operator $T$ is called
a \emph{weak $\epsilon$-Calder\'on--Zygmund operator} if $T$ satisfies the following conditions:
\begin{enumerate}
\item $T$ is a bounded operator on $L^2(X)$,
\item there exists a weak $\epsilon$-Calder\'on--Zygmund kernel $K$ such that, for any
$f\in L^2(X)$ with bounded support and for any $x\notin\overline{\mathop{\rm supp}f}$,
$$
Tf(x)=\int_X K(x,y)f(y)\,d\mu(y).
$$
\end{enumerate}
For any weak $\epsilon$-Calder\'on--Zygmund operator $T$, let
$$
\|T\|_{\mathrm{WCZO}_\epsilon(X)}:=\|T\|_{\mathcal L(L^2(X))}+C_K.
$$
\end{definition}

Now, we turn to the multiparameter case. In what follows, let $X_1$ and $X_2$ be two
spaces of homogeneous type and
\begin{equation}\label{eq-diag}
\Delta:=\left\{(x_1,x_2,y_1,y_2)\in(X_1\times X_2)^2:\ x_1=y_1\textup{ or }x_2=y_2\right\}.
\end{equation}
\begin{definition}\label{def-wpczk}
Let $\epsilon\in(0,\eta]$ with $\eta$ the same as in Theorem \ref{thm-wave}.
A linear operator $T$ is said to belong to the \emph{$\epsilon$-Journ\'e's class}
if it has the following properties:
\begin{enumerate}
\item $T$ is bounded on $L^2$,
\item there exists a function $(X_1\times X_2)^2\setminus\Delta\to\mathbb C$ such that, for any
$f\in L^2$ with $\mathop{\rm supp}f\subset B_1\times B_2$ for some ball $B_i$ in $X_i$ with $i\in\{1,2\}$,
and, for any $(x_1,x_2)\in(2A_0B_1)^\complement\times(2A_0B_2)^\complement$
$$
Tf(x_1,x_2)=\int_{X_1\times X_2}K(x_1,x_2,y_1,y_2)f(y_1,y_2)\,d\mu(y_1,y_2),
$$
\item for any $x_1,y_1\in X_1$ with $x_1\neq y_1$, there exists a weak $\epsilon$-Calder\'on--Zygmund
operator $T^{(2)}(x_1,y_1)$ on $X_2$, whose kernel is $K(x_1,\cdot,y_1,*)$, 
such that, for any $a\in L^2$ satisfying $\mathop{\rm supp}a\subset B_1\times B_2$ 
with balls $B_1\subset X_1$ and $B_2\subset X_2$ and for any $x_2\in(2A_0B_2)^\complement$,
$$
Ta(x_1,x_2)=\int_{X_1}T^{(2)}(x_1,y_1)a(y_1,\cdot)(x_2)\,d\mu_1(y_1);
$$
moreover, there exists a positive constant $C$ such that, for any $x_1,y_1,y_1'\in X_1$ with
$x_1\neq y_1$ and $d_1(y_1,y_1')\le(2A_0)^{-1}d_2(x_1,y_1)$,
$$
\left\|T^{(2)}(x_1,y_1)-T^{(2)}(x_1,y_1')\right\|_{{\rm WCZO}_\epsilon(X_2)}
\le\left[\frac{d_1(y_1,y_1')}{d_1(x_1,y_1)}\right]^\epsilon\frac{C}{V_1(x_1,y_1)},
$$
\item property (iii) also holds with $X_1$ and $X_2$ interchanged.
\end{enumerate}
\end{definition}

The next theorem shows that an operator belonging to Journ\'e's class is bounded from the Hardy space
to the corresponding Lebesgue space.

\begin{theorem}\label{thm-bdj}
Let $\epsilon\in(0,\eta]$ and $T$ belongs to $\epsilon$-Journ\'e's class with $\eta$ the same as in
Theorem \ref{thm-wave}. Then, for any $p\in(\frac{\omega}{\omega+\epsilon},1]$ with $\omega$ the same as in
\eqref{eq-doub}, $T$ can be extended to a bounded linear operator from $H^p$ to $L^p$.
\end{theorem}

To prove this theorem, we need the following technical lemma.

\begin{lemma}\label{lem-can}
Let $T^{(2)}$ be a bounded linear operator on $L^2(X_2)$. If $f\in L^2$ satisfies the following
two conditions:
\begin{enumerate}
\item[\rm(i)] $\mathop{\rm supp} f\subset B_1\times X_2$  for some ball $B_1\subset X_1$,
\item[\rm(ii)] for any $z_2\in X_2$, $\int_{X_1}f(y_1,z_2)\,d\mu_1(y_1)=0$,
\end{enumerate}
then, for almost every $x_2\in X_2$, it holds that
\begin{equation}\label{eq-tcan}
\int_{X_1}T^{(2)}f(y_1,\cdot)(x_2)\,d\mu_1(y_1)=T^{(2)}\left(\int_{X_1}f(y_1,\cdot)\,d\mu_1(y_1)\right)
(x_2)=0.
\end{equation}
\end{lemma}

\begin{proof}
Let
$$
\mathcal S(B_1,X_2):=\mathop{\rm span}\left\{f_1\otimes f_2:\ f_1\in L^2(X_1),\ \mathop{\rm supp} f_1\subset B_1,\ 
f_2\in L^2(X_2)\right\}.
$$
For any $f$ as in the present lemma, we first show that there exists a sequence
$\{f_n\}_{n=1}^\infty$ in $\mathcal S(B_1,X_2)$ such that
$\lim_{n\to\infty}\|f_n-f\|_{L^2}=0$. Indeed, since $f\in L^2$, it follows from \eqref{eq-re} that
$$
f=\sum_{\genfrac{}{}{0pt}{}{k_1\in\mathbb Z}{\alpha_1\in\mathcal G_{1,k_1}}}
\sum_{\genfrac{}{}{0pt}{}{k_2\in\mathbb Z}{\alpha_2\in\mathcal G_{2,k_2}}}
\left<f,\psi_{\alpha_1,\alpha_2}^{k_1,k_2}\right>\psi_{1,\alpha_1}^{k_1}\otimes\psi_{2,\alpha_2}^{k_2}
$$
in $L^2$. Thus, for any $n\in\mathbb N$, we may choose $g_n$ as the finite linear combination of
$\{\psi_{1,\alpha_1}^{k_1}\otimes\psi_{2,\alpha_2}^{k_2}:\ k_1,k_2\in\mathbb Z,\alpha_1\in\mathcal G_{1,k_1},
\alpha_2\in\mathcal G_{2,k_2}\}$ such that $\lim_{n\to\infty}\|f-g_n\|_{L^2}=0$. Now, for any $n\in\mathbb N$,
let $f_n:=g_n(\mathbf 1_{B_1}\otimes\mathbf 1_{X_2})$. Then, for any $n\in\mathbb N$,
$f_n\in\mathcal S(B_1,X_2)$ and, by the choice of $g_n$, we find that
$$
\|f-f_n\|_{L^2}=\left\|f-g_n\left(\mathbf 1_{B_1}\otimes\mathbf 1_{X_2}\right)\right\|_{L^2}
=\left\|(f-g_n)\left(\mathbf 1_{B_1}\otimes\mathbf 1_{X_2}\right)\right\|_{L^2}
\le\|f-g_n\|_{L^2}\to 0
$$
as $n\to\infty$. Thus, the above claim holds.

Next, we show \eqref{eq-tcan}. We first assume that $f\in\mathcal S(B_1,X_2)$. Then we have,
for any $(x_1,x_2)\in X_1\times X_2$, $f(x_1,x_2)=\sum_{k=1}^N f_{1,k}(x_1)f_{2,k}(x_2)$,
where, for any $k\in\{1,\ldots,N\}$, $f_{1,k}\in L^2(X_1)$ with $\mathop{\rm supp}f_{1,k}\subset B_1$
and $f_{2,k}\in L^2(X_2)$. By this and $T^{(2)}\in\mathcal L(L^2(X_2))$, we conclude that,
for any $x\in X_2$,
\begin{align}\label{eq-can=}
&\int_{X_1} T^{(2)}f(y_1,\cdot)(x_2)\,d\mu_1(y_1)\nonumber\\
&\quad=\sum_{k=1}^N\int_{X_1}f_{1,k}(y_1)T^{(2)}f_{2,k}(x_2)\,d\mu_1(y_1)
=\sum_{k=1}^N T^{(2)}f_{2,k}(x_2)\int_{X_1}f_{1,k}(y_1)\,d\mu_1(y_1)\nonumber\\
&\quad=\sum_{k=1}^N T^{(2)}\left(\int_{X_1}f_{1,k}(y_1)\,d\mu_1(y_1)\cdot f_{2,k}\right)(x_2)\nonumber\\
&\quad=T^{(2)}\left(\sum_{k=1}^N\int_{X_1}f_{1,k}(y_1)f_{2,k}(\cdot)\,d\mu_1(y_1)
\right)(x_2)=T^{(2)}\left(\int_{X_1}f(y_1,\cdot)\,d\mu_1(y_1)\right)(x_2).
\end{align}
This is the desired result.

Now, for any $f$ satisfying (i) and (ii) of the present lemma, by the previous argument,
there exists a sequence $\{f_n\}_{n=1}^\infty$ in $\mathcal S(B_1,X_2)$ such that
$\lim_{n\to\infty}\|f-f_n\|_{L^2}=0$. Then, on one side, when $y_1\notin B_1$,
we find that, for any $n\in\mathbb N$, $T^{(2)}f(y_1,\cdot)=0=T^{(2)}f_n(y_1,\cdot)$.
Using this, H\"older's inequality, and the assumption that $T^{(2)}\in\mathcal L(L^2(X_2))$, 
we obtain
\begin{align}\label{eq-can=1}
&\left\|\int_{X_1}T^{(2)}f(y_1,\cdot)(*)\,d\mu_1(y_1)
-\int_{X_1}T^{(2)}f_n(y_1,\cdot)(*)\,d\mu_1(y_1)\right\|_{L^2(X_2)}^2\nonumber\\
&\quad=\int_{X_2}\left|\int_{X_1}T^{(2)}
(f(y_1,\cdot)-f_n(y_1,\cdot))(x_2)\,d\mu_1(y_1)\right|^2\,d\mu_2(x_2)\nonumber\\
&\quad=\int_{X_2}\left|\int_{B_1}T^{(2)}
(f(y_1,\cdot)-f_n(y_1,\cdot))(x_2)\,d\mu_1(y_1)\right|^2\,d\mu_2(x_2)\nonumber\\
&\quad\le\mu_1(B_1)\int_{B_1}\int_{X_2}\left|T^{(2)}
(f(y,\cdot)-f_n(y,\cdot))(x_2)\right|^2\,d\mu_2(x_2)\,d\mu_1(y_1)\nonumber\\
&\quad\le\mu_1(B_1)\left\|T^{(2)}\right\|_{\mathcal L(L^2(X_2))}^2
\int_{B_1}\|f(y_1,\cdot)-f_n(y_1,\cdot)\|_{L^2(X_2)}^2\,d\mu_1(y_1)\nonumber\\
&\quad\le\mu_1(B_1)\left\|T^{(2)}\right\|_{\mathcal L(L^2(X_2))}^2
\|f-f_n\|_{L^2}^2\to 0
\end{align}
as $n\to\infty$. On the other hand, from H\"older's inequality and, for any $n\in\mathbb N$,
$\mathop{\rm supp}f,\mathop{\rm supp}f_n\subset B_1\times X_2$, we deduce that
\begin{align*}
&\left\|\int_{X_1}f_n(y_1,\cdot)\,d\mu_1(y_1)-\int_{X_1}f(y_1,\cdot)\,d\mu_1(y_1)\right\|_{L^2(X_2)}^2\\
&\quad=\int_{X_2}\left|\int_{B_1}[f_n(y_1,x_2)-f(y_1,x_2)]\,d\mu_1(y_1)\right|^2\,d\mu_2(x_2)\\
&\quad\le\mu_1(B_1)\int_{X_2}\int_{B_1}|f_n(y_1,x_2)-f(y_1,x_2)|^2\,d\mu_1(y_1)\,d\mu_2(x_2)\\
&\quad=\mu_1(B_1)\|f_n-f\|_{L^2}\to 0
\end{align*}
as $n\to\infty$. By this and the fact $T^{(2)}\in\mathcal L(L^2(X_2))$,
we further find that
\begin{equation}\label{eq-can=2}
\lim_{n\to\infty}\left\|T^{(2)}\left(\int_{X_1}f_n(y_1,\cdot)\,d\mu_1(y_1)\right)
-T^{(2)}\left(\int_{X_1}f(y_1,\cdot)\,d\mu_1(y_1)\right)\right\|_{L^2(X_2)}=0.
\end{equation}
From \eqref{eq-can=1}, \eqref{eq-can=}, $\{f_n\}_{n=1}^\infty\subset\mathcal S(B_1,X_2)$,
 and \eqref{eq-can=2}, we infer that
\begin{align*}
&\left\|\int_{X_1}T^{(2)}f(y_1,\cdot)(*)\,d\mu_1(y_1)
-T^{(2)}\left(\int_{X_1}f(y_1,\cdot)\,d\mu_1(y_1)\right)(*)\right\|_{L^2(X_2)}\\
&\quad\le\left\|\int_{X_1}T^{(2)}f(y_1,\cdot)(*)\,d\mu_1(y_1)
-\int_{X_1}T^{(2)}f_n(y_1,\cdot)(*)\,d\mu_1(y_1)\right\|_{L^2(X_2)}\\
&\quad\quad+\left\|\int_{X_1}T^{(2)}f_n(y_1,\cdot)(*)\,d\mu_1(y_1)
-T^{(2)}\left(\int_{X_1}f_n(y_1,\cdot)\,d\mu_1(y_1)\right)(*)\right\|_{L^2(X_2)}\\
&\quad\quad+\left\|T^{(2)}\left(\int_{X_1}f_n(y_1,\cdot)\,d\mu_1(y_1)\right)(*)
-T^{(2)}\left(\int_{X_1}f(y_1,\cdot)\,d\mu_1(y_1)\right)(*)\right\|_{L^2(X_2)}\\
&\quad\to 0
\end{align*}
as $n\to\infty$. Using this and (ii), we conclude that, for almost every $x_2\in X_2$,
$$
\int_{X_1}T^{(2)}f(y_1,\cdot)(x_2)\,d\mu_1(y_1)
=T^{(2)}\left(\int_{X_1}f(y_1,\cdot)\,d\mu_1(y_1)\right)(x_2)=0.
$$
This proves \eqref{eq-tcan}, which then completes the proof of Lemma \ref{lem-can}.
\end{proof}

Applying Lemma \ref{lem-can}, we now can give the proof of Theorem \ref{thm-bdj}.

\begin{proof}[Proof of Theorem \ref{thm-bdj}]
By Theorem \ref{thm-bddhl}, it suffices to show that, for any $(p,2,C_0)$-rectangle atom
$a$ supported in $[C_0B_1(Q_1)]\times [C_0B_2(Q_2)]$ with $C_0$ sufficiently large,
some $Q_1\in\mathcal D_1$, and $Q_2\in\mathcal D_2$,
\begin{equation}\label{eq-ta}
\int_{\{[\gamma C_0B_1(Q_1)]\times [\gamma C_0B_2(Q_2)]\}^\complement}|Ta(x_1,x_2)|^p
\, d\mu(x_1,x_2)\lesssim\gamma^{-\epsilon}.
\end{equation}
To achieve this, notice that
\begin{align}\label{eq-tasum}
&\int_{\{[\gamma C_0B_1(Q_1)]\times [\gamma C_0B_2(Q_2)]\}^\complement}|Ta(x_1,x_2)|^p
\, d\mu(x_1,x_2)\nonumber\\
&\quad\le\int_{[\gamma C_0B_1(Q_1)]^\complement}\int_{[2A_0C_0B_2(Q_2)]^\complement}
|Ta(x_1,x_2)|^p\,d\mu_2(x_2)\,d\mu_1(x_1)+\int_{[\gamma C_0B_1(Q_1)]^\complement}
\int_{2A_0C_0B_2(Q_2)}\cdots\nonumber\\
&\quad\quad
+\int_{[\gamma C_0B_1(Q_1)]\setminus[2A_0C_0B_1(Q_1)]}\int_{[\gamma C_0B_2(Q_2)]^\complement}\cdots
+\int_{2A_0C_0B_1(Q_1)}\int_{[\gamma C_0B_2(Q_2)]^\complement}\cdots\nonumber\\
&\quad=:{\rm I}+{\rm II}+{\rm III}+{\rm IV}.
\end{align}
It suffices to estimate $\rm I$, $\rm II$, and $\rm III$, and $\rm IV$ respectively. To this end, we
suppose that $x_1^{(0)}$ and $x_2^{(0)}$ are respectively the centers of $Q_1$ and $Q_2$ and
that $k_0$ is the largest integer
such that $(2A_0)^{k_0}\le\gamma$. Since $\gamma\ge 2A_0$, it then follows that $k_0\ge 1$.

We first estimate $\rm I$. Indeed, let $x_1\in[\gamma C_0B_1(Q_1)]^\complement$ and
$x_2\in[2A_0C_0B_2(Q_2)]^\complement$. By (ii) and (iv) of Definition \ref{def-wpczk},
the cancellation of $a$, and H\"older's inequality, we conclude that
\begin{align*}
|Ta(x_1,x_2)|&=\left|\int_{C_0B_1(Q_1)}\int_{C_0B_2(Q_2)}K(x_1,x_2,y_1,y_2)f(y_1,y_2)
\,d\mu_2(y_2)\,d\mu_1(y_1)\right|\\
&\le\int_{C_0B_1(Q_1)}\int_{C_0B_2(Q_2)}\left|K(x_1,x_2,y_1,y_2)-K\left(x_1,x_2,x_1^{(0)},y_2\right)
-K\left(x_1,x_2,y_1,x_2^{(0)}\right)\right.\\
&\quad\left.{}+K\left(x_1,x_2,x_1^{(0)},x_2^{(0)}\right)\right|a(y_1,y_2)\,d\mu_2(y_2)\,d\mu_1(x_1)\\
&\lesssim\left[\frac{\ell(Q_1)}{d_1(x_1^{(0)},x_1)}\right]^\epsilon
\left[\frac{\ell(Q_2)}{d_2(x_2^{(0)},x_2)}\right]^\epsilon
\frac{1}{V_1(x_1^{(0)},x_1)}\frac{1}{V_2(x_2^{(0)},x_2)}\|a\|_{L^1}\\
&\lesssim\left[\frac{\ell(Q_1)}{d_1(x_1^{(0)},x_1)}\right]^\epsilon
\left[\frac{\ell(Q_2)}{d_2(x_2^{(0)},x_2)}\right]^\epsilon
\frac{1}{V_1(x_1^{(0)},x_1)}\frac{1}{V_2(x_2^{(0)},x_2)}[\mu(R)]^{1-\frac 1p}.
\end{align*}
In what follows, we denote by $\ell(Q_1)$ and $\ell(Q_2)$ the edge-lengthes of $Q_1$
and $Q_2$, respectively. Using this, \eqref{eq-doub}, and the assumption that $p>\frac{\omega}{\omega+\eta}$, we obtain
\begin{align}\label{eq-czo1}
\mathrm{I}&\lesssim\mu(R)^{p-1}\int_{[\gamma C_0B_1(Q_1)]^\complement}\int_{[2A_0C_0B_2(Q_2)]^\complement}
\left[\frac{\ell(Q_1)}{d_1(x_1^{(0)},x_1)}\right]^{p\epsilon}
\left[\frac{\ell(Q_1)}{d_1(x_1^{(0)},x_1)}\right]^{p\epsilon}\nonumber\\
&\quad\times\left[\frac{1}{V_1(x_1^{(0)},x_1)}\right]^p\left[\frac{1}{V_2(x_2^{(0)},x_2)}\right]^p
\,d\mu_2(x_2)\,d\mu_1(x_1)\nonumber\\
&\lesssim\mu(R)^{p-1}\sum_{k_1=k_0}^\infty\sum_{k_2=1}^\infty
\mu_1\left((2A_0)^{k_1}C_0B_1(Q_1)\right)^{1-p}\mu_2\left((2A_0)^{k_2}C_0B_2(Q_2)\right)^{1-p}\nonumber\\
&\quad\times(2A_0)^{-k_1p\epsilon}(2A_0)^{-k_2p\epsilon}\nonumber\\
&\lesssim\sum_{k_1=k_0}^\infty\sum_{k_2=1}^\infty 2^{-k_1[p\epsilon-\omega(1-p)]}
2^{-k_2[p\epsilon-\omega(1-p)]}
\sim 2^{-k_0[p\epsilon-\omega(1-p)]}\sim\gamma^{p\epsilon-\omega(1-p)}.
\end{align}
This is the desired estimate for $\mathrm I$.

Next, we estimate $\rm II$. To this end, let $x_1\in[\gamma C_0B_1(Q_1)]^\complement$.
By Definitions \ref{def-wczo}(ii) and \ref{def-rat}, and Lemma \ref{lem-can}, we find that,
for almost every $x_2\in2A_0C_0B_2(Q_2)$,
$$
\int_{X_1}T^{(2)}\left(x_1,x_1^{(0)}\right)a(y_1,\cdot)(x_2)\,d\mu_1(y_1)=0
$$
and hence
$$
Ta(x_1,x_2)=\int_{X_1}\left[T^{(2)}(x_1,y_1)-T^{(2)}\left(x_1,x_1^{(0)}\right)\right]a(y_1,\cdot)
(x_2)\,d\mu_1(y_1)=\int_{C_0B_1(Q_1)}\ldots.
$$
Using this, H\"older's inequality, and (iii) and (iv) of Definition \ref{def-wpczk}, we conclude that
\begin{align*}
&\int_{2A_0C_0B_2(Q_2)}|Ta(x_1,x_2)|^p\,d\mu_2(x_2)\\
&\quad\lesssim\mu_2(Q_2)^{1-p}\left[\int_{2A_0C_0B_2(Q_2)}|Ta(x_1,x_2)|\,d\mu_2(x_2)\right]^p\\
&\quad\sim\mu_2(Q_2)^{1-p}\left\{\int_{B_1(Q_1)}\left|\int_{2A_0C_0B_2(Q_2)}\right.\right.\\
&\qquad\left.\times\left[T^{(2)}(x_1,y_1)-T^{(2)}\left(x_1,x_1^{(0)}\right)\right]
a(y_1,\cdot)(x_2)\Bigg|\,d\mu_2(x_2)\,d\mu_1(y_1)\right\}^p\\
&\quad\lesssim\mu_1(Q_1)^{\frac p2}\mu_2(Q_2)^{1-\frac p2}\\
&\qquad\times\left[\int_{B_1(Q_1)}\left\|T^{(2)}(x_1,y_1)-T^{(2)}\left(x_1,x_1^{(0)}\right)
\right\|_{{\rm WCZO}_\epsilon(X_2)}^2\|a(y_1,\cdot)\|_{L^2(X_2)}^2\,d\mu_1(y_1)\right]^{\frac p2}\\
&\quad\lesssim\mu_1(Q_1)^{\frac p2}\mu_2(Q_2)^{1-\frac p2}\left[\frac 1{V_1(x_1,x_1^{(0)})}\right]^p
\left[\frac{\ell(Q_1)}{d_1(x_1,x_1^{(0)})}\right]^{p\epsilon}\|a\|_{L^2}^p\\
&\quad\lesssim\mu_1(Q_1)^{p-1}\left[\frac 1{V_1(x_1,x_1^{(0)})}\right]^p
\left[\frac{\ell(Q_1)}{d_1(x_1,x_1^{(0)})}\right]^{p\epsilon}.
\end{align*}
From this, \eqref{eq-doub}, and the assumption that $p>\frac{\omega}{\omega+\epsilon}$, we further deduce that
\begin{align}\label{eq-czo2}
{\rm II}&\lesssim\mu_1(Q_1)^{p-1}\int_{[\gamma C_0B_1(Q_1)]^\complement}
\left[\frac 1{V_1(x_1,x_1^{(0)})}\right]^p
\left[\frac{\ell(Q_1)}{d_1(x_1,x_1^{(0)})}\right]^{p\epsilon}\,d\mu_1(x_1)\nonumber\\
&\lesssim\mu_1(Q_1)^{p-1}\sum_{k=k_0}^\infty\mu_1\left((2A_0)^kC_0B_1(Q_1)\right)^{1-p}
(2A_0)^{-kp\epsilon}\nonumber\\
&\lesssim\sum_{k=k_0}^\infty(2A_0)^{-k[p\epsilon-\omega(1-p)]}\sim(2A_0)^{k_0[p\epsilon-\omega(1-p)]}
\sim\gamma^{p\epsilon-\omega(1-p)}.
\end{align}
This is the desired estimate for $\rm II$.

Finally, for $\rm III$ and $\rm IV$, applying the symmetry of $T$ [see Definition \ref{def-wpczk}(iv)]
and an argument similar to that used  respectively in the estimations of $\rm I$ and $\rm II$,
we find that
\begin{equation*}
{\rm III}+{\rm IV}\lesssim\gamma^{p\epsilon-\omega(1-p)}.
\end{equation*}
Combining this with \eqref{eq-czo1}, \eqref{eq-czo2}, and \eqref{eq-tasum},
we immediately obtain \eqref{eq-ta}. This finishes the proof of Theorem \ref{thm-bdj}.
\end{proof}

\section{Boundedness of Product Calder\'on--Zygmund Operators \\
on Product Hardy Spaces}\label{s-czo}

In this section, we establish the boundedness of product Calder\'on--Zygmund operators on product Hardy spaces.
We first recall the concept of one-parameter Calder\'on--Zygmund operators on a given space of homogeneous type  $X$,
which can be referred to \cite[Definition 7.2]{whhy21}.

Let $s \in (0,1]$. The \emph{space $\dot C^s(X)$} is defined to be the set 
of all functions $f$ on $X$ such that
$$
\|f\|_{\dot C^s(X)}:=\sup_{x\neq y}\frac{|f(x)-f(y)|}{[d(x,y)]^s}<\infty.
$$
A function $\varphi$ is called an \emph{$s$-bump function} associated with the ball
$B(x,r)$, for some $x\in X$ and $r\in(0,\infty)$, if $\varphi$ is supported in $B(x,r)$ and there exists a constant $C\in(0,\infty)$ such that, for any $\epsilon\in(0,s]$,
\begin{equation}\label{eq-defbump}
\|\varphi\|_{\dot C^\epsilon(X)}\le Cr^{-\epsilon}.
\end{equation}
If $\varphi$ further satisfies \eqref{eq-defbump} with $C:=1$ therein, then $\varphi$ is called a \emph{normalized
$s$-bump function}.

The space $C^s(X)$ denotes the space of all functions $f\in\dot C^s(X)\cap L^\infty(X)$,
equipped with the norm
$$
\|\cdot\|_{C^s(X)}:=\|\cdot\|_{\dot C^s(X)}+\|\cdot\|_{L^\infty(X)}.
$$

Let $\mathring C^s_{\rm b}(X)$ be the space of all functions $f\in C^s(X)$ with bounded support and
$$
\int_X f(x)\,d\mu(x)=0,
$$
equipped with the inductive topology, and denote by $(\mathring C^s_{\rm b}(X))'$ the dual space of
$\mathring C^s_{\rm b}(X)$, equipped with the weak-$*$ topology.

\begin{definition}[{\cite{whhy21}}]
Let $(X,d,\mu)$ be a space of homogeneous type and $\epsilon\in(0,\eta)$ with $\eta$ the same as in Theorem  \ref{thm-wave}.
A linear operator $T$, which is bounded from
$\mathring C_\mathrm b^{\epsilon'}(X)$ to $(\mathring C_\mathrm b^{\epsilon'}(X))'$ for any given
$\epsilon'\in(0,\epsilon)$, is called an \emph{$\epsilon$-Calder\'on--Zygmund operator}
if there exists a constant $C\in(0,\infty)$ such that the following
conditions hold:
\begin{enumerate}
\item if $\varphi$ is a normalized $\epsilon$-bump function associated with a ball $B(x,r)$ for some $x\in X$ and
$r\in(0,\infty)$, then, for any $\epsilon'\in(0,\epsilon)$,
$\|T\varphi\|_{\dot C^{\epsilon'}(X)}\le C\|\varphi\|_{\dot C^{\epsilon'}(X)}$,
\item[(ii)] there exists a function $K:\ X\times X\setminus\{(x,x):\ x\in X\}
\to\mathbb C$ such that, for any
$\varphi, \psi\in C^{\epsilon'}_{\mathrm b}(X)$ with $\epsilon'\in(0,\epsilon)$ and disjoint supports,
$$
\langle T\varphi,\psi\rangle=\int_X\left[\int_X K(x,y)\varphi(y)\,d\mu(y)\right]\psi(x)\,d\mu(x),
$$
\item[(iii)] the function $K$ in (ii) has the following properties:
\begin{enumerate}
\item for any $x, y\in X$ with $x\neq y$,
$$
|K(x,y)|\le\frac C{V(x,y)},
$$
\item for any $x, x', y\in X$ with $x\neq y$ and $d(x,x')\le(2A_0)^{-1}d(x,y)$,
$$
|K(x,y)-K(x',y)|\le\left[\frac{d(x,x')}{d(x,y)}\right]^\epsilon\frac C{V(x,y)},
$$
\end{enumerate}
\item[(iv)] the conditions (i), (ii), and (iii)
also hold with $x$ and $y$ interchanged, that is, these
conditions also hold for $T^*$, the adjoint operator of $T$, defined by setting, for any
$\epsilon'\in(0,\infty)$ and $\varphi,\psi\in\dot C^{\epsilon'}(X)$,
$\langle T^*\varphi,\psi\rangle:=\langle T\psi,\varphi\rangle$.
\end{enumerate}
Suppose $\epsilon\in(0,\eta)$ with $\eta$ the same as in Theorem  \ref{thm-wave}.
Let $\mathrm{CZO}_\epsilon(X)$ denote
the set of all $\epsilon$-Calder\'on--Zygmund operators. For any $T\in\mathrm{CZO}_\epsilon(X)$, let
$$
\|T\|_{\mathrm{CZO}_\epsilon(X)}
:=\inf\left\{C\in(0,\infty):\ \textup{(i), (iii), and (iv) hold}\right\}.
$$
\end{definition}

Now we introduce the product case.
Let $X_1$ and $X_2$ be two spaces of homogeneous type and $s_1,s_2\in(0,1)$.
The \emph{product $(s_1,s_2)$-H\"older space $\dot C^{(s_1,s_2)}$} is defined to be the set of all
functions $f$ on $X_1\times X_2$ such that
$$
\|f\|_{\dot C^{(s_1,s_2)}}:=\sup_{\substack{x_1\neq y_1 \\ x_2\neq y_2}}
\frac{|f(x_1,x_2)-f(x_1,y_2)-f(y_1,x_2)+f(y_1,y_2)|}{[d_1(x_1,y_1)]^{s_1}[d_2(x_2,y_2)]^{s_2}}
<\infty.
$$
For any given $s\in(0,1]$, we denote $\dot C^{(s,s)}$ simply by $\dot C^s$. The \emph{space $C^s$}
is defined by setting $C^s:=\dot C^s\cap L^\infty$, equipped with the norm
$$
\|\cdot\|_{C^s}:=\|\cdot\|_{\dot C^s}+\|\cdot\|_{L^\infty}.
$$
The space $\mathring C^s_\mathrm b$ is defined to be the set
of all functions $f\in C^s$ with bounded support satisfying that, for any $x_1\in X_1$ and $x_2\in X_2$,
$$
\int_{X_1} f(y_1,x_2)\,d\mu_1(y_1)=0=\int_{X_2} f(x_1,y_2)\,d\mu_2(y_2),
$$
equipped with the inductive topology. Denote by $(\mathring C^s_\mathrm b)'$ the dual space of
$\mathring C^s_\mathrm b$, equipped with the weak-$*$ topology.

\begin{definition}[{\cite{HYY24}}]\label{def-pcz}
Let $X_1$ and $X_2$ be two spaces of homogeneous type, $\epsilon\in(0,1)$, and $\Delta$
be the same as in \eqref{eq-diag}.
A function $K:\ (X_1\times X_2)^2\setminus
\Delta\to\mathbb C$ is called a \emph{product $\epsilon$-Calder\'on--Zygmund kernel} if there exists a constant
$C\in(0,\infty)$ such that $K$ has the following properties:
\begin{enumerate}
\item[(i)] (the \emph{size condition}) for any $(x_1,x_2,y_1,y_2)\in(X_1\times X_2)^2$ with $x_1\neq y_1$ and
$x_2\neq y_2$,
$$
|K(x_1,x_2,y_1,y_2)|\le \frac{C}{V_1(x_1,y_1)V_2(x_2,y_2)},
$$
\item[(ii)] (the \emph{single H\"ormander condition}) for any $(x_1,x_2,y_1,y_2)\in(X_1\times X_2)^2$, $y_1'\in X_1$, and
$y_2'\in X_2$ such that $x_1\neq y_1$, $x_2\neq y_2$, $d_1(y_1,y_1')\le(2A_0)^{-1}d_1(x_1,y_1)$,
and $d_2(y_2,y_2')\le(2A_0)^{-1}d_2(x_2,y_2)$,
\begin{align*}
&\left|K(x_1,x_2,y_1,y_2)-K\left(x_1,x_2,y_1',y_2\right)\right|
+\left|K(x_1,x_2,y_1,y_2)-K\left(x_1,x_2,y_1,y_2'\right)\right|\\
&\quad\le\left\{\left[\frac{d_1(y_1,y_1')}{d_1(x_1,y_1)}\right]^\epsilon
+\left[\frac{d_2(y_2,y_2')}{d_2(x_2,y_2)}\right]^\epsilon\right\}
\frac{C}{V_1(x_1,y_1)V_2(x_2,y_2)},
\end{align*}
\item[(iii)] (the \emph{double H\"ormander condition}) for any $(x_1,x_2,y_1,y_2)\in(X_1\times X_2)^2$, $y_1'\in X_1$, and
$y_2'\in X_2$ such that $x_1\neq y_1$, $x_2\neq y_2$, $d_1(y_1,y_1')\le(2A_0)^{-1}d_1(x_1,y_1)$,
and $d_2(y_2,y_2')\le(2A_0)^{-1}d_2(x_2,y_2)$,
\begin{align*}
&\left|\left[K(x_1,x_2,y_1,y_2)-K\left(x_1,x_2,y_1',y_2\right)\right]-
\left[K\left(x_1,x_2,y_1,y_2'\right)-K\left(x_1,x_2,y_1',y_2'\right)\right]\right|\\
&\quad\le\left[\frac{d_1(y_1,y_1')}{d_1(x_1,y_1)}\right]^\epsilon
\left[\frac{d_2(y_2,y_2')}{d_2(x_2,y_2)}\right]^\epsilon
\frac{C}{V_1(x_1,y_1)V_2(x_2,y_2)}.
\end{align*}
\end{enumerate}
\end{definition}

Now, it is natural to introduce the concept of product Calder\'on--Zygmund operators.

\begin{definition}[{\cite{HYY24}}]
Let $\epsilon\in(0,\eta)$ with $\eta$ the same as in Theorem \ref{thm-wave},
and let $T:\ \mathring C^{\epsilon'}_\mathrm b\to(\mathring C^{\epsilon'}_\mathrm b)'$
for any given $\epsilon'\in(0,\epsilon)$. Then $T$ is called a
\emph{product $\epsilon$-Calder\'on--Zygmund operator} if there exist a product
$\epsilon$-Calder\'on--Zygmund kernel $K$ and a constant $C\in(0,\infty)$ such that
\begin{enumerate}
\item[(i)] (the \emph{weak boundedness property}) for any given
$i\in\{1,2\}$ and any normalized $\epsilon$-bump function
$\varphi_i$ associated with $B_i(x_i,r_i)\subset X_i$ for some $x_i\in X_i$ and $r_i\in(0,\infty)$,
$$
\|T(\varphi_1\otimes\varphi_2)\|_{\dot C^{\epsilon'}}\le C r_1^{-\epsilon'}r_2^{-\epsilon'},
$$
\item[(ii)] for any given $\epsilon'\in(0,\epsilon)$ and for any $i\in\{1,2\}$ and
$\varphi_i,\psi_i\in C_{\mathrm b}^{\epsilon'}(X_i)$
with disjoint supports,
\begin{align*}
\langle T(\varphi_1\otimes\varphi_2),\psi_1\otimes\psi_2\rangle
&=\int_{X_1\times X_2}\left[\int_{X_1\times X_2} K(x_1,x_2,y_1,y_2)\varphi_1(y_1)\varphi_2(y_2)\,d\mu_1(y_1)\,d\mu_2(y_2)\right]\\
&\quad\times\psi_1(x_1)\psi_2(x_2)\,d\mu_1(x_1)\,d\mu_2(x_2),
\end{align*}
\item[(iii)] for any given $x_2\in X_2$ and any given normalized $\epsilon$-bump function $\varphi_2$ on $X_2$, there exists
$T^{\varphi_2,x_2}\in\mathrm{CZO}_\epsilon(X_1)$ such that, for any $\epsilon$-bump function $\varphi_1$ on $X_1$,
$$
\langle T(\varphi_1\otimes\varphi_2),\psi_1\otimes\psi_2\rangle=\int_{X_2}
\left<T^{\varphi_2,x_2}\varphi_1,\psi_1\right>
\psi_2(x_2)\,d\mu_2(x_2);
$$
moreover, if $\mathop{\rm supp} \varphi_2\subset B(x_2^{(0)},r_2)$ for some $x_2^{(0)}\in X_2$ and $r_2\in(0,\infty)$,
then, for any $x_2,x_2'\in X_2$ and $\epsilon'\in(0,\epsilon)$,
$$
\left\|T^{\varphi_2,x_2}-T^{\varphi_2,x_2'}\right\|_{{\rm CZO}_\epsilon(X_1)}
\le C\left[\frac{d(x_2,x_2')}{r_2}\right]^{\epsilon'},
$$
\item[(iv)]  the properties (i) through (iii) also hold with $x$ and $y$
or $X_1$ and $X_2$ interchanged.
\end{enumerate}
\end{definition}

\begin{theorem}\label{thm-czhp}
Let $p\in(\frac{\omega}{\omega+\eta},1]$, $\epsilon\in(\omega(\frac{1}{p}-1),\eta)$, and $T$ be a product $\epsilon$-Calder\'on--Zygmund
operator. Further assume that $T$ is bounded  on $L^{2}$.
Then $f$ can be extended to a bounded linear operator on $H^p$. Moreover, there exists a positive constant
$C$ such that, for any $f\in H^p$, $\|Tf\|_{H^p}\le C\|f\|_{H^p}$.
\end{theorem}

To prove Theorem \ref{thm-czhp}, we need the following lemma.

\begin{lemma}\label{lem-whpp}
Let $\{D_{k_1,k_2}\}_{k_1,k_2\in\mathbb Z}$ be as in Theorem \ref{thm-ctype}, $j_0\in\mathbb Z$, $p\in(\frac\omega{\omega+\eta},1]$
with $\omega$ and $\eta$ as, respectively, in \eqref{eq-doub} and Theorem \ref{thm-wave},
and, for any $i\in\{1,2\}$, $k_i\in\mathbb Z$, and $\alpha_i\in\mathcal A_{i,k_i+j_0}$,
$x_{i,\alpha_i}^{k_i+j_0}\in Q_{i,\alpha_i}^{k_i+j_0}$ be an arbitrary point.
Then there exists a positive constant $C$, independent of $x_{1,\alpha_{1}}^{k_{1}+j_0}$ 
and $x_{2,\alpha_{2}}^{k_{2}+j_0}$, such that,
for any $f\in L^2$,
\begin{align}\label{eq-whpp}
 \|f\|_{H^p}
 \le C\left\|\left[\sum_{\genfrac{}{}{0pt}{}{k_1\in{\mathbb Z}}{\alpha_1\in\mathcal  A_{1,k_1+j_0}}}
\sum_{\genfrac{}{}{0pt}{}{k_2\in{\mathbb Z}}{\alpha_2\in\mathcal A_{2,k_2+j_0}}}
\left|D_{k_1,k_2}f\left(x_{1,\alpha_1}^{k_1+j_0},x_{2,\alpha_2}^{k_2+j_0}\right)\right|^2
\mathbf 1_{Q_{\alpha_1,\alpha_2}^{k_1+j_0,k_2+j_0}}\right]^{\frac 12}\right\|_{L^p}.
\end{align}
\end{lemma}

\begin{proof}
Let $f\in L^2$ and all the symbols be as in the proof Theorem \ref{thm-ctype}. By the boundedness of
$T_N$ on $L^2$, we find that $T_Nf\in L^2$ and
\begin{align*}
T_Nf(\cdot,*)&=\sum_{\genfrac{}{}{0pt}{}{k_1\in{\mathbb Z}}{\alpha_1\in\mathcal  A_{1,k_1+j_0}}}
\sum_{\genfrac{}{}{0pt}{}{k_2\in{\mathbb Z}}{\alpha_2\in\mathcal A_{2,k_2+j_0}}}
\mu\left(Q_{\alpha_1,\alpha_2}^{k_1+j_0,k_2+j_0}\right)
D_{k_1,k_2}^N\left(\cdot,*,x_{1,\alpha_1}^{k_1+j_0},x_{2,\alpha_2}^{k_2+j_0}\right)
D_{k_1,k_2}f\left(x_{1,\alpha_1}^{k_1+j_0},x_{2,\alpha_2}^{k_2+j_0}\right)
\end{align*}
in $L^2$. Thus, for any $k_1',k_2'\in\mathbb Z$, $\alpha_1'\in\mathcal G_{1,k_1'}$, and $\alpha_2'\in\mathcal G_{2,k_2'}$,
\begin{align*}
\left<T_Nf,\psi_{\alpha_1',\alpha_2'}^{k_1',k_2'}\right>&=\sum_{k_1,\alpha_1}\sum_{k_2,\alpha_2}
\mu\left(Q_{\alpha_1,\alpha_2}^{k_1+j_0,k_2+j_0}\right)
D_{k_1,k_2}^N\psi_{\alpha_1',\alpha_2'}^{k_1',k_2'}\left(x_{1,\alpha_1}^{k_1+j_0},
x_{2,\alpha_2}^{k_2+j_0}\right)
D_{k_1,k_2}f\left(x_{1,\alpha_1}^{k_1+j_0},x_{2,\alpha_2}^{k_2+j_0}\right).
\end{align*}
Consequently, applying an argument similar to that used in the proof of Lemma \ref{lem-slp}, we find that
\begin{equation*}
\left\|T_Nf\right\|_{H^p}
\lesssim\left\|\left[\sum_{k_1,\alpha_1}\sum_{k_2,\alpha_2}
\left|D_{k_1,k_2}f\left(x_{1,\alpha_1}^{k_1+j_0},x_{2,\alpha_2}^{k_2+j_0}\right)\right|^2
\mathbf 1_{Q_{\alpha_1,\alpha_2}^{k_1+j_0,k_2+j_0}}\right]^{\frac12}\right\|_{L^p}.
\end{equation*}
This, together with \eqref{eq-tnhp}, further implies that
$$
\|f\|_{H^p}=\left\|T_N^{-1}T_Nf\right\|_{H^p}\lesssim\left\|\left[\sum_{k_1,\alpha_1}\sum_{k_2,\alpha_2}
\left|D_{k_1,k_2}f\left(x_{1,\alpha_1}^{k_1+j_0},x_{2,\alpha_2}^{k_2+j_0}\right)\right|^2
\mathbf 1_{Q_{\alpha_1,\alpha_2}^{k_1+j_0,k_2+j_0}}\right]^{\frac12}\right\|_{L^p}.
$$
This finishes the proof of \eqref{eq-whpp} and hence Lemma \ref{lem-whpp}.
\end{proof}

\begin{proof}[Proof of Theorem \ref{thm-czhp}]
Since $L^2\cap H^p$ is dense in $H^p$, without loss of generality, we may assume that $f\in L^2\cap H^p$.
Then, by Theorem \ref{thm-ctype}, we have
\begin{align}\label{eq-5.2x}
f(\cdot, *)&=\sum_{\genfrac{}{}{0pt}{}{k_1\in{\mathbb Z}}{\alpha_1\in\mathcal G_{1,k_1}}}
\sum_{\genfrac{}{}{0pt}{}{k_2\in{\mathbb Z}}{\alpha_2\in\mathcal G_{2,k_2}}}
\mu\left(Q_{\alpha_1,\alpha_2}^{k_1+j_0,k_2+j_0}\right)
D_{k_1,k_2}^N\left(x_{1,\alpha_1}^{k_1+j_0},x_{2,\alpha_2}^{k_2+j_0},\cdot,*\right)
D_{k_1,k_2}(g)\left(x_{1,\alpha_1}^{k_1+j_0},x_{2,\alpha_2}^{k_2+j_0}\right)
\end{align}
with the series converging in $L^2\cap H^p$, where $g \in L^2 \cap H^p$ and
$\|g\|_{H^p} \sim \|f\|_{H^p}$, Consequently, from this and the boundedness of $T$ on $L^2$, we deduce that
\begin{align*}
Tf(\cdot,*)&=\sum_{k_1,\alpha_1}\sum_{k_2,\alpha_2}
\mu\left(Q_{\alpha_1,\alpha_2}^{k_1+j_0,k_2+j_0}\right)
T\left(D_{k_1,k_2}^N\left(x_{1,\alpha_1}^{k_1+j_0},x_{2,\alpha_2}^{k_2+j_0},\cdot,*\right)\right)
 D_{k_1,k_2}(g)\left(x_{1,\alpha_1}^{k_1+j_0},x_{2,\alpha_2}^{k_2+j_0}\right),
\end{align*}
and, similarly, for any $i\in\{1,2\}$, $k_i'\in\mathbb Z$,
$\alpha_i\in\mathcal A_{i,k_i}$, and $x_{i,\alpha_{i}'}^{k_{i}'}\in Q_{i,\alpha_{i}'}^{k_{i}'}$,
from \eqref{eq-5.2x} and the boundedness of $D_{k_1',k_2'}T$ on $L^2$, we infer that
\begin{align*}
&D_{k_1',k_2'}Tf\left(x_{1,\alpha_1'}^{k_1'},x_{2,\alpha_2'}^{k_2'}\right)\\
&\quad=\sum_{k_1,\alpha_1}\sum_{k_2,\alpha_2}\mu\left(Q_{\alpha_1,\alpha_2}^{k_1+j_0,k_2+j_0}\right)
D_{k_1,k_2}T\left(D_{k_1,k_2}^N\left(x_{1,\alpha_1}^{k_1+j_0},x_{2,\alpha_2}^{k_2+j_0},\cdot,*\right)\right)
\left(x_{1,\alpha_1'}^{k_1'},x_{2,\alpha_2'}^{k_2'}\right)\\
&\qquad\times D_{k_1,k_2}g\left(x_{1,\alpha_1}^{k_1+j_0},x_{2,\alpha_2}^{k_2+j_0}\right).
\end{align*}
Using the following  almost orthogonality estimate in \cite[Theorem 5.4]{HYY24} that,  for any given
$\epsilon'\in(0,\epsilon)$ and for any $k_1, k_2, k_1', k_2'\in\mathbb Z$ and
$x_{\alpha_1,\alpha_2}^{k_1,k_2}\in Q_{\alpha_1,\alpha_2}^{k_1,k_2}$,
\begin{align*}
&\left|D_{k_1,k_2}T\left(D_{k_1,k_2}^N\left(x_{1,\alpha_1}^{k_1+j_0},x_{2,
\alpha_2}^{k_2+j_0},\cdot,*\right)\right)
\left(x_{1,\alpha_1'}^{k_1'},x_{2,\alpha_2'}^{k_2'}\right)\right|\\
&\quad\lesssim\delta^{|j_1-k_1|\epsilon'}\delta^{|j_2-k_2|\epsilon'}
P_{1,\epsilon'}\left(x_{1,\alpha_1}^{k_1+j_0},x_{1,\alpha_1'}^{k_1'};\delta^{k_1'\wedge k_1}\right)
P_{2,\epsilon'}\left(x_{2,\alpha_2}^{k_2+j_0},x_{2,\alpha_2'}^{k_2'};\delta^{k_1'\wedge k_1}\right)
\end{align*}
and using Proposition \ref{prop-estmax}, we find that, for any $(x_1,x_2)\in X_1\times X_2$,
\begin{align} \label{eq-T-bound}
&\sum_{\alpha_1'\in\mathcal A_{1,k_1'}}\sum_{\alpha_2'\in\mathcal A_{2,k_2'}}
\left|D_{k_1',k_2'}Tf\left(x_{1,\alpha_1'}^{k_1'},x_{2,\alpha_2'}^{k_2'}\right)\right|^2
\mathbf 1_{Q_{\alpha_1',\alpha_2'}^{k_1',k_2'}}(x_1,x_2)    \notag\\
&\quad\lesssim\sum_{\alpha_1'\in\mathcal A_{1,k_1'}}\sum_{\alpha_2'\in\mathcal A_{2,k_2'}}
\left[\sum_{k_1,\alpha_1}\sum_{k_2,\alpha_2}\delta^{|k_1-k_1'|\epsilon'}\delta^{|k_2-k_2'|\epsilon'}
\mu\left(Q_{\alpha_1,\alpha_2}^{k_1+j_0,k_2+j_0}\right)\right.  \notag\\
&\quad\quad\left.{}\times
P_{1,\epsilon'}\left(x_{1,\alpha_1}^{k_1+j_0},x_{1,\alpha_1'}^{k_1'};\delta^{k_1'\wedge k_1}\right)P_{2,\epsilon'}\left(x_{2,\alpha_2}^{k_2+j_0},x_{2,\alpha_2'}^{k_2'};\delta^{j_1\wedge k_1}\right)
\mathbf 1_{Q_{\alpha_1',\alpha_2'}^{k_1',k_2'}}(x_1,x_2)\right]^2  \notag\\
&\quad\sim\left[\sum_{k_1,\alpha_1}\sum_{k_2,\alpha_2}\delta^{|k_1-k_1'|\epsilon'}
\delta^{|k_2-k_2'|\epsilon'}
\mu\left(Q_{\alpha_1,\alpha_2}^{k_1+j_0,k_2+j_0}\right)P_{1,\epsilon'}
\left(x_{1,\alpha_1}^{k_1+j_0},x_1;\delta^{j_1\wedge k_1}\right)
P_{2,\epsilon'}\left(x_{2,\alpha_2}^{k_2+j_0},x_2;\delta^{j_1\wedge k_1}\right)\right]^2\notag\\
&\quad\lesssim\left[\sum_{k_1,\alpha_1}\sum_{k_2,\alpha_2}\delta^{|k_1-k_1'|\epsilon'}
\delta^{|k_2-k_2'|\epsilon'}
\mu\left(Q_{\alpha_1,\alpha_2}^{k_1+j_0,k_2+j_0}\right)P_{1,\epsilon''}
\left(x_{1,\alpha_1}^{k_1+j_0},x_1;\delta^{k_1'\wedge k_1}\right)
P_{2,\epsilon'}\left(x_{2,\alpha_2}^{k_2+j_0},x_2;\delta^{k_1'\wedge k_1}\right)\right]^2\notag\\
&\quad\lesssim\left\{\sum_{k_1, k_2\in\mathbb Z}\delta^{|k_1-k_1'|\epsilon'}\delta^{\omega(1-\frac{1}{r})[k_1-(k_1\wedge k_1')]}
\delta^{|k_2-k_2'|\epsilon'}\delta^{\omega(1-\frac{1}{r})[k_2-(k_2\wedge k_2')]}\right.\notag\\
&\quad\quad\left.{}\times\left[\mathcal M_{\rm str}\left(
\sum_{\alpha_1\in\mathcal A_{1,k_1}}\sum_{\alpha_2\in\mathcal A_{2,k_2}}\left|D_{k_1,k_2}g\left(x_{1,\alpha_1}^{k_1+j_0},
x_{2,\alpha_2}^{k_2+j_0}\right)\right|^r\mathbf 1_{Q_{\alpha_1,\alpha_2}^{k_1+j_0,k_2+j_0}}\right)(x_1,x_2)\right]^{ \frac{1}{r}}\right\}^2\notag\\
&\quad\lesssim\sum_{k_1, k_2\in\mathbb Z}\delta^{|k_1-k_1'|\epsilon'}\delta^{\omega(1-\frac{1}{r})[k_1-(k_1\wedge k_1')]}
\delta^{|k_2-k_2'|\epsilon'}\delta^{\omega(1-\frac{1}{r})[k_2-(k_2\wedge k_2')]}\notag\\
&\quad\quad\times\left[\mathcal M_{\rm str}\left(
\sum_{\alpha_1\in\mathcal A_{1,k_1}}\sum_{\alpha_2\in\mathcal A_{2,k_2}}\left|D_{k_1,k_2}g\left(x_{1,\alpha_1}^{k_1+j_0},
x_{2,\alpha_2}^{k_2+j_0}\right)\right|^r\mathbf 1_{Q_{\alpha_1,\alpha_2}^{k_1+j_0,k_2+j_0}}\right)(x_1,x_2)\right]^{\frac{2}{r}},
\end{align}
where we chose $r\in(\frac\omega{\omega+\epsilon'},1]$.
Since $\epsilon\in(\omega(\frac1p-1),\eta)$, it follows that we can choose $\epsilon'$ and
$r$ such that $r\in(\frac{\omega}{\omega+\epsilon'},p)$.
By Lemma \ref{lem-whpp}, \eqref{eq-T-bound}, Proposition  \ref{prop-fs}, 
Lemma \ref{lem-slp},  and Theorem \ref{thm-ctype},
we further obtain
\begin{align*}
 \left\|Tf\right\|_{H^p}
&\lesssim\left\|\left[\sum_{\genfrac{}{}{0pt}{}{k_1'\in{\mathbb Z}}{\alpha_1'\in\mathcal  A_{1,k_1'+j_0}}}
\sum_{\genfrac{}{}{0pt}{}{k_2'\in{\mathbb Z}}{\alpha_2'\in\mathcal A_{2,k_2'+'j_0}}}
\left|D_{k_1',k_2'}Tf\left(x_{1,\alpha_1'}^{k_1'},x_{2,\alpha_2'}^{k_2'}\right)\right|^2
\mathbf 1_{Q_{\alpha_1', \alpha_2'}^{k_1'+j_0, k_2'+j_0}}\right]^{\frac 12}\right\|_{L^p}\\
&\lesssim\left\|\left\{\sum_{k_1,k_2\in\mathbb Z}\left[\mathcal M_{\rm str}\left(
\sum_{\alpha_1\in\mathcal A_{1,k_1}}\sum_{\alpha_2\in\mathcal A_{2,k_2}}
\left|D_{k_1,k_2}g\left(x_{1,\alpha_1}^{k_1+j_0},x_{2,\alpha_2}^{k_2+j_0}\right)\right|^r
\mathbf 1_{Q_{\alpha_1,\alpha_2}^{k_1+j_0,k_2+j_0}}\right)\right]^{\frac2r}\right\}^{\frac12}\right\|_{L^p}\\
&\sim\left\|\left\{\sum_{k_1, k_2\in\mathbb Z}\left[\mathcal M_{\rm str}\left(
\sum_{\alpha_1\in\mathcal A_{1,k_1}}\sum_{\alpha_2\in\mathcal A_{2,k_2}}\left|D_{k_1,k_2}g
\left(x_{1,\alpha_1}^{k_1+j_0},x_{2,\alpha_2}^{k_2+j_0}\right)\right|^r
\mathbf 1_{Q_{\alpha_1,\alpha_2}^{k_1+j_0,k_2+j_0}}\right)
\right]^{\frac{2}{r}}\right\}^{\frac{r}{2}}\right\|_{L^{\frac{p}{r}}}^{\frac{1}{r}}\\
&\lesssim\left\|\left\{\sum_{k_1,k_2\in\mathbb Z}\left(\sum_{\alpha_1\in\mathcal A_{1,k_1}}\sum_{\alpha_2\in\mathcal A_{2,k_2}}
\left|D_{k_1,k_2}g\left(x_{1,\alpha_1}^{k_1+j_0},x_{2,\alpha_2}^{k_2+j_0}\right)\right|^r
\mathbf 1_{Q_{\alpha_1,\alpha_2}^{k_1+j_0,k_2+j_0}}\right)^{\frac{2}{r}}\right\}^{\frac{r}{2}}
\right\|_{L^{\frac{p}{r}}}^{\frac{1}{r}}\\
&\lesssim\left\|\left\{\sum_{k_1,\alpha_1}\sum_{k_2,\alpha_2}
\left|D_{k_1,k_2}g\left(x_{1,\alpha_1}^{k_1+j_0},x_{2,\alpha_2}^{k_2+j_0}\right)\right|^2
\mathbf 1_{Q_{\alpha_1,\alpha_2}^{k_1+j_0,k_2+j_0}}\right\}^{\frac{1}{2}}\right\|_{L^{p}}
\lesssim\|g\|_{H^p}\sim\|f\|_{H^p}.
\end{align*}
This finishes the proof of Theorem \ref{thm-czhp}.
\end{proof}

\begin{remark}
In \cite{hlpw21}, Han et al. introduce a new and interesting kind
of atoms on product Hardy spaces and characterize these Hardy
spaces via these new atoms. However, these atoms may {\emph not} be useful for the boundedness of product Calder\'on--Zygmund
operators on product Hardy spaces. This is because a product Calder\'on--Zygmund operator should operate on a function with bounded
support, while an atom in \cite{hlpw21} may not have bounded support. Thus, it is \emph{unknown} whether an atom in
\cite{hlpw21} can be acted on by a product Calder\'on--Zygmund operator.
\end{remark}

 \bigskip \bigskip

\noindent Ziyi He

\medskip

\noindent Key Laboratory of Mathematics and Information Networks (Ministry of Education of China),
 School of Mathematical Sciences, Beijing University of Posts and
Telecommunications, Beijing 100876, The People's Republic of China

\smallskip

\noindent{{\it E-mail:}} \texttt{ziyihe@bupt.edu.cn}

\bigskip

\noindent Dachun Yang

\medskip

\noindent Laboratory of Mathematics and Complex Systems (Ministry of Education of China),
School of Mathematical Sciences, Beijing Normal University, Beijing 100875,
The People's Republic of China

\smallskip

\noindent{\it E-mail:} \texttt{dcyang@bnu.edu.cn}

\bigskip

\noindent Taotao Zheng

\medskip

\noindent
Department of Mathematics,
Zhejiang University of Science and Technology,
Hangzhou 310023, The People's Republic of China

\smallskip

\noindent{{\it E-mail:}}
\texttt{zhengtao@zust.edu.cn}


\begin{thebibliography}{99}

\bibitem{agh20}
R. Alvarado, P. G\'orka and P. Haj{\l}asz, Sobolev embedding for $M^{1,p}$ spaces is equivalent to a lower bound of the measure,
J. Funct. Anal. 279 (2020), Paper No. 108628, 39 pp.

\vspace{-0.3cm}

\bibitem{am15}
R. Alvarado and M. Mitrea, Hardy spaces on Ahlfors-Regular Quasi Metric Spaces. A Sharp Theory, Lecture Notes in Mathematics 2142,
Springer, Cham, 2015.

\vspace{-0.3cm}

\bibitem{awyy23}
R. Alvarado, F. Wang, D. Yang and W. Yuan,
Pointwise characterization of Besov and Triebel--Lizorkin spaces on spaces of homogeneous type,
Studia Math.  268 (2023), 121--166.

\vspace{-0.3cm}

\bibitem{ayy22}
R. Alvarado, D. Yang and W. Yuan,
A measure characterization of embedding and extension domains for Sobolev,
Triebel--Lizorkin, and Besov spaces on spaces of homogeneous type,
J. Funct. Anal. 283 (2022), Paper No. 109687, 71 pp.

\vspace{-0.3cm}

\bibitem{ayy24}
R. Alvarado, D. Yang and W. Yuan,
Optimal embeddings for Triebel--Lizorkin and Besov spaces on quasi-metric measure spaces,
Math. Z. 307 (2024), Paper No. 50, 59 pp.

\vspace{-0.3cm}

\bibitem{ah13} P. Auscher and T. Hyt\"{o}nen, Orthonormal bases of regular wavelets in spaces of homogeneous
type, Appl. Comput. Harmon. Anal. 34 (2013), 266--296.

\vspace{-0.3cm}

\bibitem{ah15}
P. Auscher and T. Hyt\"onen,
Addendum to Orthonormal bases of regular wavelets in spaces of
homogeneous type [Appl. Comput. Harmon. Anal. 34(2) (2013) 266--296],
Appl. Comput. Harmon. Anal. 39 (2015), 568--569.

\vspace{-0.3cm}

\bibitem{bdk20} T. A. Bui, X. T. Duong and L. D. Ky, Hardy spaces associated 
to critical functions and applications to $T1$ theorems,
J. Fourier Anal. Appl. 26 (2020),  1--67.

\vspace{-0.3cm}

\bibitem{bdl17} T. A. Bui, X. T. Duong and J. Li, VMO spaces associated with 
operators with Gaussian upper bounds on product domains,
J. Geom. Anal. 27 (2017), 1065--1085.

\vspace{-0.3cm}

\bibitem{bdl18} T. A. Bui, X. T. Duong and F. K. Ly, Maximal function characterizations for new local Hardy-type spaces on spaces of
homogeneous type, Trans. Amer. Math. Soc. 370 (2018), 7229--7292.

\vspace{-0.3cm}

\bibitem{bdl20} T. A. Bui, X. T. Duong and F. K. Ly, Maximal function
characterizations for Hardy spaces on spaces of homogeneous type
with finite measure and applications, J. Funct. Anal. 278 (2020), Paper No. 108423, 55 pp.

\vspace{-0.3cm}

\bibitem{cgr22} J. Cao and A. Grigor'yan, Heat kernels and Besov spaces on metric
measure spaces, J. Anal. Math. 148 (2022), 637--680.

\vspace{-0.3cm}

\bibitem{cgl21} J. Cao, A. Grigor'yan and L. Liu, Hardy's inequality
and Green function on metric measure spaces, J. Funct. Anal. 281 (2021),
Paper No. 109020, 78 pp.

\vspace{-0.3cm}

\bibitem{car74}
L. Carleson, A counterexample for measures bounded on $H^p$ for the bidisc. Mittag-Leffler Report No. 7, 1974.

\vspace{-0.3cm}

\bibitem{CF80}
 S. Y. A. Chang and R. Fefferman, A continuous version of duality of $H^1$ with BMO on the bidisc, Ann. of Math. (2) 112 (1980), 179--201.

\vspace{-0.3cm}

\bibitem{CF82}
 S. Y. A. Chang and R. Fefferman,  The Calder\'on--Zygmund decomposition on product domains, Amer.
J. Math. 104 (1982), no. 3, 455--468.

\vspace{-0.3cm}

\bibitem{cyz10} D.-C. Chang, D. Yang and Y. Zhou, Boundedness of sublinear
operators on product Hardy spaces and its application,
J. Math. Soc. Japan 62 (2010), 321--353.

\vspace{-0.3cm}

\bibitem{clow24} D.-C. Chang, J. Li, A. Ottazzi
and L. Wu, Taylor polynomial and Hardy spaces on the Shilov
boundary of product domains in ${\mathbb C}^{2n}$, Ann. Math.
Sci. Appl. 9 (2024), 643--669.

\vspace{-0.3cm}

\bibitem{cl19} D.-C. Chang and S. Li, On the boundedness of multipliers,
commutators and the second derivatives of Green's operators on
$H^1$ and BMO, Ann. Scuola Norm. Sup. Pisa Cl. Sci. (4) 28 (1999), 341--356.

\vspace{-0.3cm}

\bibitem{cyy24} D.-C. Chang, S. Yang and Z. Yang, A two-weight boundedness
criterion on spaces of homogeneous type with its application to some
elliptic boundary value problems, Math. Methods Appl. Sci. 47 (2024),
12989--13006.

\vspace{-0.3cm}

\bibitem{cg22}
G. Cleanthous and A. G. Georgiadis,
Product $(\alpha_1,\alpha_2)$-modulation spaces,
Sci. China Math.
65 (2022), 1599--1640.

\vspace{-0.3cm}

\bibitem{c74} R. R. Coifman, A real variable characterization of $H^{p}$,
Studia Math., 51 (1974), 269--274.

\vspace{-0.3cm}
\bibitem{cw71} R. R. Coifman and G. Weiss, Analyse Harmonique Non-Commutative sur Certains Espaces Homog\`enes.
(French) \'Etude de Certaines int\'egrales singuli\`eres, Lecture Notes in Mathematics 242. Springer-Verlag,
Berlin--New York, 1971.

\vspace{-0.3cm}

\bibitem{cw77} R. R. Coifman and G. Weiss, Extensions of Hardy spaces
and their use in analysis, Bull. Amer.
Math. Soc. 83 (1977), 569--645.

\vspace{-0.3cm}

\bibitem{dk21} N. A. Dao and S. G. Krantz, Lorentz boundedness
and compactness characterization of integral commutators on
spaces of homogeneous type, Nonlinear Anal. 203 (2021),
Paper No. 112162, 22 pp.

\vspace{-0.3cm}

\bibitem{FS72} C. Fefferman and E. M. Stein, $H^p$ spaces of several variables, Acta Math. 129 (1972),
137--193.

\vspace{-0.3cm}

\bibitem{F87}
R. Fefferman, Harmonic analysis on product spaces, Ann. of Math. (2) 126 (1987), 109--130.

\vspace{-0.3cm}

\bibitem{fmy20} X. Fu, T. Ma and D. Yang, Real-variable characterizations of Musielak--Orlicz
Hardy spaces on spaces of homogeneous type, Ann. Acad. Sci. Fenn. Math. 45 (2020), 343--410.


\vspace{-0.3cm}

\bibitem{gkp21} 		
A. G. Georgiadis, G. Kyriazis and P. Petrushev,
Product Besov and Triebel--Lizorkin spaces with application to nonlinear approximation,
Constr. Approx. 53 (2021), 39--83.

\vspace{-0.3cm}

\bibitem{gkp-2312}
A.G. Georgiadis, G. Kyriazis and P. Petrushev, Spaces of distributions on product metric spaces associated with operators,
Dissertationes Math. 603 (2025), 1--114.

\vspace{-0.3cm}

\bibitem{gly09} L. Grafakos, L. Liu and D. Yang, Vector-valued singular integrals and maximal functions on
spaces of homogeneous type, Math. Scand. 104 (2009), 296--310.

\vspace{-0.3cm}
\bibitem{gs79} R. Gundy and E. M. Stein, $H^p$ theory for the polydisc,  Proc. Nat. Acad. Sci.  76 (1979), 1026--1029.

\vspace{-0.3cm}

\bibitem{hhl16} Ya. Han, Yo. Han and J. Li, Criterion of the boundedness of singular integrals on
spaces of homogeneous type, J. Funct. Anal. 271 (2016), 3423--3464.

\vspace{-0.3cm}

\bibitem{hhl17} Ya. Han, Yo. Han and J. Li, Geometry and Hardy spaces on spaces of homogeneous
type in the sense of Coifman and Weiss, Sci. China Math. 60  (2017),  2199--2218.

\vspace{-0.3cm}

\bibitem{hll10} Y. Han,  M.-Y. Lee, C.-C. Lin and Y.-C. Lin,
Calder\'on--Zygmund operators on product Hardy spaces,
J. Funct. Anal. 258 (2010), 2834--2861.

\vspace{-0.3cm}

\bibitem{hll16-Pisa} Y. Han, J. Li and C.-C. Lin, Criterion of the $L^2$ boundedness and sharp endpoint estimates
for singular integral operators on product spaces of homogeneous type, Ann. Sc. Norm. Super. Pisa Cl. Sci.
(5) 16 (2016), 845--907.


\vspace{-0.3cm}

\bibitem{hll13-Trans}
Y. Han, J. Li and  G. Lu, Multiparameter Hardy space theory on Carnot--Caratheodory
spaces and product spaces of homogeneous type, Trans. Amer. Math. Soc. 365 (2013), 319--360.

\vspace{-0.3cm}

\bibitem{hlpw21} Y. Han, J. Li, M. C. Pereyra and L. A. Ward, Atomic decomposition of product Hardy spaces via
wavelet bases on spaces of homogeneous type, New York J. Math. 27 (2021), 1173--1239.

\vspace{-0.3cm}

\bibitem{hlw18} Y. Han, J. Li and L. A. Ward, Hardy space theory on spaces of homogeneous type via
orthonormal wavelet bases, Appl. Comput. Harmon. Anal. 45 (2018), 120--169.

\vspace{-0.3cm}

\bibitem{HMY06} Y. Han, D. M\"{u}ller and D. Yang, Littlewood--Paley characterizations for Hardy spaces on
spaces of homogeneous type, Math. Nachr. 279 (2006), 1505--1537.

\vspace{-0.3cm}

\bibitem{HMY08} Y. Han, D. M\"{u}ller and D. Yang, A theory of Besov and
Triebel--Lizorkin spaces on metric
measure spaces modeled on Carnot--Carath\'eodory spaces,
Abstr. Appl. Anal. 2008, Art. ID 893409,
1--250.

\vspace{-0.3cm}

\bibitem{hy05} Y. Han and D. Yang, $H^p$ boundedness of
Calder\'on--Zygmund operators on product spaces,
Math. Z. 249 (2005), 869–-881.

\vspace{-0.3cm}

\bibitem{hhllyy19}
Z. He, Y. Han, J. Li, L. Liu, D. Yang and W. Yuan,
A complete real-variable theory of Hardy
spaces on spaces of homogeneous type,
J. Fourier Anal. Appl. 25 (2019), 2197--2267.

\vspace{-0.3cm}

\bibitem{hlyy19}
Z. He, L. Liu, D. Yang and W. Yuan,
New Calder\'{o}n reproducing formulae with exponential
decay on spaces of homogeneous type,
Sci. China Math. 62 (2019), 283--350.

\vspace{-0.3cm}

\bibitem{hwyy21}
Z. He, F. Wang, D. Yang and W. Yuan,
Wavelet characterization of Besov and Triebel--Lizorkin
spaces on spaces of homogeneous type and its applications,
Appl. Comput. Harmon. Anal. 54 (2021), 176--226.

\vspace{-0.3cm}

\bibitem{HYY24}
Z. He, X. Yan and D. Yang,
Calder\'on reproducing formulae on product spaces of homogeneous type and
their applications, Math. Nachr. 298 (2025), 1839--1921.

\vspace{-0.3cm}

\bibitem{hyy21}
Z. He, D. Yang and W. Yuan,
Real-variable characterizations of local Hardy spaces on spaces
of homogeneous type,
Math. Nachr. 294 (2021), 900--955.

\vspace{-0.3cm}

\bibitem{ho07} K.-P. Ho,
Annihilator, completeness and convergence of wavelet system,
Nagoya Math. J.  188 (2007), 59--105.

\vspace{-0.3cm}

\bibitem{ho08}
K.-P. Ho,  Remarks on Littlewood--Paley analysis, Canad. J. Math.
60 (2008), 1283--1305.

\vspace{-0.3cm}

\bibitem{ho17}
K.-P. Ho,  Strong maximal operator and singular integral operators in weighted Morrey spaces on product domains,
Math. Nachr. 290 (2017), 2629--2640.

\vspace{-0.3cm}

\bibitem{hk12} T. Hyt\"{o}nen and A. Kairema, Systems of dyadic cubes in a doubling metric space,
Colloq. Math. 126 (2012), 1--33.

\vspace{-0.3cm}

\bibitem{k19}
N. Karak,
Measure density and embeddings of Haj{\l}asz--Besov and Haj{\l}asz--Triebel--Lizorkin spaces,
J. Math. Anal. Appl. 475 (2019), 966--984.

\vspace{-0.3cm}

\bibitem{k20}
N. Karak,
Lower bound of measure and embeddings of Sobolev, Besov and Triebel--Lizorkin spaces,
Math. Nachr. 293 (2020), 120--128.

\vspace{-0.3cm}

\bibitem{kd23} N. Karak and D. Mondal, Besov and Triebel--Lizorkin
capacity in metric spaces, Math. Slovaca 73 (2023), 937--948.

\vspace{-0.3cm}

\bibitem{kd24} N. Karak and D. Mondal,  Capacity in Besov and Triebel--Lizorkin
spaces with generalized smoothness, Georgian Math. J. 31 (2024), 973--985.

\vspace{-0.3cm}

\bibitem{k24} S. G. Krantz, A new theory of atomic $H^p$ spaces with
applications to smoothness of functions, Expo. Math. 42 (2024),
Paper No. 125532, 18 pp.

\vspace{-0.3cm}

\bibitem{kl01-1} S. G. Krantz and S.-Y. Li, Boundedness and compactness
of integral operators on spaces of homogeneous type and applications. I,
J. Math. Anal. Appl. 258 (2001), 629--641.

\vspace{-0.3cm}

\bibitem{kl01-2} S. G. Krantz and S.-Y. Li, Boundedness and compactness
of integral operators on spaces of homogeneous type and applications. II,
J. Math. Anal. Appl. 258 (2001), 642--657.

\vspace{-0.3cm}

\bibitem{l78} R. H. Latter, A characterization of $H^{p}({\mathbb R}^{n})$ in terms of atoms.
Studia Math., 62 (1978), 93--101.

\vspace{-0.3cm}

\bibitem{lmv20-2}
K. Li, H. Martikainen and E. Vuorinen, Bilinear Calder\'on--Zygmund theory on product spaces,
J. Math. Pures Appl. (9)  138 (2020), 356--412.

\vspace{-0.3cm}

\bibitem{lmv20}
K. Li, H. Martikainen and E. Vuorinen,
Bloom type upper bounds in the product BMO setting, J. Geom. Anal. 30 (2020), 3181--3203.

\vspace{-0.3cm}

\bibitem{lcfy17} L. Liu,  D.-C. Chang, X. Fu and D. Yang,
Endpoint boundedness of commutators on spaces of homogeneous
type, Appl. Anal. 96 (2017), 2408–-2433.

\vspace{-0.3cm}

\bibitem{lcfy18} L. Liu,  D.-C. Chang, X. Fu and D. Yang,
Endpoint estimates of linear commutators on Hardy spaces
over spaces of homogeneous type,
Math. Methods Appl. Sci. 41 (2018), 5951--5984.

\vspace{-0.3cm}

\bibitem{MS79b}
R. A. Mac\'{i}as and C. Segovia, A decomposition into atoms of distributions on spaces of
homogeneous type, Adv. in Math. 33 (1979), 271--309.

\vspace{-0.3cm}

\bibitem{MS79} R. A. Mac\'{i}as and C. Segovia, Lipschitz functions on spaces of homogeneous type,
Adv. in Math. 33 (1979), 257--270.

\vspace{-0.3cm}

\bibitem{NS04} A. Nagel and E. M. Stein, On the product theory of singular integrals, Rev. Mat.
Iberoamericana 20 (2004), 531--561.

\vspace{-0.3cm}

\bibitem{n06}
E. Nakai, The Campanato, Morrey and H\"older spaces on spaces of homogeneous type,
Studia Math. 176 (2006), 1--19.

\vspace{-0.3cm}

\bibitem{ny97} E. Nakai and K. Yabuta, Pointwise multipliers for functions of weighted bounded mean
oscillation on spaces of homogeneous type, Math. Japon. 46 (1997), 15--28.

\vspace{-0.3cm}

\bibitem{SHYY17} Y. Sawano, K.-P. Ho, D. Yang, and S. Yang, Hardy spaces for ball quasi-Banach
function spaces, Dissertationes Math. (Rozprawy Mat.) 525 (2017), 1--102.

\vspace{-0.3cm}

\bibitem{sw60} E. M. Stein and G. Weiss, On the theory of harmonic functions of several variables. I. The
theory of $H^p$-spaces, Acta Math. 103 (1960), 25--62.

\vspace{-0.3cm}

\bibitem{whhy21}
F. Wang, Y. Han,  Z. He and D. Yang, Besov and Triebel--Lizorkin spaces on spaces of homogeneous type
with applications to boundedness of Calder\'{o}n--Zygmund operators, Dissertationes Math.
565 (2021), 1--113.

\vspace{-0.3cm}

\bibitem{yan24} X. Yan, Littlewood-Paley $g_\lambda^*$-function characterizations of
Musielak--Orlicz Hardy spaces on spaces of homogeneous type, Probl. Anal. Issues Anal. 13(31)
(2024), 100--123.

\vspace{-0.3cm}

\bibitem{yhyy22} X. Yan, Z. He, D. Yang and W. Yuan, Hardy spaces associated with ball
quasi-Banach function spaces on spaces of homogeneous type: Littlewood--Paley characterizations
with applications to boundedness of Calder\'on--Zygmund operators, Acta Math. Sin. (Engl. Ser.) 38 (2022),
1133--1184.

\vspace{-0.3cm}

\bibitem{yhyy23} X. Yan, Z. He, D. Yang and W. Yuan, Hardy spaces associated with ball
quasi-Banach function spaces on spaces of homogeneous type:
characterizations of maximal functions, decompositions, and dual spaces, Math. Nachr. 296 (2023),
3056--3116.

\vspace{-0.3cm}
\bibitem{ylk17} D. Yang, Y. Liang and L. D. Ky, Real-Variable Theory of Musielak--Orlicz Hardy Spaces,
Lecture Notes in Mathematics 2182, Springer, Cham, 2017.

\vspace{-0.3cm}

\bibitem{yz08} D. Yang and Y. Zhou, Boundedness of sublinear operators in Hardy spaces on RD-spaces via atoms,
J. Math. Anal. Appl. 339 (2008), 622--635.

\vspace{-0.3cm}

\bibitem{YZ10} D. Yang and Y. Zhou, Radial maximal function characterizations of Hardy spaces on
RD-spaces and their applications, Math. Ann. 346 (2010), 307--333.

\vspace{-0.3cm}

\bibitem{YZ11} D. Yang and Y. Zhou, New properties of Besov and Triebel--Lizorkin spaces on RD-spaces.
Manuscripta Math. 134 (2011), 59--90.

\vspace{-0.3cm}

\bibitem{zcy17} J. Zhang,  D.-C. Chang and D. Yang, Characterizations
of Sobolov spaces associated to operators satisfying off-diagonal
estimates on balls, Math. Methods Appl. Sci. 40 (2017), 2907--2929.

\vspace{-0.3cm}

\bibitem{ZXT24} T. Zheng, Y. Xiao and X. Tao, Weighted estimates for product singular integral
operators in Journ\'es class on RD-spaces,
Forum Math.  37 (2025), 593--627.

\vspace{-0.3cm}

\bibitem{zhy20} X. Zhou, Z. He and D. Yang, Real-variable characterizations
of Hardy--Lorentz spaces on spaces of homogeneous type with applications to real
interpolation and boundedness of Calder\'on--Zygmund operators,
Anal. Geom. Metr. Spaces 8 (2020), 182--260.

\vspace{-0.3cm}

\bibitem{ZSY16} C. Zhuo, Y. Sawano and D. Yang, Hardy spaces with variable exponents on RD-spaces and
applications. Dissertationes Math. (Rozprawy Mat.) 520 (2016), 1--74.

\end{thebibliography}
\end{document}